\documentclass[a4paper,11pt,twoside]{amsart}
\usepackage[english]{babel}
\usepackage[utf8]{inputenc}

\usepackage[a4paper,inner=3cm,outer=3cm,top=4cm,bottom=4cm,pdftex]{geometry}
\usepackage{fancyhdr}
\usepackage{color}
\usepackage{bold-extra}
\usepackage{ mathrsfs }

\usepackage{graphics}
\usepackage{aliascnt}
\usepackage[hidelinks]{hyperref}

\usepackage{amsmath}
\usepackage{amsfonts}
\usepackage{amssymb}
\usepackage{amsthm}
\usepackage{comment}
\usepackage{mathtools}
\usepackage{enumerate}
\usepackage{enumitem} 

\newtheorem{theorem}{Theorem}[section]
\newtheorem{corollary}[theorem]{Corollary}

\newtheorem{lemma}[theorem]{Lemma}
\newtheorem{proposition}[theorem]{Proposition}
\theoremstyle{definition}
\newtheorem{definition}[theorem]{Definition}
\newtheorem{remark}[theorem]{Remark}

\numberwithin{equation}{section}

\newcommand\eps{\varepsilon}

\newcommand\E{\mathbb{E}}
\newcommand\R{\mathbb{R}}
\newcommand\Z{\mathbb{Z}}
\newcommand\N{\mathbb{N}}
\newcommand\C{\mathbb{C}}

\newcommand\n{\mathbf{n}}
\newcommand\Poly{{\operatorname{Poly}}}
\newcommand\TV{{\operatorname{TV}}}
\newcommand\Lip{{\operatorname{Lip}}}

\let\oldpmod\pmod
\renewcommand{\pmod}[1]{\hspace{-0.1cm}\oldpmod {#1}}

\begin{document}

\title[Higher uniformity I]{Higher uniformity of arithmetic functions in short intervals I.  All intervals}

\author[K. Matom\"aki]{Kaisa Matom\"aki}
\address{Department of Mathematics and Statistics \\
University of Turku, 20014 Turku\\
Finland}
\email{ksmato@utu.fi}

\author[X. Shao]{Xuancheng Shao}
\address{Department of Mathematics, University of Kentucky\\
715 Patterson Office Tower\\
Lexington, KY 40506\\
USA}
\email{xuancheng.shao@uky.edu}

\author[T. Tao]{Terence Tao}
\address{Department of Mathematics, UCLA\\
405 Hilgard Ave\\
Los Angeles CA 90095\\
USA}
\email{tao@math.ucla.edu}

\author[J. Ter\"av\"ainen]{Joni Ter\"av\"ainen}
\address{Department of Mathematics and Statistics\\ University of Turku, 20014 Turku\\
Finland}
\email{joni.p.teravainen@gmail.com}

\begin{abstract} We study higher uniformity properties of the M\"obius function $\mu$, the von Mangoldt function $\Lambda$, and the divisor functions $d_k$ on short intervals $(X,X+H]$ with $X^{\theta+\eps} \leq H \leq X^{1-\eps}$ for a fixed constant $0 \leq \theta < 1$ and any $\eps>0$. 

More precisely, letting $\Lambda^\sharp$ and $d_k^\sharp$ be suitable approximants of $\Lambda$ and $d_k$ and $\mu^\sharp = 0$, we show for instance that, for any nilsequence $F(g(n)\Gamma)$, we have 
\[
\sum_{X < n \leq X+H} (f(n)-f^\sharp(n)) F(g(n) \Gamma) \ll H \log^{-A} X
\]
when $\theta = 5/8$ and $f \in \{\Lambda, \mu, d_k\}$ or $\theta = 1/3$ and $f = d_2$.

As a consequence, we show that the short interval Gowers  norms $\|f-f^\sharp\|_{U^s(X,X+H]}$ are also asymptotically small for any fixed $s$ for these choices of $f,\theta$.  As applications, we prove an asymptotic formula for the number of solutions to linear equations in primes in short intervals, and show that multiple ergodic averages along primes in short intervals converge in $L^2$. 

Our innovations include the use of multi-parameter nilsequence equidistribution theorems to control type $II$ sums, and an elementary decomposition of the neighborhood of a hyperbola into arithmetic progressions to control type $I_2$ sums.
\end{abstract}

\maketitle

\tableofcontents

\section{Introduction}
\label{sec:intro}
In this paper we shall study correlations of arithmetic functions $f \colon \mathbb{N} \to \mathbb{C}$ with arbitrary nilsequences $n \mapsto F(g(n) \Gamma)$ in short intervals. For simplicity, we will restrict attention  to the following model examples of functions $f$:
\begin{itemize}
\item The \emph{M\"obius function} $\mu(n)$, defined to equal $(-1)^j$ when $n$ is the product of $j$ distinct primes, and $0$ otherwise.
\item The \emph{von Mangoldt function} $\Lambda(n)$, defined to equal $\log p$ when $n$ is a power $p^j$ of a prime $p$ for some $j \geq 1$, and $0$ otherwise.
\item The \emph{$k^{\mathrm{th}}$ divisor function} $d_k(n)$, defined to equal the number of representations of $n$ as the product $n=n_1\dotsm n_k$ of $k$ natural numbers, where $k \geq 2$ is fixed.  (In particular, all implied constants in our asymptotic notation are understood to depend on $k$.)
\end{itemize}
By a ``nilsequence'', we mean a function of the form $n \mapsto F(g(n)\Gamma)$, where $G/\Gamma$ is a filtered nilmanifold and $F \colon G/\Gamma \to \C$ is a Lipschitz function.  The precise definitions of these terms will be given in Section~\ref{nilmanifold-sec}, but a simple example of a nilsequence to keep in mind for now is $F(g(n) \Gamma) = e(\alpha n^d)$ for some real number $\alpha$, some natural number $d \geq 0$, and with $e(\theta) \coloneqq e^{2\pi i\theta}$.

When $f$ is non-negative and $F(g(n) \Gamma)$ is ``major arc'' in some sense (e.g., if $F(g(n)\Gamma) = e(\alpha n^s)$ with $\alpha$ very close to a rational $a/q$ with small denominator $q$), there is actually correlation between $f$ and $F(g(n) \Gamma)$, but we shall deal with this by first subtracting off a suitable approximation $f^\sharp$ from $f$. In the case of the M\"obius function $\mu$, we may set $\mu^\sharp = 0$.  On the other hand, the functions $\Lambda, d_k$ are non-negative and one therefore needs to construct non-trivial approximants $\Lambda^\sharp, d_k^\sharp$ to such functions before one can expect to obtain discorrelation; we shall choose
\begin{equation}\label{lambdar-def}
 \Lambda^\sharp(n) \coloneqq \frac{P(R)}{\varphi(P(R))}1_{(n,P(R))=1}, \quad \text{where}\quad  P(w)\coloneqq\prod_{p<w}p,\quad R \coloneqq \exp( (\log X)^{1/10} )
\end{equation}
and
\begin{equation}\label{dks-def}
 d_k^\sharp(n) \coloneqq \sum_{\substack{m \leq R_k^{2k-2}\\ m|n}} P_m(\log n), \quad \text{where } R_k \coloneqq X^{\eta}\text{ and }\eta = \tfrac{1}{10k}
\end{equation}
and the polynomials $P_m(t)$ (which have degree $k-1$) are given by the formula
\begin{equation}
\label{eq:Pm(t)-def}
P_m(t) \coloneqq \sum_{j=0}^{k-1} \binom{k}{j} \sum_{\substack{n_1,\dots,n_j \leq R_k < n_{j+1},\dots,n_{k-1} \leq R_k^2\\ n_1 \dotsm n_{k-1} = m}} \frac{\left( t - \log(n_1 \dotsm n_j R_k^{k-j})\right)^{k-j-1}}{(k-j-1)! \log^{k-j-1} R_k}.
\end{equation}
We will discuss these choices of approximants more in Section~\ref{ssec:longinterval} (which can be read independently of the rest of the paper), but let us already here note that the approximants lead to type $I$ sums and are thus easier to handle than the original functions, and that the choice of the parameter $R$ in $\Lambda^{\sharp}$ allows for an arbitrary power of log saving in~\eqref{mangoldt-discor} below. Moreover, the approximants are nonnegative, which is helpful for some applications (in particular in the proof of Theorem~\ref{thm_gowers} below). For future use, we record the fact that our correlation estimates for $d_k - d_k^{\sharp}$ work for $d_k^{\sharp}$ defined as in~\eqref{dks-def} with any fixed $0 < \eta \leq \frac{1}{10k}$, as long as we allow implied constants to depend on $\eta$.

For technical reasons, it can be beneficial to consider ``maximal discorrelation'' estimates. Loosely following Robert and Sargos~\cite{robert-sargos} we adopt the convention\footnote{Strictly speaking, this is an abuse of notation, since the expression $|\sum_{n \in I \cap \Z} f(n)|^*$ depends not only on the value of the sum $\sum_{n \in I \cap \Z} f(n)$, but also on the individual summands $f(n)$ and the range $I \cap \Z$. In particular, we caution that $\sum_{n \in I \cap \Z} f(n) = \sum_{m \in J \cap \Z} g(m)$ does \emph{not} necessarily imply that $|\sum_{n \in I \cap \Z} f(n)|^* = |\sum_{m \in J \cap \Z} g(m)|^*$.} that, for an interval $I$,
\begin{equation}\label{maximal-sum}
 \left|\sum_{n \in I \cap \Z} f(n)\right|^* \coloneqq \sup_{P \subset I \cap \Z} \left|\sum_{n \in P} f(n)\right|,
\end{equation}
where $P$ ranges over all arithmetic progressions in $I \cap \Z$.

Now we are ready to state our main theorem\footnote{For definitions of undefined terms such as ``filtered nilmanifold'' and $\Poly(\Z \to G)$, see Definitions~\ref{def:filtNilman} and~\ref{def:FiltGroup} below.  For our conventions for asymptotic notation such as $\ll$, see Section~\ref{notation-sec}.}.
\begin{theorem}[Discorrelation estimate]\label{discorrelation-thm} Let $X \geq 3$, $X^{\theta+\varepsilon} \leq H \leq X^{1-\varepsilon}$ for some $0 < \theta < 1$ and $\eps > 0$, and let $\delta \in (0,1)$. Let $G/\Gamma$ be a filtered nilmanifold of some degree $d$ and dimension $D$, and complexity at most $1/\delta$, and let $F \colon G/\Gamma \to \C$ be a Lipschitz function of norm at most $1/\delta$.
\begin{itemize}
\item[(i)] If $\theta = 5/8$, then for all $A > 0$,
\begin{align}
\sup_{g \in \Poly(\Z \to G)} \left| \sum_{X < n \leq X+H} \mu(n) \overline{F}(g(n)\Gamma) \right|^* &\ll_{A,\eps,d,D} \delta^{-O_{d,D}(1)} H \log^{-A} X \label{mobius-discor}
\end{align}
\item[(ii)] If $\theta = 5/8$, then for all $A > 0$,
\begin{align}
\sup_{g \in \Poly(\Z \to G)} \left| \sum_{X < n \leq X+H} (\Lambda(n) - \Lambda^\sharp(n)) \overline{F}(g(n)\Gamma) \right|^* &\ll_{A,\eps,d,D} \delta^{-O_{d,D}(1)} H \log^{-A} X. \label{mangoldt-discor}
\end{align}
\item[(iii)] Let $k \geq 2$.  Set $\theta = 1/3$ for $k=2$, $\theta=5/9$ for $k=3$, and $\theta =5/8$ for $k \geq 4$. Then
\begin{equation}\label{dk-discor}
\sup_{g \in \Poly(\Z \to G)}\left| \sum_{X < n \leq X+H} (d_k(n) - d_k^\sharp(n)) \overline{F}(g(n)\Gamma) \right|^* \ll_{\eps,d,D} \delta^{-O_{d,D}(1)} H X^{-c_{k,d,D} \eps}
\end{equation}
for some constant $c_{k,d,D}>0$ depending only on $k,d,D$.
\item[(iv)] If $\theta = 3/5$, then
\begin{equation}\label{mobius-discor-alt}
\sup_{g \in \Poly(\Z \to G)} \left| \sum_{X < n \leq X+H} \mu(n) \overline{F}(g(n)\Gamma)  \right|^* \ll_{\eps,d,D} \delta^{-O_{d,D}(1)} H \log^{-1/4} X.
\end{equation}
\item[(v)] Let $k \geq 4$. If $\theta = 3/5$, then
\begin{equation}\label{dk-discor-alt}
\sup_{g \in \Poly(\Z \to G)}\left| \sum_{X < n \leq X+H} (d_k(n) - d_k^\sharp(n)) \overline{F}(g(n)\Gamma) \right|^* \ll_{\eps,d,D} \delta^{-O_{d,D}(1)} H \log^{\frac{3}{4}k-1} X.
\end{equation}
\end{itemize}
\end{theorem}

The dependency of the implied constants on $A$ in~\eqref{mobius-discor} and~\eqref{mangoldt-discor} is ineffective due to the possible existence of Siegel zeros. All the other implied constants are effective.

\begin{remark}\label{rem:1-eps} One could extend the theorem to cover the range $X^{1-\varepsilon}\leq H\leq X$ without difficulty; however, this is not the most interesting regime and there are some places in the proof where the restriction to $H\leq X^{1-\varepsilon}$ is convenient. In the cases of~\eqref{mobius-discor},~\eqref{mobius-discor-alt}, the result for $X^{\theta+\varepsilon}\leq H\leq X^{1-\varepsilon}$ directly implies the result for $X^{1-\varepsilon}\leq H\leq X$ by splitting  long sums into shorter ones. In the cases of~\eqref{mangoldt-discor},~\eqref{dk-discor},~\eqref{dk-discor-alt}, it turns out that there is some flexibility in the choice of the approximant (one can certainly vary $R$ in~\eqref{lambdar-def} or $R_k$ in~\eqref{dks-def} by a multiplicative factor $\asymp 1$), and then one can make a similar splitting argument. We leave the details to the interested reader.
\end{remark}

In applications $d,D,\delta$ will often be fixed; however, the fact that the constants here depend in a polynomial fashion on $\delta$ will be useful for induction purposes.

Note that polynomial phases $F(g(n)\Gamma) = e(P(n))$, with $P \colon \Z \to \R$ a polynomial of degree $d$, are a special case of nilsequences --- in this case the filtered nilmanifold is the unit circle $\R/\Z$ (with $\R = (\R,+)$ being the filtered nilpotent group with $\R_i = \R$ for $i \leq d$ and $\R_i = \{0\}$ for $i>d$) and $F(\alpha) = e(\alpha)$ for all $\alpha \in \R/\Z$. In particular the results of Theorem~\ref{discorrelation-thm} hold for polynomial phases, that is, with $G/\Gamma=\mathbb{R}/\mathbb{Z}$, $D=1$, and with $\overline{F}(g(n)\Gamma)$ replaced with $e(P(n))$. Before moving on, let us for the convenience of the reader state the following corollary of our theorem in the polynomial phase case.

\begin{corollary}[Discorrelation of $\mu$ and $\Lambda$ with polynomial phases in short intervals]\label{cor:polynomial}
Let $X \geq 3$ and let $X^{\theta+\varepsilon} \leq H \leq X^{1-\varepsilon}$ for some $0 < \theta < 1$ and $\eps > 0$. Let $d\geq 1$ and let $P:\mathbb{Z}\to \mathbb{R}$ be any polynomial of degree $d$.
\begin{itemize}
\item[(i)]
If $\theta = 5/8$, then, for all $A > 0$,
\begin{align*}
\left|\sum_{X < n\leq X+H}\mu(n)e(P(n))\right|\ll_{d,A,\varepsilon} \frac{H}{\log^A X}
\end{align*}
\item[(ii)]
If $\theta = 5/8$ and $A > 0$, we have
\begin{align*}
\left|\sum_{X < n\leq X+H}\Lambda(n)e(P(n))\right|\leq \frac{H}{\log^A X},
\end{align*}
unless there exists $1\leq q\leq (\log X)^{O_{d,A,\varepsilon}(1)}$ such that one has the ``major arc'' property
\begin{align}\label{erg7}
\max_{1\leq j\leq d} H^j \|q\alpha_j\|_{\R/\Z}\leq (\log X)^{O_{d,A,\varepsilon}(1)},
\end{align}
where $\alpha_j$ is the degree $j$ coefficient of the polynomial $n\mapsto P(n+X)$ and $\|y\|_{\R/\Z}$ denotes the distance from $y$ to the nearest integer(s).
\item[(iii)]
If $\theta = 3/5$, then
\begin{align*}
\left|\sum_{X < n\leq X+H}\mu(n)e(P(n))\right|\ll_{d,\varepsilon} \frac{H}{\log^{1/10} X}.
\end{align*}
\end{itemize}
\end{corollary}
The claims (i) and (iii) are immediate from Theorem~\ref{discorrelation-thm}, but (ii) requires a short argument, provided in Section~\ref{sec:apps}. One could state an analogous result in the case of $d_k$ (with the same exponents as in Theorem~\ref{discorrelation-thm}).

Let us now discuss the literature on the topic, starting with results concerning the M\"obius function. A discorrelation estimate such as Theorem~\ref{discorrelation-thm}(i) with arbitrary $F(g(n) \Gamma)$ was previously only known in case of long intervals due to the work of Green and the third author~\cite[Theorem 1.1]{gt-mobius}. Namely, they have shown that
\begin{equation}\label{green-tao}
 \sup_{g \in \Poly(\Z \to G)} \left| \sum_{n \leq X} \mu(n) \overline{F}(g(n)\Gamma) \right| \ll_{A,G/\Gamma,F} X \log^{-A} X
\end{equation}
for any $X \geq 2$, $A>0$, filtered nilmanifold $G/\Gamma$, and Lipschitz function $F \colon G/\Gamma \to \C$. This result of Green and the third author is a vast generalization of a classical result of Davenport~\cite{davenport}, which states that
\begin{equation}\label{davenport}
\sup_{\alpha \in \R} \left| \sum_{n \leq X} \mu(n) e(-\alpha n)\right| \ll_A X \log^{-A} X,
\end{equation}
and of the Siegel--Walfisz theorem (see e.g.~\cite[Corollary 5.29]{ik}), which states that
\begin{equation}
\label{eq:S-W}
\max_{a, q \in \mathbb{N}} \Bigl|\sum_{\substack{n \leq X\\ n = a\ (q)}} \mu(n) \Bigr| \ll_A X \log^{-A} X.
\end{equation}
As is well known, the bounds of $O_A(X \log^{-A} X)$ here cannot be improved unconditionally with current technology, due to the possible existence of Siegel zeroes (unless one subtracts a correction term to account for the contribution of such zero; see~\cite[Theorem 2.7]{tt-quant}).

On the other hand, for short intervals there has been a lot of activity in the special case of polynomial phase twists.

Theorem~\ref{discorrelation-thm}(i) was previously only known in the linear phase case when $F(g(n)\Gamma) = e(\alpha n)$ for any $\alpha \in \mathbb{R}$ by work of Zhan~\cite{zhan}. More precisely Zhan~\cite[Theorem 5]{zhan} established that
\begin{equation}
\label{eq:Zhan-mu}
\sup_{\alpha \in \R} \left|\sum_{X < n \leq X+H} \mu(n) e(-\alpha n)\right| \ll_{A,\varepsilon} H \log^{-A} X
\end{equation}
whenever $X^{5/8+\varepsilon} \leq H \leq X$ and $A \geq 1$. Hence Theorem~\ref{discorrelation-thm}(i) can be seen as a vast extension of Zhan's work.

Concerning higher degree polynomials, the most recent result is due to the first two authors~\cite[Theorem 1.4]{matomaki-shao} giving, for any polynomial $P(n)$ of degree $\leq d$,
\begin{equation}\label{mu-poly}
\sum_{X < n \leq X+H} \mu(n) e(-P(n))  \ll_{A,d,\eps} H \log^{-A} X
\end{equation}
for all $A > 0$ and $X^{2/3+\varepsilon} \leq H \leq X$. In particular a special case of Theorem~\ref{discorrelation-thm}(i) (recorded here as Corollary~\ref{cor:polynomial}(i)) supersedes this result by showing it with the exponent $2/3$ lowered to $5/8$.

All the previous results mentioned so far for the M\"obius function exist also for the von Mangoldt function as long as $F(g(n) \Gamma)$ or $e(-P(n))$ is ``minor arc'' in certain sense (for results corresponding to~\eqref{green-tao},~\eqref{davenport},~\eqref{eq:S-W},~\eqref{eq:Zhan-mu} and~\eqref{mu-poly} see respectively~\cite[Section 7]{gt-mobius},~\cite[Theorem 13.6]{ik},~\cite[Corollary 5.29]{ik},~\cite[Theorems 2--3]{zhan}, and~\cite[Theorem 1.1]{matomaki-shao}). It is very likely that with our choice of approximant these arguments also extend to cover major arc cases and maximal correlations, although we will not detail this here as such claims follow in any case from Theorem~\ref{discorrelation-thm}.

Theorem~\ref{discorrelation-thm}(iv) generalizes (albeit with a slightly weaker logarithmic saving) a result of the first and fourth authors~\cite[Theorem 1.5]{MatoTera} that gave, for $0 < A < 1/3$,
\begin{equation}\label{eq:MatoTera}
 \sup_{\alpha \in \R} \left|\sum_{X < n \leq X+H} \mu(n) e(-\alpha n)\right| \ll_{A,\eps} H \log^{-A} X
\end{equation}
in the regime $X \geq H \geq X^{3/5+\varepsilon}$ (actually~\cite[Remark 5.2]{MatoTera} allows one to enlarge the range of $A$ to $0 < A < 1$).

The literature on correlations between $d_k$ and Fourier or higher order phases is sparse. A variant of the long interval case~\eqref{green-tao} (with a weaker error term) follows from work of Matthiesen~\cite[Theorem 6.1]{MatthiesenGenFour}.

Furthermore, it should be possible to adapt the existing results on polynomial correlations of $\Lambda(n)$ also to the case of $d_k(n)$, but with power savings. More precisely, one should be able to follow the approach of Zhan~\cite{zhan} to obtain discorrelation with linear phases $e(\alpha n)$ for $X \geq H \geq X^{5/8+\varepsilon}$ (for $k=2$ one can replace $5/8$ by $1/2$ and for $k=3$ one can replace $5/8$ by $3/5$) and the work of the first two authors~\cite{matomaki-shao} to obtain discorrelation with polynomial phases for $X \geq H \geq X^{2/3+\varepsilon}$ (for $k=2$ one can replace $2/3$ by $1/2$). We omit the details of these extensions of~\cite{zhan, matomaki-shao} as they follow from our Theorem~\ref{discorrelation-thm}.

We note that in the case $k=2$ the exponent $1/3$ in Theorem~\ref{discorrelation-thm}(iii) matches the classical Voronoi exponent for the error term in long sums of the divisor function without any twist, and the result seems to be new even in the case of linear phases.

In the most major arc case $F(g(n) \Gamma) = 1$, shorter intervals can be reached than in Theorem~\ref{discorrelation-thm}; see Theorem~\ref{thm:major-arc} below. Furthermore if one only wants discorrelation in almost all intervals, for instance by seeking to bound
\[
\int_X^{2X} \sup_{g \in \Poly(\Z \to G)} \left| \sum_{x < n \leq x+H} (f(n)-f^\sharp(n)) \overline{F}(g(n)\Gamma) \right|^* dx,
\]
much shorter intervals can be reached with aid of additional ideas. We will return to this question and its applications in a follow-up paper~\cite{MRSTT-almost}.

\begin{remark}\label{variants-rem}	 It should be clear to experts from an inspection of our arguments that the methods used in this paper could also treat other arithmetic functions with similar structure to $\mu$, $\Lambda$, or $d_k$.  For instance, all of the results for the M\"obius function $\mu$ here have counterparts for the Liouville function $\lambda$; the results for the von Mangoldt function $\Lambda$ have counterparts (with somewhat different normalizations) for the indicator function $1_{\mathbb{P}}$ of the primes ${\mathbb{P}}$, and the results for $d_2$ have counterparts for the function $r_2(n) \coloneqq \sum_{a,b \in \Z: a^2+b^2=n} 1$ counting the number of representations of $n$ as the sum of two squares.  We sketch the modifications needed to establish these variants in Appendix~\ref{variants-app}.  We also conjecture that the methods can be extended to treat the indicator function $1_S$ of the set $S \coloneqq \{ a^2+b^2: a,b \in \Z \}$ of sums of two squares, or the indicator $1_{S_\eta}$ of $X^\eta$-smooth numbers, although in those two cases a technical difficulty arises that the construction of a sufficiently accurate approximant to these indicator functions is non-trivial.  Again, see Appendix~\ref{variants-app} for further discussion.

On the other hand, our arguments do not seem to easily extend to the Fourier coefficients $\lambda_f(n)$ of holomorphic cusp forms. The coefficients $\lambda_f(n)$ are analogous to $d_2(n)$ in many ways (though with vanishing approximant $\lambda^\sharp_f = 0$), and it is reasonable to conjecture parallel results for these two functions.  For instance, in~\cite{EHK} it was established that
\[
\sup_\alpha \left| \sum_{X < n \leq X+H} \lambda_f(n) e(\alpha n) \right| \ll HX^{-c_\eps}
\]
for $X^{2/5+\eps} \leq H \leq X$. See also~\cite{he-wang} for a result with general nilsequences but long intervals. Unfortunately, the methods we use in this paper rely heavily on the convolution structure of the functions involved and do not obviously extend to give results for $\lambda_f$.
\end{remark}

\subsection{Gowers uniformity in short intervals}\label{subsec:gowers}

Just as discorrelation estimates with polynomial phases are important for applications of the circle method, discorrelation estimates with nilsequences are important in higher order Fourier analysis due to the connection with the Gowers uniformity norms that we next discuss.

For any non-negative integer $s \geq 1$, and any function $f \colon \Z \to \C$ with finite support, define the (unnormalized) Gowers uniformity norm
$$ \| f \|_{U^{s}(\Z)} \coloneqq \left( \sum_{x,h_1,\dots,h_{s} \in \Z} \prod_{\omega \in \{0,1\}^{s}} \mathcal{C}^{|\omega|} f(x+\omega_1 h_1+\dots+\omega_{s} h_{s}) \right)^{1/2^{s}}$$
where $\omega = (\omega_1,\dots,\omega_{s})$, $|\omega| \coloneqq \omega_1+\dots+\omega_{s}$, and $\mathcal{C} \colon z \mapsto \overline{z}$ is the complex conjugation map. Then for any interval $(X,X+H]$ with $H \geq 1$ and any $f \colon \Z \to \C$ (not necessarily of finite support), define the \emph{Gowers uniformity norm over} $(X,X+H]$ by
\begin{equation}\label{gow}
 \| f \|_{U^{s}(X,X+H]} \coloneqq \| f 1_{(X,X+H]} \|_{U^{s}(\Z)} /  \| 1_{(X,X+H]} \|_{U^{s}(\Z)}
\end{equation}
where $1_{(X,X+H]} \colon \Z \to \C$ is the indicator function of $(X,X+H]$.

Using the inverse theorem for Gowers norms (see Proposition~\ref{prop_inverse}) we can deduce the following theorem from Theorem~\ref{discorrelation-thm} and a construction of pseudorandom majorants in Section~\ref{gowers-sec}.

\begin{theorem}[Gowers uniformity estimate]\label{thm_gowers}
Let $X^{\theta+\varepsilon}\leq H\leq X^{1-\varepsilon}$ for some fixed $0 < \theta < 1$ and $\eps > 0$. Let $s\geq 1$ be a fixed integer. Also denote $\Lambda_{w}(n):=\frac{W}{\varphi(W)}1_{(n,W)=1}$, where $W:=\prod_{p\leq w}p$ and $X$ is large enough in terms of $w$.
\begin{itemize}
\item[(i)]   If $\theta = 5/8$, then
\begin{align}\label{erg14}
&\|\Lambda-\Lambda_w\|_{U^s(X,X+H]}=o_{w\to \infty}(1),
\end{align}
and for any $1\leq a\leq W$ with $(a,W)=1$ we have
\begin{align}\label{erg14b}
&\left\|\frac{\varphi(W)}{W}\Lambda(W\cdot+a)-1\right\|_{U^s(X,X+H]}=o_{w\to \infty}(1).
\end{align}
\item[(ii)] Let $k \geq 2$.  Set $\theta = 1/3$ for $k=2$, $\theta=5/9$ for $k=3$, and $\theta =3/5$ for $k \geq 4$. Then
\begin{equation}\label{dk-unif}
\|d_k-d_k^{\sharp}\|_{U^s(X, X+H]}=o(\log^{k-1} X),
\end{equation}
and for any $W'$ satisfying $W\mid W'\mid W^{\lfloor w\rfloor}$ and for any $1\leq a\leq W'$ with $(a,W')=1$ we have
\begin{align}\label{dk-unifb}
\|d_k(W'\cdot+a)-d_k^{\sharp}(W'\cdot+a)\|_{U^s(X, X+H]}=o_{w\to \infty}\left(\left(\frac{\varphi(W')}{W'}\right)^{k-1}\log^{k-1} X\right).
\end{align}
\item[(iii)]  If $\theta = 3/5$, then
\begin{equation}\label{mobius-unif-alt}
\|\mu\|_{U^s(X, X+H]}=o(1).
\end{equation}
\end{itemize}
In all these estimates the $o(1)$ notation is with respect to the limit $X \to \infty$ (holding $s,\eps,k$ fixed).
\end{theorem}

\textbf{Remarks.}

\begin{itemize}
    \item

    The model $\Lambda_w$ with $w$ fixed is simple to work with and arises in various applications of Gowers uniformity (e.g. to ergodic theory). This also motivates our choice of the $\Lambda^{\sharp}$ model in~\eqref{lambdar-def} (although that is defined with a larger value of $w$ to produce better error terms).

    \item Since the bounds in this theorem (unlike in Theorem~\ref{discorrelation-thm}) are qualitative in nature, it should be possible to use Heath-Brown's trick from~\cite{Heath-Brown} to extend the range of $H$ from $X^{\theta+\varepsilon}\leq H\leq X^{1-\varepsilon}$ to $X^{\theta}\leq H\leq X^{1-\varepsilon}$. Also the range $X^{1-\varepsilon}\leq H\leq X$ could be covered, as in Remark~\ref{rem:1-eps}. We leave the details to the interested reader.

\item  In the case $s=2$, we obtain significantly stronger estimates thanks to the polynomial nature of the $U^2$ inverse theorem.  Specifically, when $\theta=5/8+\varepsilon$, we have
$$
\|\mu\|_{U^2(X, X+X^{\theta}]}, \|\Lambda-\Lambda^{\sharp}\|_{U^2(X, X+X^{\theta}]} \ll_{A,\eps} \log^{-A} X$$
for all $A > 0$ and
\begin{equation}\label{dk-u2}
\|d_k\|_{U^2(X, X+X^{\theta}]} \ll_\eps X^{-c_{k}\eps}
\end{equation}
for some $c_{k}>0$, with~\eqref{dk-u2} also holding when $(k,\theta) = (3,5/9), (2,1/3)$, and finally
$$
\|\mu\|_{U^2(X, X+X^{\theta}]} \ll_\eps \log^{-1/20} X$$
when $\theta = 3/5$. All of these follow directly by combining Theorem~\ref{discorrelation-thm} for $d=1$ (that is, for Fourier phases in place of nilsequences) with the polynomial form of the $U^2$ inverse theorem, which states that if $f:[N]\to \mathbb{C}$ is $1$-bounded and $\|f\|_{U^2[N]}\geq \delta$ for some $\delta>0$, then $|\sum_{n\leq N}f(n)e(\alpha n)|^{*}\gg \delta^4 N$ for some $\alpha \in \mathbb{R}$. This form of the inverse theorem follows directly from the Fourier representation of the $U^2[N]$ norm and Parseval's theorem, where the Gowers norm $U^2[N]$ is defined analogously as in~\eqref{gow}.

\end{itemize}

\subsection{Applications}\label{subsec:applications}

\subsubsection{Polynomial phases}
We already stated Corollary~\ref{cor:polynomial} concerning polynomial phases. But let us here mention that in a recent work of Kanigowski--Lema{\'n}czyk--Radziwi{\l\l}~\cite{klr} on the prime number theorem for analytic skew products, a key analytic input (\cite[Theorem 9.1]{klr}) was that Corollary~\ref{cor:polynomial}(ii) holds for $H=X^{2/3-\eta}$ (with a weaker error term of $o_{\eta\to 0}(H)$), thus going just beyond the range of validity of~\cite[Theorem 1.1]{matomaki-shao}. Corollary~\ref{cor:polynomial} allows taking $\eta<1/24$ with strongly logarithmic savings for the error terms. Similar remarks apply to the recent work of Kanigowski~\cite{kanigowski}.

\subsubsection{An application to ergodic theory}

In a seminal work, Host and Kra~\cite{host-kra} showed that, for any measure-preserving system $(X,\mathcal{X},\mu,T)$, any bounded functions $f_1,\ldots, f_k:X\to \mathbb{C}$, and any intervals $I_N$ whose lengths tend to infinity as $N\to \infty$, the multiple ergodic averages
\begin{align*}
\frac{1}{|I_N|}\sum_{n\in I_N}f_1(T^nx)\cdots f_k(T^{kn}x)
\end{align*}
converge in $L^2(\mu)$ as $N\to \infty$. Since this work, it has therefore become a natural and active question to determine for which sequences of intervals $(I_N)_N$ and weights $w:\mathbb{N}\to \mathbb{C}$ we have the $L^2$-convergence of
\begin{align*}
\frac{1}{|I_N|}\sum_{n\in I_N}w(n)f_1(T^nx)\cdots f_k(T^{kn}x)
\end{align*}
as $N\to \infty$. The case of $I_N=[1,N]$ and with the weight being the primes, that is $w(n)=1_{\mathbb{P}}(n)$, was settled in the works of Frantzikinakis--Host--Kra~\cite{fhk} and Wooley--Ziegler~\cite{wooley-ziegler} (the results of~\cite{fhk} in the cases $k\geq 4$ were originally conditional on the Gowers uniformity of the von Mangoldt function).  Analogous results also exist for weights $w$ supported on a sequence given by a Hardy field~\cite{frantzikinakis-hardy} or random sequences~\cite{flw}; see also~\cite{le} for related results concerning correlation sequences $n \mapsto \int_X f_1(T^nx)\cdots f_k(T^{kn}x)\ d\mu(x)$.  As an application of Theorem~\ref{thm_gowers}, we can extend the result on prime weights  to short collections of intervals $(I_N)_N$.

\begin{theorem}[Multiple ergodic averages over primes in short intervals] \label{thm_ergodic}
Let $k\geq 1$, $\varepsilon>0$ and $\kappa \in [5/8+\varepsilon, 1-\varepsilon]$. Let $h_1,\ldots, h_k$ be distinct positive integers. Let $(X,\mathcal{X},\mu,T)$ be a measure-preserving system. Let $f_1,\ldots, f_k:X\to \mathbb{C}$ be bounded and measurable. Then the multiple ergodic averages
\begin{align*}
\mathbb{E}_{N < p\leq N+N^{\kappa}} f_1(T^{h_1p}x)\cdots f_k(T^{h_kp}x)
\end{align*}
converge in $L^2(\mu)$.
\end{theorem}
The results of~\cite{fhk} and~\cite{wooley-ziegler} correspond to the case $\kappa=1$ . According to the best of our knowledge, Theorem~\ref{thm_ergodic} is the first result of its kind with $\kappa<1$.

\subsubsection{Linear equations in short intervals}

The work of Green and the third author~\cite{green-tao} on linear equations in primes (together with~\cite{gt-mobius},~\cite{gtz}) provides for any finite complexity systems of linear forms $(\psi_1,\ldots, \psi_t):\mathbb{Z}^d\to \mathbb{Z}^t$ an asymptotic formula for
\begin{align}\label{erg11}
\sum_{\mathbf{n}\in K\cap \mathbb{Z}^d}\prod_{i=1}^t \Lambda(\psi_i(\mathbf{n})),
\end{align}
whenever $K\subset [-X,X]^d$ is a convex body containing a positive proportion of the whole cube $[-X,X]^d$, that is, $\textnormal{vol}(K)\gg X^d$. One may ask if one can establish similar results when $K$ is a smaller region in $[-X,X]^d$, of volume $\asymp X^{\theta d}$ with $\theta<1$. Note that for a single linear form, this boils down to asymptotics for primes in short intervals (where the exponent $\theta=7/12$ from~\cite{huxley},~\cite{Heath-Brown} is the best one known). Using Theorem~\ref{thm_gowers}, we can indeed give asymptotics for~\eqref{erg11} in small regions.

\begin{theorem}[Generalized Hardy--Littlewood conjecture in small boxes for finite complexity systems]\label{thm_lineq}
Let $X \geq 3$ and $X^{5/8+\varepsilon} \leq H \leq X^{1-\varepsilon}$ for some fixed $\eps > 0$. Let $d,t,L\geq 1$. Let $\Psi=(\psi_1,\ldots, \psi_t)$ be a system of affine-linear forms, where each $\psi_i:\mathbb{Z}^d\to \mathbb{Z}$ has the form $\psi_i(\mathbf{x})=\dot{\psi_i}\cdot \mathbf{x}+\psi_i(0)$ with $\dot{\psi_i}\in \mathbb{Z}^d$ and $\psi_i(0)\in \mathbb{Z}$ satisfying $|\dot{\psi_i}|\leq L$ and $|\psi_i(0)|\leq LX$. Suppose that $\dot{\psi_i}$ and $\dot{\psi_j}$ are linearly independent whenever $i\neq j$. Let $K\subset (X,X+H]^d$ be a convex body. Then
\begin{align}\label{erg10}
\sum_{\mathbf{n}\in K\cap \mathbb{Z}^d}\prod_{i=1}^t \Lambda(\psi_i(\mathbf{n}))=\beta_{\infty}\prod_p \beta_p+o_{t,d,L}(H^d),
\end{align}
where $\Lambda$ is extended as $0$ to the nonpositive integers and the Archimedean factor is given by
\begin{align*}
\beta_{\infty}=\textnormal{vol}(K\cap \Psi^{-1}(\mathbb{R}_{>0}^t))
\end{align*}
and the local factors are given by
\begin{align*}
\beta_p=\mathbb{E}_{\mathbf{n}\in (\mathbb{Z}/p\mathbb{Z})^d}\prod_{i=1}^t\frac{p}{p-1}1_{\psi_i(\mathbf{n})\neq 0}.
\end{align*}
\end{theorem}

\begin{remark}
From Theorem~\ref{thm_gowers} and the proof method of Theorem~\ref{thm_lineq}, one can also deduce similar correlation results  when in~\eqref{erg10} one replaces $\Lambda$ with $\mu$ or $d_k$ (with the value of $\theta$ as in Theorem~\ref{thm_gowers}, and with no main term in the case of $\mu$, and a different local product in the case of $d_k$). More specifically, under the assumption of Theorem~\ref{thm_lineq}, we have
\begin{align}
\sum_{\mathbf{n}\in K\cap \mathbb{Z}^d}\prod_{i=1}^t \mu(\psi_i(\mathbf{n}))= o_{t,d,L}(H^d),
\end{align}
and, for a positive integer $k$,
\begin{align*}
\sum_{\mathbf{n}\in K\cap \mathbb{Z}^d}\prod_{i=1}^t d_k(\psi_i(\mathbf{n}))=\beta_{\infty}\prod_p \beta_p+o_{t,d,L}(H^d\log^{t(k-1)}X),
\end{align*}
where $d_k$ is extended as $0$ to the nonpositive integers and the Archimedean factor is given by
\begin{align*}
\beta_{\infty}=\int_K \prod_{i=1}^t \frac{\log_+^{k-1}\psi_i(\mathbf{x})}{(k-1)!} d\mathbf{x} = O_{t,d,L} (H^d \log^{t(k-1)}X),
\end{align*}
and the local factors are given by
\begin{align*}
\beta_p=\frac{\E_{\mathbf{n} \in \Z_p^d} \prod_{i=1}^t d_{k,p}(\psi_i(\mathbf{n}))}{\prod_{i=1}^t \E_{m \in \Z_p} d_{k,p}(m)} = \E_{\mathbf{n} \in \Z_p^d} \prod_{i=1}^t \Big(\frac{p-1}{p}\Big)^{k-1} d_{k,p}(\psi_i(\mathbf{n})).
\end{align*}
Here $\log_+ y := \log \max(y, 1)$, $\Z_p$ is the $p$-adics (with the usual Haar probability measure),
$$ d_{k,p}(m) = \binom{k-1+v_p(m)}{k-1}, $$
and $v_p(m)$ is the number of times $p$ divides $m$. These local factors are natural extensions of the ones defined in~\cite[Remark 1.2]{mrt-div} in the special case of two linear forms $\psi_1(n) = n, \psi_2(n) = n+h$.
\end{remark}

We have the following immediate corollary to Theorem~\ref{thm_lineq}.

\begin{corollary}[Linear equations in primes in short intervals]\label{cor_lineq} Let $X \geq 3$ and $X^{5/8+\varepsilon} \leq H \leq X^{1-\varepsilon}$ for some fixed $\eps > 0$. Let $d,t,L\geq 1$. Let $\Psi=(\psi_1,\ldots, \psi_t):\mathbb{Z}^d\to \mathbb{Z}^t$ be a system of affine-linear forms, where each $\psi_i$ has the form $\psi_i(\mathbf{x})=\dot{\psi_i}\cdot \mathbf{x}+\psi_i(0)$ with $\dot{\psi_i}\in \mathbb{Z}^d$ and $\psi_i(0)\in \mathbb{Z}$ satisfying $|\dot{\psi_i}|\leq L$ and $|\psi_i(0)|\leq LX$. Suppose that $\dot{\psi_i}$ and $\dot{\psi_j}$ are linearly independent whenever $i\neq j$.  Suppose that, for every prime $p$, the system of equations $\Psi(\mathbf{n})=0$ is solvable with $\mathbf{n}\in ((\mathbb{Z}/p\mathbb{Z})\setminus\{0\})^d$. Then the number of solutions to $\Psi(\mathbf{n})=0$ with $\mathbf{n}\in (\mathbb{P}\cap (X,X+H])^d$ is
\begin{align*}
\gg \frac{\textnormal{vol}((X,X+H]^d\cap \Psi^{-1}(\mathbb{R}_{>0}^t))}{\log^d X}+o_{d,t,L}\left(\frac{H^d}{\log^d X}\right).
\end{align*}
\end{corollary}

Thus, for example, for any $\varepsilon>0$ and any large enough odd $N$ there is a solution to
\begin{align*}
p_1+p_2+p_3=N,\quad p_1,p_2,p_3,2p_1-p_2\in \mathbb{P}
\end{align*}
with $p_i\in [N/3-N^{5/8+\varepsilon},N/3+N^{5/8+\varepsilon}]$. Without the condition $2p_1-p_2 \in \mathbb{P}$, this is due to Zhan~\cite{zhan}. The exponent $5/8$ in Zhan's result has been improved using sieve methods (see e.g.~\cite{BH1998}) and more recently using the transference principle~\cite{MMS2017}. It would probably be possible to use a sieve method also to improve on Corollary~\ref{cor_lineq}; it would suffice to find a suitable minorant function for $\Lambda(n)$ that has positive average and is Gowers uniform in shorter intervals. Such a minorant could be constructed with our arithmetic information using Harman's sieve method~\cite{harman-book}, but we do not do so here.

\subsection{Methods of proof}

We now describe (in somewhat informal terms) the general strategy of proof of our main theorems, although for various technical reasons the actual rigorous proof will not quite follow the intuitive plan that is outlined here.

To prove Theorem~\ref{discorrelation-thm}, the first step, which is standard, is to apply Heath--Brown's identity (Lemma~\ref{hb-identity}) together with a combinatorial lemma regarding subsums of a finite number of non-negative reals summing to one (Lemma~\ref{combinatorial}) to decompose $\mu, \Lambda, d_k$ (up to small errors) into three standard types of sums:
\begin{itemize}
\item[($I$)] \emph{Type $I$} sums, which are roughly of the form $\alpha * 1 = \alpha*d_1$ for some arithmetic function $\alpha \colon \N \to \C$ supported on some interval $[1, A_I]$ that is not too large, and with $\alpha$ bounded in an $L^2$ averaged sense.
\item[($I_2$)] \emph{Type $I_2$} sums, which are roughly of the form $\alpha * d_2$ for some arithmetic function $\alpha \colon \N \to \C$ supported on some interval $[1, A_{I_2}]$ that is not too large, and with $\alpha$ bounded in an $L^2$ averaged sense.
\item[($II$)] \emph{Type $II$} sums, which are roughly of the form $\alpha * \beta$ for some arithmetic functions $\alpha, \beta \colon \N \to \C$ with $\alpha$ supported on some interval $[A_{II}^-, A_{II}^+]$ that is neither too long nor too close to $1$ or $X$, and with $\alpha,\beta$ bounded in an $L^2$ averaged sense.
\end{itemize}
This decomposition is detailed in Section~\ref{reduction-sec}. The precise ranges of parameters $A_I, A_{I_2}, A_{II}^-$, $A_{II}^+$ that arise in this decomposition depend on the choice of $\theta$ (and, in the case of $d_k$ for small $k$, on the value of $k$); this is encoded in the combinatorial lemma given here as Lemma~\ref{combinatorial}.

The treatment of these types of sums (in Theorem~\ref{inverse}) depends on the behavior of the nilsequence $F(g(n) \Gamma)$, in particular whether it is ``major arc'' or ``minor arc''. This splitting into different behaviors will be done somewhat differently for different types of sums.

In case of type $I$ and type $I_2$ sums, one can use the equidistribution theory of nilmanifolds to essentially reduce to two cases, the \emph{major arc case} in which the nilsequence $F(g(n) \Gamma)$ behaves like (or ``pretends to be'') the constant function $1$ (or some other function of small period), and the \emph{minor arc case} in which $F$ has mean zero and $g(n) \Gamma$ is highly equidistributed in the nilmanifold $G/\Gamma$.  The contribution of type $I$ and type $I_2$ major arc sums can be treated by standard methods, namely an application of Perron's formula and mean value theorems for Dirichlet series; see Section~\ref{major-arc-sec}.

The contribution of type $I$ minor arc sums can be treated by a slight modification of the arguments in~\cite{gt-mobius}, which are based on the ``quantitative Leibman theorem'' (Theorem~\ref{qlt}~below) that characterizes when a nilsequence is equidistributed, as well as a classical lemma of Vinogradov (Lemma~\ref{vin} below) that characterizes when a polynomial modulo $1$ is equidistributed. (Actually it will be convenient to rely primarily on a corollary of Lemma~\ref{vin} that asserts that if typical dilates of a polynomial are equidistributed modulo $1$, then the polynomial itself is equidistributed modulo $1$: see Corollary~\ref{smooth-dilate} below.)

Our treatment of type $I_2$ minor arc sums is more novel.  A model case is that of treating the $d_2$-type correlation
$$ \sum_{X < n \leq X+H} d_2(n) \overline{F}(g(n) \Gamma).$$
From the definition of the divisor function $d_2$, we can expand this sum as a double sum
\begin{equation}\label{nm-h}
 \sum_{n,m: X < nm \leq X+H} \overline{F}(g(nm) \Gamma).
\end{equation}
We are not able to obtain non-trivial estimates on such sums in the regime $H \leq X^{1/3}$.  However, when $H \geq X^{1/3+\eps}$, it turns out by elementary geometry of numbers that the hyperbola neighborhood $\{ (n,m) \in \Z^2: X < nm \leq X+H\}$ may be partitioned\footnote{This partition is reminiscent of the classical Hardy--Littlewood partition of the unit circle into major and minor arcs, except that we are partitioning (a neighborhood of) a hyperbola rather than a circle.} into arithmetic progressions $P \subset \Z^2$ that mostly have non-trivial length; see Theorem~\ref{decomp} for a precise statement.  This decomposition lets us efficiently decompose the sum~\eqref{nm-h} into short sums of the form
$$ \sum_{(n,m) \in P} \overline{F}(g(nm) \Gamma)$$
that turn out to exhibit cancellation for most progressions $P$ in the type $I_2$ minor arc case, mainly thanks to the quantitative Leibman theorem (Theorem~\ref{qlt}) and a corollary of the Vinogradov lemma (Corollary~\ref{smooth-dilate}); see Section~\ref{I2-sec}.

It remains to handle the contribution of type $II$ sums, which are of the form
$$ \sum_{X < n \leq X+H} \alpha*\beta(n) \overline{F}(g(n) \Gamma)$$
which we can expand as
\begin{equation}\label{Expo}
 \sum_{A_{II}^- \leq a \leq A_{II}^+} \alpha(a) \sum_{X/a < b \leq X/a + H/a} \beta(b) \overline{F}(g(ab) \Gamma).
\end{equation}

To treat these sums, we can use a Fourier decomposition and the equidistribution theory of nilmanifolds to reduce (roughly speaking) to treating the following three special cases of these sums:
\begin{itemize}
\item \emph{Type $II$ major arc sums} that are essentially of the form
$$ \sum_{X < n \leq X+H} \alpha*\beta(n) n^{iT} $$
for some real number $T = X^{O(1)}$ of polynomial size (one can also consider generalizations of such sums when the $n^{iT}$ factor is twisted by an additional Dirichlet character $\chi$ of bounded conductor).

\item \emph{Abelian Type $II$ minor arc sums} in which $F(g(n)\Gamma) = e(P(n))$ is a polynomial phase that does not ``pretend'' to be a character $n^{iT}$ (or more generally $\chi(n) n^{iT}$ for some Dirichlet character $\chi$ of bounded conductor) in the sense that the Taylor coefficients of $e(P(n))$ around $X$ do not align with the corresponding coefficients of such characters.

\item \emph{Non-abelian Type $II$ minor arc sums}, in which $g(n) \Gamma$ is highly equidistributed in a nilmanifold $G/\Gamma$ arising from a non-abelian nilpotent group $G$, and $F$ exhibits non-trivial oscillation in the direction of the center $Z(G)$ of $G$ (which one can reduce to be one-dimensional).
\end{itemize}

One can treat the contribution of the type $II$ major arc sums by applying Perron's formula and Dirichlet polynomial estimates of  Baker--Harman--Pintz~\cite{baker-harman-pintz} in the regime, so long as one actually has a suitable triple convolution (with one of the subfactors having well-controlled correlations with $n^{iT}$); see Lemma~\ref{le:BHP}.  As already implicitly observed by Zhan~\cite{zhan}, this case can be treated (with favorable choices of parameters) for any of the three functions $\mu, \Lambda, d_k$ in the case $\theta = 5/8$.  As observed in~\cite{MatoTera}, in the case of the M\"obius function $\mu$, it is possible to lower $\theta$ to $3/5$ and still obtain triple convolution structure after removing a small exceptional error term from $\mu$ (which is responsible for the final discorrelation bounds not saving arbitrary powers of $\log X$); see Lemma~\ref{comb-mu}.

It remains to treat the contribution of non-abelian and abelian type $II$ minor arc sums.
It turns out that we will be able to establish good estimates for such sums~\eqref{Expo} in the regime
$$ X^\eps \frac{X}{H} \lll A_{II}^- < A_{II}^+ \lll X^{-\eps} H.$$
In this regime, the inner intervals $(X/a, X/a+H/a]$ in~\eqref{Expo} have non-negligible length (at least $X^\eps$), and furthermore they exhibit non-trivial overlap with each other ($(X/a, X/a+H/a]$ will essentially be identical to $(X/a', X/a'+H/a']$ whenever $a' = \left(1 + O\left(X^{-\eps} \frac{H}{X}\right)\right) a$).

As a consequence, many of the dilated nilsequences $b \mapsto \overline{F}(g(ab) \Gamma)$ appearing in~\eqref{Expo} will correlate with the same portion of the sequence $\beta$.  To handle this situation we introduce a nilsequence version of the large sieve inequality in Proposition~\ref{large-sieve}, which we establish with the aid of the equidistribution theory for nilsequences, as well as Goursat's lemma.  The upshot of this large sieve inequality is that for many nearby pairs $a',a$ there is an algebraic relation between the sequences $b \mapsto g(ab)$ and $b \mapsto g(a'b)$, namely that one has an identity of the form
$$ g(a' \cdot) = \eps_{aa'} g(a \cdot) \gamma_{aa'}$$
where $\eps_{aa'} \colon \Z \to G$ is a ``smooth'' polynomial map and $\gamma_{aa'} \colon \Z \to G$ is a ``rational'' polynomial map; see~\eqref{gepq} for a precise statement.  This can be viewed as an assertion that the map $g$ is ``approximately dilation-invariant'' in some weak sense.  This turns out to imply a non-trivial lack of two-dimensional equidistribution for the map
$$ (a,a',b,b') \mapsto (g(ab) \Gamma, g(ab') \Gamma, g(a'b) \Gamma, g(a'b') \Gamma)$$
which is incompatible with the non-abelian nature of $G$ thanks to a commutator argument of Furstenberg and Weiss~\cite{furstenberg-weiss}; see Section~\ref{typeII-nonabelian-sec}.  This resolves the non-abelian case.  In the abelian case, one can replace the maps $g$ by the ordinary polynomials $P$, and one can then proceed by adapting the arguments by the first two authors in~\cite{matomaki-shao} to show that $e(P(n))$ necessarily ``pretends'' to be like a character $n^{iT}$, which resolves the abelian type $II$ minor arc case.  Combining all these cases yields Theorem~\ref{discorrelation-thm}.

\subsubsection{The result on Gowers norms}

 The proof of Theorem~\ref{thm_gowers} (in Section~\ref{gowers-sec}) requires in addition to Theorem~\ref{discorrelation-thm} and the inverse theorem for the Gowers norms also a construction of pseudorandom majorants for ($W$-tricked versions of) $\Lambda$ and $d_k$ over \emph{short intervals} $(X,X+H]$. By this we mean functions $\nu_1,\nu_2$ that majorize the functions $\Lambda,d_k$ (after $W$-tricking and suitable normalization), and such that  $\nu_i-1$ restricted to $(X,X+H]$ is Gowers uniform. In the case of long intervals (that is,  $H=X$), the existence of such majorants is well known from works of Green and the third author~\cite{green-tao-AP} and Matthiesen~\cite{matthiesen-linear}. Fortunately, it turns out that the structure of these well-known majorants as type I sums of small ``level'' enables us to show that they work as majorants also over short intervals $(X,X+H]$; see Lemmas~\ref{le_pseudo} and~\ref{le_pseudoinshort}. These lemmas combined with the implementation of the $W$-trick (which in the case of $d_k$ requires additionally two simple lemmas, namely Lemmas~\ref{le_wtrick} and~\ref{le_upperbound}) leads to the proof of Theorem~\ref{thm_gowers}.

\begin{remark} In this remark we discuss the obstructions to improving the value of $\theta$ in the various components of Theorem~\ref{discorrelation-thm}.  In most of these results, the primary obstruction arises (roughly speaking) from portions of $\mu$, $\Lambda$, or $d_k$ that look something like
\begin{equation}\label{1na}
1_{(X^{\alpha_1},2X^{\alpha_1}]} * \dots * 1_{(X^{\alpha_m},2X^{\alpha_m}]}
\end{equation}
for various tuples $(\alpha_1,\dots,\alpha_m)$ of positive real numbers that add up to $1$.  More specifically:
\begin{itemize}
\item[(a)]  For the $\theta=5/8$ results in Theorem~\ref{discorrelation-thm}(i)--(iii), the primary obstruction arises from convolutions~\eqref{1na} with $(\alpha_1,\dots,\alpha_m)$ equal to $(1/4,1/4,1/4,1/4)$, when correlated against characters $n^{iT}$ with $T \asymp X^{O(1)}$, as this lies just outside the reach of our twisted major arc type $I$ and type $II$ estimates when $\theta$ goes below $5/8$.  This obstruction was already implicitly observed by Zhan~\cite{zhan}.
\item[(b)]  For the $\theta=3/5$ result in Theorem~\ref{discorrelation-thm}(iv), the primary obstruction are convolutions~\eqref{1na} with $(\alpha_1,\dots,\alpha_m)$ equal to  $(2/5,1/5,1/5,1/5)$ or $(1/5,1/5,1/5,1/5,1/5)$, when correlated against ``minor arc'' nilsequences, such as $e(\alpha n)$ for some minor arc $\alpha$.  Such convolutions become just out of reach of our type $I$, type $II$, and type $I_2$ estimates when $\theta$ goes below $3/5$.  This obstruction was already observed in~\cite{MatoTera}.
\item[(c)]  For the $\theta=1/3$ result in Theorem~\ref{discorrelation-thm}(iii), the primary obstruction is of a different nature from the preceding cases: it is that our treatment of minor arcs in this case relies crucially on the ability to partition the neighborhood of a hyperbola into arithmetic progressions (see Theorem~\ref{decomp}), and this partition is no longer available in any useful form once $\theta$ goes below $1/3$.
\item[(d)]  For the $\theta=5/9$ result in Theorem~\ref{discorrelation-thm}(iii), the primary obstruction arises from convolutions~\eqref{1na} with $(\alpha_1,\dots,\alpha_m)$ equal to $(1/3,1/3,1/3)$, when correlated against minor arc nilsequences, for reasons similar to those in the previous case (c).
\end{itemize}
\end{remark}

\subsection{Acknowledgments}

KM was supported by Academy of Finland grant no. 285894. XS was supported by NSF grant DMS-1802224. TT was supported by a Simons Investigator grant, the James and Carol Collins Chair, the Mathematical Analysis \& Application Research Fund Endowment, and by NSF grant DMS-1764034. JT was supported by a Titchmarsh Fellowship,  Academy of Finland grant no. 340098, and funding from European Union's Horizon
Europe research and innovation programme under Marie Sk\l{}odowska-Curie grant agreement No
101058904. We are greatly indebted to Maksym Radziwi{\l}{\l} for many helpful discussions during the course of this project and would like to thank Lilian Matthiesen for discussions concerning~\cite{MatthiesenGenFour}. We are grateful to the anonymous referee for a careful reading of the paper and for numerous helpful comments and corrections, and to James Leng for a correction regarding Proposition \ref{corr-crit}.

\subsection{Notation}\label{notation-sec}

The parameter $X$ should be thought of as being large.

We use $Y \ll Z$, $Y = O(Z)$, or $Z \gg Y$ to denote the estimate $|Y| \leq CZ$ for some constant $C$.  If we wish to permit this constant to depend (possibly ineffectively) on one or more parameters we shall indicate this by appropriate subscripts, thus for instance $O_{\eps,A}(Z)$ denotes a quantity bounded in magnitude by $C_{\eps,A} Z$ for some quantity $C_{\eps,A}$ depending only on $\eps,A$.  We write $Y \asymp Z$ for $Y \ll Z \ll Y$. When working with $d_k$, all implied constants are permitted to depend on $k$.  We also write $y \sim Y$ to denote the assertion $Y < y \leq 2Y$.

If $x$ is a real number (resp. an element of $\R/\Z$), we write $e(x) \coloneqq e^{2\pi i x}$ and let $\|x\|_{\R/\Z}$ denote the distance of $x$ to the nearest integer (resp. zero).

We use $1_E$ to denote the indicator of an event $E$, thus $1_E$ equals $1$ when $E$ is true and $0$ otherwise.  If $S$ is a set, we write $1_S$ for the indicator function $1_S(n) \coloneqq 1_{n \in S}$.

Unless otherwise specified, all sums range over natural number values, except for sums over $p$ which are understood to range over primes.  We use $d|n$ to denote the assertion that $d$ divides $n$, $(n,m)$ to denote the greatest common divisor of $n$ and $m$, $n = a \ (q)$ to denote the assertion that $n$ and $a$ have the same residue mod $q$, and $f*g(n) \coloneqq \sum_{d|n} f(d) g(n/d)$ to denote the Dirichlet convolution of two arithmetic functions $f,g \colon \N \to \C$.

The \emph{height} of a rational number $a/b$ with $a,b$ coprime is defined as $\max(|a|, |b|)$.

\section{Basic tools}

\subsection{Total variation}

The notion of maximal summation defined in~\eqref{maximal-sum} interacts well with the notion of total variation, which we now define.

\begin{definition}[Total variation]\label{tv-def}  Given any function $f: P \to \C$ on an arithmetic progression $P$, the \emph{total variation norm} $\|f\|_{\TV(P)}$ is defined by the formula
$$ \|f\|_{\TV(P)} \coloneqq \sup_{n \in P} |f(n)| + \sup_{n_1 < \dots < n_k} \sum_{j=1}^{k-1} |f(n_{j+1})-f(n_j)|$$
where the second supremum ranges over all increasing finite sequences $n_1 < \dots < n_k$ in $P$ and all $k \geq 1$.  We remark that in this finitary setting one can simply take $n_1,\dots,n_k$ to be the elements of $P$ in increasing order, if one wishes.
We adopt the convention that $\|f\|_{\TV(P)}=0$ when $P$ is empty.  For any natural number $q \geq 1$, we also define
$$ \|f\|_{\TV(P;q)} \coloneqq \sum_{a \in \Z/q\Z} \|f\|_{\TV(P \cap (a+q\Z))}.$$
\end{definition}

Informally, if $f$ is bounded in $\TV(P;q)$ norm, then $f$ does not vary much on each residue class modulo $q$ in $P$.  From the fundamental theorem of calculus we see that if $f \colon I \to \C$ is a continuously differentiable function then
\begin{equation}\label{tvp}
 \|f\|_{\TV(P)} \ll \sup_{t \in I} |f(t)| + \int_I |f'(t)|\ dt
\end{equation}
for all arithmetic progressions $P$ in $I$.  Also, from the identity $ab-a'b' = (a-a')b + (b-b')a'$ we see that
\begin{equation}\label{tv-prod}
\| fg \|_{\TV(P;q)} \ll \|f\|_{\TV(P;q)} \|g\|_{\TV(P;q)}
\end{equation}
for any functions $f,g \colon P \to \C$ defined on an arithmetic progression, and any $q \geq 1$.

We can now record some basic properties of maximal summation:

\begin{lemma}[Basic properties of maximal sums]\label{basic-prop}\
\begin{itemize}
\item[(i)]  (Triangle inequalities) For any subprogression $P'$ of an arithmetic progression $P$, and any $f \colon P \to \C$ we have
$$ \left| \sum_{n \in P} f(n) 1_{P'}(n) \right|^* = \left| \sum_{n \in P'} f(n) \right|^* \leq \left| \sum_{n \in P} f(n) \right|^*$$
and
$$ \left|\sum_{n \in P} f(n)\right| \leq\left |\sum_{n \in P} f(n)\right|^* \leq \sum_{n \in P} |f(n)|.$$
If $P$ can be partitioned into two subprogressions as $P = P_1 \uplus P_2$, then
\begin{equation}\label{pp1}
\left |\sum_{n \in P} f(n)\right|^* \leq \left|\sum_{n \in P_1} f(n)\right|^* + \left|\sum_{n \in P_2} f(n)\right|^*.
\end{equation}
Finally, the map $f \mapsto |\sum_{n \in P} f(n)|^*$ is a seminorm.
\item[(ii)]  (Local stability) If $x_0 \in \R$, $H > 0$, and $f \colon \Z \to \C$, then
$$ \left|\sum_{x_0 < n \leq x_0+H} f(n)\right|^* \leq \frac{2}{H} \int_{x_0-H/2}^{x_0+H/2} \left|\sum_{x < n \leq x+H} f(n)\right|^*\ dx.$$
\item[(iii)]  (Summation by parts) Let $P$ be an arithmetic progression, and let $f,g \colon P \to \C$ be functions.  Then we have
\begin{equation}\label{np}
 \left| \sum_{n \in P} f(n) g(n) \right|^* \leq \|g\|_{\TV(P)} \left| \sum_{n \in P} f(n) \right|^*
\end{equation}
and more generally
\begin{equation}\label{npq}
 \left| \sum_{n \in P} f(n) g(n) \right|^* \leq \|g\|_{\TV(P;q)} \left| \sum_{n \in P} f(n) \right|^*
\end{equation}
for any $q \geq 1$.
\end{itemize}
\end{lemma}

\begin{proof}
The claims (i) all follow easily the triangle inequality and the observation that the intersection of two arithmetic progressions is again an arithmetic progression; for instance,~\eqref{pp1} follows from the observation that any subprogression $P'$ of $P$ is partitioned into subprogressions $P' \cap P_1, P' \cap P_2$ of $P_1, P_2$ respectively.  To prove (ii), we observe from (i) that for any $0 < t < H/2$ we have
\begin{align*}
\left|\sum_{x_0 < n \leq x_0+H} f(n)\right|^*  &\leq \left|\sum_{x_0 < n \leq x_0+H/2} f(n)\right|^* +  \left|\sum_{x_0+H/2 < n \leq x_0+H} f(n)\right|^*\\
&\leq \left|\sum_{x_0-t < n \leq x_0-t+H} f(n)\right|^* +  \left|\sum_{x_0+t < n \leq x_0+t+H} f(n)\right|^*
\end{align*}
and the claim then follows by averaging in $t$.

To prove the first claim~\eqref{np} of (iii), it will suffice by the monotonicity properties of total variation and maximal sums to show that
\begin{equation}
\label{eq:P'claim}
\left| \sum_{n \in P'} f(n) g(n) \right| \leq \|g\|_{\TV(P')} \left| \sum_{n \in P'} f(n) \right|^*
\end{equation}
for all subprogressions $P'$ of $P$.  Clearly we may assume $P'$ is non-empty. If we order the elements of $P'$ as $n_1 < n_2 < \dots < n_k$, then from summation by parts we have
$$ \sum_{n \in P'} f(n) g(n) = \sum_{j=1}^{k-1} (g(n_j) - g(n_{j+1})) \sum_{i=1}^j f(n_i) + g(n_k) \sum_{i=1}^k f(n_i).$$
Since each segment $\{n_1,\dots,n_j\}$ of $P'$ is again a subprogression of $P'$, we have from the triangle inequality that
$$ \left|\sum_{n \in P'} f(n) g(n)\right| \leq \sum_{j=1}^{k-1} |g(n_j) - g(n_{j+1})| \left| \sum_{n \in P'} f(n) \right|^* + |g(n_k)| \left| \sum_{n \in P'} f(n) \right|^*$$
and the claim~\eqref{eq:P'claim} now follows from Definition~\ref{tv-def}. Thus~\eqref{np} holds.  To prove the second claim~\eqref{npq}, partition $P$ into subprogressions $P \cap (a+q\Z)$, apply~\eqref{np} to each subprogression, and sum using (i).
\end{proof}

\subsection{Vinogradov lemma}

If $P \colon \Z \to \R/\Z$ is a polynomial of degree $d$, and $I$ is an interval of length $|I| \geq 1$, we define the \emph{smoothness norm}
$$ \| P \|_{C^\infty(I)} \coloneqq \sup_{0 \leq j \leq d} \sup_{n \in I} |I|^j \| \partial^j_1 P(n) \|_{\R/\Z}$$
where $\partial_1$ is the difference operator $\partial_1 P(n) \coloneqq P(n) - P(n-1)$.  We remark that this definition deviates very slightly from that in~\cite[Definition 2.7]{green-tao-ratner}; in particular, we allow the index $j$ to equal zero and we allow $n$ to range over $I$ rather than being set to the origin. We use the same notation $\|P\|_{C^{\infty}(I)}$ for a polynomial $P \colon \Z \to \R$ after reducing its coefficients modulo $1$.

The following lemma asserts, roughly speaking, that a polynomial $P$ is (somewhat) equidistributed unless it is smooth.

\begin{lemma}[Vinogradov lemma]\label{vin}  Let $0 < \eps, \delta < 1/2$, $d \geq 0$, and let $P \colon \Z \to \R/\Z$ be a polynomial of degree at most $d$. Let $I$ be an interval of length $|I| \geq 1$, and suppose that
$$ \| P(n) \|_{\R/\Z} \leq \eps$$
for at least $\delta |I|$ integers $n \in I$.  Then either $\delta \ll_d \eps$, or else one has
$$ \| qP \|_{C^\infty(I)} \ll_d \delta^{-O_d(1)} \eps $$
for some integer $1 \leq q \ll_d \delta^{-O_d(1)}$.
\end{lemma}

\begin{proof} By applying a translation, we may assume that $I$ takes the form $(0,N]$ for some $N \geq 1$.  We may also assume $\eps \leq \delta/2$, since we are clearly done otherwise.  We may now invoke~\cite[Lemma 4.5]{green-tao-ratner} to conclude that there exists $1 \leq q \ll_d \delta^{-O_d(1)} \eps$ such that
\begin{equation}\label{sup1}
\sup_{1 \leq j \leq d} \sup_{n \in I} |I|^j \| q \partial^j_1 P(n) \|_{\R/\Z} \ll_d \delta^{-O_d(1)} \eps.
\end{equation}
This is almost what we want, except that we have to also control the $j=0$ contribution.  But from hypothesis we have at least one $n_0 \in I$ such that $\|P(n_0)\|_{\R/\Z} \leq \eps$, and from~\eqref{sup1} we have $\| q \partial_1 P(n) \|_{\R/\Z} \ll_d \delta^{-O_d(1)} |I|^{-1} \eps$ for all $n \in I$.  From the triangle inequality we then conclude that
$$ \| qP(n) \|_{\R/\Z} \ll_d \delta^{-O_d(1)} \eps$$
for all $n \in I$, and the claim follows.
\end{proof}

The following handy corollary of Lemma~\ref{vin} asserts, roughly speaking, that if many dilates of a polynomial are smooth, then the polynomial itself is smooth.

\begin{corollary}[Concatenating dilated smoothness]\label{smooth-dilate}  Let $0 < \delta < 1/2$, $d \geq 0$, and let $P \colon \Z \to \R/\Z$ be a polynomial of degree at most $d$.  Let $A \geq 1$, let $I$ be an interval with $|I| \geq 2A$, and suppose that
\begin{equation}\label{pdil}
 \| P(a \cdot) \|_{C^\infty(\frac{1}{a} I)} \leq \frac{1}{\delta}
\end{equation}
for at least $\delta A$ integers $a$ in $[A,2A]$, where $\frac{1}{a} I \coloneqq \{ \frac{t}{a}: t \in I \}$ is the dilate of $I$ by $\frac{1}{a}$.  Then either $|I| \ll_d \delta^{-O_d(1)} A$, or else one has
$$ \| q P \|_{C^\infty(I)} \ll_d \delta^{-O_d(1)}$$
for some integer $1 \leq q \ll_d \delta^{-O_d(1)}$.
\end{corollary}

\begin{proof}  We allow all implied constants to depend on $d$.  We may assume that $|I| \geq C \delta^{-C} A$ for a large constant $C$ depending on $d$, as the claim is immediate otherwise.

We now claim that for each $0 \leq j \leq d$ that there exists a decomposition
\begin{equation}\label{epq}
 P = P_j + Q_j
\end{equation}
where $P_j \colon \Z \to \R/\Z$ is a polynomial of degree at most $d$ with
\begin{equation}\label{paj}
 \| q_j P_j \|_{C^\infty(I)} \ll \delta^{-O(1)}
\end{equation}
for some $1 \leq q_j \ll\delta^{-O(1)}$, and $Q_j \colon \Z \to \R/\Z$ is a polynomial of degree at most $j$.  For $j=d$ one can simply set $P_d=0$ and $Q_d=P$.  Now suppose by downward induction that $0 \leq j < d$ and the claim has already been proven for $j+1$.  From~\eqref{paj} (for $P_{j+1}$) we have
$$ \| q_{j+1} P_{j+1} \|_{C^\infty(I)} \ll \delta^{-O(1)}.$$
Routine Taylor expansion then gives
$$ \| q_{j+1} P_{j+1}(a \cdot) \|_{C^\infty(\frac{1}{a} I)} \ll \delta^{-O(1)}$$
for all $a \in [A,2A]$, thus by~\eqref{pdil} and the triangle inequality we have
$$ \| q_{j+1} Q_{j+1}(a \cdot) \|_{C^\infty(\frac{1}{a} I)} \ll \delta^{-O(1)}$$
for $\geq \delta A$ choices of $a \in [A,2A]$.

Now write $Q_{j+1}(n) = \alpha_{j+1} \binom{n}{j+1} + Q_j(n)$ where $Q_j$ is of degree at most $j$.  Taking $j+1$-fold derivatives, we see that
$$ \| a^{j+1} q_{j+1} \alpha_{j+1} \|_{\R/\Z} \ll \delta^{-O(1)} (A/|I|)^{j+1}$$
for $\geq \delta A$ choices of $a \in [A,2A]$.  Applying Lemma~\ref{vin} to the polynomial $a \to a^{j+1}q_{j+1}\alpha_{j+1}$ (and recalling that $|I|/A \geq C \delta^{-C}$ for a suitably large $C$ by assumption), we conclude that there is $1 \leq q \ll \delta^{O(1)}$ such that
$$ \| q (\cdot)^{j+1} q_{j+1} \alpha_{j+1} \|_{C^\infty([A,2A])} \ll \delta^{-O(1)} (A/|I|)^{j+1} $$
and hence on taking $j+1$-fold derivatives
$$ \| (j+1)! q q_{j+1} \alpha_{j+1} \|_{\R/\Z} \ll \delta^{-O(1)} |I|^{-j-1}.$$
If one then sets $q_j \coloneqq (j+1)! qq_{j+1}$ and $P_j(n) \coloneqq P_{j+1}(n) + \alpha_{j+1} \binom{n}{j+1}$, we obtain the decomposition~\eqref{epq}, and~\eqref{paj} follows from the triangle inequality.  This closes the induction.  Applying the claim with $j=0$, we obtain the corollary.
\end{proof}

\subsection{Equidistribution on nilmanifolds}\label{nilmanifold-sec}

We now recall some of the basic notation and results from~\cite{green-tao-ratner} concerning equidistribution of polynomial maps on nilmanifolds.

\begin{definition}[Filtered group] \label{def:FiltGroup}  Let $d \geq 1$. A \emph{filtered group} is a group $G$ (which we express in multiplicative notation $G = (G,\cdot)$ unless explicitly indicated otherwise) equipped with a filtration $G_\bullet = (G_i)_{i=0}^\infty$ of nested groups $G \geq G_0 \geq G_1 \geq \dots$ such that $[G_i,G_j] \leq G_{i+j}$ for all $i,j \geq 0$.  We say that this group has degree at most $d$ if $G_i$ is trivial for all $i>d$.  Given a filtered group of degree at most $d$, a \emph{polynomial map} $g \colon \Z \to G$ from $\Z$ to $G$ is a map of the form $g(n) = g_0 g_1^{\binom{n}{1}} \dots g_d^{\binom{n}{d}}$ where $g_i \in G_i$ for all $0 \leq i \leq d$; the collection of such maps will be denoted $\Poly(\Z \to G)$.
\end{definition}

The well-known Lazard--Leibman theorem (see e.g.,~\cite[Proposition 6.2]{green-tao-ratner}) asserts that $\Poly(\Z \to G)$ is a group under pointwise multiplication; also, from~\cite[Corollary 6.8]{green-tao-ratner} we see that if $g \colon \Z \to G$ is a polynomial map then so is $n \mapsto g(an+b)$ for any integers $a,b$.

If $G$ is a simply connected nilpotent Lie group, we write $\log G$ for the Lie algebra.  From the Baker--Campbell--Hausdorff formula\footnote{The reader may consult~\cite[Appendix B]{MRTTZ} for more details on the use of the Baker--Campbell--Hausdorff formula in the context of quantitative nilmanifold theory.} (see e.g.~\cite[Theorem 3.3]{hall}) we see that the exponential map $\exp \colon \log G \to G$ is a homeomorphism and hence has an inverse $\log \colon G \to \log G$.

\begin{definition}[Filtered nilmanifolds] \label{def:filtNilman}  Let $d, D \geq 1$ and $0 < \delta < 1$.  A \emph{filtered nilmanifold} $G/\Gamma$ of degree at most $d$, dimension $D$, and complexity at most $1/\delta$ consists of the following data:
\begin{itemize}
\item A filtered simply connected nilpotent Lie group $G$ of dimension $D$ equipped with a filtration $G_\bullet = (G_i)_{i=0}^\infty$ of degree at most $d$, with $G_0=G_1=G$ and all $G_i$ closed connected subgroups of $G$.
\item A lattice (i.e., a discrete cocompact subgroup $\Gamma$) of $G$, with the property that $\Gamma_i \coloneqq \Gamma \cap G_i$ is a lattice of $G_i$ for all $i \geq 0$.
\item  A linear basis $X_1,\dots,X_D$ (which we call a \emph{Mal'cev basis}) of $\log G$.
\end{itemize}
Furthermore we assume the following axioms:
\begin{itemize}
\item[(i)] For all $1 \leq i,j \leq D$ we have $[X_i,X_j] = \sum_{i,j < k \leq D} c_{ijk} X_k$ for some rational numbers $c_{ijk}$ of height at most $1/\delta$.
\item[(ii)]  For all $0 \leq i \leq d$, the vector space $G_i$ is spanned by the $X_j$ with $D - \dim G_i < j \leq D$.
\item[(iii)]  We have $\Gamma = \{ \exp(n_1 X_1) \dotsm \exp(n_D X_D): n_1,\dots,n_D \in \Z \}$.
\end{itemize}
It is easy to see that $G/\Gamma$ has the structure of a smooth compact $D$-dimensional manifold, which we equip with a probability Haar measure $d\mu_{G/\Gamma}$.  We define the metric $d_G$ on $G$ to be the largest right-invariant metric such that $d_G( \exp(t_1 X_1) \dotsm \exp(t_D X_D), 1) \leq \sup_{1 \leq i \leq D} |t_i|$ for all $t_1,\dots,t_D \in \R$.  We then define a metric $d_{G/\Gamma}$ on $G/\Gamma$ by the formula $d_{G/\Gamma}(x, y) \coloneqq \inf_{g\Gamma = x, h \Gamma = y} d_G(g,h)$.  The Lipschitz norm of a function $F \colon G/\Gamma \to \C$ is defined to be the quantity
$$ \sup_{x \in G/\Gamma} |F(x)| + \sup_{x,y \in G/\Gamma: x \neq y} \frac{|F(x)-F(y)|}{d_{G/\Gamma}(x,y)}.$$

A \emph{horizontal character} $\eta$ associated to a filtered nilmanifold is a continuous homomorphism $\eta \colon G \to \R$ that maps $\Gamma$ to the integers.

An element $\gamma$ of $G$ is said to be \emph{$M$-rational} for some $M \geq 1$ if one has $\gamma^r \in \Gamma$ for some natural number $1 \leq r \leq M$.  A subnilmanifold $G'/\Gamma'$ of $G/\Gamma$ (thus $G'$ is a closed connected subgroup of $G$ with $\Gamma'_i \coloneqq G'_i \cap \Gamma$ cocompact in $G'_i$ for all $i$) is said to be \emph{$M$-rational} if each element $X'_1,\dots,X'_{\dim G'}$ of the Mal'cev basis associated to $G$ is a linear combination of the $X_i$ with all coefficients rational of height at most $M$.

A \emph{rational subgroup} $G'$ of complexity at most $1/\delta$ is a closed connected subgroup of $G$ with the property that $\log G'$ admits a linear basis consisting of $\dim G'$ vectors of the form $\sum_{i=1}^D a_i X_i$, where each $a_i$ is a rational of height at most $1/\delta$.
\end{definition}

It is easy to see that every horizontal character takes the form $\eta(g) = \lambda( \log g)$ for some linear functional $\lambda \colon \log G \to \R$ that annihilates $\log [G,G]$ and maps $\log \Gamma$ to the integers.  From this one can verify that the number of horizontal characters of Lipschitz norm at most $1/\delta$ is at most $O_{d,D}( \delta^{-O_{d,D}(1)} )$.

From several applications of Baker--Campbell--Hausdorff formula we see that if $G$ has degree at most $d$ and $\gamma_1, \gamma_2 \in G$ are $M$-rational, then $\gamma_1 \gamma_2$ is $O_d(M^{O_d(1)})$-rational.

We have the following basic dichotomy between equidistribution and smoothness:

\begin{theorem}[Quantitative Leibman theorem]\label{qlt}  Let $0 < \delta < 1/2$, let $d,D \geq 1$, let $I$ be an interval with $|I| \geq 1$, and let $G/\Gamma$ be a filtered nilmanifold of degree at most $d$, dimension at most $D$, and complexity at most $1/\delta$.  Let $F \colon G/\Gamma \to \C$ be Lipschitz of norm at most $1/\delta$ and of mean zero (i.e., $\int_{G/\Gamma} F\ d\mu_{G/\Gamma} = 0$).  Suppose that $g \colon \Z \to G$ is a polynomial map with
$$ \Big|\sum_{n \in I} F(g(n)\Gamma)\Big|^* \geq \delta |I|.$$
Then there exists a non-trivial horizontal character $\eta \colon G \to \R/\Z$ of Lipschitz norm $O_{d,D}(\delta^{-O_{d,D}(1)})$ such that
$$ \| \eta \circ g\|_{C^\infty(I)} \ll_{d,D} \delta^{-O_{d,D}(1)}.$$
\end{theorem}

\begin{proof} By applying a translation we may assume $I = (0,N]$ for some $N \geq 1$.  The claim now follows from~\cite[Theorem 3.5]{shao-teravainen}.
\end{proof}

Let $G/\Gamma$ be a filtered nilmanifold of dimension $D$ and complexity at most $1/\delta$, and let $G'$ be a rational subgroup of complexity at most $1/\delta$.  In~\cite[Proposition A.10]{green-tao-ratner} it is shown that $G'/\Gamma'$ can be equipped with the structure of a filtered nilmanifold of complexity $O_{d,D}(\delta^{-O_{d,D}(1)})$, where $\Gamma' \coloneqq \Gamma \cap G'$, $G'_i \coloneqq G_i \cap G'$, and the metrics $d_G, d_{G'}$ are comparable on $G'$ up to factors of $O_{d,D}(\delta^{-O_{d,D}(1)})$; one can view $G'/\Gamma'$ as a subnilmanifold of $G/\Gamma$.

One can easily verify from basic linear algebra and the Baker--Campbell--Hausdorff formula that the following groups are rational subgroups of $G$ of complexity $O_{d,D}(\delta^{-O_{d,D}(1)})$:
\begin{itemize}
\item The groups $G_i$ in the filtration for $0 \leq i \leq d$.
\item The kernel $\ker \eta$ of any horizontal character $\eta$ of Lipschitz norm $O_{d,D}(\delta^{-O_{d,D}(1)})$.
\item The center $Z(G) = \{ \exp(X): X \in \log G; [X,Y] = 0 \,\,\forall Y \in \log G \}$ of $G$.
\item The intersection $G' \cap G''$ or commutator $[G',G'']$ of two rational subgroups $G',G''$ of $G$ of complexity $O_{d,D}(\delta^{-O_{d,D}(1)})$.
\item The product $G' N$ of two rational subgroups $G',N$ of $G$ of complexity $O_{d,D}(\delta^{-O_{d,D}(1)})$, with $N$ normal.
\end{itemize}

We can quotient out a filtered nilmanifold by a normal subgroup to obtain another filtered nilmanifold, with polynomial control on complexity:

\begin{lemma}[Quotienting by a normal subgroup]\label{quotient-normal}  Let $G/\Gamma$ be a filtered nilmanifold of degree at most $d$, dimension $D$ and complexity at most $1/\delta$.  Let $N$ be a normal rational subgroup of $G$ of complexity at most $1/\delta$, and let $\pi \colon G \mapsto G/N$ be the quotient map.  Then $\pi(G)/\pi(\Gamma)$ can be given the structure of a filtered nilmanifold of degree at most $d$, dimension $D - \dim N$, and complexity $O_{d,D}(\delta^{-O_{d,D}(1)})$, such that
\begin{equation}\label{gh}
d_{\pi(G)}( \pi(g), \pi(h) ) \asymp_{d,D} \delta^{-O_{d,D}(1)} \inf_{n \in \mathbb{N}} d_G(g, nh)
\end{equation}
for any $g,h \in G$.
\end{lemma}

\begin{proof}  We allow all implied constants to depend on $d,D$.  Let $\tilde \pi \colon \log G \to \log G / \log N \equiv \log(G/N)$ be the quotient map of $\log G$ by the Lie algebra ideal $\log N$, then $\pi \circ \exp = \exp \circ \tilde \pi$.  For each $0 \leq i \leq d$, the vectors $\tilde \pi(X_j)$ for $D - \dim G_i < j \leq D$ span the linear subspace $\tilde \pi(\log G_i)$ of $\log(G/N)$, and the linear relations between those vectors are are generated by $O(1)$ equations with coefficients rational of height $O(\delta^{-O(1)})$.  From this and linear algebra we may find a basis $\tilde X_1,\dots,\tilde X_{\dim(G/N)}$ of $\log(G/N)$ such that for each $0 \leq i \leq d$, $\tilde \pi(\log G_i)$ is the span of $\tilde X_j$ for $\dim(G/N) - \dim \tilde \pi(\log G_i) < j \leq \dim(G/N)$, and each $\tilde X_j$ is a linear combination of the
$\tilde \pi(X_1),\dots,\tilde \pi(X_D)$ with coefficients rational of height $O(\delta^{-O(1)})$.  Meanwhile, $\pi(\Gamma)$ is generated by $\pi(X_1),\dots,\pi(X_D)$.  From this and the Baker--Campbell--Hausdorff formula we see that the basis $\tilde X_1,\dots,\tilde X_{\dim(G/N)}$ is a \emph{$O(\delta^{-O(1)})$-rational weak basis} for $\pi(G)/\pi(\Gamma)$ in the sense of~\cite[Definition A.7]{green-tao-ratner}.  Applying~\cite[Proposition A.9]{green-tao-ratner} to this weak basis, we obtain a Mal'cev basis that gives $\pi(G)/\pi(\Gamma)$ the structure of a filtered nilmanifold with the stated degree, dimension, and complexity.  It remains to establish the bound~\eqref{gh}.  By right translation invariance we can take $g$ to be the identity.  For the upper bound, it suffices (since $\pi$ is $N$-invariant) to show that
$$ d_{\pi(G)}(1, \pi(h) ) \ll \delta^{-O(1)} d_G(1, h),$$
but this follows from the fact that $\tilde \pi \colon \log G \to \tilde \pi(\log G)$ has operator norm $O(\delta^{-O(1)})$ when using the $X_1,\dots,X_D$ basis for $\log G$ and the $\tilde X_1,\dots,\tilde X_{\dim(G/N)}$ basis for $\tilde \pi(\log G)$ to define norms.

Now we need to establish the lower bound.  By~\cite[Lemma A.4]{green-tao-ratner} it suffices to show that
$$ \| Y \| \gg \delta^{-O(1)} \inf_{Y' \in \tilde \pi^{-1}(Y)} \|Y'\|$$
for any $Y \in \tilde \pi(\log G)$, where again we use the norm given by the $X_1,\dots,X_D$ basis for $\log G$ and the $\tilde X_1,\dots,\tilde X_{\dim(G/N)}$.  But this is easily verified for each $Y = \tilde X_i$, and the claim then follows by linearity.
\end{proof}

A \emph{central frequency} is a continuous homomorphism $\xi \colon Z(G) \to \R$ which maps $Z(G) \cap \Gamma$ to the integers $\Z$ (that is to say, a horizontal character on $Z(G)$, or a Fourier character of the central torus $Z(G) / (Z(G) \cap \Gamma)$).  A function $F \colon G/\Gamma \to \C$ is said to \emph{oscillate with central frequency} $\xi$ if one has the identity
$$ F(zx) = e(\xi(z)) F(x)$$
for all $x \in G/\Gamma$ and $z \in Z(G)$.  As with horizontal characters, the number of central frequencies $\xi$ of Lipschitz norm at most $1/\delta$ is $O_{d,D}(\delta^{-O_{d,D}(1)})$.  If $\xi$ is such a central frequency, one can readily verify that the kernel $\ker \xi$ is a rational normal subgroup of $G$ of complexity $O_{d,D}(\delta^{-O_{d,D}(1)})$.

We have the following convenient decomposition\footnote{The decomposition in~\cite{green-tao-ratner} uses the action of the vertical group $G_d$ (which is a subgroup of the center $Z(G)$) rather than the entire center, but the arguments are otherwise nearly identical.  One can think of Proposition~\ref{central} as a slight refinement of~\cite[Lemma 3.7]{green-tao-ratner}, in that the components exhibit central oscillation rather than merely vertical oscillation.} (cf.,~\cite[Lemma 3.7]{green-tao-ratner}):

\begin{proposition}[Central Fourier approximation]\label{central}  Let $d,D \geq 1$ and $0 < \delta < 1$.  Let $G/\Gamma$ be a filtered nilmanifold of degree at most $d$, dimension $D$, and complexity at most $1/\delta$.  Let $F \colon G/\Gamma \to \C$ be a Lipschitz function of norm at most $1/\delta$.  Then we can decompose
$$ F = \sum_\xi F_\xi + O(\delta)$$
where $\xi$ ranges over central frequencies of Lipschitz norm at most $O_{d,D}(\delta^{-O_{d,D}(1)})$, and each $F_\xi$ has Lipschitz norm
$O_{d,D}(\delta^{-O_{d,D}(1)})$ and oscillates with central frequency $\xi$.  Furthermore, if $F$ has mean zero, then so do all of the $F_\xi$.
\end{proposition}

\begin{proof}  We allow all implied constants to depend on $d,D$.  Since $Z(G) / (Z(G) \cap \Gamma)$ is an abelian filtered nilmanifold of complexity $O(\delta^{-O(1)})$, it can be identified with a torus $\R^m/\Z^m$, where $m=O(1)$ and the metric on $Z(G)$ is comparable to the metric on $\R^m$ up to factors of $O(\delta^{-O(1)})$; the identification of $\log Z(G)$ with $\R^m$ induces a logarithm map $\log \colon Z(G) \to \R^m$ and an exponential map $\exp \colon \R^m \to Z(G)$.  Central frequencies $\xi$ can then be identified with elements $k_\xi$ of $\Z^m$, with $\xi(z) = k_\xi \cdot \log(z)$ for any $z \in Z(G)$.

Let $\varphi: \R^m \to \R$ be a fixed bump function (depending only on $m$) that equals $1$ at the origin, and let $R>1$ be a parameter to be chosen later.  For any central frequency $\xi$, we set
$$ F_\xi(x) \coloneqq \varphi(k_\xi/R) \int_{\R^m/\Z^m} F(zx) e(-\xi(z))\ dz$$
where $dz$ is Haar probability measure on the torus $\R^m/\Z^m$, which acts centrally on $G/\Gamma$ in the obvious fashion.  It is easy to see that $F_\xi$ has Lipschitz norm $O(\delta^{-O(1)})$, oscillates with central frequency $\xi$, and vanishes unless $\xi$ has Lipschitz norm $O( \delta^{-O(1)} R^{O(1)} )$; also, if $F$ has mean zero, then so do all of the $F_\xi$.  From the Fourier inversion formula we have
$$ \varphi(k_\xi/R) = \int_{\R^m} \hat \varphi(y) e( k_\xi \cdot y/R )\ dy = \int_{\R^m} \hat \varphi(y) e( \xi(\exp(y/R)) )\ dy ,$$
where $\hat \varphi(y) \coloneqq \int_{\R^m} \varphi(\zeta) e(-\zeta \cdot y)\ d\zeta$, as well as the Fourier inversion formula on the torus,
$$ \sum_\xi F_\xi(x) = \int_{\R^m} \hat \varphi(y) F( \exp(y/R) x )\ dy.$$
On the other hand, from the Lipschitz nature of $F$ we have
$$ F( \exp(y/R) x )  = F(x) + O( \delta^{-O(1)} |y| / R ).$$
Since $\hat \varphi$ is rapidly decreasing and has total integral $1$, we obtain
$$ F = \sum_\xi F_\xi + O(\delta^{-O(1)}/R),$$
and the claim follows by choosing $R = O(\delta^{-O(1)})$ suitably.
\end{proof}

Next we shall recall a fundamental factorization theorem for polynomial sequences. Before we can state it, we need to define a few notions.

\begin{definition}[Smoothness, total equidistribution, rationality]  Let $G/\Gamma$ be a filtered nilmanifold, $g \in \Poly(\Z \to G)$ be a polynomial sequence, $I \subset \R$ be an interval of length $|I| \geq 1$, and $M>0$.
\begin{itemize}
\item[(i)] We say that $g$ is \emph{$(M,I)$-smooth} if one has
$$ d_G(g(n), 1_G) \leq M; \quad d_G(g(n), g(n-1)) \leq M/|I|$$
for all $n \in I$.
\item[(ii)]  We say that $g$ is \emph{totally $1/M$-equidistributed} in $G/\Gamma$ on $I$ if one has
$$ \left| \frac{1}{|P|} \sum_{n \in P} F(g(n)\Gamma) - \int_{G/\Gamma} F \right| \leq \frac{1}{M} \|F\|_{\Lip}$$
whenever $F \colon G/\Gamma \to \C$ is Lipschitz and $P$ is an arithmetic progression in $I$ of cardinality at least $|I|/M$.
\item[(iii)]  We say that $g$ is \emph{$M$-rational} if there exists $1 \leq r \leq M$ such that $g(n)^r \in \Gamma$ for all $n \in \Z$.
\end{itemize}
\end{definition}

From Taylor expansion and the Baker--Campbell--Hausdorff formula it is not difficult to see that if $G/\Gamma$ has degree at most $d$ and $g$ is $M$-rational, then the map $n \mapsto g(n) \Gamma$ is $q$-periodic for some period $1 \leq q \ll_d M^{O_d(1)}$.

\begin{lemma}\label{factor-simple}
Let $d,D \geq 1$ and $0 < \delta < 1$.  Let $G/\Gamma$ be a filtered nilmanifold of degree at most $d$, dimension $D$, and complexity at most $1/\delta$.  Let $g \in \Poly(\Z \to G)$, and let $I$ be an interval with $|I| \geq 1$.  Suppose that
\begin{equation}\label{etag-smooth}
 \|\eta \circ g\|_{C^{\infty}(I)} \leq 1/\delta
\end{equation}
for some non-trivial horizontal character $\eta : G \rightarrow \R/\Z$ of Lipschitz norm at most $1/\delta$.
Then there is a decomposition $g = \eps g' \gamma$ into polynomial maps $\eps, g', \gamma \in \Poly(\Z \to G)$ such that
\begin{itemize}
\item[(i)] $\eps$ is $(\delta^{-O_{d,D}(1)},I)$-smooth;
\item[(ii)]  $g'$ takes values in $G' = \ker\eta$;
\item[(iii)] $\gamma$ is $\delta^{-O_{d,D}(1)}$-rational.
\end{itemize}
\end{lemma}

\begin{proof}
This is a slight variant of~\cite[Lemma 7.9]{green-tao-ratner}, the main difference being that our hypothesis~\eqref{etag-smooth} involves $\eta \circ g$ rather than $\eta \circ g_2$ (where $g_2$ is the nonlinear part of $g$). The argument in the proof of~\cite[Lemma 7.9]{green-tao-ratner} can be modified in an obvious manner as follows. By translation we may assume that $I = [1, |I|]$. Let $\psi: G\rightarrow \R^D$ be the Mal'cev coordinate map.  Suppose that
$$ \psi(g(n)) = t_0 + \binom{n}{1}t_1 + \binom{n}{2}t_2 + \cdots + \binom{n}{d}t_d $$
for some $t_0, t_1,\cdots,t_d \in \R^D$ with $\psi^{-1}(t_i) \in G_i$. Our assumption on $\|\eta \circ g\|_{C^{\infty}(I)}$ implies that for some $k \in \Z^D$ with $|k| \leq \delta^{-1}$, we have
$$ \|k \cdot t_i\|_{\R/\Z} \ll \delta^{-O_{d,D}(1)}/|I| $$
for each $1 \leq i \leq d$. Choose $u_i \in \R^D$ with $\psi^{-1}(u_i) \in G_i$, such that
$$ k \cdot u_i \in \Z, \ \ |t_i - u_i| \ll \delta^{-O_{d,D}(1)}/|I|. $$
Then choose $v_i \in \R^D$ with $\psi^{-1}(v_i) \in G_i$, all of whose coordinates are rationals over some denominator $\ll \delta^{-O_{d,D}(1)}$, such that
$$ k \cdot u_i = k \cdot v_i $$
for each $1 \leq i \leq d$.
Define $\eps,\gamma$ by
$$ \psi(\eps(n)) = t_0 + \sum_{i=1}^d \binom{n}{i}(t_i - u_i), \ \ \psi(\gamma(n)) = \sum_{i=1}^d \binom{n}{i}v_i, $$
and then define $g'$ by
$$ g'(n) = \eps(n)^{-1}g(n) \gamma(n)^{-1}. $$
One can verify that they satisfy the desired properties.
\end{proof}

\begin{theorem}[Factorization theorem]\label{factor}  Let $d,D \geq 1$ and $0 < \delta < 1$.  Let $G/\Gamma$ be a filtered nilmanifold of degree at most $d$, dimension $D$, and complexity at most $1/\delta$.  Let $g \in \Poly(\Z \to G)$ and $A>0$, and let $I$ be an interval with $|I| \geq 1$.  Then there exists an integer $1/\delta \leq M \ll_{A,D,d} \delta^{-O_{A,D,d}(1)}$ and a decomposition $g = \eps g' \gamma$ into polynomial maps $\eps, g', \gamma \in \Poly(\Z \to G)$ such that
\begin{itemize}
\item[(i)] $\eps$ is $(M,I)$-smooth;
\item[(ii)]  There is an $M$-rational subnilmanifold $G'/\Gamma'$ of $G/\Gamma$ such that $g'$ takes values in $G'$ and is totally $1/M^A$-equidistributed on $I$ in $G'/\Gamma'$, and more generally in $G'/\Gamma''$ whenever $\Gamma''$ is a subgroup of $\Gamma'$ of index at most $M^A$; and
\item[(iii)] $\gamma$ is $M$-rational.
\end{itemize}
\end{theorem}

\begin{proof}  See~\cite[Theorem 1.19]{green-tao-ratner} (after rounding $I$ to integer endpoints and translating to be of the form $[1,N]$).  The additional requirement in (ii) that one has equidistribution in the larger nilmanifolds $G'/\Gamma''$ is not stated in~\cite[Theorem 1.19]{green-tao-ratner} but follows easily from the proof, the point being that if a sequence $g' \in \Poly(\Z \to G')$ fails to be totally $1/M^A$-equidistributed in $G'/\Gamma''$, then one has $\|\eta \circ g' \|_{C^\infty(I)} \ll_{d,D} M^{O_{d,D}(A)}$ for some non-trivial horizontal character $\eta$ on $G'/\Gamma''$ of Lipschitz norm $O_{d,D}(M^{O_{d,D}(A)})$, which on multiplying $\eta$ by the index of $\Gamma''$ in $\Gamma'$ also gives $\|\eta' \circ g' \|_{C^\infty(I)} \ll_{d,D} M^{O_{d,D}(A)}$ for some non-trivial horizontal character $\eta'$ on $G'/\Gamma'$ of Lipschitz norm $O_{d,D}(M^{O_{d,D}(A)})$.  As a consequence, one can replace all occurrences of $G'/\Gamma'$ in the proof of~\cite[Theorem 1.19]{green-tao-ratner} with $G'/\Gamma''$ with only negligible changes to the arguments.
\end{proof}

We will also need a multidimensional version of this theorem.

\begin{theorem}[Multidimensional factorization theorem]\label{multi-factor}  Let $t,d,D \geq 1$ and $0 < \delta < 1$.  Let $G/\Gamma$ be a filtered nilmanifold of degree at most $d$, dimension $D$, and complexity at most $1/\delta$.  Let $g \in \Poly(\Z^t \to G)$ and $A>0$, and let $I_1,\dots,I_t$ intervals with $|I_1|, \dots, |I_t| \geq C \delta^{-C}$, for some $C$ that is sufficiently large depending on $t,d,D,A$.  Then there exists an integer $1/\delta \leq M \ll_{A,D,d,t} \delta^{-O_{A,D,d,t}(1)}$ and a decomposition $g = \eps g' \gamma$ into polynomial maps $\eps, g', \gamma \in \Poly(\Z^t \to G)$ such that
\begin{itemize}
\item[(i)] $\eps$ is $(M,I_1 \times \dots \times I_t)$-smooth, in the sense that $d_G(\eps(n), 1_G) \leq M$ and $d_G( \eps(n+e_i), 1_G) \leq M/|I_i|$ for all $n \in I_1 \times \dots \times I_t$ and $i=1,\dots,t$, where $e_1,\dots,e_t$ are the standard basis of $\Z^d$;
\item[(ii)]  There is an $M$-rational subnilmanifold $G'/\Gamma'$ of $G/\Gamma$ such that $g'$ takes values in $G'$ and is totally $1/M^A$-equidistributed in $G'/\Gamma'$, and more generally in $G'/\Gamma''$ whenever $\Gamma''$ is a subgroup of $\Gamma'$ of index at most $M^A$, in the sense that
$$ \left| \frac{1}{|P_1 \times \dots \times P_t|} \sum_{n \in P_1 \times \dots \times P_t} F(g'(n)\Gamma) - \int_{G'/\Gamma''} F \right| \leq \frac{1}{M} \|F\|_{\Lip}$$
whenever $F \colon G/\Gamma \to \C$ is Lipschitz and for each $i=1,\dots,t$, $P_i$ is an arithmetic progression in $I_i$ of cardinality at least $|I_i|/M$; and
\item[(iii)] $\gamma$ is $M$-rational, in the sense that there exists $1 \leq r \leq M$ such that
$g(n)^r \in \Gamma$ for all $n \in \Z^t$.
\end{itemize}
\end{theorem}

\begin{proof} This follows from~\cite[Theorem 10.2]{green-tao-ratner}, after implementing the corrections in~\cite{green-tao-nilratner-erratum}, and the modifications indicated in the proof of Theorem~\ref{factor}.
\end{proof}

As a first application of Theorem~\ref{factor}, we can obtain a criterion for correlation between nilsequences with a non-trivial central frequency:

\begin{proposition}[Correlation criterion]\label{corr-crit}  Let $d,D \geq 1$ and $0 < \delta < 1$.  Let $G/\Gamma$ be a filtered nilmanifold of degree at most $d$, dimension $D$, and complexity at most $1/\delta$, whose center $Z(G)$ is one-dimensional.  Let $g_1,g_2 \in \Poly(\Z \to G)$, let $I$ be an interval with $|I| \geq 1$, and let $F \colon G/\Gamma \to \C$ be Lipschitz of norm at most $1/\delta$ and having a non-zero central frequency $\xi$.  Suppose that one has the correlation
$$ \left|\sum_{n \in I} F(g_1(n) \Gamma) \overline{F}(g_2(n) \Gamma)\right|^* \geq \delta |I|.$$
Then at least one of the following holds:
\begin{itemize}
\item[(i)]  There exists a non-trivial horizontal character $\eta \colon G \to \R/\Z$ of Lipschitz norm $O_{d,D}(\delta^{-O_{d,D}(1)})$ such that $\| \eta \circ g_i \|_{C^\infty(I)} \ll_{d,D} \delta^{-O_{d,D}(1)}$ for some $i\in \{1,2\}$.
\item[(ii)] There exists a factorization
$$ g_2 = \eps (\phi \circ g_1) \gamma$$
where $\eps$ is $(O_{d,D}(\delta^{-O_{d,D}(1)}),I)$-smooth, $\phi \colon G \to G$ is a Lie group automorphism whose associated Lie algebra isomorphism $\log \phi \colon \log G \to \log G$ has matrix coefficients that are all rational of height $O_{d,D}(\delta^{-O_{d,D}(1)})$ in the Mal'cev basis $X_1,\dots,X_D$ of $\log G$, and $\gamma$ is $O_{d,D}(\delta^{-O_{d,D}(1)})$-rational.
\end{itemize}
\end{proposition}

\begin{proof}  We allow all implied constants to depend on $d,D$.  The product of the filtered nilmanifold $G/\Gamma$ with itself is again a filtered nilmanifold $(G \times G)/(\Gamma \times \Gamma)$, with the obvious filtration $(G \times G)_i \coloneqq G_i \times G_i$ and Mal'cev basis $(X_i,0), (0,X_i)$, $i=1,\dots,D$.  This product filtered nilmanifold has degree at most $d$, dimension $2D$, and complexity at most $O(\delta^{-O(1)})$.  The pair $(g_1,g_2)$ can be then viewed as an element of $\Poly(\Z \to G \times G)$.  If we let $F \otimes \overline{F} \colon (G \times G)/(\Gamma \times \Gamma) \to \C$ be the function
$$ F \otimes \overline{F}(x_1, x_2) \coloneqq F(x_1) \overline{F}(x_2)$$
then $F$ is Lipschitz with norm $O( \delta^{-O(1)})$ and one has
\begin{equation}\label{fgg}
 \left|\sum_{n \in I} F \otimes \overline{F}((g_1,g_2)(n) (\Gamma \times \Gamma))\right|^* \geq \delta |I|.
\end{equation}
Let $A>1$ be sufficiently large depending on $d,D$.  Applying Theorem~\ref{factor} to $(g_1,g_2)$ (with $\delta$ replaced by $\delta^A$) we can find $\delta^{-A} \leq M \ll_A \delta^{-O_A(1)}$ and a factorization
\begin{equation}\label{g1}
 (g_1,g_2) = (\eps_1,\eps_2) (g'_1, g'_2) (\gamma_1,\gamma_2)
\end{equation}
where $\eps_1, g'_1, \gamma_1 \in \Poly(\Z \to G_1)$, $\eps_2, g'_2, \gamma_2 \in \Poly(\Z \to G_2)$ such that
\begin{itemize}
\item[(i)] $(\eps_1,\eps_2)$ is $(M,I)$-smooth;
\item[(ii)]  There is an $M$-rational subnilmanifold $G'/\Gamma'$ of $(G \times G)/(\Gamma \times \Gamma)$ such that $(g'_1,g'_2)$ takes values in $G'$ and is totally $1/M^A$-equidistributed in $G'/\Gamma''$ for any subgroup $\Gamma''$ of $\Gamma'$ of index at most $M^A$; and
\item[(iii)] $(\gamma_1,\gamma_2)$ is $M$-rational.
\end{itemize}
We caution that $G'$ is a subgroup of $G \times G$ rather than $G$. From~\eqref{fgg} we thus have
$$
 \left|\sum_{n \in I} F \otimes \overline{F}( (\eps_1,\eps_2)(n) (g'_1,g'_2)(n) (\gamma_1,\gamma_2)(n) (\Gamma \times \Gamma))\right|^* \geq \delta |I|.$$
Since $(\gamma_1,\gamma_2)$ is $M$-rational, it is $O(M^{O(1)})$-periodic, and then by the pigeonhole principle (and Lemma~\ref{basic-prop}(i)) we can thus find $M$-rational $(\gamma_1^0, \gamma_2^0) \in G \times G$ such that
$$
\left |\sum_{n \in I} F \otimes \overline{F}( (\eps_1,\eps_2)(n) (g'_1,g'_2)(n) (\gamma^0_1,\gamma_2^0) (\Gamma \times \Gamma))\right|^* \gg M^{-O(1)} |I|.$$
By shifting $\gamma_1^0, \gamma_2^0$ by elements of $\Gamma$ if necessary we may assume that they lie at distance $O(M^{O(1)})$ from the identity.
If we partition $I$ into subintervals $J$ of length $\asymp M^{-C} |I|$ for some large constant $C$, we see from the pigeonhole principle (and Lemma~\ref{basic-prop}(i)) that we can find one such $J$ for which
$$
 \left|\sum_{n \in J} F \otimes \overline{F}( (\eps_1,\eps_2)(n) (g'_1,g'_2)(n) (\gamma^0_1,\gamma_2^0) (\Gamma \times \Gamma))\right|^* \gg M^{-O(1)} |J|.$$
As $(\eps_1,\eps_2)$ is $(M,I)$-smooth, it fluctuates by $O(M^{1-C})$ on $J$ and stays a distance $O(M)$ from the identity, hence by the Lipschitz nature of $F \otimes \overline{F}$ we conclude (for $C=O(1)$ large enough) that there exists $(\eps_1^0, \eps_2^0) \in G \times G$ at distance $O(M)$ from the identity such that
$$
 \left|\sum_{n \in J} F \otimes \overline{F}( (\eps^0_1,\eps^0_2) (g'_1,g'_2)(n) (\gamma^0_1,\gamma_2^0) (\Gamma \times \Gamma))\right|^* \gg M^{-O(1)} |J|.$$
Allowing implied constants to depend on $C$, we conclude that
$$
 \left|\sum_{n \in I} F \otimes \overline{F}( (\eps^0_1,\eps^0_2) (g'_1,g'_2)(n) (\gamma^0_1,\gamma_2^0) (\Gamma \times \Gamma))\right|^* \gg M^{-O(1)} |I|.$$

From the Baker--Campbell--Hausdorff formula and the $M$-rationality of $(\gamma_1^0, \gamma_2^0)$, we see that $(\gamma^0_1,\gamma_2^0) (\Gamma \times \Gamma) (\gamma^0_1, \gamma_2^0)^{-1}$ can be covered by $O(M^{O(1)})$ cosets of $\Gamma \times \Gamma$, and conversely.  Thus if we set
$$ \Gamma'' \coloneqq G' \cap (\Gamma \times \Gamma) \cap (\gamma^0_1,\gamma_2^0) (\Gamma \times \Gamma) (\gamma^0_1, \gamma_2^0)^{-1}$$
then $G' \cap (\Gamma \times \Gamma)$ can be covered by $O(M^{O(1)})$ cosets of $\Gamma''$, thus $\Gamma''$ is a subgroup of $G' \cap (\Gamma \times \Gamma)$ of index $O(M^{O(1)})$ such that
\begin{equation}\label{go}
 \Gamma'' (\gamma^0_1,\gamma_2^0) \subset (\gamma^0_1,\gamma_2^0) (\Gamma \times \Gamma).
\end{equation}
Indeed, one can take $\Gamma''$ to be the intersection of $G' \cap (\Gamma \times \Gamma)$ and $(\gamma^0_1,\gamma_2^0) (\Gamma \times \Gamma) (\gamma^0_1, \gamma_2^0)^{-1}$.  One can then write the above claim as
$$
 \left|\sum_{n \in I} F'( (g'_1,g'_2)(n) \Gamma'')\right|^* \gg M^{-O(1)} |I|$$
where $F' \colon G' / \Gamma'' \to \mathbb{C}$ is defined by
$$ F'( (g'_1,g'_2) \Gamma'' ) \coloneqq F(\eps^0_1 g'_1 \gamma^0_1 \Gamma) \overline{F}(\eps^0_2 g'_2 \gamma^0_2 \Gamma) $$
for any $(g'_1,g'_2) \in G'$, with the inclusion~\eqref{go} ensuring that this function is well-defined.  Since $F$ is Lipschitz with norm $1/\delta \leq M$, and $\eps^0_1, \gamma^0_1, \eps^0_2, \gamma^0_2$ are at distance $O(M^{O(1)})$ from the identity, this function is Lipschitz with norm $O(M^{O(1)})$, hence by total equidistribution of $(g'_1,g'_2)$ we conclude (for $A$ large enough) that
\begin{equation}\label{mo}
 \left|\int_{G'/\Gamma''} F'\right| \gg M^{-O(1)}.
\end{equation}

Suppose that the projection $K \coloneqq \{ g_1 \in G: (g_1,g_2) \in G' \hbox{ for some } g_2 \in G \}$ is not all of $G$.  This is a proper closed connected subgroup of $G$ with
$$ \log K = \{ X \in \log G: (X,Y) \in \log G' \hbox{ for some } Y \in \log G \};$$
thus $\log K$ is the projection of $\log G'$ to $\log G$.  Since $\log G'$ is $M^{O(1)}$-rational, $\log K$ is also. Hence there exists a non-trivial horizontal character $\eta \colon G \to \R/\Z$ of Lipschitz norm $O(M^{O(1)})$ that annihilates $K$, so in particular $\eta(g'_1(n)) = 0$ for all $n$.  From~\eqref{g1} we then have
$$ \eta(g_1(n)) = \eta(\eps_1(n)) + \eta(\gamma_1(n)).$$
Since $\gamma_1$ is $M$-rational, $M\eta(\gamma_1(n)) = 0$.  Thus if we replace $\eta$ by $M\eta$ we have
$$ \eta(g_1(n)) = \eta(\eps_1(n)).$$
Since $(\eps_1,\eps_2)$ is $(M,I)$ smooth we thus conclude that
$$ \| \eta \circ g_1 \|_{C^\infty(I)} \ll M^{O(1)}$$
and we are in conclusion (i) of the proposition.  Thus we may assume that the projection $\{ g_1 \in G: (g_1,g_2) \in G' \hbox{ for some } g_2 \in G \}$ is all of $G$.  Similarly we may assume that $\{ g_2 \in G: (g_1,g_2) \in G' \hbox{ for some } g_1 \in G \}$ is all of $G$.

Now suppose that the slice $H \coloneqq \{ g \in G: (g,1) \in G'\}$ is non-trivial.  This is a non-trivial closed connected subgroup of $G$; since $G'$ is normalized by itself, we conclude that $H$ is normalized by $K$, and is hence normal in $G$ since\footnote{We thank James Leng for pointing out the need to perform the $K=G$ reduction before analyzing $H$, which was not done in a previous version of this manuscript.} $K=G$. By considering the final non-trivial element of the series $H$, $[H,G]$, $[[H,G],G]$, $\dots$, we conclude that $H$ contains a non-trivial closed connected \emph{central} subgroup of $G$.  Since $Z(G)$ is one-dimensional, we conclude that $H$ contains $Z(G)$.  In particular, $G'$ contains $Z(G) \times \{1\}$.

Since $F$ has central frequency $\xi$, we see that
$$ F'((z,1)(g_1,g_2)) = e(\xi \cdot z) F'(g_1,g_2)$$
for all $z \in Z(G)$.  By invariance of Haar measure, this implies that
$$ \int_{G'/\Gamma''} F' = e(\xi \cdot z) \int_{G'/\Gamma''} F'.$$
Since $\xi$ is non-trivial, this implies that $\int_{G'/\Gamma''} F' =0$, contradicting~\eqref{mo}.  Thus the slice $\{ g \in G: (g,1) \in G'\}$ is trivial.  Similarly the slice $\{ g \in G: (1,g) \in G' \}$ is trivial.

Applying Goursat's lemma, we now conclude that $G'$ takes the form
$$ G' = \{ (g_1, \phi(g_1)): g_1 \in G \}$$
for some group automorphism $\phi \colon G \to G$.  Since $G'$ is a $O(M^{O(1)})$-rational subgroup of $G \times G$, $\phi$ must be a Lie group automorphism whose associated Lie algebra automorphism $\log \phi \colon \log G \to \log G$ has coefficients that are rational of height $O(M^{O(1)})$ in the Mal'cev basis.  Since $(g'_1(n),g'_2(n))$ takes values in $G'$, we have
$$ g'_2(n) = \phi(g'_1(n))$$
and hence by~\eqref{g1} and some rearranging
$$ g_2(n) = \eps_2(n) \phi(\eps_1(n))^{-1} \phi(g_1(n)) \phi(\gamma_1(n))^{-1} \gamma_2(n).$$
It is then routine to verify that conclusion (ii) of the proposition holds.
\end{proof}

As a consequence of this criterion, we can establish the following large sieve inequality for nilsequences, which is a more quantitative variant of the one in~\cite[Proposition 4.11]{MRTTZ}.

\begin{proposition}[Large sieve]\label{large-sieve}
Let $d,D \geq 1$ and $0 < \delta < 1$.  Let $G/\Gamma$ be a filtered nilmanifold of degree at most $d$, dimension $D$, and complexity at most $1/\delta$, whose center $Z(G)$ is one-dimensional.  Let $g_1,\dots,g_K \in \Poly(\Z \to G)$, let $I$ be an interval with $|I| \geq 1$, and let $F \colon G/\Gamma \to \C$ be Lipschitz of norm at most $1/\delta$ and having a non-zero central frequency $\xi$.  Suppose that there is a function $f \colon \Z \to \C$ with $\sum_{n \in I} |f(n)|^2 \leq \frac{1}{\delta} |I|$ such that
\begin{equation}\label{kp}
 \left|\sum_{n \in I} f(n) \overline{F}(g_i(n) \Gamma)\right|^* \geq \delta |I|
\end{equation}
for all $i=1,\dots,K$. Then at least one of the following holds:
\begin{itemize}
\item[(i)]  There exists a non-trivial horizontal character $\eta \colon G \to \R/\Z$ of Lipschitz norm $O_{d,D}(\delta^{-O_{d,D}(1)})$ such that $\| \eta \circ g_i \|_{C^\infty(I)} \ll_{d,D} \delta^{-O_{d,D}(1)}$ for $\gg_{d,D} \delta^{O_{d,D}(1)} K$ values of $i=1,\dots,K$.
\item[(ii)] For $\gg_{d,D} \delta^{O_{d,D}(1)} K^2$ pairs $(i,j) \in \{1,\dots,K\}^2$, there exists a factorization
$$ g_i = \eps_{ij} g_j \gamma_{ij}$$
where $\eps_{ij}$ is $(O_{d,D}(\delta^{-O_{d,D}(1)}),I)$-smooth and $\gamma_{ij}$ is $O_{d,D}(\delta^{-O_{d,D}(1)})$-rational.
\end{itemize}
\end{proposition}

\begin{proof}  We allow implied constants to depend on $d,D$. From~\eqref{kp} one can find progressions $P_i \subset I$ for $i=1,\dots,K$ such that
$$ \left|\sum_{n \in I} f(n) 1_{P_i}(n) \overline{F}(g_i(n) \Gamma)\right| \geq \delta |I|$$
and thus
$$ \left|\sum_{i=1}^K \theta_i \sum_{n \in I} f(n) 1_{P_i}(n) \overline{F}(g_i(n) \Gamma)\right| \geq \delta K |I|$$
for some complex numbers $\theta_i$ with $|\theta_i| \leq 1$.  By interchanging the sums and applying Cauchy--Schwarz, we have
$$ \left|\sum_{i=1}^K \theta_i \sum_{n \in I} f(n) 1_{P_i}(n) \overline{F}(g_i(n) \Gamma)\right|^2  \leq \frac{1}{\delta} |I|  \sum_{n \in I} \left| \sum_{i=1}^K \theta_i 1_{P_i}(n) \overline{F}(g_i(n) \Gamma)\right|^2 $$
and thus
$$\sum_{n \in I} \left| \sum_{i=1}^K \theta_i 1_{P_i}(n) \overline{F}(g_i(n) \Gamma)\right|^2 \geq \delta^{3} K^2 |I|.$$
From the triangle inequality we have
$$  \sum_{n \in I} \left|\sum_{i=1}^K \theta_i 1_{P_i}(n) \overline{F}(g_i(n) \Gamma)\right|^2
\leq \sum_{1 \leq i,j \leq K} \left|\sum_{n \in I} F(g_i(n) \Gamma) \overline{F}(g_j(n) \Gamma)\right|^*$$
and thus
$$\sum_{1 \leq i,j \leq K} \left|\sum_{n \in I} F(g_i(n) \Gamma) \overline{F}(g_j(n) \Gamma)\right|^* \geq \delta^{3} K^2 |I|.$$
The inner sum is $O(\delta^{-2} |I|)$, thus we have
$$ \left|\sum_{n \in I} F(g_i(n) \Gamma) \overline{F}(g_j(n) \Gamma)\right|^* \gg \delta^{O(1)} |I|$$
for $\gg \delta^{O(1)} K^2$ pairs $(i,j) \in \{1,\dots,K\}^2$.  For each such pair, we apply Proposition~\ref{corr-crit}.  If conclusion (i) of that proposition holds for $\gg \delta^{O(1)} K^2$ pairs $(i,j)$, then by the pigeonhole principle (noting that there are only $O(\delta^{-O(1)})$ choices for $\eta$) we obtain conclusion (i) of the current proposition.  Thus we may assume that conclusion (ii) of Proposition~\ref{corr-crit} holds for $\gg \delta^{O(1)} K^2$ pairs $(i,j) \in \{1,\dots,K\}^2$, thus we have
$$ g_i = \eps_{ij} \phi_{ij}(g_j) \gamma_{ij}$$
for all such pairs $(i,j)$, where $\eps_{ij}$ is $(O(\delta^{-O(1)}),I)$-smooth, $\gamma_{ij}$ is $O(\delta^{-O(1)})$-rational, and  $\phi_{ij} \colon G \to G$ is a Lie group automorphism whose associated Lie algebra isomorphism $\log \phi \colon \log G \to \log G$ has matrix coefficients that are all rational of height $O(\delta^{-O(1)})$ in the Mal'cev basis $X_1,\dots,X_D$ of $\log G$.  The total number of choices for $\phi_{ij}$ is $O(\delta^{-O(1)})$, so by the pigeonhole principle we may assume that $\phi_{ij} = \phi$ is independent of $i,j$.  By Cauchy--Schwarz, we may thus find $\gg \delta^{O(1)} K^3$ triples $(i,i',j) \in \{1,\dots,K\}^3$ such that
$$ g_i = \eps_{ij} \phi(g_j) \gamma_{ij}; \quad g_{i'} = \eps_{i'j} \phi(g_j) \gamma_{i'j}$$
where $\eps_{ij}, \eps_{i'j}, \gamma_{ij}, \gamma_{i'j}$ are as above.  This implies that
$$ g_i = \eps_{ij} \eps_{i'j}^{-1} g_{i'} \gamma_{i'j}^{-1} \gamma_{ij}.$$
Pigeonholing in $j$ and relabeling $i,i'$ as $i,j$, we obtain conclusion (ii) of the current proposition.
\end{proof}

\subsection{Combinatorial lemmas}
The following lemma is a standard consequence of Heath-Brown's identity.

\begin{lemma}\label{hb-identity}
Let $X \geq 2$, and let $L \in \mathbb{N}$ be fixed. We may find a collection $\mathcal{F}$ of $(\log X)^{O(1)}$ functions $f \colon \mathbb{N} \to \mathbb{R}$, such that
$$ \Lambda(n) = \sum_{f \in \mathcal{F}} f(n) $$
for each $X/2 \leq n \leq 4X$, and each $f \in \mathcal{F}$ takes the form
$$ f = a^{(1)}* \cdots * a^{(\ell)} $$
for some $\ell \leq 2L$, where $a^{(i)}$ is supported on $(N_i, 2N_i]$ for some $N_i \geq 1/2$, and each $a^{(i)}(n)$ is either $1_{(N_i, 2N_i]}(n)$, $(\log n)1_{(N_i, 2N_i]}(n)$, or $\mu(n)1_{(N_i, 2N_i]}$. Moreover, $N_1N_2\cdots N_{\ell} \asymp X$, and $N_i \ll X^{1/L}$ for each $i$ with $a^{(i)}(n) = \mu(n) 1_{(N_i, 2N_i]}(n)$. The same statement holds for $\mu$ in place of $\Lambda$ (but $(\log n)1_{(N_i, 2N_i]}(n)$ does not appear).
\end{lemma}

\begin{proof}
Using Heath-Brown's identity (see~\cite[(13.37), (13.38)]{ik} with $K = L$ and $z = (2X)^{1/L}$), we have
$$ \Lambda(n) = \sum_{1 \leq j \leq L} (-1)^{j-1} \binom{L}{j} \sum_{m_1,\ldots, m_j \leq (2X)^{1/L}} \mu(m_1) \cdots \mu(m_j) \sum_{m_1\cdots m_jn_1 \cdots n_j=n} \log n_1 $$
and
$$ \mu(n) = \sum_{1 \leq j \leq L} (-1)^{j-1} \binom{L}{j} \sum_{m_1,\ldots,m_j \leq (2X)^{1/L}} \mu(m_1) \cdots \mu(m_j) \sum_{m_1\cdots m_jn_1\cdots n_{j-1}=n} 1. $$
The conclusion follows after dyadic division of the ranges of variables.
\end{proof}

The following Shiu's bound~\cite[Theorem 1]{shiu} will be used multiple times to control sums of divisor functions in short intervals in arithmetic progressions.

\begin{lemma}\label{shiu}
Let $A \geq 1$ and $\varepsilon > 0$ be fixed. Let $X \geq H \geq X^{\varepsilon}$ and $1\leq q \leq H^{1-\varepsilon}$. Let $f$ be a non-negative multiplicative function such that $f(p^{\ell}) \leq A^{\ell}$ for every prime power $p^{\ell}$ and $f(n)\ll_{c} n^{c}$ for every $c > 0$. Then, for any integer $a$ coprime to $q$, we have
$$ \sum_{\substack{X < n \leq X+H \\ n \equiv a\pmod{q}}} f(n) \ll \frac{H}{\varphi(q) \log X} \exp\Big( \sum_{\substack{p \leq 2X \\ p\nmid q}} \frac{f(p)}{p} \Big). $$
\end{lemma}

For proving Theorem~\ref{discorrelation-thm}(iv)--(v), we need a more flexible combinatorial decomposition of the multiplicative functions $\mu, d_k$, where we introduce an extra variable $p \in (P, Q]$ in the factorization. Before stating this, let us quickly prove a lemma that will in particular allow us to write, for $P < Q \leq X^{1/(\log \log X)^2}$,
\[
1_{(n, \prod_{P < p \leq Q} p) = 1} = \sum_{\substack{d \mid (n, \prod_{P < p \leq Q} p) \\ d \leq X^{\varepsilon}}} \mu(d) + \text{  acceptable error}
\]
in our sums. This can be seen as a simple version of the fundamental lemma of the sieve that is sufficient to our needs.
\begin{lemma}
\label{le:FLS}
Let $k, r \geq 1$ and $\varepsilon > 0$ be fixed. Let $X \geq H \geq X^\varepsilon$ and $X \geq D \geq  Q > P \geq 2$. Then, for any $C \geq 1$,
\begin{equation}
\label{eq:FLSclaim}
\sum_{\substack{X < mn \leq X+H \\ p \mid m \implies p \in (P, Q] \\ m > D}} d_k(mn)^r \ll_C H\frac{(\log X)^{2k^r e^C}}{\exp(C\frac{\log D}{\log Q})}.
\end{equation}
\end{lemma}

\begin{proof}
Write $\ell = mn$ and note that since $m > D$, we have $\Omega(\ell) \geq \frac{\log D}{\log Q}$. Hence the left hand side of~\eqref{eq:FLSclaim} is
\begin{align*}
\leq \sum_{\substack{X < \ell \leq X+H \\\Omega(\ell)\geq \frac{\log D}{\log Q}}} d_2(\ell) d_k(\ell)^r \leq e^{-C\frac{\log D}{\log Q}}\sum_{X< \ell \leq X+H} e^{C\Omega(\ell)} d_2(\ell) d_k(\ell)^r \ll_C H\frac{(\log X)^{2k^r e^C}}{\exp(C\frac{\log D}{\log Q})}
\end{align*}
by Lemma~\ref{shiu}.
\end{proof}

Now we state the lemma allowing us to introduce an extra variable $p \in (P, Q]$ in the factorization. It is a slight variant of~\cite[Lemma 3.1]{MatoTera} (see also~\cite[Remark 3.2]{MatoTera}).

\begin{lemma}\label{lem:MatoTera}
Let $\eps > 0$ and $k \geq 1$ be fixed. Let $X\geq 3$, $X^{\eps} \leq H \leq X$, and let $2 \leq P < Q \leq X^{1/(\log\log X)^2}$. Write $\mathcal{P}(P, Q) = \prod_{P < p \leq Q} p$. Let $f$ be any multiplicative function satisfying $|f(n)| \leq d_k(n)$. Then for any sequence $\{\omega_n\}$ with $|\omega_n| \leq 1$, we have
$$
\sum_{\substack{X < n \leq X+H \\ (n, \mathcal{P}(P, Q)) > 1}} f(n) \omega_n = \sum_{\substack{X < prn \leq X+H \\ P < p \leq Q \\ r \leq X^{\eps/2}}} a_r f(p) f(n) \omega_{prn} + O\left( \frac{H(\log X)^{4k}}{P} + \frac{H}{\exp((\log \log X)^{2})}\right),
 $$
where $\{a_r\}$ is an explicit sequence satisfying $|a_r| \leq d_{k+1}(r)$.
\end{lemma}

\begin{proof}
This is very similar to~\cite[Remark 3.2]{MatoTera} but for completeness we provide the proof in a somewhat simpler form.

By Ramar\'e's identity
\begin{align}\label{ramare}
f(n)\omega_n 1_{(n,\mathcal{P}(P, Q))>1}=\sum_{P< p\leq Q}\sum_{pm=n}\frac{f(pm) \omega_{pm}}{\omega_{(P,Q]}(pm)}
\end{align}
where $\omega_{(P,Q]}(m)$ is the number of distinct prime divisors of $m$ on $(P,Q]$; this identity follows directly since the number of representations $n=pm$ with $P<p\leq Q$ is $\omega_{(P,Q]}(n)$.

We write $m$ uniquely as $m=m_1m_2$ with $m_1$ having all of its prime factors from $(P,Q]$ and $m_2$ having no prime factors from that interval. Summing over $n$ and then spotting the condition $(m_2, \mathcal{P}(P, Q)) = 1$ using M\"obius inversion, we see that
\begin{align}
\nonumber
\sum_{\substack{X < n \leq X+H \\ (n, \mathcal{P}(P, Q)) > 1}} f(n) \omega_n &= \sum_{P< p\leq Q}\sum_{\substack{X/p\leq m_1m_2\leq (X+H)/p\\p'\mid m_1\Longrightarrow p'\in (P,Q]\\(m_2, \mathcal{P}(P, Q)) = 1}}\frac{f(p m_1 m_2)}{\omega_{(P,Q]}(p m_1)}\omega_{m_1 m_2 p} \\
\label{eq:RamareMT}
&=\sum_{P< p\leq Q}\sum_{\substack{X/p\leq m_1 d m_2\leq (X+H)/p\\d \mid \mathcal{P}(P, Q) \\ p'\mid m_1\Longrightarrow p'\in (P,Q]}}\frac{\mu(d) f(p m_1 d m_2)}{\omega_{(P,Q]}(p m_1)}\omega_{m_1 d m_2 p}.
\end{align}

Let us show that we can restrict the summation to $dm_1 \leq X^{\varepsilon/2}$. Writing $m =dm_1$ and $n = p m_2$, we see that by Lemma~\ref{le:FLS} with $C = 4/\varepsilon$ the contribution of $dm_1 > X^{\varepsilon/2}$ is bounded by
\begin{align*}
\leq \sum_{\substack{X< mn \leq X+H \\ p \mid m \implies p \in (P, Q] \\ m > X^{\varepsilon/2}}} d_2(m) d_2(n) d_k(mn) \leq \sum_{\substack{X< mn \leq X+H \\ p \mid m \implies p \in (P, Q] \\ m > X^{\varepsilon/2}}} d_{2k}(mn)^3 \ll \frac{H}{\exp((\log \log X)^2)}.
\end{align*}

Furthermore, since in~\eqref{eq:RamareMT} all prime factors of $pd m_1$ are from $(P, Q]$, we have
\begin{equation}
\label{eq:sq-freeRep}
f(p m_1 d m_2) = f(p)f(d m_1)f(m_2) \quad \text{and} \quad \omega_{(P,Q]}(p m_1) = \omega_{(P, Q]}(m_1) + 1
\end{equation}
unless there exists a prime $q\in (P,Q]$ such that $q^2 \mid pm_1dm_2 =: \ell$. Applying Lemma~\ref{shiu}, the error introduced by making the changes~\eqref{eq:sq-freeRep} to~\eqref{eq:RamareMT} is
\begin{align*}
\ll \sum_{P< q\leq Q}\sum_{\substack{X< \ell\leq X+H\\q^2\mid \ell}}d_4(\ell) d_k(\ell) \ll \sum_{P< q\leq Q}\frac{H}{q^2}(\log X)^{4k-1} \ll \frac{H}{P}(\log X)^{4k-1}.
\end{align*}

Thus~\eqref{eq:RamareMT} equals
\begin{align*}
\sum_{\substack{X \leq p m_1 d m_2\leq X+H \\p'\mid d m_1\Longrightarrow p'\in (P,Q] \\ P < p \leq Q, dm_1 \leq X^{\varepsilon/2}}}\frac{\mu(d) f(p) f(dm_1) f(m_2)}{\omega_{(P,Q]}(m_1) + 1}\omega_{m_1 d m_2 p} + O\left(\frac{H}{\exp((\log \log X)^{2})} + \frac{H}{P}(\log X)^{4k-1} \right),
\end{align*}
and the claim follows with
\begin{align*}
a_r:=f(r) 1_{\substack{p \mid r \implies p \in (P, Q]}}\sum_{r = dm_1} \frac{\mu(d)}{\omega_{(P,Q]}(m_1)+1},
\end{align*}
\end{proof}

The following combinatorial lemma will be used to arrange each component arising from Lemma~\ref{hb-identity} into a desired form, such as a type $I$ sum, a type $II$ sum, or a type $I_2$ sum.

\begin{lemma}\label{combinatorial}  Let $\alpha_1,\dots,\alpha_k$ be nonnegative real numbers with $\sum_{i=1}^k \alpha_i = 1$ and let $\frac{1}{3} \leq \theta \leq 1$.  For any $I \subset \{1,\dots,k\}$, write $\alpha_I \coloneqq \sum_{i \in I} \alpha_i$.  Consider the following statements:
\begin{itemize}
\item[($I$)]  One has $\alpha_i \geq 1 - \theta$ for some $1 \leq i \leq k$.
\item[($I_2^{\mathrm{maj}}$)]  One has $\alpha_{\{i,j\}} \geq 1 - \theta$ for some $1 \leq i < j \leq k$.
\item[($I_2$)]  One has $\alpha_{\{i,j\}} \geq \frac{3}{2}(1-\theta)$ for some $1 \leq i < j \leq k$.
\item[($II^{\mathrm{maj}}$)]  There exists a partition $\{1,\dots,k\} = I \uplus J \uplus J'$ such that $2\theta-1 \leq \alpha_I \leq 4\theta-2$ and $|\alpha_J - \alpha_{J'}| \leq 2\theta-1$.
\item[($II^{\mathrm{min}}$)] There exists a partition $\{1,\dots,k\} = J \uplus J'$ such that $|\alpha_J - \alpha_{J'}| \leq 2\theta-1$ (or equivalently, $\alpha_J, \alpha_{J'} \in [1-\theta,\theta]$; or equivalently, $\alpha_J \in [1-\theta, \theta]$).
\end{itemize}
Then the following claims hold.
\begin{itemize}
\item[(i)]  Suppose that $\theta = 5/8$.  Then at least one of ($I$) or ($II^{\mathrm{maj}}$) holds.
\item[(ii)]  Suppose that $\theta \geq 3/5$.  Then at least one of ($I$), ($I_2$), or ($II^{\mathrm{min}}$) holds.
\item[(iii)]  Suppose that $\theta = 7/12$.  Then at least one of ($I$), ($I_2^{\mathrm{maj}}$), or ($II^{\mathrm{maj}}$) holds.
\item[(iv)] Suppose that $k = 5$ and $\theta = 11/20$. Then at least one of ($I_2^{\mathrm{maj}}$) or ($II^{\mathrm{maj}}$) holds.
\item[(v)] Suppose that $k \in \{3,4\}$ and $\theta \geq 1/2$. Then ($I_2^{\mathrm{maj}}$) holds.
\item[(vi)]  Suppose that $k=3$ and $\theta \geq 5/9$ or $k=2$ and $\theta \geq 1/3$. Then ($I_2$) holds.
\end{itemize}
\end{lemma}

\begin{remark} The different conclusions ($I$), ($I_2^{\mathrm{maj}}$), ($I_2$), ($II^{\mathrm{maj}}$), ($II^{\mathrm{min}}$) in Lemma~\ref{combinatorial} correspond to different types of sums that behave well on intervals $(X,X+H]$ with $H$ much larger than $X^\theta$:
\begin{itemize}
\item Exponents obeying ($I$) correspond to ``type $I$ sums'' which behave well for both major and minor arc correlations.
\item Exponents obeying ($I_2^{\mathrm{maj}}$) correspond to ``type $I_2$ sums'' which behave well for major arc correlations.
\item Exponents obeying ($I_2$) correspond to ``type $I_2$ sums'' which behave well for both major and minor arc correlations.
\item Exponents obeying ($II^{\mathrm{maj}}$) correspond to ``type $II$ sums'' which behave well for major arc correlations.
\item Exponents obeying ($II^{\mathrm{min}}$) correspond to ``type $II$ sums'' which behave well for minor arc correlations, or for major arc correlations when one can extract a medium-sized prime factor from the sum.
\end{itemize}
\end{remark}

\begin{proof}  We first handle the easy case (vi).  If $k=2$ and $\theta \geq 1/3$, then $\frac{3}{2}(1-\theta) \leq 1$ and ($I_2$) follows simply by taking $\{i, j\} = \{1, 2\}$. If $k=3$ and $\theta \geq \frac{5}{9}$, then $\frac{3}{2}(1-\theta) \leq \frac{2}{3}$ and ($I_2$) follows by noting that the sum of the two largest of the reals $\alpha_1,\alpha_2,\alpha_3$ is necessarily at least $\frac{2}{3}$.

Now we prove (v). If $k=4$ and $\theta \geq 1/2$, then by the pigeonhole principle one of $\alpha_{\{1,2\}}$, $\alpha_{\{3,4\}}$ is at least $\frac{1}{2} \geq 1-\theta$, and we obtain ($I_2^{\mathrm{maj}}$) in this case. The case $k=3$ follows similarly, with some room to spare.

In a similar spirit in case (iv), when $k=5$ and $\theta = \frac{11}{20}$, then one of the $\alpha_i$ must be at most $\frac{1}{5}$; without loss of generality $\alpha_5 \leq \frac{1}{5}$.  Since $1-\theta = \frac{9}{20}$, we obtain ($I_2^{\mathrm{maj}}$) except when $\alpha_{\{1,2\}}, \alpha_{\{3,4\}} \leq \frac{9}{20}$, which by $\sum_{i=1}^5 \alpha_i=1$ forces $\alpha_{\{3,4\}}, \alpha_{\{1,2\}} \geq 1 - \frac{9}{20} - \frac{1}{5} = \frac{7}{20}$.  Thus $|\alpha_{\{1,2\}} - \alpha_{\{3,4\}}| \leq \frac{9}{20}-\frac{7}{20} = \frac{1}{10} = 2\theta-1$.  Also we have
$$ \alpha_5 = 1- \alpha_{1, 2} - \alpha_{3, 4} \geq 1 - \frac{9}{20} - \frac{9}{20} = \frac{1}{10} = 2\theta-1$$
and
$$ \alpha_5 \leq \frac{1}{5} = 4\theta-2$$
and so we obtain ($II^{\mathrm{maj}}$) in this case.  This establishes (iv).

In the remaining cases (i)--(iii) we assume, without loss of generality, that
\[
\alpha_1 \geq \alpha_2 \geq \dotsb \geq \alpha_k.
\]

In case (ii) when $\theta \geq 3/5$ we obtain ($I$) unless $\alpha_j < 1-\theta$ for each $j$ and ($I_2$) unless $\alpha_{\{i, j\}} < \frac{3}{2}(1-\theta) \leq \theta$ for any distinct $i, j$. But if $\alpha_{\{i, j\}} \in [1-\theta, \theta]$ for some distinct $i, j$, then we have ($II^{\mathrm{min}}$). Hence we can assume that $\alpha_{i, j} < 1-\theta$ for any distinct $i, j$. In particular, for any $j \neq 1$ we have
\[
\alpha_j \leq \frac{\alpha_1 + \alpha_j}{2} \leq \frac{1-\theta}{2} \leq 2\theta -1.
\]
Consequently there must be an index $r \in \{3, \dotsc, k\}$ such that $\alpha_1 + \sum_{j = 2}^r \alpha_j \in [1-\theta, \theta]$, and hence ($II^{\mathrm{min}}$) holds.

Let us now consider (i). Now $\theta = 5/8$ and we obtain ($I$) unless $\alpha_j < 3/8$ for every $j$ (and in particular we can assume that $k \geq 3$). Note that $2\theta -1 = 1/4$ in this case. If now $\alpha_3 > 1/4$, then $\alpha_1, \alpha_2 \in [1/4, 3/8]$ and we have ($II^{\mathrm{maj}}$) with $J = \{1\}, J' = \{2\}$, and $I = \{3, \dotsc, k\}$.

On the other hand, if $\alpha_3 \leq 1/4$, we set $J_0 = \{1\}$ and $J_0' = \{2, \dotsc, r\}$ with $r \geq 2$ the greatest integer such that $\alpha_{J_0'} < \alpha_{J_0}$. Then necessarily $|\alpha_{J_0}-\alpha_{J_0'}| \leq 1/4 = 2\theta -1$. Furthermore $\alpha_{J_0'} + \alpha_{J_0} \leq 2 \cdot \alpha_1 \leq 3/4$. If also $\alpha_{J_0'} + \alpha_{J_0} \geq 1/2$ then we have ($II^{\mathrm{maj}}$) with $J = J_0, J' = J_0'$, and $I = \{1, \dotsc, k\} \setminus (J_0 \cup J_0')$. Otherwise we add indices $j \geq r+1$ one by one to $J_0$ or $J_0'$ depending on whether $\alpha_{J_0} < \alpha_{J_0'}$ or not. We continue this process until $\alpha_{J_0} + \alpha_{J_0'} \in [1/2, 3/4]$, and we again obtain ($II^{\mathrm{maj}}$).

Let us finally turn to (iii). Now $\theta = 7/12$ and $2\theta -1 = 1/6$. We obtain ($I_2^{\mathrm{maj}}$) unless $\alpha_{\{i, j\}} < 1-\theta = 5/12$ for any distinct $i, j$. In particular we can assume that $\alpha_1 + \alpha_2 + \alpha_3 + \alpha_4 < 5/6 < 1$ and thus $k \geq 5$.

If $\alpha_5 > 1/6$, then $\alpha_{\{2, 3\}}, \alpha_{\{1, 4\}} \in [1/3, 5/12]$. Consequently $1-\alpha_{\{1, 4\}}-\alpha_{\{2, 3\}} \in [1/6, 1/3]$ and we obtain ($II^{\mathrm{maj}}$) with $J = \{1, 4\}, J' = \{2, 3\}$, and $I = \{1, \dotsc, k\} \setminus \{1, 2, 3, 4\}$.

On the other hand if $\alpha_5 \leq 2\theta-1 = 1/6$, we can argue similarly to case (i): We set $J_0 = \{1, 2\}$ and $J_0' = \{3, \dotsc, r\}$ with $r \geq 4$ the greatest integer such that $\alpha_{J_0'} \leq \alpha_{J_0}$. Then necessarily $|\alpha_{J_0} - \alpha_{J_0'}| \leq 1/6 = 2\theta-1$. Furthermore $\alpha_{J_0} + \alpha_{J_0'} \leq 2 \alpha_{1, 2} \leq 5/6$. If also $\alpha_{J_0} + \alpha_{J_0'} \geq 2/3$ then we have ($II^{\mathrm{maj}}$) with $J = J_0$ and $J' = J_0'$. Otherwise we add indices $j \geq r+1$ one by one to $J_0$ or $J_0'$ depending on whether $\alpha_{J_0} < \alpha_{J_0'}$ or not. We continue this process until $\alpha_{J_0} + \alpha_{J_0'} \in [2/3, 5/6]$, and we again obtain ($II^{\mathrm{maj}}$).
\end{proof}

\begin{remark}
\label{rem:obstructionscomb} The following counterexamples, with $\eps$ small, show that $\theta$ in the various components of Lemma~\ref{combinatorial} cannot be decreased (apart from the $k=3$ case of (v)):
\begin{itemize}
\item $\theta = 5/8-\eps$, $(\alpha_1,\dots,\alpha_k) = (1/4,1/4,1/4,1/4)$;
\item $\theta = 3/5-\eps$, $(\alpha_1,\dots,\alpha_k) \in\{(2/5,1/5,1/5,1/5), (1/5,1/5,1/5,1/5,1/5)\}$;
\item $\theta = 7/12-\eps$, $(\alpha_1,\dots,\alpha_k) = (1/6,1/6,1/6,1/6,1/6,1/6)$;
\item $\theta = 11/20-\eps$, $(\alpha_1,\dots,\alpha_k) = (1/5,1/5,1/5,1/5,1/5)$;
\item $\theta = 1/2-\eps$, $(\alpha_1,\dots,\alpha_k) = (1/4,1/4,1/4,1/4)$;
\item $\theta = 5/9-\eps$, $(\alpha_1,\dots,\alpha_k) = (1/3,1/3,1/3)$;
\item $\theta = 1/3-\eps$, $(\alpha_1,\dots,\alpha_k) = (\alpha,1-\alpha)$ for any $\alpha \in (0, 1)$.
\end{itemize}
\end{remark}

\section{Major arc estimates}\label{major-arc-sec}

In the proof of Theorem~\ref{discorrelation-thm} we shall use Theorem~\ref{inverse} below to reduce to ``major arc'' cases where more-or-less $F(g(n) \Gamma) = 1$ (or $F(g(n) \Gamma) = n^{it}$ in case of type $II$ sums). The purpose of this section is to establish the following estimates corresponding to the case $F(g(n) \Gamma) = 1$ as well as an auxiliary result (Lemma~\ref{le:BHP} below) on trilinear sums in case $F(g(n) \Gamma) = n^{it}$. 

\begin{theorem}[Major arc estimate]\label{thm:major-arc}  Let $X \geq 3$ and $X^{\theta+\eps} \leq H \leq X^{1-\eps}$ for some $0 < \theta < 1$ and $\eps > 0$.
\begin{itemize}
\item[(i)]  (Huxley type estimates) Set $\theta = 7/12$. Then, for all $A > 0$,
\begin{align*}
 \left| \sum_{X < n \leq X+H} \mu(n) \right|^* &\ll_{A,\eps} \frac{H}{\log^{A} X}
\end{align*}
 and
\begin{align*}
 \left| \sum_{X < n \leq X+H} (\Lambda(n) - \Lambda^\sharp(n)) \right|^* &\ll_{A,\eps} \frac{H}{\log^{A} X}.
\end{align*}
\item[(ii)]
Let $k \geq 2$.  Set $\theta = 1/3$ for $k=2$, $\theta=1/2$ for $k=3, 4$, $\theta =11/20$ for $k =5$, and $\theta = 7/12$ for $k \geq 6$. Then
$$
\left|\sum_{X < n \leq X+H} (d_k(n) - d^\sharp_k(n))\right|^* \ll_\eps \frac{H}{X^{c_k}} + \frac{H}{X^{\eps/1000}}$$
for some constant $c_k>0$ depending only on $k$.
\end{itemize}
\end{theorem}

We remark that if we replace the maximal sums $|\cdot|^*$ here by the ordinary sums $|\cdot|$, then the $\theta=7/12$ case of Theorem~\ref{thm:major-arc} can also be extracted after some computation from the work of Ramachandra~\cite{ramachandra} (see in particular Remarks 4, 5 of that paper), with a pseudopolynomial gain $O( \exp(-c(\log X)^{1/3} / (\log\log X)^{1/3} ) )$, while the cases $k=4, 5$ of Theorem~\ref{thm:major-arc}(ii) follow from~\cite[(4.23)]{hl-divisor}) and~\cite{hb-divisor}. Here we will provide the proofs from our viewpoint.  It may be possible to improve the error terms in (i) to be pseudopolynomial in nature even for the maximal sums, if one adjusts the approximants $\mu^\sharp, \Lambda^\sharp$ to take into account the possibility of a Siegel zero, in the spirit of~\cite[Proposition 2.2]{tt-quant}.

For the $\theta=7/12$ result, the primary obstruction arises from convolutions~\eqref{1na} with $(\alpha_1,\dots,\alpha_m)$ equal to $(1/6,1/6,1/6,1/6,1/6,1/6)$, as this lies just outside the reach of our untwisted major arc type $I$ and type $II$ estimates when $\theta$ goes below $7/12$ (cf., the third item of Remark~\ref{rem:obstructionscomb}).  This obstruction has long been known; see e.g.,~\cite{Heath-Brown}. Note that this obstruction does not arise for $k<6$, which explains the fact that better exponents than $7/12$ are available for $d_2, d_3, d_4, d_5$. The corresponding obstructions can be found in the other items of Remark~\ref{rem:obstructionscomb}.

It would probably be possible to obtain Theorem~\ref{thm:major-arc}(ii) for $\theta = 131/416 \approx 0.315$ when $k = 2$ and for $\theta = 43/96 \approx 0.448$ when $k=3$ --- corresponding to the progress in the Dirichlet divisor problem~\cite{huxley-divisor, kolesnik} --- but we do not attempt to compute this here (it requires checking that the arguments in the literature, when adapted to the Dirichlet divisor problem in an arithmetic progression, give a polynomial dependence on the common difference of the arithmetic progression, and it also does not directly improve the exponents in Theorem~\ref{discorrelation-thm}).  

Let us now  explain the strategy of the proof of Theorem~\ref{thm:major-arc}. Let $f \in \{\mu, \Lambda, d_k\}.$ By adjusting the implied constants, it suffices to show the claims with
\[
\left| \sum_{X < n \leq X+H} (f(n) - f^\sharp(n)) \right|^* \quad \text{replaced by} \quad \max_{a, q \in \mathbb{N}} \left|\sum_{\substack{X < n \leq X+H \\ n \equiv a \pmod{q}}} (f(n) - f^\sharp(n))\right|.
\]
 In the cases $f=\mu,\Lambda$ we take $H' \coloneqq X/\log^{20A} X$ and in the case $f= d_k$ we take $H' \coloneqq X^{1-1/100k}$. We use the triangle inequality to write
\begin{align}
\label{eq:fdeclongapp}
\begin{aligned}
&\left|\frac{1}{H} \sum_{\substack{X < n \leq X+H \\ n \equiv a \pmod{q}}} (f(n) - f^\sharp(n))\right| \leq \left|\frac{1}{H} \sum_{\substack{X < n \leq X+H \\ n \equiv a \pmod{q}}} f(n) - \frac{1}{H'} \sum_{\substack{X < n \leq X+H' \\ n \equiv a \pmod{q}}} f(n)\right| \\
&\quad + \left|\frac{1}{H'} \sum_{\substack{X < n \leq X+H' \\ n \equiv a \pmod{q}}} (f(n) - f^\sharp(n))\right| + \left|\frac{1}{H} \sum_{\substack{X < n \leq X+H \\ n \equiv a \pmod{q}}} f^\sharp(n) - \frac{1}{H'} \sum_{\substack{X < n \leq X+H' \\ n \equiv a \pmod{q}}} f^\sharp(n)\right|.
\end{aligned}
\end{align}
Then we show that each of the three differences on the right-hand side is small. Let us next state the required results.

To attack the second difference in~\eqref{eq:fdeclongapp}, we show in Section~\ref{ssec:longinterval} that Theorem~\ref{thm:major-arc} holds in long intervals.
\begin{proposition}[Long intervals]\label{prop:LongIntervals} Let $X \geq H_2 \geq 2$.
\begin{itemize}
\item[(i)] Let $A > 0$ and\footnote{Actually, thanks to Lemma~\ref{basic-prop}(i), it would suffice to consider the case $H_2=X$ here.} $X/\log^A X \leq H_2 \leq X$. Then
\begin{equation}
\label{eq:mulong}
\max_{a, q \in \mathbb{N}} \left| \sum_{\substack{X < n \leq X+H_2\\n\equiv a\pmod q}} \mu(n) \right| \ll_{A} \frac{H_2}{\log^{A} X}.
\end{equation}
and
\begin{equation}
\label{eq:Lambdalong}
\max_{a, q \in \mathbb{N}} \left| \sum_{\substack{X < n \leq X+H_2\\n\equiv a\pmod q}} (\Lambda(n) - \Lambda^\sharp(n)) \right| \ll_{A} \frac{H_2}{\log^{A} X}.
\end{equation}
\item[(ii)] Let $k \geq 2$ and $X^{1-\frac{1}{50k}} \leq H_2 \leq X$. Then
\begin{equation}
\label{eq:dklong}
\max_{a, q \in \mathbb{N}} \left| \sum_{\substack{X < n \leq X+H_2\\n\equiv a\pmod q}} (d_k(n) - d_k^\sharp(n))\right| \ll \frac{H_2^2}{X} \log^{k-2} X.
\end{equation}
\end{itemize}
\end{proposition}

Furthermore, using the definitions of our approximants $\Lambda^\sharp(n)$ and $d_k^\sharp(n)$ as type $I$ sums, it will be straightforward to show that the third difference on the right of~\eqref{eq:fdeclongapp} is small; in Section~\ref{ssec:approximants} we shall show the following.

\begin{lemma}[Long and short averages of approximant]\label{le:Approx} Let $X \geq H_2 \geq H_1 \geq X^{1/4} \geq 2$.
\begin{itemize}
\item[(i)] One has
\begin{align}
\max_{a, q \in \mathbb{N}}\left| \frac{1}{H_1}\sum_{\substack{X < n \leq X+H_1 \\ n \equiv a \pmod{q}}} \Lambda^\sharp(n) - \frac{1}{H_2}\sum_{\substack{X < n \leq X+H_2 \\ n \equiv a \pmod{q}}} \Lambda^\sharp(n)\right| &\ll \exp(-(\log X)^{1/10}).
\end{align}
\item[(ii)] Let $k \geq 2$. Then
\begin{equation}
\max_{a, q \in \mathbb{N}} \left| \frac{1}{H_1}\sum_{\substack{X < n \leq X+H_1 \\ n \equiv a \pmod{q}}} d_k^\sharp(n) - \frac{1}{H_2}\sum_{\substack{X < n \leq X+H_2 \\ n \equiv a \pmod{q}}} d_k^\sharp(n)\right|\ll  \frac{1}{X^{1/100}} + \frac{H_2}{X}\log^{k-2} X.
\end{equation}
\end{itemize}
\end{lemma}

Our ability to handle the first difference in~\eqref{eq:fdeclongapp} is what determines the exponent $\theta$. Concerning the first difference we prove the following proposition in Section~\ref{ssec:H1H2compPropproof}. 
\begin{proposition}[Long and short averages of arithmetic function]
\label{prop:H1H2comp}\ 
\begin{itemize}
\item[(i)] Let $X/\log^{20 A} X \geq H_2 \geq H_1 \geq X^{7/12+\varepsilon}$. Then
\begin{align*}
\max_{a, q \in \mathbb{N}}\left| \frac{1}{H_1}\sum_{\substack{X < n \leq X+H_1 \\ n \equiv a \pmod{q}}} \Lambda(n) - \frac{1}{H_2}\sum_{\substack{X < n \leq X+H_2 \\ n \equiv a \pmod{q}}} \Lambda(n)\right| &\ll_{A, \eps} \frac{1}{\log^A X}
\end{align*}
and
\begin{align*}
\max_{a, q \in \mathbb{N}}\left| \frac{1}{H_1}\sum_{\substack{X < n \leq X+H_1 \\ n \equiv a \pmod{q}}} \mu(n)\right| &\ll_{A, \eps} \frac{1}{\log^A X}.
\end{align*}
\item[(ii)]
Let $k \geq 2$.  Set $\theta = 1/3$ for $k=2$, $\theta=1/2$ for $k=3, 4$, $\theta =11/20$ for $k =5$, and $\theta = 7/12$ for $k \geq 6$. There exists $c_k > 0$ such that if $X^{1-1/(100k)} \geq H_2 \geq H_1 \geq X^{\theta+\varepsilon}$, then
\[
\max_{a, q \in \mathbb{N}} \left| \frac{1}{H_1}\sum_{\substack{X < n \leq X+H_1 \\ n \equiv a \pmod{q}}} d_k(n) - \frac{1}{H_2}\sum_{\substack{X < n \leq X+H_2 \\ n \equiv a \pmod{q}}} d_k(n) \right| \ll_{\varepsilon, k} \frac{1}{X^{c_k}} + \frac{1}{X^{\eps/1000}}
\]
\end{itemize}
\end{proposition}

Theorem~\ref{thm:major-arc} now follows from~\eqref{eq:fdeclongapp} together with Propositions~\ref{prop:H1H2comp} and~\ref{prop:LongIntervals} and Lemma~\ref{le:Approx}.

The case $k=2$ of Proposition~\ref{prop:H1H2comp}(ii) can be treated using classical methods on the Dirichlet divisor problem. In $k \geq 3$ cases of Proposition~\ref{prop:H1H2comp}(ii), we write $d_k(n) = \sum_{n = m_1 \dotsm m_k} 1$, split $m_j$ into dyadic intervals $m_j \sim M_j \asymp X^{\alpha_j}$ and classify resulting dyadic sums using Lemma~\ref{combinatorial}(iii). On the other hand in case of Proposition~\ref{prop:H1H2comp}(i) we first use Heath-Brown's identity and then Lemma~\ref{combinatorial}(iii) to classify the resulting sums.

For trilinear sums satisfying ($II^{\mathrm{maj}}$) from Lemma~\ref{combinatorial} we shall deduce in Section~\ref{ssec:BHP} the following consequence of the work of Baker, Harman and Pintz~\cite{baker-harman-pintz}. Part (ii) of the lemma will be used in handling certain type $II$ sums in Section~\ref{reduction-sec}. 

\begin{lemma} 
\label{le:BHP}
Let $1/2 \leq \theta < 1$ and $\eps > 0$. Let also $W \leq X^{\varepsilon/200}$ and $X^{\theta+\eps} \leq H_1 \leq H_2 \leq X/W^4$. Let $L, M_1 ,M_2 \geq 1$ be such that $M_j = X^{\alpha_j}$ and $L M_1 M_2 \asymp X$. Let $a_{m_1} ,b_{m_2} ,v_\ell$ be bounded by $d_2^C$ for some $C \geq 1$.

Assume that $a, q \in \mathbb{N}$, $\theta \in \{11/20, 7/12, 3/5, 5/8\}$ and that $\alpha_1, \alpha_2 > 0$ obey the bounds 
\[
|\alpha_1 - \alpha_2| \leq 2\theta-1 + \frac{\varepsilon}{100} \quad \text{and} \quad 1 - \alpha_1 - \alpha_2 \leq 4\theta-2 + \frac{\varepsilon}{100}.
\]
\begin{itemize}
\item[(i)]
If
\begin{equation}
\label{eq:BHP-supt-cond}
\max_{r \mid (a,q)}\,\, \max_{\chi \pmod{\frac{q}{(a,q)}}} \sup_{W \leq |t| \leq \frac{XW^4}{H_1}} \left| \sum_{\ell \sim L/r} \frac{v_{\ell r} \chi(\ell)}{\ell^{1/2+it}}\right| \ll_C \frac{(L/r)^{1/2}}{W^{1/3}},
\end{equation}
then
\[
\begin{split}
&\Big | \frac{1}{H_1}\sum_{\substack{X < m_1 m_2 \ell \leq X+H_1 \\ m_j \sim M_j, \ell \sim L \\ m_1 m_2 \ell \equiv a \pmod{q}}} a_{m_1} b_{m_2} v_\ell - \frac{1}{H_2} \sum_{\substack{X < m_1 m_2 \ell \leq X+H_2 \\  m_j \sim M_j, \ell \sim L \\ m_1 m_2 \ell \equiv a \pmod{q}}} a_{m_1} b_{m_2} v_\ell \Big | \ll d_3(q) \frac{\log^{O_C(1)} X}{W^{1/3}}.
\end{split}
\]
\item[(ii)]
If
\begin{equation}
\label{eq:BHP-supt-cond-2}
\max_{r \mid (a,q)}\,\, \max_{\chi \pmod{\frac{q}{(a,q)}}}  \sup_{|t| \leq \frac{XW^4}{H_1}} \left| \sum_{\ell \sim L/r} \frac{v_{\ell r} \chi(\ell)}{\ell^{1/2+it}}\right| \ll_C \frac{(L/r)^{1/2}}{W^{1/3}},
\end{equation}
then
\[
\begin{split}
&\Big | \frac{1}{H_1}\sum_{\substack{X < m_1 m_2 \ell \leq X+H_1 \\ m_j \sim M_j, \ell \sim L \\ m_1 m_2 \ell \equiv a \pmod{q}}} a_{m_1} b_{m_2} v_\ell \Big | \ll d_3(q)\frac{\log^{O_C(1)} X}{W^{1/3}}.
\end{split}
\]
\end{itemize}
\end{lemma}

For sums satisfying ($I_2^{\mathrm{maj}}$) from Lemma~\ref{combinatorial} we shall use standard methods to deduce in Section~\ref{ssec:BHP} the following lemma.

\begin{lemma} 
\label{le:typeI2} Let $\theta \in [1/2, 1)$ and $\eps > 0$. Let $W \leq X^{\varepsilon/4}$ and let $X^{\theta+\eps} \leq H_1 \leq H_2 \leq X/W^4$. Let $L, M_1 ,M_2 \geq 1$ be such that $M_j = X^{\alpha_j}$ and $LM_1 M_2 \asymp X$. Let $v_\ell$ be bounded by $d_2^C(\ell)$. 
Assume that $a, q \in \mathbb{N}$ and 
\begin{equation}
\label{eq:typeI2cond}
\alpha_1 + \alpha_2 \geq 1-\theta.
\end{equation}
Then
\[
\begin{split}
&\Big | \frac{1}{H_1}\sum_{\substack{X < m_1 m_2 \ell \leq X+H_1 \\ m_j \sim X^{\alpha_j} \\ m_1 m_2 \ell \equiv a \pmod{q}}} v_\ell - \frac{1}{H_2} \sum_{\substack{X < m_1 m_2 \ell \leq X+H_2 \\ m_j \sim X^{\alpha_j} \\ m_1 m_2 \ell \equiv a \pmod{q}}} v_\ell \Big | \ll d_3(q) \frac{\log^{O_C(1)} X}{W^{1/6}}.
\end{split}
\]
\end{lemma}

\subsection{Proof of Proposition~\ref{prop:LongIntervals}}
\label{ssec:longinterval}
The bound~\eqref{eq:mulong} follows immediately from the Siegel--Walfisz theorem~\eqref{eq:S-W} and the triangle inequality.

Before turning to the proof of~\eqref{eq:Lambdalong}, let us discuss the choice of $\Lambda^\sharp$. The prime number theorem with classical error term (see, e.g.,~\cite[Theorem 6.9]{mv}) gives
\begin{equation}\label{xlambda}
 \sum_{n \leq X} \Lambda(n) = X + O( X \exp(-c\sqrt{\log X}) ),
\end{equation}
so that if one is interested only in the correlation of $\Lambda(n)$ with a constant function, one can select the simple approximant $1$. However, this is not sufficient even for the maximal correlation with the constant function. There is some flexibility\footnote{For instance, a Fourier-analytic approximant $\Lambda^\sharp(n) \coloneqq \sum_{q \leq Q} \frac{\mu(q) c_q(n)}{\phi(q)}$ is used in~\cite{HB-ternary}, where $c_q(n) \coloneqq \sum_{1 \leq a \leq q: (a,q)=1} e(an/q)$ denotes the Ramanujan sum.  Another option is to use a truncated convolution sum, $\Lambda^\sharp(n) \coloneqq-\sum_{d\mid n,d\leq R}\mu(d)\log d$, following e.g.~\cite[\S 19.2]{ik}.} in how to select the approximant, but (following~\cite{tt-quant}) we use the Cram\'er--Granville model~\eqref{lambdar-def}, which has the benefits of being a nonnegative model function and one that is known to be pseudorandom (which will be helpful in Section~\ref{gowers-sec}).

\begin{proof}[Proof of~\eqref{eq:Lambdalong}] 
It suffices to show that, for any $a, q \in \mathbb{N}$ and any $H_2 \in [X/\log^A X, X]$, we have
\[
\left| \sum_{\substack{X < n \leq X+H_2 \\ n \equiv a \pmod{q}}} (\Lambda(n) - \Lambda^\sharp(n))\right| \ll \frac{H_2}{\log^A X}.
\]
We can clearly assume that $q < R$ and $(a, q) = 1$.

Let $D=\exp((\log X)^{3/5})$. By the fundamental lemma of the sieve (see e.g.~\cite[Fundamental Lemma 6.3 with $y = D, z = R$, and $\kappa = 1$]{ik}), there exist real numbers $\lambda_d^{+}\in [-1,1]$ such that, for any $H\geq 2$, $q<R$, and $a \in \mathbb{N}$ with $(a,q)=1$, we have 
\begin{align*}
\sum_{\substack{X<n\leq X+H\\n=a(q)}}\Lambda^{\sharp}(n)&\leq\frac{P(R)}{\varphi(P(R))}\sum_{\substack{d\leq D \\ d \mid P(R)}} \lambda_d^{+}\sum_{\substack{X<n\leq X+H\\n=a(q)\\d\mid n}}1\\
&=\prod_{p<R}\left(1-\frac{1}{p}\right)^{-1}\sum_{\substack{d\leq D\\d \mid P(R) \\ (d,q)=1}}\lambda_d^{+}\frac{H}{dq}+O(D\log R)\\
&=\frac{H}{\varphi(q)}\left(1+O\left(\exp\left(-\frac{\log D}{\log R}\right)\right)\right) +O(D\log R), 
\end{align*}
and also by the fundamental lemma we have a lower bound of the same shape. Hence, for $H\geq X^{\varepsilon}$ we have
\begin{align}\label{eq:lambdasharp}
\sum_{\substack{X<n\leq X+H\\n=a(q)}}\Lambda^{\sharp}(n)=\frac{H}{\varphi(q)}+O_{\varepsilon}(H\exp(-(\log X)^{1/2})),    
\end{align}
so~\eqref{eq:Lambdalong} follows by the Siegel--Walfisz theorem and the triangle inequality.
\end{proof}
\begin{remark}
One could improve the error term in~\eqref{eq:Lambdalong} by adjusting the approximant $\Lambda^\sharp$ to account for a potential Siegel zero; see for instance~\cite[Theorem 5.27]{ik} or~\cite[Proposition 2.2]{tt-quant}.  However, we will not do so here.
\end{remark}

Before turning to the proof of~\eqref{eq:dklong} let us discuss the construction of the approximant $d_k^\sharp$ which is a somewhat non-trivial task. The classical Dirichlet hyperbola method gives the asymptotic
\begin{equation}\label{chit}
\sum_{\substack{n \leq X\\ n=a\ (q)}} d_k(n) = X P_{k,a,q}(\log X) + O_{q,\eps}( X^{1-1/k+\eps} )
\end{equation}
for any fixed $a,q$, any $\eps>0$, and some explicit polynomial $P_{k,a,q}$ of degree $k-1$ with coefficients depending only on $k,a,q$. Better error terms are known here; see e.g.,~\cite [Section 13]{ivic}.

From~\eqref{chit}, the triangle inequality, and Taylor expansion one has
$$ \sum_{\substack{X < n \leq X+H\\ n=a\ (q)}} d_k(n) = H \left(P_{k,a,q}(\log X) + P'_{k,a,q}(\log X) + O_{q,\eps}\left( \frac{X^{1-1/k+\eps}}{H} + \frac{H}{X^{1-\eps}} \right) \right) $$
for any $\eps > 0$ whenever $2 \leq H \leq X$.

Hence we have to choose the approximant $d_k^\sharp$ to also obey estimates such as
\begin{equation}
\label{eq:dksharpclaim}
\sum_{\substack{X \leq n < X+H\\ n=a\ (q)}} d^\sharp_k(n) = H \left(P_{k,a,q}(\log X) + P'_{k,a,q}(\log X) + O_\eps(X^{-\kappa_k} + HX^{\eps-1}) \right)
\end{equation}
for some $\kappa_k >0$, with exactly the same choice of polynomial $P_{k,a,q}$. 

The delta method of Duke, Friedlander and Iwaniec~\cite{dfi} can be used to build an approximant of a Fourier-analytic nature, basically by isolating the major arc components of $d_k$; see~\cite{ivic-divisor},~\cite{conrey-gonek},~\cite{ng-thom}, and~\cite[Proposition 4.2]{mrt-div} for relevant calculations in this direction.  However, the approximant that is (implicitly) constructed in these papers is very complicated, and somewhat difficult to deal with for our purposes (for instance, it is not evident whether it is non-negative).

The simpler approximant
$$ d_k(n, A) \coloneqq A^{1-k} \sum_{\substack{m|n\\ m \leq n^A}} d_{k-1}(m)$$
was recently proposed by Andrade and Smith~\cite{smith-andrade} for various choices of parameter $0 < A < 1$.  Unfortunately the polynomial $P_{k,a,q,A}(\log X)$ associated to this approximant usually only agrees with $P_{k,a,q}(\log X)$ to leading order (see~\cite[Theorem 2.1]{smith-andrade}), and so with this approximant one cannot hope to get polynomial saving like in our Theorem~\ref{discorrelation-thm}(iii).

Our approximant~\eqref{dks-def} with $P_m(t)$ as in~\eqref{eq:Pm(t)-def} can be seen as a more complicated variant of the Andrade--Smith approximant. Note that the constraint $m \leq R_k^{2k-2}$ in~\eqref{dks-def} is redundant, as $P_m$ vanishes for $m > R_k^{2k-2}$. Note also that (by adjusting the value of $c_{k,d,D}$ in Theorem~\ref{discorrelation-thm}) one could take $R_k$ to be any sufficiently small power of $X$, and that, for any $n \ll X$,
\begin{align}
\label{eq:dkapp<<dk}
\begin{aligned}
d_k^\sharp(n) &= \sum_{\substack{m \leq R_k^{2k-2}\\ m|n}} \sum_{j=0}^{k-1} \binom{k}{j} \sum_{\substack{n_1,\dots,n_j \leq R_k < n_{j+1},\dots,n_{k-1} \leq R_k^2\\ n_1 \dotsm n_{k-1} = m}} \frac{\left(\log n - \log(n_1 \dotsm n_j R_k^{k-j})\right)^{k-j-1}}{(k-j-1)! \log^{k-j-1} R_k} \\
&\ll \sum_{m \mid n} d_{k-1}(m) = d_k(n)
\end{aligned}
\end{align}

Recall we chose $R_k = X^{\frac{1}{10k}}$ in~\eqref{dks-def}. The motivation for our approximant $d_k^\sharp$ can be seen by noting that, sorting a factorization $n = n_1 \dotsm n_k$ into terms $n_1,\dotsc,n_j \leq R_k$ and terms $n_{j+1},\dots,n_k > R_k$, we get the generalized Dirichlet hyperbola identity
\begin{equation}
\label{eq:genhypid}
d_k(n) = \sum_{j=0}^{k-1} \binom{k}{j}  \sum_{n_1,\dots,n_j \leq R_k}  \sum_{\substack{n_{j+1}, \dotsc, n_{k-1} > R_k \\  \frac{n}{n_1\dots n_{k-1}} > R_k}} 1_{n_1 \dotsm n_{k-1}|n}.
\end{equation}
The polynomials $P_m(t)$ are chosen to match with the contribution from the sum over $n_{j+1}, \dotsc, n_{k-1}$ as can be seen from the proof of~\eqref{eq:dklong} that we now give.

\begin{proof}[Proof of~\eqref{eq:dklong}] It suffices to show that, for any $k \geq 2$, any $a, q \in \mathbb{N}$, and any $H_2 \in [X^{1-1/(50k)}, X]$, we have
\[
\left| \sum_{\substack{X < n \leq X+H_2 \\ n \equiv a \pmod{q}}} (d_k(n) - d_k^\sharp(n))\right| \ll \frac{H_2^2}{X} \log^{k-2} X.
\]
Since $d_k(n) = O_\varepsilon(n^\varepsilon)$, we can clearly assume that $q \leq X^{\frac{1}{40k}}$. Using~\eqref{eq:genhypid} we obtain
\[
\begin{split}
\sum_{\substack{X < n \leq X+H_2 \\ n \equiv a \pmod{q}}} d_k(n) &= \sum_{\substack{a_i \pmod{q} \\a_1 \dotsm a_k \equiv a \pmod{q}}} \sum_{j=0}^{k-1} \binom{k}{j} \sum_{\substack{n_1,\dotsc,n_j \leq R_k \\ n_i \equiv a_i \pmod{q}}} \sum_{\substack{n_{j+1}, \dotsc, n_{k-1} > R_k \\  \frac{X}{n_1\dotsm n_{k-1}} > R_k \\ n_i \equiv a_i \pmod{q}}} \left(\frac{H_2}{q n_1 \dotsm n_{k-1}} + O(1)\right) \\
& \qquad + O\left(\sum_{n_1,\dots,n_j \leq R_k} \sum_{\substack{n_{j+1}, \dotsc, n_{k-1} > R_k \\  \frac{X+H_2}{n_1\dotsm n_{k-1}} > R_k > \frac{X}{n_1\dotsm n_{k-1}}}} \left(\frac{H_2}{n_1 \dotsm n_{k-1}} + 1 \right)\right).
\end{split}
\]
Let us consider the two error terms. The first error term contributes, using the inequality $1 < X/(R_k n_1 \dotsm n_{k-1})$,
\[
\ll \sum_{a_k \pmod{q}} \sum_{n_1, \dotsc, n_{k-1} \leq X} \frac{X}{R_k n_1 \dotsm n_{k-1}} \ll q \frac{X}{R_k} \log^{k-1} X \ll \frac{H_2^2}{X} \log^{k-2} X
\]
since $q \leq X^{\frac{1}{40k}}$, $R_k = X^{\frac{1}{10k}},$ and $H_2 \geq X^{1-\frac{1}{50k}}$. The second error term contributes, using $n_1 \dotsm n_{k-1} \asymp X/R_k$ and Shiu's bound (Lemma~\ref{shiu}), 
\[
\ll \sum_{\substack{n_1, \dotsc, n_{k-1} \leq 2X \\ \frac{X}{R_k} < n_1 \dotsm n_{k-1} \leq \frac{X+H_2}{R_k}}} \frac{R_k H_2}{X} = \frac{R_k H_2}{X}  \sum_{\frac{X}{R_k} < n < \frac{X+H_2}{R_k}} d_{k-1}(n) \ll \frac{H_2^2}{X} \log^{k-2} X.
\]
Hence
\begin{align}
\begin{aligned}
\label{eq:dksumdec}
&\sum_{\substack{X < n \leq X+H_2 \\ n \equiv a \pmod{q}}} d_k(n) \\
& = \frac{H_2}{q} \sum_{\substack{a_i \pmod{q} \\a_1 \dotsm a_k \equiv a \pmod{q}}} \sum_{j=0}^{k-1} \binom{k}{j} \sum_{\substack{n_1,\dotsc,n_j \leq R_k \\ n_i \equiv a_i \pmod{q}}} \frac{1}{n_1 \dotsm n_j} \sum_{\substack{n_{j+1}, \dotsc, n_{k-1} > R_k \\  \frac{X}{n_1\dotsm n_{k-1}} > R_k \\ n_i \equiv a_i \pmod{q}}} \frac{1}{n_{j+1} \dotsm n_{k-1}} \\
&\qquad + O\left(\frac{H_2^2}{X} \log^{k-2} X\right).
\end{aligned}
\end{align}
For any $B \geq A \geq 1$, we have
\[
\sum_{\substack{A < n < B \\ n\equiv a \pmod{q}}} \frac{1}{n} = \frac{1}{q} \int_A^B \frac{1}{t} dt + O\left(\frac{1}{A}\right).
\]
Applying this $k-1-j$ times, we see that\footnote{To obtain the second equality we use the classical formula $\int_{x_1,\dots,x_d \geq 0: x_1+\dots + x_d \leq L} 1\ dx_1 \dots dx_d = \frac{L^d}{d!}$ for the volume of a simplex (easily proven by induction on $d$ and the Fubini--Tonelli theorem combined with the change of variables $x_i = \log \frac{t_{i+j}}{R}$ for $i=1,\dots,k-j-1$).}
\begin{align}
\begin{aligned}
\label{eq:dkninnersum}
&\sum_{\substack{n_{j+1}, \dotsc, n_{k-1} > R_k \\  \frac{X}{n_1 \dotsm n_{k-1}} > R_k \\ n_i \equiv a_i \pmod{q}}} \frac{1}{n_{j+1} \dotsm n_{k-1}} \\
& = \frac{1}{q^{k-1-j}} \int_{\substack{t_{j+1},\dots,t_{k-1} > R_k \\ t_{j+1} \dotsm t_{k-1} \leq \frac{X}{n_1\dotsm n_j R_k} }} \frac{dt_{j+1} \dotsm dt_{k-1}}{t_{j+1} \dots t_{k-1}} + O\left(\frac{(\log X)^{k-1-j-1}}{q^{k-1-j-1}} \cdot \frac{1}{R_k}\right)\\
&=  \frac{1}{q^{k-1-j}} \frac{\log^{k-j-1} \frac{X}{n_1 \dotsm n_j R_k^{k-j}}}{(k-j-1)!} + O\left(\frac{(\log X)^{k-j-2}}{q^{k-j-2}} \cdot \frac{1}{R_k}\right).
\end{aligned}
\end{align}
Since $R_k = X^{\frac{1}{10k}}, q \leq X^{\frac{1}{40k}}$, and $H_2 \geq X^{1-\frac{1}{50k}},$ the error term contributes to~\eqref{eq:dksumdec}
\begin{align*}
&\ll \frac{H_2}{q}  \sum_{j=0}^{k-1} \sum_{\substack{a_{j+1}, \dotsc, a_k \pmod{q}}}\, \sum_{\substack{n_1,\dotsc,n_j \leq R_k}} \frac{1}{n_1 \dotsm n_j} \cdot \frac{(\log X)^{k-j-2}}{q^{k-j-2}} \cdot \frac{1}{R_k} \\
&\ll \frac{H_2}{q} \sum_{j=0}^{k-1} q^{k-j} (\log X)^j \cdot \frac{(\log X)^{k-j-2}}{q^{k-j-2}} \cdot \frac{1}{R_k} \ll \frac{H_2^2}{X} \log^{k-2} X.
\end{align*}
Hence~\eqref{eq:dksumdec} and~\eqref{eq:dkninnersum} give
\[
\begin{split}
\sum_{\substack{X < n \leq X+H_2 \\ n \equiv a \pmod{q}}} d_k(n) & = \frac{H_2}{q^{k-j}} \sum_{\substack{a_i \pmod{q} \\a_1 \dotsm a_k \equiv a \pmod{q}}} \sum_{j=0}^{k-1} \binom{k}{j} \sum_{\substack{n_1,\dotsc,n_j \leq R_k \\ n_i \equiv a_i \pmod{q}}} \frac{\log^{k-j-1} \frac{X}{n_1 \dotsm n_j R_k^{k-j}}}{(k-j-1)! n_1 \dotsm n_j} \\
&\qquad + O\left(\frac{H_2^2}{X} \log^{k-2} X\right).
\end{split}
\]
On the other hand, by definition,
\[
\begin{split}
&\sum_{\substack{X < n \leq X+H_2 \\ n \equiv a \pmod{q}}} d_k^\sharp(n) \\
&=  \sum_{\substack{a_i \pmod{q} \\a_1 \dotsm a_k \equiv a \pmod{q}}} \sum_{j=0}^{k-1} \binom{k}{j} \sum_{\substack{n_1,\dots,n_j \leq R_k \\ n_i \equiv a_i \pmod{q}}} \frac{\log^{k-j-1} \frac{X}{n_1 \dotsm n_j R_k^{k-j}} + O\left(\frac{H_2}{X} \log^{k-j-2} X\right)}{(k-j-1)! \log^{k-j-1} R_k} \\
& \qquad \cdot \sum_{\substack{R_k < n_{j+1},\dots,n_{k-1} \leq R_k^2 \\ n_i \equiv a_i \pmod{q}}} \left(\frac{H_2}{q n_1 \dotsm n_{k-1}} + O(1)\right) \\
& =  \frac{H_2}{q} \sum_{\substack{a_i \pmod{q} \\a_1 \dotsm a_k \equiv a \pmod{q}}} \sum_{j=0}^{k-1} \binom{k}{j} \sum_{\substack{n_1,\dots,n_j \leq R_k \\ n_i \equiv a_i \pmod{q}}} \frac{\log^{k-j-1} \frac{X}{n_1 \dotsm n_j R_k^{k-j}}}{(k-j-1)! \log^{k-j-1} R_k} \\
& \qquad \sum_{\substack{R_k < n_{j+1},\dots,n_{k-1} \leq R_k^2 \\ n_i \equiv a_i \pmod{q}}} \frac{1}{n_1 \dotsm n_{k-1}}\\
& + O\left(\sum_{a_k \pmod{q}}  \sum_{j=0}^{k-1} \sum_{\substack{n_1,\dots,n_j \leq R_k}} \frac{H_2}{X \log X} \sum_{\substack{R_k < n_{j+1},\dots,n_{k-1} \leq R_k^2}} \left(\frac{H_2}{q n_1 \dotsm n_{k-1}} + 1\right)\right) \\
&+ O\left(\sum_{a_k \pmod{q}}  \sum_{j=0}^{k-1} \sum_{\substack{n_1,\dots,n_j \leq R_k}} \sum_{\substack{R_k < n_{j+1},\dots,n_{k-1} \leq R_k^2}} 1 \right).
\end{split}
\]
The error terms contribute 
\[
\ll \frac{H_2^2}{X} \log^{k-2} X + q \frac{H_2}{X \log X} R_k^{2(k-1)} +  qR_k^{2(k-1)} \ll \frac{H_2^2}{X} \log^{k-2} X + qX^{1/2}
\]
and in the main term
\[
\begin{split}
\frac{1}{\log^{k-j-1} R_k} \sum_{\substack{R_k < n_{j+1},\dots,n_{k-1} \leq R_k^2 \\ n_i \equiv a_i \pmod{q}}} \frac{1}{n_{j+1} \dotsm n_{k-1}}
&= \left(\frac{1}{q}+O\left(\frac{1}{R_k}\right)\right)^{k-j-1}.
\end{split}
\]
The claim follows since $R_k = X^{\frac{1}{10k}}$ and $q \leq X^{\frac{1}{40k}}$.
\end{proof}

\subsection{Proof of Lemma~\ref{le:Approx}}
\label{ssec:approximants}
Note first that the claims are trivial unless $q \leq X^{1/80}$. For part (ii), note that, for $j=1, 2$,
\[
\begin{split}
&\frac{1}{H_j}\sum_{\substack{X < n \leq X+H_j \\ n \equiv a \pmod{q}}} d_k^\sharp(n)\\
&= \frac{1}{H_j} \sum_{\substack{b, c \pmod{q} \\ bc \equiv a \pmod{q}}} \sum_{\substack{m \leq X^{\frac{2k-2}{10k}} \\ m \equiv b \pmod{q}}} \left(P_m(\log X) + O\left(d_{k-1}(m) \frac{H_j}{X \log X}\right)\right)  \sum_{\substack{X/m < n \leq (X+H_j)/m \\ n \equiv c \pmod{q}}} 1 \\
&=   \frac{1}{H_j} \sum_{\substack{b, c \pmod{q} \\ bc \equiv a \pmod{q}}} \sum_{\substack{m \leq X^{\frac{k-1}{5k}} \\ m \equiv b \pmod{q}}} \left(P_m(\log X) + O\left(d_{k-1}(m) \frac{H_j}{X \log X}\right)\right)  \left(\frac{H_j}{mq} + O(1)\right) \\
& =  \sum_{\substack{b, c \pmod{q} \\ bc \equiv a \pmod{q}}} \sum_{\substack{m \leq X^{\frac{k-1}{5k}} \\ m \equiv b \pmod{q}}} \frac{P_m(\log X)}{mq} + O\left(\frac{H_j \log^{k-2} X}{X} + \frac{qX^{1/5}}{H_j}\right).
\end{split}
\]
The claim follows by subtracting this for $j =1, 2$. Part (i) follows directly from~\eqref{eq:lambdasharp} applied with $H\in \{H_1,H_2\}$ and the triangle inequality.

\subsection{Proof of Lemmas~\ref{le:BHP} and~\ref{le:typeI2}}
\label{ssec:BHP}
We first make a standard reduction to studying averages of Dirichlet polynomials. 

\begin{lemma}
\label{le:Perron}
Let $W \leq X^{1/100}$. Let $|a_n| \leq d_2(n)^C$ for some $C \geq 1$ and let $A(s, \chi):= \sum_{c_1 X < n \leq c_2 X} a_n \chi(n) n^{-s}$ for some fixed $c_2 > c_1 > 0$.  Let  $X^{1/2} \leq H_1 \leq H_2 \leq X/W^4$ and $(a, q) = 1$. 
\begin{itemize}
\item[(i)]
One has
\[
\begin{split}
&\Big | \frac{1}{H_1}\sum_{\substack{X < n \leq X+H_1 \\ n \equiv a \pmod{q}}} a_n - \frac{1}{H_2} \sum_{\substack{X < n \leq X+H_2 \\ n \equiv a \pmod{q}}} a_n \Big | \ll  \frac{\log^{O_C(1)} X}{W^2} \\
& \qquad + \frac{\log X}{X^{1/2}} \max_{\frac{X}{H_1} \leq T \leq \frac{XW^4}{H_1}} \frac{1}{\varphi(q)} \sum_{\chi \pmod{q}} \frac{X/H_1}{T} \int_{\substack{W \leq |t|\leq T}} |A(\tfrac{1}{2}+it, \chi)| \  dt.
\end{split}
\]
\item[(ii)] One has
\[
\begin{split}
&\Big | \frac{1}{H_1}\sum_{\substack{X < n \leq X+H_1 \\ n \equiv a \pmod{q}}} a_n\Big | \ll  \frac{\log^{O_C(1)} X}{W^2} \\
& \qquad + \frac{\log X}{X^{1/2}} \max_{\frac{X}{H_1} \leq T \leq \frac{XW^4}{H_1}} \frac{1}{\varphi(q)} \sum_{\chi \pmod{q}} \frac{X/H_1}{T} \int_{\substack{|t|\leq T}} |A(\tfrac{1}{2}+it, \chi)| \  dt.
\end{split}
\]
\end{itemize}
\end{lemma}

\begin{proof}
Let us first consider part (i). We begin by using the orthogonality of characters and Perron's formula (see e.g.~\cite[Corollary 5.3]{mv}) to get that, for $j = 1, 2$,
\begin{align*}
\frac{1}{H_j} \sum_{\substack{X < n \leq X+H_j \\ n \equiv a \pmod{q}}} a_n &=  \frac{1}{\varphi(q) H_j} \sum_{\chi \pmod{q}} \overline{\chi}(a) \int_{-\frac{XW^4}{H_j}}^{\frac{XW^4}{H_j}} A(\tfrac{1}{2}+it, \chi) \frac{(X+H_j)^{1/2+it}-X^{1/2+it}}{\tfrac{1}{2}+it} dt\\
&+ O\left(\frac{\log^{O_C(1)} X}{W^4}\right).
\end{align*}
The ``main term'' comes from (only $\chi_0$ contributes to actual main terms)
\begin{align*}
&\frac{1}{\varphi(q) H_j} \sum_{\chi \pmod{q}} \overline{\chi}(a) \int_{-W}^{W}  A(\tfrac{1}{2}+it, \chi)
\frac{  (X+H_j)^{1/2+it}-X^{1/2+it} } {\frac{1}{2}+it}  dt \\
&= \frac{1}{\varphi(q)} \sum_{\chi \pmod{q}} \overline{\chi}(a) \int_{-W}^{W}  A(\tfrac{1}{2}+it, \chi) X^{-1/2+it}  dt + O\left(\frac{H_j W^2}{X} \log^{O_C(1)} X\right).
\end{align*}
The error term is $O(\log^{O_C(1)} X /W^2)$ while the main term is independent of $j$. Hence
\begin{align*}
&\Big | \frac{1}{H_1}\sum_{\substack{X < n \leq X+H_1 \\ n \equiv a \pmod{q}}} a_n - \frac{1}{H_2} \sum_{\substack{X < n \leq X+H_2 \\ n \equiv a \pmod{q}}} a_n \Big | \ll  \frac{\log^{O_C(1)} X}{W^2} \\
& + \sum_{j=1}^2 \frac{1}{\varphi(q) H_j} \sum_{\chi \pmod{q}} \int_{\substack{W \leq |t|\leq \frac{XW^4}{H_j}}} \left|A(\tfrac{1}{2}+it, \chi)\right| \left|\frac{(X+H_j)^{1/2+it}-X^{1/2+it}}{\tfrac{1}{2}+it}\right| dt
\end{align*}

Since $|\frac{  (X+H_j)^{1/2+it}-X^{1/2+it} } {1/2+it}  | \ll \min\{ H_j X^{-1/2}, X^{1/2}/(1+|t|)\}$, the second line contributes
\[
\begin{split}
&\ll \sum_{j = 1}^2 \frac{1}{\varphi(q) H_j} \sum_{\chi \pmod{q}} \frac{H_j}{X^{1/2}} \int_{\substack{W \leq |t|\leq \frac{X}{H_j}}}
|A(\tfrac{1}{2}+it, \chi)|  dt \\
&\qquad + \sum_{j = 1}^2 \frac{1}{\varphi(q) H_j} \sum_{\chi \pmod{q}} \int_{\substack{\frac{X}{H_j} \leq |t|\leq \frac{XW^4}{H_j}}}
|A(\tfrac{1}{2}+it, \chi)| \frac{X^{1/2}}{1+|t|} dt.
\end{split}
\]
Splitting the second integral dyadically, we see that this is 
\[
\ll \frac{\log X}{X^{1/2}} \sum_{j = 1}^2 \max_{\frac{X}{H_j} \leq T \leq \frac{XW^4}{H_j}} \frac{1}{\varphi(q)} \sum_{\chi \pmod{q}} \frac{X/H_j}{T} \int_{\substack{W \leq |t|\leq T}} |A(\tfrac{1}{2}+it, \chi)| \  dt.
\]
Since $H_2 \geq H_1$, the contribution of the part with $j=1$ is larger than the contribution of the part with $j=2$. Hence part (i) follows.

Part (ii) follows similarly, except there is no need to handle a main term separately.
\end{proof}

\begin{proof}[Proof of Lemma~\ref{le:BHP}]
By Shiu's bound (Lemma~\ref{shiu}) we can clearly assume that $q \leq W^{1/2} \leq X^{\varepsilon/400}$. Let us consider, for $j = 1, 2$,
\[
\frac{1}{H_j}\sum_{\substack{X < m_1 m_2 \ell \leq X+H_j \\ m_1 m_2 \ell \equiv a \pmod{q} \\ m_j \sim M_j, \ell \sim L}} a_{m_1} b_{m_2} v_\ell.
\]
We first split the sums according to $r_1 = (m_1, q)$, $r_2 = (m_2, q/r_1)$ and $r_3 = (\ell, q/(r_1 r_2))$, writing $m_j = r_j m_j'$ and $\ell = r_3 \ell'$. Then $m_1' m_2' \ell' r_1 r_2 r_3 \equiv a \pmod{\frac{q}{r_1 r_2 r_3} r_1 r_2 r_3}$ and necessarily $r_1 r_2 r_3 = (a, q)$. We have
\begin{align*}
&\frac{1}{H_j}\sum_{\substack{X < m_1 m_2 \ell \leq X+H_j \\ m_1 m_2 \ell \equiv a \pmod{q} \\ m_j \sim M_j, \ell \sim L}} a_{m_1} b_{m_2} v_\ell \\
&= \sum_{\substack{r_1 r_2 r_3 = (a, q)}} \frac{1}{H_j}\sum_{\substack{X/(r_1 r_2 r_3) < m_1' m_2' \ell' \leq (X+H_j)/(r_1 r_2 r_3) \\ m_1' m_2' \ell' \equiv \frac{a}{r_1 r_2 r_3} \pmod{\frac{q}{r_1 r_2 r_3}} \\ (m_1', q/r_1) = (m_2', q/(r_1 r_2)) = (\ell', q/(r_1 r_2 r_3)) = 1 \\ m_j' \sim M_j/r_j, \ell'
 \sim L/r_3}} a_{m_1' r_1} b_{m_2' r_2} v_{\ell' r_3}.
\end{align*}
Part (i) follows from Lemma~\ref{le:Perron} (with $X/(a, q)$, $H_j/(a, q)$, $q/(a, q)$, and $a/(a,q)$ in place of $X$, $H_j$, $q$, and $a$) if, for any $T \in [X/H_1, XW^4/H_1]$ and any $r_1 r_2 r_3 = (a, q)$ and any $\chi \pmod{q/(a, q)}$, one has
\begin{align*}
&\int_{\substack{W \leq |t|\leq T}} \Bigl|\sum_{\substack{m_1' \sim M_1/r_1 \\ (m_1, q/r_1) = 1}} \frac{a_{m_1' r_1} \chi(m_1')}{m_1'^{1/2+it}} \sum_{\substack{m_2' \sim M_2/r_2 \\ (m_2', q/(r_1 r_2)) = 1}} \frac{b_{m_2' r_2} \chi(m_2')}{m_2'^{1/2+it}} \sum_{\substack{\ell' \sim L/r_3 \\ (\ell', q/(r_1r_2r_3)) = 1}} \frac{v_{\ell' r_3} \chi(\ell')}{\ell'^{1/2+it}} \Bigr| \  dt \\
&\ll \frac{\log^{O_C(1)} X}{W^{1/3}}\frac{T}{X/H_1} \left(\frac{X}{(a, q)}\right)^{1/2}.
\end{align*}
But, using the assumption~\eqref{eq:BHP-supt-cond}, this follows from a slight variant of~\cite[Lemma 9]{baker-harman-pintz} with $g=1$ in cases $\theta \in \{7/12, 3/5, 5/8\}$ and with $g=2$ in case $\theta = 11/20$ (alternatively see~\cite[Lemma 7.3]{harman-book}). The idea in the proofs of these lemmas is to first split the integral to level sets according to the absolute values of the three Dirichlet polynomials appearing, and then to apply appropriate mean and large value results individually for the three Dirichlet polynomials to obtain upper bounds for the sizes of the level sets. Combining these upper bounds using case-by-case study and H\"older's inequality leads to the lemmas.

Part (ii) follows similarly.
\end{proof}

In fact, one can establish Lemma~\ref{le:BHP} for $\theta \in [7/12, 5/8]$ by using~\cite[Lemma 9]{baker-harman-pintz} with $g=1$, and for $\theta \in [11/20, 9/16]$ by using~\cite[Lemma 9]{baker-harman-pintz} with $g=2$ (see~\cite[end of Section 7.2]{harman-book}), but we shall not need this more general result.

\begin{proof}[Proof of Lemma~\ref{le:typeI2}]
By Shiu's bound~ (Lemma~\ref{shiu}) we can assume that $q \leq W^{1/6}$. Notice first that if for either $i =1$ or $i=2$, we have $\theta+\varepsilon - (1-\alpha_i) \geq \varepsilon,$ then we can obtain the claim by simply moving the sum over $m_i$ inside. Hence we can assume that $\alpha_1, \alpha_2 < 1-\theta$.
 
Arguing as in proof of Lemma~\ref{le:BHP} and doing a dyadic splitting it suffices to show that, for any $T \in [W, XW^4/H_1]$ and any $r_1 r_2 r_3 = (a, q)$,
\begin{align}
\label{eq:typeI2majorclaim}
&\frac{1}{\varphi(\frac{q}{(a, q)})} \sum_{\chi \pmod{\frac{q}{(a, q)}}} \int_{T}^{2T} \Bigl|\sum_{\substack{m_1 \sim M_1/r_1 \\ (m_1, q/r_1) = 1}} \frac{\chi(m_1)}{m_1^{1/2+it}} \sum_{\substack{m_2 \sim M_2/r_2 \\ (m_2, q/(r_1 r_2)) = 1}} \frac{\chi(m_2)}{m_2^{1/2+it}} \sum_{\substack{\ell \sim L/r_3 \\ (\ell, q/(r_1r_2r_3)) = 1}} \frac{\chi(\ell) v_{\ell r_3}}{\ell^{1/2+it}} \Bigr| \  dt\\
\nonumber
&\ll \frac{\log^{O(1)} X}{W^{1/6}} \max\left\{\frac{T}{X/H_1}, 1\right\} \left(\frac{X}{(a, q)}\right)^{1/2}.
\end{align}
By the fourth moment estimate for Dirichlet $L$-functions we have (see~\cite[Lemma 10.11]{harman-book}), for any $M, T \geq 2$ and $d \mid (a, q)$,
\[
\begin{split}
\sum_{\chi \pmod{\frac{q}{(a, q)}}} \int_{T}^{2T} \Biggl|\sum_{\substack{m \sim M \\ (m, q/d) = 1}} \frac{\chi(m)}{m^{1/2+it}} \Biggr|^4 dt & \ll \sum_{\chi \pmod{\frac{q}{d}}} \int_{T}^{2T} \left|\sum_{\substack{m \sim M}} \frac{\chi(m)}{m^{1/2+it}} \right|^4 dt \\
&\ll\left(q^3T + \frac{qM^2}{T^3}\right) \log^{O(1)} (MT) .
\end{split}
\]
Hence, using also H\"older and the mean value theorem (see e.g.~\cite[Theorem 9.12 with $k = q$ and $Q =1$]{ik}), the left-hand side of~\eqref{eq:typeI2majorclaim} is
\[
\begin{split}
&\ll  \log^{O(1)} X  \left(q^2T + \frac{X^{2\alpha_1}}{T^3} \right)^{1/4} \left(q^2 T + \frac{X^{2\alpha_2}}{T^3} \right)^{1/4} \left(T + \frac{X^{1-\alpha_1-\alpha_2}}{q}\right)^{1/2} \\
&\ll q \log^{O(1)} X \left(T + T^{1/2} X^{1/2-\alpha_1/2-\alpha_2/2} + X^{\alpha_1/2} + X^{\alpha_2/2} + \frac{X^{1/2-\alpha_1/2}}{T^{1/2}} + \frac{X^{1/2-\alpha_2/2}}{T^{1/2}} + \frac{X^{1/2}}{T^{3/2}}\right).
\end{split}
\]
One can see that this is always at most the right-hand side of~\eqref{eq:typeI2majorclaim} by considering each term separately --- depending on the term, the worst case is either $T = W$ or $T = X/H_1$.
\end{proof}

\subsection{Proof of Proposition~\ref{prop:H1H2comp}}
\label{ssec:H1H2compPropproof}
Let us first show the $k= 2$ case of Proposition~\ref{prop:H1H2comp}(ii). It follows from classical arguments leading to the exponent $1/3+\varepsilon$ in the Dirichlet divisor problem (see e.g.~\cite[Section I.6.4]{Tenenbaum-book}). For completeness, we provide the proof here. By a trivial bound we can assume that $q \leq X^{\varepsilon/4}$.

First note that
\[
\begin{split}
\frac{1}{H_j} \sum_{\substack{X < n \leq X+H_j \\ n \equiv a \pmod{q}}} d_2(n) 
&= \frac{2}{H_j} \sum_{\substack{X < mn \leq X+H_j \\ m \leq X^{1/2} \\ mn \equiv a \pmod{q}}} 1 + O\Biggl(\frac{1}{H_j} \sum_{m \in (X^{1/2}, (X+H_j)^{1/2}]} \sum_{\substack{X/m < n \leq (X+H_j)/m \\ mn \equiv a \pmod{q}}} 1\Biggr).
\end{split}
\]
The error term contributes
\[
\ll \frac{1}{H_j} \cdot \left(\frac{H_j}{X^{1/2}} + 1\right) \cdot \left(\frac{H_j}{X^{1/2}} + 1\right) \ll \frac{H_j}{X} + \frac{1}{H_j}.
\]

Hence it suffices to show that, for any $M \in [1/2, X^{1/2}]$, we have
\[
\frac{1}{H_1} \sum_{\substack{X < mn \leq X+H_1 \\ m \sim M \\ mn \equiv a \pmod{q}}} 1 = \frac{1}{H_2} \sum_{\substack{X < mn \leq X+H_2 \\ m \sim M \\ mn \equiv a \pmod{q}}} 1 + O\left(\frac{1}{X^{\varepsilon/5}}\right).
\]
Now, for $j = 1, 2$,
\[
\begin{split}
&\sum_{\substack{X < mn \leq X+H_j \\ m \sim M \\ mn \equiv a \pmod{q}}} 1 = \sum_{\substack{0 \leq b, c < q \\ bc \equiv a \pmod{q}}} \sum_{\substack{m \sim M \\ m \equiv b \pmod{q}}} \Biggl(\sum_{\substack{1 \leq n \leq \frac{X+H_j}{m} \\ n \equiv c \pmod{q}}} 1 - \sum_{\substack{1 \leq n \leq \frac{X}{m} \\ n \equiv c \pmod{q}}} 1 \Biggr)  \\
&=\sum_{\substack{0 \leq b, c < q  \\ bc \equiv a \pmod{q}}} \sum_{\substack{m \sim M \\ m \equiv b \pmod{q}}} \left(\left\lfloor \frac{X+H_j}{mq} - \frac{c}{q} \right\rfloor - \left\lfloor \frac{X}{mq} - \frac{c}{q} \right\rfloor\right) \\ 
&= \sum_{\substack{0 \leq b, c < q  \\ bc \equiv a \pmod{q}}} \sum_{\substack{m \sim M \\ m \equiv b \pmod{q}}} \left(\frac{H_j}{mq} + \left(\frac{1}{2}-\left\{\frac{X+H_j}{mq} -\frac{c}{q}\right\}\right) - \left(\frac{1}{2} - \left\{\frac{X}{mq} - \frac{c}{q}\right\}\right)\right).
\end{split}
\]
Hence it suffices to show that, for $j =1, 2$ and $\xi \in \{X/q, (X+H_j)/q\}$,
\begin{equation}
\label{eq:d2claimM}
 \sum_{\substack{0 \leq b, c < q  \\ bc \equiv a \pmod{q}}} \sum_{\substack{m \sim M \\ m \equiv b \pmod{q}}} \left(\frac{1}{2}-\left\{\frac{\xi}{m} -\frac{c}{q}\right\}\right) = O\left(\frac{H_j}{X^{\varepsilon/5}}\right).
\end{equation}
The left-hand side is trivially $O(qM) = O(X^{\varepsilon/4}M)$ and so~\eqref{eq:d2claimM} is immediate in case $M \leq H_j/X^{\varepsilon/2}$, and so we can concentrate on showing~\eqref{eq:d2claimM} for $j$ and $M$ for which $M > H_j/X^{\varepsilon/2}$. 

For any $K \geq 1$ we have the Fourier expansion (see e.g.~\cite[Section I.6.4]{Tenenbaum-book})
\[
\frac{1}{2} -\{ y\} = \sum_{k \neq 0} v_k e(ky) + O(1/K) \quad \text{with} \quad v_k \ll \min\{1/k, K/k^2\}.
\]
Taking $K_j = MX^{\varepsilon/2}/H_j$ (which is $\geq 1$) and writing $m = b+rq$, it suffices to show that, for $j =1, 2$ and $\xi \in \{X/q, (X+H_j)/q\}$,
\[
\sum_{|k| > 0} \min\left\{ \frac{1}{k}, \frac{MX^{\varepsilon/2}/H_j}{k^2}\right\} \left| \sum_{(M-b)/q < r \leq (2M-b)/q} e(k \xi/(b+rq)) \right| = O(X^{-\varepsilon/2} H_j/q^2).
\]
The second derivative of the phase has size $\asymp kXq/M^3$, so that by van der Corput's exponential sum bound (see e.g.~\cite[Theorem 5 in Section I.6.3]{Tenenbaum-book} or~\cite[Corollary 8.13]{ik}), the left-hand side is
\[
\begin{split}
&\ll \sum_{0 < |k| \leq MX^{\varepsilon/2}/H_j}\frac{1}{k} \left(\left(\frac{kXq}{M^3}\right)^{1/2} \frac{M}{q} + \left(\frac{M^3}{kX q}\right)^{1/2} \right) \\
&\qquad +\sum_{|k| > MX^{\varepsilon/2}/H_j}\frac{MX^{\varepsilon/2}/H_j}{k^2} \left(\left(\frac{kXq}{M^3}\right)^{1/2} \frac{M}{q} + \left(\frac{M^3}{kX q}\right)^{1/2} \right) \\
&\ll \frac{X^{1/2 + \varepsilon/4}}{H_j^{1/2} q^{1/2}} + \frac{M^{3/2}}{q^{1/2} X^{1/2}}.
\end{split}
\]
This is $\ll X^{-\varepsilon/2} H_j/q^2$ since $H_2 \geq H_1 \geq X^{1/3+\varepsilon}$, $q \leq X^{\varepsilon/4}$, and $M \leq X^{1/2}$. This establishes the $k=2$ case of Proposition~\ref{prop:H1H2comp}.

The cases $k=3, 4$ of Proposition~\ref{prop:H1H2comp}(ii) follow from dyadic splitting, Lemma~\ref{combinatorial}(v), and Lemma~\ref{le:typeI2} with $W = \min\{X^{\frac{1}{400k}}, X^{\varepsilon/4}\}$, so we can concentrate on Proposition~\ref{prop:H1H2comp}(i) and cases $k \geq 5$ of Proposition~\ref{prop:H1H2comp}(ii). To apply Lemma~\ref{le:BHP} we need parts (i) and (ii) of the following lemma (part (iii) will be used in the proof of Lemma~\ref{comb-mu} below):

\begin{lemma}[Dirichlet polynomial bounds]
\label{le:nitcancel}
Let $0 \leq T_0 \leq X$ and $\alpha \in (0, 1]$. 
\begin{itemize}
\item[(i)] There exists $\delta = \delta(\alpha)$ such that, for any character $\chi$ of modulus $q \leq X^{\alpha/2}$ and any $L \in [X^\alpha, X]$,
\[
\sup_{T_0 \leq |t| \leq X} \sup_{I\subset [L,2L]}\left| \sum_{\ell \in I} \frac{\chi(\ell)}{\ell^{1/2+it}}\right| \ll_{\alpha} L^{1/2} X^{-\delta} + L^{1/2} \frac{\log X}{(T_0 +1)^{1/2}}.
\]
\item[(ii)] For any $A > 0$, any $1\leq r\leq X$, and any character $\chi$ of modulus $q \leq \log^A X$, one has
\[
\sup_{|t| \leq X} \sup_{I\subset [X^{\alpha},2X^{\alpha}]}\left| \sum_{\ell \in I} \frac{\mu(r\ell) \chi(\ell)}{\ell^{1/2+it}}\right| \ll_{\alpha, A} \frac{X^{\alpha/2}}{\log^A X}.
\]
\item[(iii)] Let $\varepsilon > 0$. For any $A > 0$, any $P \in [\exp((\log X)^{2/3+\varepsilon}), X^2]$ and any character $\chi$ of modulus $q \leq \log^A X$,
\[
\sup_{T_0 \leq |t| \leq X} \sup_{I\subset [P,2P]}\left| \sum_{p \in I} \frac{\chi(p)}{p^{1/2+it}}\right| \ll_{\varepsilon, A} \frac{P^{1/2}}{T_0} + \frac{P^{1/2}}{\log^A X}.
\]
\end{itemize}
\end{lemma}

\begin{proof}
Parts (ii) and (iii) follow by standard contour integration arguments, using the known zero-free region for $L(s, \chi)$ (see e.g.,~\cite[Lemma 2]{mr-p} for a similar argument without the character). 

Let us concentrate on part (i). By partial summation, splitting into residue classes $a \pmod{q}$ and writing $\ell = mq + a$, it suffices to show that, for any $a \in \{1, \dotsc, q\}$ and $|t| \in [T_0, X]$, we have
\begin{equation}
\label{eq:expchimsumclaim}
\sum_{m \in \frac{1}{q}I} e\left(\frac{t}{2\pi}  \log (m q + a)\right) \ll L \frac{X^{-\delta}}{q} + L \frac{\log X}{q(T_0+1)^{1/2}}.
\end{equation}
The $\nu$th derivative of the phase $g(m) = \frac{t}{2\pi}  \log (m q + a)$ satisfies
\[
|g^{(\nu)}(m)| \frac{m^\nu}{\nu!} \asymp_\nu |t|
\]
for any $\nu \ge 1$.  
We apply the Weyl bound in the form of~\cite[Theorem 8.4]{ik}. When $T_0 \leq |t| \leq L/q$, we use~\cite[Theorem 8.4]{ik} with $k= 2$, obtaining
\[
\sum_{m \in \frac{1}{q}I} e\left(\frac{t}{2\pi}  \log (m q + a)\right) \ll \left(\frac{|t|}{L^2/q^2} + \frac{1}{|t|}\right)^{1/2} \frac{L}{q} \log X \ll \frac{L^{1/2}}{q^{1/2}} \log X + \frac{L \log X}{q(T_0+1)^{1/2}}. 
\]
Recalling that $q \leq L^{1/2}$, the bound~\eqref{eq:expchimsumclaim} follows with $\delta = \alpha/5$.

On the other hand, when $L/q < |t| \leq X$, we use~\cite[Theorem 8.4]{ik} with $k= \lfloor \frac{2}{\alpha} + 2\rfloor$, obtaining
\begin{equation}
\begin{split}
\sum_{m \in \frac{1}{q}I} e\left(\frac{t}{2\pi}  \log (m q + a)\right) &\ll_\alpha \left(\frac{|t|}{(L/q)^k} + \frac{1}{|t|}\right)^{\frac{4}{k 2^k}} \frac{L}{q} \log X \\
&\ll_\alpha \left(\frac{X}{(L^{1/2})^k} + \frac{1}{L^{1/2}}\right)^{\frac{4}{k 2^k}} \frac{L}{q}  \log X\\
& \ll_\alpha \frac{L^{1-\frac{2}{k2^k}}}{q} \log X
\end{split}
\end{equation}
and~\eqref{eq:expchimsumclaim} follows.
\end{proof}

Let us now get back to the proof of Proposition~\ref{prop:H1H2comp}(ii). Recall that we can assume that $k \geq 5$. The claim follows trivially unless $q \leq \min\{X^{2c_k}, X^{\varepsilon/900}\}$. We can request that $c_k \leq \frac{1}{4000k}$. By dyadic splitting it suffices to show that, for any $N_j \in [1/2, X]$ with $N_1 \dotsm N_k \asymp X$, one has
\begin{equation}
\label{eq:dkn1nkclaim}
\max_{\substack{a, q \in \mathbb{N} \\ q \leq X^{1/(2000k)}}} \left| \frac{1}{H_1}\sum_{\substack{X < n_1 \dotsm n_k \leq X+H_1 \\ n_i \sim N_i \\ n_1 \dotsm n_k \equiv a \pmod{q}}} 1 - \frac{1}{H_2}\sum_{\substack{X < n_1 \dotsm n_k \leq X+H_2 \\ n_i \sim N_i \\n_1 \dotsm n_k \equiv a \pmod{q}}} 1 \right| \ll \frac{1}{X^{2c_k}} + \frac{1}{X^{\varepsilon/800}}.
\end{equation}
We can find $\alpha_1, \dotsc, \alpha_k \in [0,1]$ with $\alpha_1 + \dotsb + \alpha_k = 1$ such that $N_i \asymp X^{\alpha_i}$ for each $i = 1, \dotsc, k$.

In case $k=5$ and $\theta = 11/20$ we start by applying Lemma~\ref{combinatorial}(iv). In case $(I_2^{\mathrm{maj}})$ holds we apply Lemma~\ref{le:typeI2} with $W = \min\{X^{\varepsilon/4}, X^{8c_k}\}$ to obtain~\eqref{eq:dkn1nkclaim}. In case $(II^{\mathrm{maj}})$ holds we wish to apply Lemma~\ref{le:BHP}. In order to do this, we need to show that~\eqref{eq:BHP-supt-cond} holds with
\begin{equation}
\label{eq:cmdefk5}
v_m = \sum_{\substack{m = \prod_{i \in I} m_i \\ m_i \sim N_i}} 1
\end{equation}
and $W = \min\{X^{\varepsilon/200}, X^{20 c_k}\}$ for any $L \asymp \prod_{i \in I} N_i$. Now there exists $i_0 \in I$ such that $\alpha_{i_0} \geq (2\theta-1)/k = \frac{1}{10k}$. We have (using $d(r) d_{|I|-1}(m) \ll W^{1/100}$) 
\begin{align*}
\left|\sum_{\ell \sim L/r} \frac{v_{\ell r} \chi(\ell)}{\ell^{1/2+it}}\right| &\leq \sum_{r = r_1 r_2} \sum_{\frac{L}{2r_2 X^{\alpha_{i_0}}} < m \leq \frac{2L}{r_2 X^{\alpha_{i_0}}}} \frac{d_{|I|-1}(m)}{m^{1/2}} \left|\sum_{\substack{m_{i_0} \sim X^{\alpha_{i_0}}/r_1 \\ m_{i_0} \sim L/(mr)}} \frac{\chi(m_{i_0})}{m_{i_0}^{1/2+it}}\right| \\
& \ll \left(\frac{L}{X^{\alpha_{i_0}}}\right)^{1/2} W^{1/100} \max_{r = r_1 r_2} \frac{1}{r_2^{1/2}} \max_{y \sim X^{\alpha_{i_0}}/r_1} \left|\sum_{\substack{X^{\alpha_{i_0}}/r_1 < m \leq y}} \frac{\chi(m)}{m^{1/2+it}}\right|.
\end{align*}
Hence~\eqref{eq:BHP-supt-cond} follows for~\eqref{eq:cmdefk5} if we show that
\begin{equation}
\label{eq:supt-cond-claim}
\max_{\substack{r_1 r_2 \mid q \\ \chi \pmod{\frac{q}{(a, q)}}}} \sup_{W \leq |t| \leq \frac{XW^4}{H_1}} \max_{y \sim X^{\alpha_{i_0}}/r_1} \left| \sum_{X^{\alpha_{i_0}}/r_1 < m \leq y} \frac{\chi(m)}{m^{1/2+it}}\right| \ll \frac{(X^{\alpha_{i_0}}/r_1)^{1/2}}{W^{1/3+1/100}},
\end{equation}
Note that $X^{\alpha_{i_0}}/r_1 \geq X^{\frac{1}{10k} - 2c_k} \geq X^{\frac{1}{20k}}$. We apply Lemma~\ref{le:nitcancel}(i) with $T_0 = W$. Taking $c_k \leq \delta(\frac{1}{20k})/30$ we obtain that the left-hand-side of~\eqref{eq:supt-cond-claim} is
\[
\ll \left(\frac{X^{\alpha_{i_0}}}{r_1}\right)^{1/2} \cdot \frac{\log X}{W^{1/2}} \ll \frac{(X^{\alpha_{i_0}}/r_1)^{1/2}}{W^{1/3+1/100}}.
\] 
Hence~\eqref{eq:dkn1nkclaim} follows from Lemma~\ref{le:BHP}. The case $k \geq 6$ and $\theta = 7/12$ follows similarly using Lemma~\ref{combinatorial}(iii).

A similar method allows us to establish Proposition~\ref{prop:H1H2comp}(i). We start by applying Heath-Brown's identity (Lemma~\ref{hb-identity}) with $L = \lceil 2/\eps \rceil$, writing $N_i = X^{\alpha_i}$. Then we apply Lemma~\ref{combinatorial}(iii) to these $\alpha_i$.

In case $(II^{\mathrm{maj}})$ holds we argue as above but with $W = \log^A X$ for some large $A > 0$. On the other hand, in case $\alpha_{i_0} \geq 1-\theta-\varepsilon/2$ for some $i_0$, we write $M = \frac{1}{N_{i_0}}\prod_{\substack{j = 1}}^\ell N_{j}$ and move the summation over $n_{i_0} \sim X^{\alpha_{i_0}}$ inside. Then it suffices to show in this case that, for any $B \geq 1$,
\[
\max_{\substack{a, q \in \mathbb{N}}} \sum_{\substack{M < m \leq 2^\ell M}} d_{\ell-1}(m) \left| \frac{1}{H_1}\sum_{\substack{X/m < n_{i_0} \leq (X+H_1)/m \\ n_{i_0} \sim N_{i_0} \\ n_{i_0} m \equiv a \pmod{q}}} a_{n_{i_0}} - \frac{1}{H_2}\sum_{\substack{X/m < n_{i_0} \leq (X+H_2)/m \\ n_{i_0} \sim N_{i_0} \\ n_{i_0} m \equiv a \pmod{q}}} a_{n_{i_0}} \right| \ll \frac{1}{(\log X)^B}
\]
for $a_{n_{i_0}} = \mathbf{1}_{(N_{i_0}, 2N_{i_0}]}(n_{i_0})$ and $a_{n_{i_0}} = \mathbf{1}_{(N_{i_0}, 2N_{i_0}]}(n_{i_0}) \log n_{i_0}$. But here $H_2/M \geq H_1/M \geq X^{\varepsilon/2}$, so the claim is easy to establish.

In the remaining case $(I_2^{\mathrm{maj}})$ holds and $\alpha_i, \alpha_j > \varepsilon/2$. Thus the corresponding coefficients from Heath-Brown's identity are either $1_{(N_i, 2N_i]}(n)$ or $(\log n) 1_{(N_i, 2N_i]}(n)$ and the claim follows from Lemma~\ref{le:typeI2} (and partial summation if needed).

\subsection{Major arc estimates with restricted prime factorization}
When proving Theorem~\ref{discorrelation-thm}(iv)--(v) we need the following quick consequence of Theorem~\ref{thm:major-arc}. One could obtain stronger results, but this is sufficient for our needs.
\begin{corollary}\label{cor:major-arc-Ramare}  Let $X \geq 3$ and $X^{7/12+\eps} \leq H \leq X^{1-\eps}$ for some $\eps > 0$. Let $2 \leq P < Q \leq X^{1/(\log\log X)^2}$ and write $\mathcal{P}(P, Q) = \prod_{P < p \leq Q} p$.
\begin{itemize}
\item[(i)] For all $A > 0$,
\begin{align*}
 \left|\sum_{\substack{X < n \leq X+H}} 1_{(n, \mathcal{P}(P, Q)) > 1} \mu(n) \right|^* &\ll_{A,\eps} \frac{H}{\log^{A} X} + \frac{H(\log X)^4}{P}.
\end{align*}
\item[(ii)]
Let $k \geq 2$. For all $A > 0$,
$$
\left|\sum_{\substack{X < n \leq X+H}} 1_{(n, \mathcal{P}(P, Q)) > 1} (d_k(n) - d^\sharp_k(n))\right|^* \ll_{A, \eps} \frac{H}{\log^A X} + \frac{H(\log X)^{4k}}{P}$$
\end{itemize}
\end{corollary}

\begin{proof} 
Let us first show (i). By Lemma~\ref{lem:MatoTera} it suffices to show that
\[
\left|\sum_{\substack{X < prn \leq X+H \\ P < p \leq Q \\ r \leq X^{\eps/2}}} a_r \mu(n)\right|^\ast \ll_{A, \eps} \frac{H}{\log^A X}
\]
whenever $|a_r| \leq d_2(r)$. By the triangle inequality and Theorem~\ref{thm:major-arc} the left-hand side is
\begin{align*}
&\ll \sum_{P < p \leq Q} \sum_{r \leq X^{\eps/2}} d_2(r) \left|\sum_{\substack{X/(pr) < n \leq (X+H)/(pr) }} \mu(n)\right|^\ast \\
&\ll_{A, \eps} \sum_{P < p \leq Q} \sum_{r \leq X^{\eps/2}} d_2(r) \frac{H}{pr (\log X)^{A+3}} \ll \frac{H}{\log^A X}.
\end{align*}

Let us now turn to (ii). By Theorem~\ref{thm:major-arc} and the triangle inequality it suffices to show the claim with $1_{(n, \mathcal{P}(P, Q)) > 1}$ replaced by $1_{(n, \mathcal{P}(P, Q)) = 1}$. Hence by M\"obius inversion we need to show that
\begin{equation}
\label{eq:mamuRclaim}
\left|\sum_{\substack{X < n \leq X+H}} \sum_{d \mid (n,\mathcal{P}(P, Q))} \mu(d) (d_k(n) - d^\sharp_k(n))\right|^* \ll_A \frac{H}{\log^A X}.
\end{equation}
Write $D := \min\{X^{\varepsilon/2000}, X^{c_k/2}\}$. Since $d_k^\sharp(m) \ll d_k(m)$ (see~\eqref{eq:dkapp<<dk}), the contribution of $d > D$ to the left-hand side of~\eqref{eq:mamuRclaim} is by Lemma~\ref{le:FLS} at most
$$
\ll \sum_{\substack{X < dn \leq X+H \\ d > D \\ d \mid \mathcal{P}(P, Q)}} d_k(dn) \ll_{A} \frac{H}{\log^A X}.
$$
On the other hand, the contribution of $d \leq D$ to the left-hand side of~\eqref{eq:mamuRclaim} is by the triangle inequality and Theorem~\ref{thm:major-arc}
\begin{align*}
&\left|\sum_{\substack{X < n \leq X+H}} \sum_{\substack{d \leq D \\ d \mid \mathcal{P}(P, Q)}} \mu(d)   1_{n \equiv 0 \pmod{d}}(d_k(n) - d^\sharp_k(n))\right|^* \\
&\ll \sum_{d \leq D} \left|\sum_{\substack{X < n \leq X+H}}  (d_k(n) - d^\sharp_k(n))\right|^* \ll_{\varepsilon} \frac{H}{X^{\varepsilon/2000}} + \frac{H}{X^{c_k/2}}.
\end{align*}
\end{proof}

\section{\texorpdfstring{Reduction to type $I$, type $II$, and type $I_2$ estimates}{Reduction to type I, type II, and type I2 estimates}}\label{reduction-sec}

To complement the major arc estimates in Theorem~\ref{thm:major-arc}, we will establish later in the paper some ``inverse theorems'' that provide discorrelation between an arithmetic function $f$ and a nilsequence $F(g(n)\Gamma)$ assuming that $f$ is of\footnote{Informally, we use type $I_k$ to refer to expressions resembling $\alpha * d_k$ for some arithmetic function $\alpha$ supported on a relatively short range, with the classical type $I$ sums corresponding to the case $k=1$, and type $II$ sums to refer to convolutions $\alpha*\beta$ where both $\alpha$ and $\beta$ are supported away from $1$.} ``type $I$'', ``type $II$'', or ``type $I_2$'', and the nilsequence is ``minor arc'' in a suitable sense.  To make this precise, we give some definitions:

\begin{definition}[Type $I$, $II$, $I_2$ sums]\label{struct-sum}  Let $0 < \delta < 1$ and $A_I, A_{II}^-, A_{II}^+, A_{I_2} \geq 1$.  
\begin{itemize}
\item[(i)]  (Type $I$ sum) A \emph{$(\delta,A_I)$ type $I$ sum} is an arithmetic function of the form $f = \alpha *\beta$, where $\alpha$ is supported in $[1,A_I]$, and one has the bounds
\begin{equation}\label{abound}
\sum_{n \leq A} |\alpha(n)|^2 \leq \frac{1}{\delta} A
\end{equation}
and
\begin{equation}\label{btv}
\| \beta \|_{\TV(\N; q)} \leq \frac{1}{\delta}
\end{equation}
for all $A \geq 1$ and some $1 \leq q \leq \frac{1}{\delta}$.
\item[(ii)]  (Type $II$ sum) A \emph{$(\delta, A_{II}^-, A_{II}^+)$ type $II$ sum} is an arithmetic function of the form $f = \alpha * \beta$, where $\alpha$ is supported on $[A_{II}^-,A_{II}^+]$, and one has the bound~\eqref{abound} and the bounds
\begin{equation}\label{bbound-2}
\sum_{n \leq B} |\beta(n)|^2 \leq \frac{1}{\delta} B \quad \text{and} \quad \sum_{n \leq B} |\beta(n)|^4 \leq \frac{1}{\delta^2} B
\end{equation}
for all $A,B \geq 1$.  
(The type $II$ sums become vacuous if $A_{II}^- > A_{II}^+$.)
\item[(iii)]  (Type $I_2$ sum)  A \emph{$(\delta, A_{I_2})$ type $I_2$ sum} is an arithmetic function of the form $f = \alpha * \beta_1 * \beta_2$, where $\alpha$ is supported on $[1,A_{I_2}]$ and obeys the bound~\eqref{abound} for all $A \geq 1$, and $\beta_1, \beta_2$ obey the bound~\eqref{btv} for some $1 \leq q \leq \frac{1}{\delta}$.
\end{itemize}
\end{definition}

We now state the inverse theorems we will establish here.

\begin{theorem}[Inverse theorems]\label{inverse}  Let $d,D \geq 1$, $2 \leq H \leq X$, $0 < \delta < \frac{1}{\log X}$, let $G/\Gamma$ be a filtered nilmanifold of degree at most $d$, dimension at most $D$, and complexity at most $1/\delta$.  Let $F \colon G/\Gamma \to \C$ be Lipschitz of norm at most $1/\delta$ and mean zero.  Let $f \colon \N \to \C$ be an arithmetic function such that
\begin{equation}\label{invo}
\left| \sum_{X < n \leq X+H} f(n) F(g(n) \Gamma) \right|^* \geq \delta H.
\end{equation}
for some polynomial map $g \colon \Z \to G$.
\begin{itemize}
\item[(i)]  (Type $I$ inverse theorem)  If $f$ is a $(\delta,A_I)$ type $I$ sum for some $A_I \geq 1$, then either
$$ H \ll_{d,D} \delta^{-O_{d,D}(1)} A_I$$
or else there exists a non-trivial horizontal character $\eta \colon G \to \R/\Z$ of Lipschitz norm $O_{d,D}( \delta^{-O_{d,D}(1)})$ such that
$$ \| \eta \circ g \|_{C^\infty(X,X+H]} \ll_{d,D} \delta^{-O_{d,D}(1)}.$$
\item[(ii)]  (Type $II$ inverse theorem, non-abelian case)  If $f$ is a $(\delta,A_{II}^-, A_{II}^+)$ type $II$ sum for some $A_{II}^+ \geq A_{II}^- \geq 1$, $G$ is non-abelian with one-dimensional center, and $F$ oscillates with a non-zero central frequency $\xi$ of Lipschitz norm at most $1/\delta$, then either
$$ H \ll_{d,D} \delta^{-O_{d,D}(1)} \max( A_{II}^+, X/A_{II}^- )$$
or else there exists a non-trivial horizontal character $\eta \colon G \to \R/\Z$ of Lipschitz norm $O_{d,D}( \delta^{-O_{d,D}(1)})$ such that
\begin{equation}\label{ut}
 \| \eta \circ g \|_{C^\infty(X,X+H]} \ll_{d,D} \delta^{-O_{d,D}(1)}.
\end{equation}
\item[(iii)]  (Type $II$ inverse theorem, abelian case)  If $f$ is a $(\delta,A_{II}^-, A_{II}^+)$ type $II$ sum for some $A_{II}^+ \geq A_{II}^- \geq 1$ and $F(g(n)\Gamma) = e(P(n))$ for some polynomial $P \colon \Z \to \R$ of degree at most $d$, then either
$$ H \ll_{d} \delta^{-O_{d}(1)} \max( A_{II}^+, X/A_{II}^- )$$
or else there exists a real number $T \ll_{d} \delta^{-O_d(1)} (X/H)^{d+1}$ such that
$$ \| e(P(n)) n^{-iT} \|_{\TV( (X,X+H] \cap \Z; q)} \ll_d \delta^{-O_d(1)} $$
for some $1 \leq q \ll_d \delta^{-O_d(1)}$.
\item[(iv)] (Type $I_2$ inverse theorem) If $f$ is a $(\delta,A_{I_2})$ type $I_2$ sum for some $A_{I_2} \geq 1$, then either
\begin{equation}\label{hale}
 H \ll_{d,D} \delta^{-O_{d,D}(1)} X^{1/3} A_{I_2}^{2/3}
\end{equation}
or else there exists a non-trivial horizontal character $\eta \colon G \to \R/\Z$ of Lipschitz norm $O_{d,D}( \delta^{-O_{d,D}(1)})$ such that
$$ \| \eta \circ g \|_{C^\infty(X,X+H]} \ll_{d,D} \delta^{-O_{d,D}(1)}.$$
\end{itemize}
\end{theorem}

In this section we show how Theorem~\ref{inverse}, when combined with the major arc estimates in Theorem~\ref{thm:major-arc}, gives Theorem~\ref{discorrelation-thm}.

\subsection{Combinatorial decompositions}

We start by describing the combinatorial decompositions that allow us to reduce sums involving $\mu,\Lambda,d_k$ to type $I$, type $II$, and type $I_2$ sums. Lemma~\ref{comb-lambda} will be used to prove~\eqref{mobius-discor} and~\eqref{mangoldt-discor}, Lemma~\ref{comb-divisor} will be used to prove~\eqref{dk-discor}, and Lemma~\ref{comb-mu} will be used to prove~\eqref{mobius-discor-alt} and~\eqref{dk-discor-alt}.

The model function $\Lambda^{\sharp}$ is not quite a type $I$ sum, but we can approximate it well by the type I sum\footnote{One could alternatively use a type $I$ approximant coming from the $\beta$-sieve, using the fundamental lemma of the sieve (see e.g.~\cite[Lemma 6.3]{ik}) but the simper approximant $\Lambda_I^\sharp$ is sufficient for us.}
\begin{equation}\label{lambdasharp-i-def}
\Lambda_I^\sharp(n) := \frac{P(R)}{\varphi(P(R))}\sum_{\substack{d \leq X^{\theta/2}\\d\mid (n, P(R))}} \mu(d).
\end{equation}
Indeed by~\eqref{lambdar-def}, M\"obius inversion and Lemma~\ref{le:FLS} we have
\begin{equation}\label{eq:lambdasharp2}
 \sum_{X < n \leq X+H} |\Lambda^\sharp_I(n)-\Lambda^{\sharp}(n)|\leq \frac{P(R)}{\varphi(P(R))} \sum_{\substack{X < dn \leq X+H \\ d > X^{\theta/2} \\ d \mid P(R)}} 1 \ll H\exp(-(\log X)^{1/20}).
\end{equation}
In practice, this bound allows us to substitute $\Lambda^\sharp$ with the type $I$ sum $\Lambda^\sharp_I$ with negligible cost.

\begin{lemma}[Combinatorial decompositions of $\mu,\Lambda,$ and $\Lambda^\sharp_I$]\label{comb-lambda}
Let $X^{\theta+\eps} \leq H \leq X$ for $\theta = 5/8$ and some fixed $\eps > 0$. For each $g \in \{\mu, \Lambda, \Lambda^\sharp_I\}$, we may find a collection $\mathcal{F}$ of $O((\log X)^{O(1)})$ functions $f \colon \mathbb{N} \to \mathbb{R}$ such that
$$ g(n) = \sum_{f \in \mathcal{F}} f(n) $$
for each $X/2 \leq n \leq 4X$, and each component $f \in \mathcal{F}$ satisfies one of the following:
\begin{itemize}
\item[(i)] $f$ is a $(\log^{-O(1)} X, O(X^{\theta}))$ type $I$ sum;
\item[(ii)] $f$ is a $(\log^{-O(1)} X, O(X^{(3\theta-1)/2}))$ type $I_2$ sum.
\item[(iii)] $f$ is a $(\log^{-O(1)} X, A_{II}^-, A_{II}^+)$ type $II$ sum for some $X^{1-\theta} \ll A_{II}^- \leq A_{II}^+ \ll X^{\theta}$, and it obeys the bound
\begin{equation}\label{eq-comb-lambda} 
\sup_{(X/H)(\log X)^{50A} \leq |T| \leq X^A}\left|\sum_{X < n \leq X+H} f(n) n^{iT}\right|^* \ll_A H\log^{-A} X
\end{equation}
for all sufficiently large $A \geq 1$.
\end{itemize}
\end{lemma}

\begin{lemma}[Combinatorial decompositions of $d_k$ and $d_k^{\sharp}$]\label{comb-divisor}
Let $k \geq 2$.
Let $X^{\theta + \eps} \leq H \leq X$ for $\theta = \theta_k$ and some fixed $\eps > 0$, where $\theta_2 = 1/3$, $\theta_3 = 5/9$, and $\theta_k = 5/8$ for $k \geq 4$. For each $g \in \{d_k, d_k^\sharp\}$, we may find a collection $\mathcal{F}$ of $O((\log X)^{O(1)})$ functions $f \colon \mathbb{N} \to \mathbb{R}$ such that
$$ g(n) = \sum_{f \in \mathcal{F}} f(n) $$
for each $X/2 \leq n \leq 4X$, and each component $f \in \mathcal{F}$ satisfies one of the following:
\begin{itemize}
\item[(i)] $f$ is a $(\log^{-O(1)} X, O(X^{\theta}))$ type $I$ sum;
\item[(ii)] $f$ is a $(\log^{-O(1)} X, O(X^{(3\theta-1)/2}))$ type $I_2$ sum.
\item[(iii)] $f$ is a $(\log^{-O(1)} X, A_{II}^-, A_{II}^+)$ type $II$ sum for some $X^{1-\theta} \ll A_{II}^- \leq A_{II}^+ \ll X^{\theta}$ and it obeys the bound
\begin{equation}\label{eq-comb-divisor} 
\sup_{(X/H)X^{2c} \leq |T| \leq X^A}\left|\sum_{X < n \leq X+H} f(n) n^{iT}\right|^* \ll_{A,k} H X^{-c}
\end{equation}
for all $A > 0$, where $c = c_{k,A} > 0$ is a sufficiently small constant.
\end{itemize}
\end{lemma}

\begin{lemma}[Flexible combinatorial decompositions of $\mu, d_k,$ and $d_k^\sharp$]\label{comb-mu}
Let $X^{3/5 + \eps} \leq H \leq X$ for some fixed $\eps > 0$, let $\exp((\log x)^{2/3+\eps}) \leq P \leq Q \leq X^{1/(\log \log X)^2}$, and write $\mathcal{P}(P, Q) = \prod_{P < p \leq Q} p$. We can find a collection $\mathcal{F}$ of functions,  where $|\mathcal{F}| =O((\log X)^{O(1)})$, such that for any sequence $\{\omega_n\}$ with $|\omega_n| \leq 1$,
$$ \sum_{X < n \leq X+H} 1_{(n, \mathcal{P}(P, Q)) > 1} \mu(n)\omega_n = \sum_{f \in \mathcal{F}}  \sum_{X < n \leq X+H}f(n)\omega_n + O\left(\frac{H \log^4 X}{P} + \frac{H}{\exp((\log \log X)^2)}\right). $$
Moreover, each component $f \in \mathcal{F}$ satisfies one of the following:
\begin{itemize}
\item[(i)] $f$ is a $(\log^{-O(1)} X, X^{3/5+\eps/10})$ type $I$ sum;
\item[(ii)] $f$ is a $(\log^{-O(1)} X, X^{2/5+\eps/10})$ type $I_2$ sum.
\item[(iii)] $f$ is a $(\log^{-O(1)} X, X^{2/5-\eps/10}, X^{3/5+\eps/10})$ type $II$ sum and it obeys the bound
\begin{equation}\label{eq-comb-mu} 
\sup_{(X/H)(\log X)^{20A} \leq |T| \leq X^A}\left|\sum_{X < n \leq X+H} f(n) n^{iT}\right|^* \ll_A H\log^{-A} X
\end{equation}
for all sufficiently large $A > 0$.
\end{itemize}
Similarly, for fixed $k \geq 2$ we can find a collection $\mathcal{F}$ of functions, where $|\mathcal{F}| = O((\log X)^{O(1)})$, such that for any sequence $\{\omega_n\}$ with $|\omega_n| \leq 1$,
$$ \sum_{X < n \leq X+H} d_k(n)\omega_n 1_{(n, \mathcal{P}(P, Q)) > 1} = \sum_{f \in \mathcal{F}}  \sum_{X < n \leq X+H} f(n)\omega_n + O\left(\frac{H \log^{4k} X}{P} + \frac{H}{\exp((\log \log X)^2)}\right). $$
Moreover, each component $f \in \mathcal{F}$ is one of (i), (ii), or (iii) above, and a similar decomposition holds also with $d_k^\sharp$ in place of $d_k$.
\end{lemma}

We will prove Lemmas~\ref{comb-lambda},~\ref{comb-divisor} and~\ref{comb-mu} by first decomposing the relevant functions into certain Dirichlet convolutions (using Lemma~\ref{hb-identity} in the proof of Lemma~\ref{comb-lambda} and Lemma~\ref{lem:MatoTera} in the proof of Lemma~\ref{comb-mu}). We then use Lemma~\ref{combinatorial} to arrange each convolution into either type $I$, type $II$, or type $I_2$ sums. In the case of type $II$ sums, Lemma~\ref{combinatorial} also allows us to arrange them into a triple convolution for which Lemma~\ref{le:BHP} is applicable.

\begin{remark}
\label{rem:AlternativeMajor}
Let us briefly discuss the type $II$ conditions such as~\eqref{eq-comb-lambda}, concentrating on the case of the von Mangoldt function.

One may observe from the proof of Theorem~\ref{discorrelation-thm}(ii) below that if our major arc estimate (Theorem~\ref{thm:major-arc}(i)) held, for any $T \leq X^{O(1)}$, with $(\Lambda(n)-\Lambda^\sharp(n))n^{iT}$ in place of $\Lambda(n)-\Lambda^\sharp(n)$, we could prove Theorem~\ref{discorrelation-thm}(ii) without the need to impose in Lemma~\ref{comb-lambda} the condition~\eqref{eq-comb-lambda} concerning type $II$ sums.

Unfortunately, with current knowledge, one cannot obtain such a twisted version of Theorem~\ref{thm:major-arc}, at least not in the whole range $X^{7/12+\varepsilon} \leq H \leq X^{1-\varepsilon}$. However, inserting special cases of our type $I$ and type $I_2$ estimates into Section~\ref{major-arc-sec}, it would be possible to obtain such a twisted variant in the relevant range $X^{5/8+\varepsilon} \leq H \leq X^{1-\varepsilon}$. If we did this, we would not need to impose the condition~\eqref{eq-comb-lambda}. However, we found it more natural to work out the major arc estimates first using existing methods without needing to appeal to the more involved $I_2$ case.
\end{remark}

\begin{proof}[Proof of Lemma~\ref{comb-lambda}]
The function $\Lambda^\sharp_I$ is clearly a $(\log^{-O(1)} X, O(X^{\theta}))$ type $I$ sum by definition~\eqref{lambdasharp-i-def}.
For $\Lambda$ and $\mu$, we apply Lemma~\ref{hb-identity} with $L = 10$. Each component $f \in \mathcal{F}$ takes the form
\begin{equation}\label{hb-convolution} 
f = a^{(1)}* \cdots * a^{(\ell)} 
\end{equation}
for some $\ell \leq 20$, where each $a^{(i)}$ is supported on $(N_i, 2N_i]$ for some $N_i \geq 1/2$, and each $a^{(i)}(n)$ is either $1_{(N_i, 2N_i]}(n)$, $(\log n)1_{(N_i, 2N_i]}(n)$, or $\mu(n)1_{(N_i, 2N_i]}(n)$. Moreover, $N_1N_2\cdots N_{\ell} \asymp X$, and $N_i \leq X^{1/10}$ for each $i$ with $a^{(i)}(n) = \mu(n) 1_{(N_i, 2N_i]}(n)$.

We can find $\alpha_1,\ldots,\alpha_{\ell} \in [0,1]$ with $\sum_{i=1}^{\ell}\alpha_i=1$, such that $N_i \asymp X^{\alpha_i}$ for each $i$. If $\alpha_i > 1/10$ for some $i$, then  $a^{(i)}(n)$ is either $1_{(N_i, 2N_i]}(n)$ or $(\log n)1_{(N_i, 2N_i]}(n)$, and hence $\|a^{(i)}\|_{\TV(\N)} \ll \log X$.

Since $\theta = 5/8 \geq 3/5$, we may apply Lemma~\ref{combinatorial}(i), (ii) to conclude that either ($I$) holds, or ($I_2$) holds, or both ($II^{\mathrm{min}}$) and ($II^{\mathrm{maj}}$) hold.

First consider the case ($I$) holds, i.e. $\alpha_i \geq 1-\theta$ for some $i$. Since $\alpha_i > 1/10$, $\|a^{(i)}\|_{\TV(\N)} \ll \log X$, and~\eqref{hb-convolution} is a $(\log^{-O(1)} X, O(X^{\theta}))$ type $I$ sum of the form $\alpha*\beta$ with $\beta = a^{(i)}$ and $\alpha = a^{(1)}*\cdots*a^{(i-1)}*a^{(i+1)}*\cdots*a^{(k)}$. 

Henceforth we may assume that $\alpha_i < 1-\theta$ for each $i$. Next consider the case ($I_2$) holds. Then $\alpha_i + \alpha_j\geq \tfrac{3}{2}(1-\theta)$ for some $i<j$. Since $\alpha_i,\alpha_j \leq 1-\theta$, this implies that $\alpha_i, \alpha_j> 1/10$ and thus $\|a^{(i)}\|_{\TV(\N)}, \|a^{(j)}\|_{\TV(\N)}  \ll \log X$. Hence~\eqref{hb-convolution} is a $(\log^{-O(1)} X, O(X^{(3\theta-1)/2}))$ type $I_2$ sum of the form $f = \alpha*\beta_1*\beta_2$, with $\beta_1 = a^{(i)}$, $\beta_2 = a^{(j)}$.

Finally consider the case when both ($II^{\mathrm{min}}$) and ($II^{\mathrm{maj}}$) hold. Let $\{1,\ldots,\ell\} = J\uplus J'$ be the partition from ($II^{\mathrm{min}}$), so that $\alpha_J, \alpha_{J'} \in [1-\theta, \theta]$. Then~\eqref{hb-convolution} is a $(\log^{-O(1)} X, A_{II}^-, A_{II}^+)$  type $II$ sum of the form $f = \alpha * \beta$, where $\alpha$ (resp. $\beta$) is the convolution of those $a^{(i)}$ with $i \in J$ (resp. $i \in J'$), and $X^{1-\theta} \ll A_{II}^- \leq A_{II}^+  \ll X^{\theta}$.

It remains to establish the bound~\eqref{eq-comb-lambda}. For any subinterval $(X_1, X_1+H_1] \subset (X, X+H]$, any residue class $a\pmod{q}$, any fixed $A > 0$, and any $(X/H)(\log X)^{50A} \leq |T| \leq X^A$, we need to show that
$$ \Big|\sum_{\substack{X_1 < n \leq X_1+H_1 \\ n \equiv a\pmod{q}}} f(n) n^{iT}\Big| \ll_A H\log^{-A} X. $$
We may assume that $A$ is sufficiently large, $H_1 \geq H(\log X)^{-2A}$ and $q \leq (\log X)^{2A}$.
Let now $\{1,\ldots,\ell\} = I \uplus J \uplus J'$ be the partition from ($II^{\mathrm{maj}}$), so that
$$ 
2\theta - 1 \leq \alpha_I \leq 4\theta-2, \ \ |\alpha_J - \alpha_{J'}| \leq 2\theta-1. 
$$

Let $\{a_{m_1}'\}, \{b_{m_2}'\}, \{v_{\ell}'\}$ be the convolution of those $a^{(i)}$ with $i \in J$, $i \in J'$, $i \in I$, respectively. Note that they are supported on $m_1 \asymp X_1^{\alpha_J}$, $m_2 \asymp X_1^{\alpha_{J'}}$, $\ell \asymp X_1^{\alpha_I}$, respectively.
Thus, after dyadic division of the ranges of $m_1,m_2,\ell$, we need to show that
$$ 
\Big|\sum_{\substack{X_1 < m_1m_2\ell < X_1+H_1 \\ m_1\sim M_1, m_2\sim M_2, \ell\sim L \\ m_1m_2\ell \equiv a\pmod{q}}} a_{m_1}'m_1^{iT} b_{m_2}'m_2^{iT} v_{\ell}'\ell^{iT} \Big| \ll_A H\log^{-A} X 
$$
for $M_1 \asymp X_1^{\alpha_J}$, $M_2 \asymp X_1^{\alpha_{J'}}$, $L \asymp X_1^{\alpha_I}$. In view of Lemma~\ref{le:BHP}(ii) applied with $W = (\log X)^{10A}$ and $v_\ell = v_\ell'\ell^{iT}$, it suffices to verify the hypothesis~\eqref{eq:BHP-supt-cond-2}. There exists $i_0 \in I$ such that $\alpha_{i_0} \geq (2\theta-1)/20 = 1/80$. Now~\eqref{eq:BHP-supt-cond-2} follows if we show that
$$ 
\max_{r \mid (a,q)}\,\, \max_{\chi\pmod{\frac{q}{(a,q)}}} \sup_{|t| \leq \frac{X_1 (\log X)^{40A}}{H_1}} \Big|\sum_{m \asymp X^{\alpha_{i_0}}/r} \frac{a^{(i_0)}(mr)\chi(m)}{m^{1/2+i(t-T)}} \Big| \ll_A \frac{(X^{\alpha_{i_0}}/r)^{1/2}}{(\log X)^{10A}}. 
$$
Since $a^{(i_0)}$ is either $1$, $\log$, or $\mu$ on its support, this follows from Lemma~\ref{le:nitcancel} applied with $T_0 = (\log X)^{45A}$.
\end{proof}

\begin{proof}[Proof of Lemma~\ref{comb-divisor}]
The function $d_k^{\sharp}$ is clearly a $(\log^{-O(1)} X, O(X^{\theta}))$ type $I$ sum by definition~\eqref{dks-def}. On the other hand
$d_k$ can be decomposed into a sum of $\log^k X$ terms, each of which takes the form
$$ f =1_{(N_1, 2N_1]} * \cdots * 1_{(N_k, 2N_k]}$$
for some $N_i \geq 1/2$ with $N_1N_2 \cdots N_{k} \asymp X$. The $k \geq 4$ case of the lemma then follows in a similar way as Lemma~\ref{comb-lambda}, with the only difference being that Lemma~\ref{le:BHP} is now applied with $W = X^{10c}$ instead of a power of $\log X$.

In the case $k=2$ and $\theta = 1/3$, $f$ is clearly a $(\log^{-O(1)} X, 1)$ type $I_2$ sum. In the case $k=3$ and $\theta = 5/9$, at least one of the $N_i$'s (say $N_3$) is $\ll X^{1/3}$. Hence $f$ is a $(\log^{-O(1)} X, O(X^{1/3}))$ type $I_2$ sum of the form $f = \alpha * \beta_1 * \beta_2$, with $\alpha = 1_{(N_3, 2N_3]}$ and $\beta_j = 1_{(N_j, 2N_j]}(n)$ for $j = 1, 2$.
\end{proof}

\begin{proof}[Proof of Lemma~\ref{comb-mu}]
Let us first outline the proof for $\mu$. We first apply Lemma~\ref{lem:MatoTera} and then Heath-Brown's identity (Lemma~\ref{hb-identity}) with $L=10$ to $\mu(n)$ on the right-hand side; note that we now have extra flexibility with the $p$ variable. We obtain a collection of functions $\mathcal{F}$, where each $f \in \mathcal{F}$ takes the form
$$ f = a^{(0)} * a^{(1)} * \cdots * a^{(\ell)} $$
for some $\ell \leq 21$, where each $a^{(i)}$ is supported on $(N_i, 2N_i]$ for some $N_i \geq 1/2$, with
$$ P/2 \leq N_0 \leq Q, \ \ N_1 \leq X^{\eps/30}, \ \ N_0N_1\cdots N_{\ell} \asymp X. $$
(Here $a^{(0)}$ comes from the $p$ variable, $a^{(1)}$ comes from the $r$ variable, and $a^{(2)}*\cdots*a^{(\ell)}$ comes from applying Heath-Brown's identity to $\mu(n)$.)
Moreover, $a^{(0)}(n) = 1_{n\text{ prime}}1_{(N_0, 2N_0]}(n)$, $a^{(1)}$ is divisor-bounded, and for each $i \geq 2$, $a^{(i)}(n)$ is either $1_{(N_i, 2N_i]}(n)$ or $\mu(n)1_{(N_i, 2N_i]}(n)$, and $N_i \leq X^{1/10}$ for each $i$ with $a^{(i)} = \mu(n)1_{(N_i, 2N_i]}(n)$. 

We can find $\alpha_1,\ldots,\alpha_{\ell} \in [0,1]$ with $\sum_{i=1}^{\ell}\alpha_i=1$, such that $X^{\alpha_i-\eps/20} \leq N_i \ll X^{\alpha_i}$ for each $1 \leq i \leq \ell$. We may apply Lemma~\ref{combinatorial}(ii) to conclude that either ($I$) holds, or ($I_2$) holds, or ($II^{\mathrm{min}}$) holds.

As in the proof of Lemma~\ref{comb-lambda}, if ($I$) holds then $f$ is a desired type $I$ sum, if ($I_2$) holds then $f$ is a desired type $I_2$ sum, and if ($II^{\mathrm{min}}$) holds then $f$ is a desired type $II$ sum. It remains to establish the bound~\eqref{eq-comb-mu} in the type $II$ case.
Let $\{1,\ldots,\ell\} = J \uplus J'$ be the partition from ($II^{\mathrm{min}}$), so that $|\alpha_J - \alpha_{J'}| \leq 1/5$. In view of Lemma~\ref{le:BHP}(ii) with $W = (\log X)^{4A}$, it suffices to verify the hypothesis~\eqref{eq:BHP-supt-cond-2} for the sequence
$$ v_{\ell} = a_{\ell}^{(0)}\ell^{iT} = 1_{\ell\text{ prime}} \ell^{iT}. $$
Since $N_0 \gg P$, Lemma~\ref{le:nitcancel} implies that hypothesis~\eqref{eq:BHP-supt-cond-2} is satisfied when $(\log X)^{20A} X/H \leq |T| \leq X^A$ as required.

The claim for $d_k$ follows similarly.

In case $d_k^\sharp$ we use M\"obius inversion to write
\begin{align*}
\sum_{X < n \leq X+H} d_k^\sharp(n)\omega_n 1_{(n, \mathcal{P}(P, Q)) > 1} &=  \sum_{X < n \leq X+H} d_k^\sharp(n)\omega_n -  \sum_{X < n \leq X+H} d_k^\sharp(n)\omega_n 1_{(n, \mathcal{P}(P, Q)) = 1} \\
&= \sum_{X < n \leq X+H} d_k^\sharp(n)\omega_n -  \sum_{\substack{X < dn \leq X+H \\ d \mid \mathcal{P}(P, Q)}} \mu(d) d_k^\sharp(dn) \omega_{dn}.
\end{align*}
Now $d_k^\sharp(n)$ is immediately a $(\log^{-O(1)} X, X^{3/5})$ type $I$ sum by the definition~\eqref{dks-def}.
Using Lemma~\ref{le:FLS} we can truncate the last sum above to $d \leq X^{\varepsilon/10}$ with an admissible error $O(H/\exp((\log \log X)^2/20))$ and it remains to show that 
\[
f(n) = \sum_{\substack{d \mid (n, \mathcal{P}(P, Q)) \\ d \leq X^{\varepsilon/10}}} \mu(d) d_k^\sharp(dn) 
\]
is also a $(\log^{-O(1)} X, X^{3/5})$ type $I$ sum. But this follows easily from the definition~\eqref{dks-def} of $d_k^\sharp$.
\end{proof}

\subsection{Deduction of Theorem~\ref{discorrelation-thm}}\label{discor-sec}

In this subsection we deduce Theorem~\ref{discorrelation-thm} from Theorem~\ref{inverse}. We focus on establishing~\eqref{mangoldt-discor}. The other estimates in Theorem~\ref{discorrelation-thm} are established similarly and we mention the small differences at the end of the section. In this section we allow all implied constants to depend on $d,D$.

We induct on the dimension $D$ of $G/\Gamma$. In view of the major arc estimates (Theorem~\ref{thm:major-arc}), we may assume that $F$ has mean zero (after replacing $F$ by $F - \int F$). In view of Proposition~\ref{central} with $\delta = \log^{-A} X$, we may assume that $F$ oscillates with a central frequency $\xi \colon Z(G) \rightarrow \R$. If the center $Z(G)$ has dimension larger than $1$, or $\xi$ vanishes, then $\ker\xi$ has positive dimension and the conclusion follows from induction hypothesis applied to $G/\ker\xi$ (via Lemma~\ref{quotient-normal}). Henceforth we assume that $G$ has one-dimensional center and that $\xi$ is non-zero. (A zero-dimensional center is not possible since $G$ is nilpotent and non-trivial.) 

Let $X^{\theta + \eps} \leq H \leq X^{1-\eps}$ for $\theta = 5/8$ and $\eps > 0$. Redefining $\delta$, we see that, to prove~\eqref{mangoldt-discor}, it suffices to show the following claim: There exists a small $c> 0$ such that for any large $A$ and $\delta = \log^{-A} X$, if $G/\Gamma$ has complexity at most $\delta^{-c}$ and $F$ has Lipschitz norm at most $\delta^{-c}$, then we have
\begin{equation}
\label{eq:discorLamClaim}
| \sum_{X < n \leq X+H} (\Lambda(n) - \Lambda^{\sharp}(n)) \overline{F}(g(n) \Gamma) |^* \leq \delta H.
\end{equation}
Suppose that~\eqref{eq:discorLamClaim} fails, i.e.
\begin{equation}
\label{eq:discorLamClaimRev}
| \sum_{X < n \leq X+H} (\Lambda(n) - \Lambda^{\sharp}(n)) \overline{F}(g(n) \Gamma) |^* > \delta H.
\end{equation}
By~\eqref{eq:lambdasharp2} and the triangle inequality, we then have
\begin{equation}
\label{eq:discorLamClaim-2}
| \sum_{X < n \leq X+H} (\Lambda(n) - \Lambda^\sharp_I(n)) \overline{F}(g(n) \Gamma) |^* \gg \delta H .
\end{equation}

By Lemma~\ref{comb-lambda}, for some component $f \in \mathcal{F}$ as in that lemma, one has the bound
\begin{equation}\label{eq:discorLamClaim-3}
 | \sum_{X < n \leq X+H} f(n) \overline{F}(g(n) \Gamma) |^* \gg \delta^{O(1)} H. 
\end{equation}

Consider first the case when $f$ is a $(\log^{-O(1)} X, A_{II}^-, A_{II}^+)$ type $II$ sum with $X^{1-\theta} \ll A_{II}^- \leq A_{II}^+ \ll X^{\theta}$ obeying~\eqref{eq-comb-lambda}, and $G$ is abelian, hence one-dimensional since $G=Z(G)$.  Then we may identify $G/\Gamma$ with the standard circle $\R/\Z$ (increasing the Lipschitz constants for $F$, $\xi$ by $O(\delta^{-O(1)})$ if necessary) and $\xi$ with an element of $\Z$ of magnitude $O(\delta^{O(1)})$, and we can write
$$ F(x) = b e(\xi x)$$
for some $b = O(\delta^{-O(1)})$ and all $x \in \R/\Z$.  We can write $\xi \cdot g(n) \Gamma = P(n) \hbox{ mod } 1$ for some polynomial $P \colon \Z \rightarrow \R$ of degree at most $d$, thus by~\eqref{eq:discorLamClaimRev},~\eqref{eq:discorLamClaim-3} we have
\begin{equation}\label{fepn}
| \sum_{X < n \leq X+H} f(n) e(-P(n)) |^* \geq \delta^{O(1)} H
\end{equation}
and
\begin{equation}\label{fepn-2}
| \sum_{X < n \leq X+H} (\Lambda(n)-\Lambda^\sharp(n)) e(-P(n)) |^* \geq \delta^{O(1)} H.
\end{equation}
Theorem~\ref{inverse}(iii) implies that there exists a real number $T \ll \delta^{-O(1)} (X/H)^{d+1}$ such that
\begin{equation}
\label{eq:EP(n)n-iT} 
\| e(P(n)) n^{-iT} \|_{\TV((X, X+H] \cap \Z; q)} \ll\delta^{-O(1)}
\end{equation}
for some $1 \leq q \leq \delta^{-O(1)}$. By Lemma~\ref{basic-prop}(iii), we thus obtain
\begin{equation}\label{Dirichlet-sum} 
\left|\sum_{X < n \leq X+H} f(n) n^{-iT} \right|^* \gg \delta^{O(1)} H.
\end{equation}
By~\eqref{eq-comb-lambda}, we must have $|T| \leq \delta^{-O(1)}X/H$, and thus by~\eqref{tvp} we have
$$ \|n^{iT}\|_{\TV((X, X+H] \cap \Z; q)} \ll \delta^{-O(1)}.$$
Hence by~\eqref{eq:EP(n)n-iT} and~\eqref{tv-prod} we have
$$ \|e(P(n))\|_{\TV((X, X+H] \cap \Z; q)} \ll \delta^{-O(1)}.$$
From~\eqref{fepn-2} and Lemma~\ref{basic-prop}(iii), we conclude that
$$| \sum_{X < n \leq X+H} \Lambda(n) - \Lambda^\sharp(n)|^* \gg \delta^{O(1)} H.$$
But this contradicts the major arc estimates (Theorem~\ref{thm:major-arc}(i)).

Hence in case $f$ is a type $II$ sum we can assume that $G$ is non-abelian with one-dimensional center.  We claim that in all the remaining cases arising from Lemma~\ref{comb-lambda}, Theorem~\ref{inverse} implies that there exists a non-trivial horizontal character $\eta \colon G \rightarrow \R/\Z$ of Lipschitz norm $\delta^{-O(1)}$ such that
\begin{equation}\label{g-smooth}
\|\eta \circ g\|_{C^{\infty}(X, X+H]} \gg \delta^{-O(1)}. 
\end{equation}

Indeed, in the case when $f$ is a $(\log^{-O(1)}X, A_I)$ type $I$ sum for some $A_I = O(X^{\theta})$, the bound $H \ll (\log X)^{O(1)}A_I$ fails since $H \geq X^{\theta+\varepsilon}$. Hence~\eqref{g-smooth} follows from Theorem~\ref{inverse}(i).

In the case when $f$ is a $(\log^{-O(1)}X, A_{I_2})$ type $I_2$ sum for some $A_{I_2} = O(X^{(3\theta-1)/2})$,  the bound $H \ll (\log X)^{O(1)} X^{1/3} A_{I_2}^{2/3}$ fails since $H \geq X^{\theta+\varepsilon}$ and $X^{1/3}A_{I_2}^{2/3} = O(X^{\theta})$. Hence~\eqref{g-smooth} follows from Theorem~\ref{inverse}(iv).

In the case when $f$ is a $(\log^{-O(1)}X, A_{II}^-, A_{II}^+)$ type $II$ sum for some $X^{1-\theta} \ll A_{II}^- \ll A_{II}^+ \ll X^{\theta}$, we can assume that $G$ is non-abelian with one-dimensional center as discussed above to meet the assumption in Theorem~\ref{inverse}(ii). The bound $H \ll (\log X)^{O(1)} \max(A_{II}^+, X/A_{II}^-)$ fails since $H \geq X^{\theta+\varepsilon}$ and $\max(A_{II}^+, X/A_{II}^-) \ll X^{\theta}$, and thus~\eqref{g-smooth} follows from Theorem~\ref{inverse}(ii).

Now that we have~\eqref{g-smooth}, we can reduce the dimension (by passing to a proper subnilmanifold) and apply the induction hypothesis to conclude the proof. By~\eqref{g-smooth} and Lemma~\ref{factor-simple}, we have a decomposition $g = \eps g'\gamma$ for some $\eps, g', \gamma \in \Poly(\Z \to G)$ such that
\begin{itemize}
\item[(i)] $\eps$ is $(\delta^{-O(1)}, (X, X+H])$-smooth;
\item[(ii)]  There is a $\delta^{-O(1)}$-rational proper subnilmanifold $G'/\Gamma'$ of  $G/\Gamma$ such that $g'$ takes values in $G'$ (in fact $G' = \ker\eta$); and
\item[(iii)] $\gamma$ is $\delta^{-O(1)}$-rational.
\end{itemize}

Let $q \leq \delta^{-O(1)}$ be the period of $\gamma \Gamma$. Form a partition $(X, X+H] = P_1 \cup \cdots \cup P_r$ for some $r \leq \delta^{-O(1)}$, where each $P_i$ is an arithmetic progression of modulus $q$ and $d_G(\eps(n), \eps(n')) \leq \delta^4$ whenever $n,n' \in P_i$ (which can be ensured by the smoothness of $\eps$ as long as $|P_i| \leq \delta^CH$ for some sufficiently large constant $C$).
 By the triangle inequality in Lemma~\ref{basic-prop}(i), we have
$$ 
\left| \sum_{X < n \leq X+H} (\Lambda-\Lambda^\sharp)(n) F(g(n)\Gamma) \right|^* \leq \sum_{i=1}^r \left| \sum_{n \in P_i} (\Lambda-\Lambda^\sharp)(n) F(g(n)\Gamma) \right|^*.
$$
For each $i$, fix any $n_i \in P_i$, and write $\gamma(n_i) \Gamma = \gamma_i \Gamma$ for some $\gamma_i \in G$ which is rational of height $O(\delta^{-O(1)})$. Let $g_i \in \Poly(\Z\to G)$ be the polynomial sequence defined by
$$ g_i(n) = \gamma_i^{-1} g'(n) \gamma_i, $$
which takes values in $\gamma_i^{-1}G'\gamma_i$. Let $F_i \colon G/\Gamma \to \C$ be the function defined by
$$ F_i(x\Gamma) = F(\eps(n_i)\gamma_i x\Gamma). $$
For each $n \in P_i$ we have
\begin{align*} 
|F(g(n)\Gamma) - F_i(g_i(n)\Gamma)| &= |F(g(n)\Gamma) - F(\eps(n_i) g'(n)\gamma_i\Gamma)| \\
&\leq \|F\|_{\Lip} \cdot d_G(\eps(n) g'(n)\gamma_i, \eps(n_i) g'(n) \gamma_i) \\ 
&= \|F\|_{\Lip} \cdot d_G(\eps(n), \eps(n_i)) \leq \delta^3. 
\end{align*}
It follows that
\begin{equation}
\label{eq:LamF-Fi}
\left| \sum_{X < n \leq X+H} (\Lambda-\Lambda^\sharp)(n) F(g(n)\Gamma) \right|^* \leq \sum_{i=1}^r \left| \sum_{n \in P_i} (\Lambda-\Lambda^\sharp)(n) F_i(g_i(n)\Gamma) \right|^* + O(\delta^2H)
\end{equation}
By Lemma~\ref{basic-prop}(i) and the induction hypothesis, we have, for each $i = 1, \dotsc, r$,
\begin{equation}\label{eq:ind-hyp}
\left| \sum_{n \in P_i} (\Lambda-\Lambda^\sharp)(n) F_i(g_i(n)\Gamma) \right|^* \leq \left| \sum_{X < n \leq X+H} (\Lambda-\Lambda^\sharp)(n) F_i(g_i(n)\Gamma) \right|^* \ll \delta^CH
\end{equation}
for any sufficiently large constant $C$. Combining this with~\eqref{eq:LamF-Fi} we obtain
$$ \left| \sum_{X < n \leq X+H} (\Lambda-\Lambda^\sharp)(n) F(g(n)\Gamma) \right|^* \ll \delta^{2} H, $$
contradicting our assumption~\eqref{eq:discorLamClaimRev}. This completes the proof of~\eqref{mangoldt-discor}.

The proof of~\eqref{mobius-discor} is completely similar (with the role of $\Lambda^\sharp$ and $\Lambda^\sharp_I$ both replaced by $\mu^\sharp = 0$). For the estimate~\eqref{dk-discor} involving $d_k$, one runs the argument above with $\delta = X^{-c\eps}$ for some sufficiently small constant $c>0$, using Lemma~\ref{comb-divisor}, and with the role of $\Lambda^\sharp$ and $\Lambda^\sharp_I$ both replaced by $d_k^\sharp$. 

Let us now turn to the estimate~\eqref{mobius-discor-alt}. We choose 
\begin{equation}
\label{eq:PQdef}
P = \exp((\log x)^{2/3+\eps}) \quad \text{and} \quad Q = x^{1/(\log\log x)^2}
\end{equation}
and write $\mathcal{P}(P, Q) = \prod_{P < p \leq Q} p$. We first use Shiu's bound (Lemma~\ref{shiu}) to note that
\[
\sum_{X < n \leq X+H} \mu(n) \overline{F}(g(n)\Gamma) = \sum_{X < n \leq X+H} 1_{(n, \mathcal{P}(P, Q)) > 1} \mu(n) \overline{F}(g(n)\Gamma) + O\left(H \frac{\log P}{\log Q}\right).
\]
Now one can repeat the previous arguments with $\delta = \log^{-A} X$ and $1_{(n, \mathcal{P}(P, Q)) > 1} \mu(n)$ in place of $\Lambda$ and $0$ in place of $\Lambda^\sharp$ and $\Lambda^\sharp_I$ --- this time we use Lemma~\ref{comb-mu} to replace $1_{(n, \mathcal{P}(P, Q)) > 1} \mu(n)$ by the approximant $\sum_{f \in {\mathcal F}} f(n)$ and Corollary~\ref{cor:major-arc-Ramare} gives the required major arc estimate for $1_{(n, \mathcal{P}(P, Q)) > 1} \mu(n)$.

The estimate~\eqref{dk-discor-alt} follows similarly, noting first that, with $P, Q$ as in~\eqref{eq:PQdef} we have by Shiu's bound (Lemma~\ref{shiu})
\begin{align*}
\sum_{X < n \leq X+H} d_k(n) \overline{F}(g(n)\Gamma) &= \sum_{X < n \leq X+H} 1_{(n, \mathcal{P}(P, Q)) > 1} d_k(n) \overline{F}(g(n)\Gamma)\\
&+ O\left(H (\log X)^{k-1}\left(\frac{\log P}{\log Q}\right)^{k}\right)
\end{align*}
and then arguing as for~\eqref{mobius-discor-alt}.

\section{\texorpdfstring{The type $I$ case}{The type I case}}\label{type-i-sec}

In this section we establish the type $I$ case (i) of Theorem~\ref{inverse}, basically following the arguments in~\cite{gt-mobius}. In this section we allow all implied constants to depend on $d,D$.

Writing $f = \alpha*\beta$, we see from Lemma~\ref{basic-prop}(i) that
$$
\left|\sum_{X < n \leq X+H} f(n) F(g(n) \Gamma)\right|^* \leq \sum_{a \leq A_I} |\alpha(a)| \left|\sum_{X/a < b \leq X/a + H/a} \beta(b) F(g(ab) \Gamma)\right|^*.$$
By the pigeonhole principle (and the hypothesis $\delta \leq \frac{1}{\log X}$), we can thus find a scale $1 \leq A \leq A_I$ such that
$$
\sum_{A < a \leq 2A} |\alpha(a)| \left|\sum_{X/a < b \leq X/a + H/a} \beta(b) F(g(ab) \Gamma)\right|^* \gg \delta^{O(1)} H$$
and hence by~\eqref{abound} and the Cauchy--Schwarz inequality
$$
\sum_{A < a \leq 2A} \left(\left|\sum_{X/a < b \leq X/a + H/a} \beta(b) F(g(ab) \Gamma)\right|^*\right)^2 \gg \delta^{O(1)} H^2/A.$$
From Lemma~\ref{basic-prop}(iii) and~\eqref{btv} we conclude that
\begin{equation}
\label{eq:typeIlow}
\sum_{A < a \leq 2A} \left(\left|\sum_{X/a < b \leq X/a + H/a} F(g(ab) \Gamma)\right|^*\right)^2 \gg \delta^{O(1)} H^2/A.
\end{equation}

We may assume that $H \geq C \delta^{-C} A$ for some large constant $C$ depending on $d,D$, since otherwise we have $H \leq \delta^{-O(1)} A_I$ and can conclude.  Trivially
$$ \left|\sum_{X/a < b \leq X/a + H/a} F(g(ab) \Gamma)\right|^* \ll \delta^{-1} H/A$$
for all $A < a \leq 2A$, and hence by~\eqref{eq:typeIlow} we must have
$$ \left|\sum_{X/a < b \leq X/a + H/a} F(g(ab) \Gamma)\right|^* \gg \delta^{O(1)} H/A$$
for $\gg \delta^{O(1)} A$ choices of $a \in (A,2A]$.  For each such $a$, we apply Theorem~\ref{qlt} to find a non-trivial horizontal character $\eta \colon G \to \R/\Z$ of Lipschitz norm $O(\delta^{-O(1)})$ such that
\begin{equation}\label{etaga}
 \| \eta \circ g(a \cdot) \|_{C^\infty(X/a, X/a+H/a]} \ll \delta^{-O(1)}.
\end{equation}
This character $\eta$ could initially depend on $a$, but the number of possible choices for $\eta$ is $O(\delta^{-O(1)})$, hence by the pigeonhole principle we may refine the set of $a$ under consideration to make $\eta$ independent of $a$.  The function $\eta \circ g \colon \Z \to \R/\Z$ is a polynomial of degree at most $d$, hence by Corollary~\ref{smooth-dilate} (and the assumption $H \geq C \delta^{-C} A$) we have
$$ \| q \eta \circ g \|_{C^\infty(X, X+H]} \ll \delta^{-O(1)}$$
for some $1 \leq q \ll \delta^{-O(1)}$. Replacing $\eta$ by $q \eta$, we obtain Theorem~\ref{inverse}(i) as required.

\begin{remark} It should also be possible to establish Theorem~\ref{inverse}(i) using the variant of Theorem~\ref{factor} given in~\cite[Theorem 3.6]{HeWang}.
\end{remark}

\section{\texorpdfstring{The non-abelian type $II$ case}{The non-abelian type II case}}\label{typeII-nonabelian-sec}

In this section we establish the non-abelian type $II$ case (ii) of Theorem~\ref{inverse}.   Let $d, D, H, X, \delta, G/\Gamma, F, f, A_{II}^-, A_{II}^+$ be as in that theorem.  For the rest of this section we allow all constants to depend on $d,D$.  We will need several constants
$$ 1 < C_1 < C_2 < C_3 < C_4$$
depending on $d,D$, with each $C_i$ assumed to be sufficiently large depending on the preceding constants.

We first eliminate the role of $\alpha$ by a standard Cauchy--Schwarz argument.  By Definition~\ref{struct-sum}(ii), we can write $f = \alpha*\beta$, where  $\alpha$ is supported on $[A_{II}^-,A_{II}^+]$, and one has the bounds~\eqref{abound},~\eqref{bbound-2}
for all $A,B \geq 1$.  From~\eqref{invo} we have
$$
 \left| \sum_{n \in P} \alpha*\beta(n) F(g(n) \Gamma) \right| \geq \delta H$$
for some arithmetic progression $P \subset (X,X+H]$.  By the triangle inequality, we have
$$
 \left| \sum_{n \in P} (\alpha*\beta)(n) F(g(n) \Gamma) \right| 
\leq \sum_{A_{II}^- \leq a \leq A_{II}^+} |\alpha(a)| \left|\sum_{b: ab \in P} \beta(b) F(g(ab) \Gamma)\right|.$$
By the pigeonhole principle and the hypothesis $\delta \leq \frac{1}{\log X}$, one can thus find $A_{II}^- \leq A \leq A_{II}^+$ such that
\begin{equation}
\label{eq:absumP}
\sum_{A < a \leq 2A} |\alpha(a)| \left|\sum_{b: ab \in P} \beta(b) F(g(ab) \Gamma)\right| \gg \delta^{O(1)} H.
\end{equation}
We may assume that
\begin{equation}\label{deltac}
\delta^{-C_4} \frac{X}{H} \leq A \leq \delta^{C_4} H,
\end{equation}
since otherwise the first conclusion of Theorem~\ref{inverse}(ii) holds. Now by~\eqref{eq:absumP}, the Cauchy--Schwarz inequality, and~\eqref{abound}
\begin{equation}
\label{eq:asumlowe}
\sum_{A < a \leq 2A} \left|\sum_{b: ab \in P} \beta(b) F(g(ab) \Gamma)\right|^2 \gg \delta^{O(1)} \frac{H^2}{A}.
\end{equation}
Next, we dispose of the large values of $\beta$.  Namely, we now show that the contribution of those $b$ for which $|\beta(b)| > \delta^{-C_2}$ to the left-hand side is negligible.  They contribute
\begin{align}
\nonumber
\ll \delta^{-2} \sum_{A < a \leq 2A} \left(\sum_{b: ab \in P} 1_{|\beta(b)| > \delta^{-C_2}} |\beta(b)| \right)^2 &\ll \delta^{2C_2-2} \sum_{A < a \leq 2A} \left(\sum_{b: ab \in P} |\beta(b)|^2 \right)^2 \\
\label{eq:largebetabound}
&\ll \delta^{2C_2-2} \sum_{b_1, b_2} |\beta(b_1)|^2|\beta(b_2)|^2 \sum_{\substack{A < a \leq 2A \\ ab_1, ab_2 \in P}} 1
\end{align}
Since $P \subseteq (X, X+H]$, the inner sum can be non-empty only if $b_j \asymp X/A$ and $|b_1 - b_2| \leq H/A$ and in this case it has size $\ll H/(X/A) = AH/X$. Using also the inequality $|xy|^2 \leq |x|^4 + |y|^4$ and~\eqref{bbound-2}, we see that~\eqref{eq:largebetabound} is
\begin{align*}
&\ll \delta^{2C_2-2} \sum_{b_1 \asymp X/A} |\beta(b_1)|^4 \sum_{\substack{b_2 \\ |b_1 - b_2| \leq H/A}} \frac{AH}{X} \ll \delta^{2C_2-4} \frac{H^2}{A}.
\end{align*}

From now on in this section we allow all implied constants to depend on $C_2$.
Write
$$ \tilde \beta(b) \coloneqq \beta(b) 1_{|\beta(b)| \leq \delta^{-C_2}} = O(\delta^{-O(1)}).$$
By above and the triangle inequality,~\eqref{eq:asumlowe} holds with $\tilde\beta(b)$ in place of $\beta(b)$. Hence, by Markov's inequality, we see that, for $C_2$ large enough, we have
\begin{equation}\label{xba}
 \left|\sum_{X/a < b \leq (X+H)/a} \tilde \beta(b) F(g(ab) \Gamma)\right|^* \gg \delta^{O(1)} H/A
\end{equation}
for $\gg \delta^{O(1)} A$ choices of $a \in (A, 2A]$.
We cover $(A, 2A]$ by $O(X/H)$ boundedly overlapping intervals of the form $I_{A'} \coloneqq (A', (1+\frac{H}{X})A']$ with $A \leq A' \leq 2A$.  Note that these intervals are non-empty by the lower bound on $A$ in~\eqref{deltac}.
By the pigeonhole principle, we see that for $\gg \delta^{O(1)} X/H$ of these intervals,~\eqref{xba} holds for $\gg \delta^{O(1)} \frac{H}{X} A$ choices of $a \in I_{A'}$.  For all such $A'$ and $a$, the interval $(X/a, (X+H)/a]$ is contained in
\begin{equation}\label{jap}
J_{A'} \coloneqq\left(\left(1-\frac{10 H}{X}\right) \frac{X}{A'}, \left(1+\frac{10 H}{X}\right) \frac{X}{A'}\right],
\end{equation}
hence
$$ \left|\sum_{b \in J_{A'}} \tilde \beta(b) F(g(ab) \Gamma)\right|^* \gg \delta^{O(1)} H/A$$
for $\gg \delta^{O(1)} \frac{H}{X} A$ choices of $a \in I_{A'}$.
We can now apply Proposition~\ref{large-sieve} and the pigeonhole principle to reach one of two conclusions for $\gg \delta^{O(1)} X/H$ of the intervals $I_{A'}$:
\begin{itemize}
\item[(i)]  There exists a non-trivial horizontal character $\eta \colon G \to \R/\Z$ of Lipschitz norm $O(\delta^{-O(1)})$ such that $\| \eta \circ g(a \cdot) \|_{C^\infty(J_{A'})} \ll \delta^{-O(1)}$ for $\gg \delta^{O(1)} |I_{A'}|$ values of $a \in I_{A'}$.
\item[(ii)] For $\gg \delta^{O(1)} |I_{A'}|^2$ pairs $(a,a') \in I_{A'}^2$, there exists a factorization
\begin{equation}\label{gepq}
 g(a' \cdot) = \eps_{aa'} g(a \cdot) \gamma_{aa'}
\end{equation}
where $\eps_{aa'}$ is $(O(\delta^{-O(1)}),J_{A'})$-smooth and $\gamma_{aa'}$ is $O(\delta^{-O(1)})$-rational.
\end{itemize}

Suppose first that conclusion (i) holds for $\gg \delta^{O(1)} X/H$ of the intervals $I_{A'}$.  By pigeonholing we may make $\eta$ independent of $A'$, and then by collecting all the $a$ we see that
$$\| \eta \circ g(a \cdot) \|_{C^\infty((X/a,(X+H)/a])} \ll \delta^{-O(1)}$$
for $\gg \delta^{O(1)} A$ values of $a$ with $a \asymp A$.  Applying Corollary~\ref{smooth-dilate}, we see that either $H \ll \delta^{-O(1)} A$, or else there is another non-trivial horizontal character $\eta' \colon G \to \R/\Z$ of Lipschitz norm $O(\delta^{-O(1)})$ such that
$$\| \eta' \circ g \|_{C^\infty((X,X+H])} \ll \delta^{-O(1)}.$$
In either case the conclusion of Theorem~\ref{inverse}(ii) is satisfied.

Now suppose that conclusion (ii) holds for some $A'$ which we now fix (discarding the information collected for all other choices of $A'$). 
We will formalize the argument that follows as a proposition, as we will need this precise proposition also in our followup work~\cite{MRSTT-almost}.

\begin{proposition}[Abstract non-abelian Type II inverse theorem]\label{prop:Furstenberg-Weiss} Let $C\geq 1$, $d,D \geq 1$, $2 \leq H, A \leq X$, $0 < \delta < \frac{1}{\log X}$, and let $G/\Gamma$ be a filtered nilmanifold of degree at most $d$, dimension at most $D$, and complexity at most $1/\delta$, with $G$ non-abelian.  Let $g: \Z \to G$ be a polynomial map.  Cover $(A,2A]$ by at most $C X/H$ intervals $I_{A'} = (A',(1+\frac{H}{X}) A')$ with $A \leq A' \leq 2A$, with each point belonging to at most $C$ of these intervals.  Suppose that for at least $\frac{1}{C} \delta^{C} X/H$ of the intervals $I_{A'}$, there exist at least $\frac{1}{C} \delta^C |I_{A'}|^2$ pairs $(a,a') \in I_{A'}^2$ for which there exists a factorization
$$ g(a' \cdot) = \eps_{aa'} g(a \cdot) \gamma_{aa'}$$
where $\eps_{aa'}$ is $(C \delta^{-C}, J_{A'})$-smooth and $\gamma_{aa'}$ is $C\delta^{-C}$-rational, with $J_{A'}$ defined by~\eqref{jap}.

Then either
\begin{equation}\label{hdc}
 H \ll_{d,D,C} \delta^{-O_{d,D,C}(1)} \max( A, X/ A)
\end{equation}
or there exists a non-trivial horizontal character $\eta \colon G \to \R/\Z$ having Lipschitz norm $O_{d,D,C}(\delta^{-O_{d,D,C}(1)})$ such that 
$$ \| \eta \circ g \|_{C^\infty(X,X+H]} \ll_{d,D,C} \delta^{-O_{d,D,C}(1)}.$$
\end{proposition}

Indeed, applying this proposition (with a suitable choice of $C=O(1)$, and the other parameters given their obvious values), the conclusion~\eqref{hdc} is not compatible with~\eqref{deltac} for $C_4$ large enough, so we obtain the desired conclusion~\eqref{ut}.

It remains to establish the proposition.  We allow all implied constants to depend on $d,D,C$.
We will now proceed by analyzing the equidistribution properties of the four-parameter polynomial map
$$ (a,b,a',b') \mapsto (g(ab), g(ab'), g(a'b), g(a'b')).$$
The one-parameter equidistribution theorem in Theorem~\ref{factor} is not directly applicable for this purpose.  Fortunately, we may apply the multi-parameter equidistribution theory in Theorem~\ref{multi-factor} instead.  We conclude that either
\begin{equation}\label{ija}
 \min( |I_{A'}|, |J_{A'}| ) \ll_{C_3} \delta^{-O_{C_3}(1)},
\end{equation}
or else there exists 
\begin{equation}\label{mdeltac}
\delta^{-C_3} \leq M \ll \delta^{-O_{C_3}(1)}
\end{equation}
and a factorization
\begin{equation}\label{m-factor}
 (g(ab), g(ab'), g(a'b), g(a'b')) = \eps(a,a',b,b') g'(a,a',b,b') \gamma(a,a',b,b')
\end{equation}
where $\eps, \tilde g, \gamma \in \Poly(\Z^4 \to G^4)$ are such that
\begin{itemize}
\item[(i)]  ($\eps$ smooth) For all $(a,a',b,b') \in I_{A'} \times I_{A'} \times J_{A'} \times J_{A'}$, we have the smoothness estimates
\begin{align*}
d_G(\eps(a,a',b,b'),1) &\leq M\\
d_G(\eps(a+1,a',b,b'),\eps(a,a',b,b')) &\leq M / |I_{A'}|\\
d_G(\eps(a,a'+1,b,b'),\eps(a,a',b,b')) &\leq M / |I_{A'}|\\
d_G(\eps(a,a',b+1,b'),\eps(a,a',b,b')) &\leq M / |J_{A'}|\\
d_G(\eps(a,a',b,b'+1),\eps(a,a',b,b')) &\leq M / |J_{A'}|.
\end{align*}
\item[(ii)]  ($g'$ equidistributed)  There is an $M$-rational subnilmanifold $G'/\Gamma'$ of $G^4/\Gamma^4$ such that $g'$ takes values in $G'$ and one has the total equidistribution property
$$ \Big| \sum_{(a,a',b,b') \in P_1 \times P_2 \times P_3 \times P_4} F(g'(a,a',b,b') \Gamma'') \Big| \leq \frac{|I_{A'}|^2 |J_{A'}|^2}{M^{C_3^2}} \|F\|_{\Lip}$$
for any arithmetic progressions $P_1,P_2 \subset I_{A'}$, $P_3,P_4 \subset J_{A'}$, any finite index subgroup $\Gamma''$ of $\Gamma'$ of index at most $M^{C_3^2}$, and any Lipschitz function $F \colon G'/\Gamma'' \to \C$ of mean zero.
\item[(iii)]  ($\gamma$ rational) There exists $1 \leq r \leq M$ such that $\gamma^r(a,a',b,b') \in \Gamma^4$ for all $a,a',b,b' \in \Z$.
\end{itemize}

The alternative~\eqref{ija} of course implies~\eqref{hdc}, so we may assume we are in the opposite alternative.  Thus we may assume that we have a scale $M$ and a factorization~\eqref{m-factor} with the claimed properties.  

We know that~\eqref{gepq} holds for $\gg M^{-O(1)} |I_{A'}|^2$ pairs $(a,a') \in I_{A'}^2$.  
 By pigeonholing we may assume there is a fixed $1 \leq r \ll M^{O(1)}$ such that $\gamma_{aa'}(b)^r \in \Gamma$ for all such pairs $(a,a')$ and all $b$, and also such that $\gamma^r(a,a',b',b') \in \Gamma^4$.
  This implies that there is some lattice $\tilde \Gamma$ independent of $a,a'$ that contains $\Gamma$ as an index $O(\delta^{-O(1)})$ subgroup, such that $\gamma_{aa'}(b) \in \tilde\Gamma$ for all such pairs $(a,a')$, and $\gamma(a,a',b,b') \in \tilde \Gamma^4$; indeed, by~\cite[Lemma A.8(i), Lemma A.11(iii)]{green-tao-ratner}, we could take $\tilde \Gamma$ to be generated by $\exp( \frac{1}{Q'} X_i )$ for the Mal'cev basis $X_1,\dots,X_D$ of $G/\Gamma$, and some $Q' \ll M^{O(1)}$.  From~\eqref{gepq} we then have
$$ g(a' b) \tilde \Gamma = \eps_{aa'}(b) g(a b) \tilde \Gamma$$
for all such pairs $(a,a')$ and all $b \in \Z$.  If we introduce the subinterval
$$ 
J'_{A'} \coloneqq \left(\frac{X}{A'}, \left(1+\frac{1}{M^{C_3}} \frac{H}{X} \right) \frac{X}{A'}\right]
$$
of $J_{A'}$, then from the smoothness of $\eps_{aa'}$ we have
$$ \eps_{aa'}(b') = O_G( M^{-C_3+O(1)} ) \eps_{aa'}(b) = O_G(M^{O(1)})$$
whenever $b,b' \in J'_{A'}$, where $O_G(r)$ denotes an element of $G$ at a distance $O(r)$ from the identity. This implies that
$$
 (g(ab) \tilde \Gamma, g(ab') \tilde \Gamma, g(a'b) \tilde \Gamma, g(a'b') \tilde \Gamma) \in \Omega
$$
where $\Omega \subset (G/\tilde \Gamma)^4$ consists of all quadruples of the form
\begin{equation}\label{kap}
 (x, y, \eps x, \kappa \eps y)
\end{equation}
for some $x,y \in G/\Gamma$ and $\eps,\kappa \in G$ with $\eps = O_G(M^{O(1)})$ and $\kappa = O_G(M^{-C_3+O(1)})$ (with appropriate choices of implied constants).  We conclude that
$$
\sum_{a,a' \in I_{A'}; b,b' \in J'_{A'}} 1_\Omega( g(ab) \tilde \Gamma, g(ab') \tilde \Gamma, g(a'b) \tilde \Gamma, g(a'b') \tilde \Gamma ) \gg M^{-O(1)} |I_{A'}|^2 |J'_{A'}|^2.$$
Applying~\eqref{m-factor}, we conclude that
$$
\sum_{a,a' \in I_{A'}; b,b' \in J'_{A'}} 1_\Omega( \eps(a,a',b,b') g'(a,a',b,b') \tilde \Gamma^4 ) \gg M^{-O(1)} |I_{A'}|^2 |J'_{A'}|^2.$$
By the pigeonhole principle, we can find intervals $I'_{A'}, I''_{A'}$ in $I_{A'}$ of length $M^{-C_3} I_{A'}$ such that
$$
\sum_{a \in I'_{A'}, a' \in I''_{A'}; b,b' \in J'_{A'}} 1_\Omega( \eps(a,a',b,b') g'(a,a',b,b') \tilde \Gamma^4 ) \gg M^{-O(1)} |I'_{A'}| |I''_{A'}| |J'_{A'}|^2.$$
By the smoothness of $\eps$ we have
$$ \eps(a,a',b,b') =  O_G( M^{-C_3+O(1)} ) \eps(a_0,a'_0,b_0,b_0) = O_G(M^{O(1)})$$
where $a_0,a'_0,b_0$ are the left endpoints of $I'_{A'}, I''_{A'}, J'_{A'}$ respectively. 
Let $\varphi$ be a bump function\footnote{Indeed, one could set $\varphi(x) = \max(1 - K \mathrm{dist}(x,\Omega),0)$ for some $K = O(M^{O(C_3)})$.} supported on $\tilde\Omega$ that equals $1$ on $\Omega$, with Lipschitz norm $O(M^{O(C_3)})$, where $\tilde\Omega$ is defined similarly to $\Omega$ in~\eqref{kap} but with slightly larger choices of implied constants $O(1)$ in the definition of $\eps,\kappa$.   
 This implies that
$$
1_\Omega( \eps(a,a',b,b') g'(a,a',b,b') \tilde \Gamma^4 ) 
\leq
\varphi( \eps(a_0,a'_0,b_0,b_0) g'(a,a',b,b') \tilde \Gamma^4 ) $$
whenever $a \in I'_{A'}, a' \in I''_{A'}; b,b' \in J'_{A'}$.
Abbreviating $\eps_0 \coloneqq \eps(a_0,a'_0,b_0,b_0)  = O_G(M^{O(1)})$, we conclude that
$$
\sum_{a \in I'_{A'}, a' \in I''_{A'}; b,b' \in J'_{A'}} \varphi( \eps_0 g'(a,a',b,b') \tilde \Gamma^4 ) \gg M^{-O(1)} |I'_{A'}| |I''_{A'}| |J'_{A'}|^2.$$
Using the equidistribution properties of $g'$, we conclude that
\begin{equation}\label{thwack}
\int_{G'/(G' \cap \tilde \Gamma^4)} \varphi(\eps_0 x)\ d\mu_{G'/(G' \cap \tilde \Gamma^4)} \gg M^{-O(1)}.
\end{equation}
We now use this bound to obtain control on the group $G'$.  Let us introduce the slice
\begin{equation}\label{ll}
 L \coloneqq \{ g \in G: (1,1,1,g) \in G' \}.
\end{equation}
This is a $O(M^{O(1)})$-rational subgroup of $G$.  Suppose first that this group is non-trivial, then $L \cap \Gamma'$ contains a non-trivial element $\gamma = O_G(M^{O(1)})$.  For $0 \leq t \leq 1$, the group element $\gamma^t \coloneqq \exp( t \log \gamma) = O_G(M^{O(1)})$ is such that $(1,1,1,\gamma^t)$ lies in $G'$, and hence from~\eqref{thwack} and invariance of Haar measure we have
$$
\int_{G'/(G' \cap \tilde \Gamma^4)} \varphi(\eps_0 (1,1,1,\gamma^t) x)\ d\mu_{G'/(G' \cap \tilde \Gamma^4)} \gg M^{-O(1)}.$$
Integrating this and using the Fubini--Tonelli theorem, we have
$$
\int_{G'/(G' \cap \tilde \Gamma^4)} \int_0^1 \varphi(\eps_0 (1,1,1,\gamma^t) x)\ dt \ d\mu_{G'/(G' \cap \tilde \Gamma^4)} \gg M^{-O(1)}.$$
and thus by the pigeonhole principle there exists $x \in (G/\Gamma)^4$ such that
$$\int_0^1 \varphi(\eps_0 (1,1,1,\gamma^t) x)\ dt  \gg M^{-O(1)}.$$
In particular, we have
\begin{equation}\label{epst}
 \eps_0 (1,1,1,\gamma^t) x \in \tilde \Omega \subset (G/\Gamma)^4
\end{equation}
for a set of $t \in [0,1]$ of measure $\gg M^{-O(1)}$.  But if we let $x_1,x_2,x_3$ be the first three components of $\eps_0 x$, we see from~\eqref{kap} that in order for~\eqref{epst} to hold, the fourth coordinate of $\eps_0 (1,1,1,\gamma^t) x$ must take the form $\kappa \eps x_2$, where $\eps = O(M^{O(1)})$ is such that $x_3 = \eps x_1$.  Since the equation $x_3 = \eps x_1$ fixes $\eps$ to a double coset of $\tilde \Gamma$, there are at most $O(M^{O(1)})$ choices for $\eps$, and for each such choice, $\kappa \eps x_2$ is confined to a ball of radius $O(M^{-C_3+O(1)})$; thus the fourth coordinate of $\eps_0 (1,1,1,\gamma^t) x$ is confined to the union of $O(M^{O(1)})$ balls of radius $O(M^{-C_3+O(1)})$.  Since $\gamma$ is non-trivial, $t \in [0,1]$ is thus confined to the union of $O(M^{O(1)})$ intervals of radius $O(M^{-C_3+O(1)})$.  Thus the set of $t \in [0,1]$ obeying~\eqref{epst} has measure at most $O(M^{-C_3+O(1)})$, leading to a contradiction for $C_3$ large enough.  Thus $L$ must be trivial.

Now we apply a ``Furstenberg--Weiss'' argument~\cite{furstenberg-weiss} (see also the argument attributed to Serre in~\cite[Lemma 3.3]{ribet}).  Consider the groups
\begin{align*}
L_1 &\coloneqq \{ g \in G: (1,g',1,g) \in G' \hbox{ for some } g' \in G \}\\
L_2 &\coloneqq \{ g \in G: (1,1,g',g) \in G' \hbox{ for some } g' \in G \}.
\end{align*}
Taking logarithms, we have
\begin{align*}
\log L_1 &\coloneqq \{ X \in \log G: (X,X',0,X) \in \log G' \hbox{ for some } X' \in \log G \}\\
\log L_2 &\coloneqq \{ X \in \log G: (0,0,X',X) \in \log G' \hbox{ for some } X' \in \log G \},
\end{align*}
thus $\log L_1, \log L_2$ are projections of certain slices of $\log G'$.  Since $G'$ was a $O(M^{O(1)})$-rational subgroup of $G^4$, we conclude from linear algebra that $L_1,L_2$ are $O(M^{O(1)})$-rational subgroups of $G$; comparing with~\eqref{ll}, we also see that $[L_1,L_2] \subset L$; since $L$ is trivial, $[L_1,L_2]$ is trivial.  Since $G$ is non-abelian by hypothesis, $[G,G]$ is non-trivial; thus at least one of $L_1,L_2$ must be a proper subgroup of $G$.  For sake of discussion let us assume that $L_1$ is a proper subgroup, as the other case is similar.  Then there exists a non-trivial horizontal character $\eta_4 \colon G \to \R/\Z$ on $G/\tilde \Gamma$ of Lipschitz norm $O(M^{O(1)})$ that annihilates $L_1$, that is to say $\eta_4(g)=0$ whenever $(1,g',1,g) \in G'$ for some $g' \in G$.  Thus, the homomorphism $(1,g',1,g) \mapsto \eta_4(g)$ on $1 \times G \times 1 \times G$ annihilates the restriction of $G'$ to this group, as well as $1 \times G \times 1 \times 1$.  Taking logarithms, we obtain a linear functional on the Lie algebra $0 \times \log G \times 0 \times \log G$ (with all coefficients $O(M^{O(1)})$ in the Mal'cev basis) that annihilates the restriction of $\log G'$ to this Lie algebra, as well as to $0 \times \log G \times 0 \times 0$; by composing with a suitable linear projection we can then extend this linear functional to a linear functional on all of $(\log G)^4$ that annihilates all of $\log G'$, again with all coefficients $O(M^{O(1)})$.  Undoing the logarithm, we may find (possibly trivial) additional horizontal characters $\eta_1,\eta_3 \colon G \to \R/\Z$ on $G/\tilde \Gamma$ of Lipschitz norm $O(M^{O(1)})$ such that
$$ \eta_1(g_1)  + \eta_3(g_3) + \eta_4(g_4) = 0$$
for all $(g_1,g_2,g_3,g_4) \in G'$.  In particular, writing $g' = (g'_1,g'_2,g'_3,g'_4)$, we have
$$ \eta_1(g'_1(a,a',b,b'))  + \eta_3(g'_3(a,a',b,b')) + \eta_4(g'_4(a,a',b,b')) = 0$$
for all $a,a',b,b' \in \Z$.  Applying the factorization~\eqref{m-factor}, and noting that the horizontal characters $\eta_1,\eta_3,\eta_4$ annihilate the components of $\gamma$, we conclude that
\begin{equation}\label{aoi}
 \eta_1(g(ab))  + \eta_3(g(a'b)) + \eta_4(g(a'b')) = \tilde \eps(a,a',b,b')
\end{equation}
for all $a,a',b,b' \in \Z$, where
$$ \tilde \eps(a,a',b,b') \coloneqq \eta_1(\eps_1(a,a',b,b')) +\eta_3(\eps_3(a,a',b,b')) + \eta_4(\eps_4(a,a',b,b'))$$
and $\eps_1,\eps_2,\eps_3,\eps_4$ are the components of $\eps$.  From the smoothness properties of $\eps$, we see in particular that
$$ \| \tilde \eps(a,a',b,b'+1) - \tilde \eps(a,a',b,b') \|_{\R/\Z} \ll M^{O(1)} / |J_{A'}|$$
for $a,a' \in I_{A'}, b, b' \in J_{A'}$, and hence from~\eqref{aoi}
$$
\|\eta_4(g(a'(b'+1))) - \eta_4(g(a'b')) \|_{\R/\Z} \ll M^{O(1)} / |J_{A'}|$$
whenever $a' \in I_{A'}, b' \in J_{A'}$.  For any $a' \in I_{A'}$, the map $b' \mapsto \eta_4(g(a'b'))$ is a polynomial of degree at most $d$, so by Vinogradov's lemma (Lemma~\ref{vin}), for each such $a'$, we either have
$$ |J_{A'}| \ll M^{O(1)},$$
or else there exists $1 \leq q \ll M^{O(1)}$ such that
\begin{equation}\label{latter}
 \| q \eta_4(g(a' \cdot)) \|_{C^\infty(J_{A'})} \ll M^{O(1)}.
\end{equation}
The former possibility is not compatible with~\eqref{deltac} if $C_4$ is large enough, so we may assume the latter possibility~\eqref{latter} holds for all $a' \in I_{A'}$.  Currently the quantity $q$ may depend on $a'$, but by the pigeonhole principle we may fix a $q$ so that~\eqref{latter} holds for $\gg M^{-O(1)} |I_{A'}|$ choices of $a' \in I_{A'}$.  Applying Corollary~\ref{smooth-dilate}, we conclude that either
$$ |I_{A'}| \ll M^{O(1)},$$
or else there exists $1 \leq q' \ll M^{O(1)}$ such that
$$ \| q' \eta_4 \circ g \|_{C^\infty( [X, X+H] )} \ll M^{O(1)}.$$
In either case we obtain one of the conclusions of Proposition~\ref{prop:Furstenberg-Weiss}.  The proof of Theorem~\ref{inverse}(ii) is now complete.

\section{\texorpdfstring{The abelian type $II$ case}{The abelian type II case}}\label{Abelian-type-ii}

In this section we establish the abelian Type $II$ case (iii) of Theorem~\ref{inverse} using arguments from~\cite{matomaki-shao}. We shall need the following variant of~\cite[Proposition 2.2]{matomaki-shao}.

\begin{proposition}\label{prop:M-StypeII}
Let $\delta \in (0,1/2)$, $M \geq 2$ and $L = X/M$. Assume that $H \geq \delta^{-C}\max(L,M)$ for some sufficiently large constant $C = C(k)>0$. Let $\alpha(\ell), \beta(m) \in \mathbb{C}$. Let $k \in \mathbb{N}$ and let
\[ 
g(n) = \sum_{j=1}^k \nu_j (n-X)^j 
\]
be a polynomial of degree $k$ with real coefficients $\nu_j$. If
\[
\Biggl|\sum_{\substack{\ell, m \\ m \sim M \\ X < \ell m \leq X+H}} \alpha(\ell) \beta(m) e(g(\ell m)) \Biggl| \geq \delta H \left(\frac{1}{L}\sum_{L/2 < \ell \leq 2L} |\alpha(\ell)|^2\right)^{1/2} \left(\frac{1}{M}\sum_{m \sim M} |\beta(m)|^4\right)^{1/4}, 
\]
then there exists a positive integer $q \leq \delta^{-O_k(1)}$ such that
\[ \|q (j \nu_j + (j+1)X \nu_{j+1})\|_{\mathbb{R} / \mathbb{Z}} \leq \delta^{-O_k(1)} \frac{X}{H^{j+1}} \]
for all $1 \leq j \leq k$, with the convention that $\nu_{k+1} = 0$.
\end{proposition}

\begin{proof}
This follows from the same argument as~\cite[Proposition 2.2]{matomaki-shao}. The only difference is that we do not assume that the coefficients $\alpha(\ell)$ and $\beta(m)$ are divisor bounded and due to this in the beginning of the proof we do not estimate the sums $\sum_{L/2 < \ell \leq 2L} |\alpha(\ell)|^2$ and $\sum_{m \sim M} |\beta(m)|^4$ with bounds for averages of divisor functions but keep them as they are.
\end{proof}

Let us get back to the proof of Theorem~\ref{inverse}(iii). We can assume that
\[
\max\{A_{II}^+, X/A_{II}^-\} \ll \delta^{O_d(1)} H
\]
since otherwise the claim is immediate. Note that in particular $H \geq \delta^{-O_d(1)} X^{1/2}$. By assumption and dyadic splitting (noting that $\delta < 1/\log X$)
\begin{equation}
\label{eq:typeII1stlowerbound}
\Biggl|\sum_{\substack{x < \ell m \leq x+h \\ m \sim M \\ \ell m \equiv u \pmod{v}}} \alpha(\ell) \beta(m) e(P(\ell m))\Biggr| \geq \delta^2 H 
\end{equation}
for some $(x, x+h] \subseteq (X, X+H]$, some $M \in [X/A_{II}^+, X/A_{II}^-]$, some polynomial $P(x)$ of degree at most $d$ and some $u, v \in \mathbb{N}$ with $u \leq v$. Before applying Proposition~\ref{prop:M-StypeII} we will show that~\eqref{eq:typeII1stlowerbound} can hold only if $v \ll \delta^{-8}$ and $h \gg \delta^8 H$.  In order to show this, we give an upper bound for the left-hand side using the Cauchy--Schwarz inequality. Using also~\eqref{abound} and denoting $L = X/M$, we obtain, using the inequality $|xy| \leq |x|^2 + |y|^2$
\[
\begin{split}
\delta^4 H^2 &\leq \Biggl|\sum_{\substack{x < \ell m \leq x+h \\ m \sim M \\ \ell m \equiv u \pmod{v}}} \alpha(\ell) \beta(m) e(P(\ell m))\Biggr|^2 \\
&\ll \sum_{L/2 < \ell \leq 2L} |\alpha(\ell)|^2 \cdot \sum_{L/2 < \ell \leq 2L} \Biggl(\sum_{\substack{m \sim M \\ x < \ell m \leq x+h \\ \ell m \equiv u \pmod{v}}} |\beta(m)|\Biggr)^2 \\
&\ll \frac{L}{\delta} \sum_{\substack{m_1, m_2 \sim M \\ |m_1 -m_2| \leq 2h/L \\ (m_j, v) \mid u}} |\beta(m_1) \beta(m_2)| \sum_{\substack{L/2 < \ell \leq 2L \\ x < \ell m_1, \ell m_2 < x+h \\ \ell m_j \equiv u \pmod{v}}} 1 \\
&\ll \frac{L}{\delta} \sum_{\substack{m_1, m_2 \sim M \\ |m_1 -m_2| \leq 2h/L \\ (m_2, v) \mid u}} |\beta(m_1)|^2 \left(1+\frac{h (m_2, v)}{M v}\right). \\
\end{split}
\]
Writing $d = (m_2, v)$ and $m_2' = m_2/d$ and using~\eqref{bbound-2} we obtain
\begin{align*}
\delta^4 H^2 &\ll \frac{L}{\delta} \sum_{\substack{m_1 \sim M}} |\beta(m_1)|^2 \Biggl(\frac{h}{L}+1 + \sum_{d \mid u}\sum_{\substack{m_2' \\ |m_1 -d m_2'| \leq 2h/L}} \frac{h d}{M v}\Biggr) \\
& \ll \frac{h M}{\delta^2} + \frac{LM}{\delta^2} + \frac{LM}{\delta^2} \sum_{d \mid u} \frac{h d}{M v} \left(\frac{h}{Ld} +1\right) \\
& \ll \frac{h M}{\delta^2} + \frac{LM}{\delta^2} + \frac{h^2 d_2(u)}{v\delta^2} + \frac{hL}{\delta^2 v} \cdot \frac{u^2}{\varphi(u)}.
\end{align*}
Since $L, M \ll \delta^{O(1)} H$ and $LM \ll \delta^{O(1)} H^2$, this is a contradiction unless $v \ll \delta^{-8}$ and $h \gg \delta^8 H$. 

From~\eqref{eq:typeII1stlowerbound} together with~\eqref{abound} and~\eqref{bbound-2} we have
\[
\Biggl|\sum_{\substack{x < \ell m \leq x+h \\ m \sim M \\ \ell m \equiv u \pmod{v}}} \alpha(\ell) \beta(m) e(P(\ell m))\Biggr| \geq \delta^9 h\left(\frac{1}{L}\sum_{L/2 < \ell \leq 2L} |\alpha(\ell)|^2\right)^{1/2} \left(\frac{1}{M}\sum_{m \sim M} |\beta(m)|^4\right)^{1/4}.
\]
We can write, for some $\nu_j \in \mathbb{R}$,
\[
P(n) = \sum_{j = 0}^d \nu_j (n-X)^j.
\]
We can assume that $\nu_0 = 0$. Furthermore we can spot the condition $\ell m = u \pmod{v}$ using additive characters, so that, for some $r \pmod{v}$ we have
\[
\left|\sum_{\substack{x < \ell m \leq x+h \\ m \sim M}} \alpha(\ell) \beta(m) e\left(P(\ell m) + \frac{r\ell m}{v}\right)\right| \geq \delta^9 h\left(\frac{1}{L}\sum_{L/2 < \ell \leq 2L} |\alpha(\ell)|^2\right)^{1/2} \left(\frac{1}{M}\sum_{m \sim M} |\beta(m)|^4\right)^{1/4}.
\]

Now we are in the position to apply Proposition~\ref{prop:M-StypeII} to the polynomial $P(n) + rn/v$. By multiplying the resulting $q$ by $v$ we see that the conclusion of the proposition holds also for the coefficients of $P(n)$, ignoring $rn/v$. Hence we get that there exists a positive integer $q' \leq \delta^{-O_d(1)}$ such that
\[ 
\|q' (j \nu_j + (j+1)X \nu_{j+1})\|_{\mathbb{R} / \mathbb{Z}} \leq \delta^{-O_d(1)} \frac{X}{H^{j+1}} 
\]
for all $1 \leq j \leq d$, with the convention that $\nu_{d+1} = 0$.

Next we use a variant of the argument in the treatment of type II sums in~\cite[Proof of Theorem 1.3 in Section 4]{matomaki-shao}. We start by shifting each $\nu_j$ by $(q'j)^{-1}a_j$ for an appropriate $a_j \in \Z$ to get $\nu_j'$ such that
\begin{equation}\label{eq:alphak}
|q'(j\nu_j' + (j+1)X\nu_{j+1}')| \leq \delta^{-O_d(1)} \frac{X}{H^{j+1}}
\end{equation}
for all $1 \leq j \leq d$. Let
\[ P_1(n) = \sum_{j=1}^d \nu_j'(n-X)^j, \]
so that
\[ e(P(n)) = e(P_1(n)) e\left(-\sum_{j=1}^d \frac{a_j}{q'j} (n-X)^j\right).  \]
Choosing $q = q' d!$, we see that $e(P(n)-P_1(n))$ is constant in any arithmetic progression $\pmod{q}$ and thus
\begin{equation}
\label{eq:PP1TV}
\|e(P(n)-P_1(n))\|_{\TV( [X,X+H) \cap \Z; q)} \leq q \ll \delta^{-O_d(1)}
\end{equation}

By  induction one can deduce from~\eqref{eq:alphak} that
\begin{equation}\label{eq:alphak'}
\left|\nu_j' - \frac{(-1)^{j-1}}{jX^{j-1}} \nu_1'\right| \leq \delta^{-O_d(1)} \frac{1}{H^j}
\end{equation}
for all $1 \leq j \leq d  + 1$. In particular when $j=d+1$ this gives
\[ |\nu_1'| \leq \delta^{-O_d(1)} \frac{X^d}{H^{d+1}}. \]
We set $T = 2\pi X\nu_1'$, so that
\begin{equation}
\label{eq:Tupper}
|T| \leq \delta^{-O_d(1)} \left(\frac{X}{H}\right)^{d+1}.
\end{equation}
We write also
\[
P_2(n) = \sum_{j=1}^d \frac{(-1)^{j-1}}{jX^{j-1}} \nu_1' (n-X)^j = \frac{T}{2\pi} \sum_{j=1}^d \frac{(-1)^{j-1}}{j} \left(\frac{n-X}{X}\right)^j.
\]
By~\eqref{eq:alphak'} we have that
\begin{equation}
\label{eq:P1P2TV}
\|e(P_1(n)-P_2(n))\|_{\TV( [X,X+H) \cap \Z; q)} \leq q \delta^{-O_d(1)} \ll \delta^{-O_d(1)}.
\end{equation}

By Taylor expansion, for any $k \geq 0$ and $n \in (X, X+H]$,
\[ \log \frac{n}{X} = \log\left(1 + \frac{n-X}{X}\right) = \sum_{j=1}^{d+k} \frac{(-1)^{j-1}}{j} \left(\frac{n-X}{X}\right)^j + O\left( \left(\frac{H}{X}\right)^{d+k+1} \right),
\]
so that, using~\eqref{eq:Tupper},
\[
\begin{split}
P_2(n) &= \frac{T}{2\pi} \log \frac{n}{X} - \frac{T}{2\pi} \sum_{j=d+1}^{d+k} \frac{(-1)^{j-1}}{j} \left(\frac{n-X}{X}\right)^{j} + O\left(\delta^{-O_d(1)} \left(\frac{H}{X}\right)^{k} \right).
\end{split} 
\]
Hence
\[ 
e(P_2(n))n^{-iT} = X^{-iT} e\left(-\frac{T}{2\pi}\sum_{j=d+1}^{d+k} \frac{(-1)^{j-1}}{j} \left(\frac{n-X}{X}\right)^{j}\right) + O\left(\delta^{-O_d(1)} \left(\frac{H}{X}\right)^{k} \right).
\]
Taking $k$ large enough in terms of $\theta$, this implies that
\begin{equation}
\label{eq:P2nTTV}
\|e(P_2(n)) n^{-iT} \|_{\TV( [X,X+H) \cap \Z; q)} \ll \delta^{-O_d(1)}.
\end{equation}
Now the claim follows by combining~\eqref{eq:PP1TV},~\eqref{eq:P1P2TV}, and~\eqref{eq:P2nTTV} utilizing~\eqref{tv-prod}.

\section{\texorpdfstring{The type $I_2$ case}{The type I2 case}}\label{I2-sec}

In this section we establish the type $I_2$ case (iv) of Theorem~\ref{inverse}.  Our main tool will be the following elementary partition\footnote{In this section only, $(m,n)$ will denote the element of the lattice $\Z^2$ with coordinates $m,n$, rather than the greatest common divisor of $m$ and $n$.  We hope that this collision of notation will not cause confusion.} of the hyperbolic neighborhood $\{ (m,n) \in \Z^2: m \in J; \quad X < nm \leq X+H \}$ into arithmetic progressions, which is non-trivial when $H$ is much larger than $X^{1/3}$.

\begin{theorem}[Partition of hyperbolic neighborhood]\label{decomp}  Let $X, H, M \geq 1$ be such that
$$ X^{1/3} \leq H \leq X \quad \text{and} \quad M \ll X^{1/2},$$
and let $J$ be a subinterval of $(M,2M]$.  Then the set
\begin{equation}\label{eaq0}
\{ (m, n) \in \Z^2: m \in J; \quad X < nm \leq X+H \}
\end{equation}
can be partitioned for any integer $Q$ obeying
\begin{equation}\label{qbound}
\frac{M}{H} \leq Q \leq \frac{M}{(HX)^{1/4}}
\end{equation}
as
$$ \bigcup_{q=1}^Q \bigcup_{\substack{a \asymp \frac{X}{M^2} q \\ (a,q)=1}} \bigcup_{P \in {\mathcal P}_{a,q}} P$$
where for each pair $a,q$ of coprime integers with $1 \leq q \leq Q$ and $a \asymp \frac{X}{M^2} q$, ${\mathcal P}_{a,q}$ is a family of $O( \frac{M^3}{XQ^2q} )$ arithmetic progressions $P$ in~\eqref{eaq0}, each of spacing $(q,-a)$ and length at most $\frac{HQ}{M}$.
\end{theorem}

In particular, the cardinality of the set~\eqref{eaq0} does not exceed
\begin{equation}\label{hlog}
 \ll \sum_{1 \leq q \leq Q} \sum_{a \asymp \frac{X}{M^2} q} \frac{M^3}{XQ^2q} \frac{HQ}{M} \ll H.
\end{equation}

\begin{proof}[Proof of Theorem~\ref{decomp}]  For future reference, we note from~\eqref{qbound} and $X^{1/3} \leq H \leq X$ that
\begin{equation}\label{q-extra}
 Q \leq \frac{M}{(HX)^{1/4}} \leq \frac{M}{X^{1/3}} \leq \frac{M H^{1/2}}{X^{1/2}} \leq M.
\end{equation}

Note that if $(m,n)$ lies in~\eqref{eaq0} then $m \asymp M$ and $nm \asymp X$, thus $\frac{n}{m} \asymp \frac{X}{M^2}$.  By the Dirichlet approximation theorem, we then have
$$ \frac{n}{m} \in \left[\frac{a}{q} - \frac{1}{Qq}, \frac{a}{q} + \frac{1}{Qq}\right]$$
for some $1 \leq q \leq Q$ and some $a \asymp \frac{X}{M^2} q$ coprime to $q$.  If for any such $a,q$ we define $I_{a,q}$ to be the portion of the interval $[\frac{a}{q} - \frac{1}{Qq}, \frac{a}{q} + \frac{1}{Qq}]$ that is not contained in any other such interval $I_{a',q'}$ with $q' < q$, we see that the $I_{a,q}$ are disjoint intervals, and that we can partition~\eqref{eaq0} into sets
\begin{equation}\label{eaq}
\{ (m,n) \in \Z^2: m \in J; \frac{n}{m} \in I_{a,q}; \quad X < nm \leq X+H \}
\end{equation}
where $a,q$ range over those coprime integers with 
\begin{equation}\label{aq}
1 \leq q \leq Q; \quad \frac{a}{q} \asymp \frac{X}{M^2}.
\end{equation}
It then suffices to show that each such set~\eqref{eaq} can be partitioned into $O( \frac{M^3}{XQ^2q} )$ arithmetic progressions $P$ in $\Z^2$, each of spacing $(q,-a)$ and length at most $\frac{HQ}{M}$.

Fix $a,q$, and write $I = I_{a,q}$.  It in fact suffices to show that the set~\eqref{eaq} can be partitioned into $O( \frac{M^3}{XQ^2q} )$ arithmetic progressions $P$ of spacing $(q,-a)$ and arbitrary length, so long as we also show that the total cardinality of~\eqref{eaq} is $O( \frac{HM^2}{XQq} )$.  This is because any such progression $P$ can be partitioned into $O( \frac{M}{HQ} \# P + 1 )$ subprogressions of the same spacing $(q,-a)$ and length at most $\frac{HQ}{M}$, and
$$ \sum_P\left( \frac{M}{HQ} \# P + 1\right)  \ll \frac{M}{HQ} \frac{HM^2}{XQq} +\frac{M^3}{XQ^2q} \ll \frac{M^3}{XQ^2 q}.$$

It remains to obtain such a partition.  From Bezout's theorem we see that for any integer $c$, the set $\{ (m,n) \in \Z^2: qn+am=c\}$ is an infinite arithmetic progression of spacing $(q,-a)$.  The intersection of~\eqref{eaq} with this set is
\begin{equation}\label{eaqc}
E_c := \left\{ \left(m,\frac{c-am}{q}\right): m, \frac{c-am}{q} \in \Z; m \in J; \frac{c}{mq} - \frac{a}{q} \in I; X < \frac{(c-am)m}{q} \leq X+H \right\}.
\end{equation}
The constraints
$$m \in J; \frac{c}{mq} - \frac{a}{q} \in I; X < \frac{(c-am)m}{q} \leq X+H $$
confine $m$ to the union of at most two intervals in the real line, and hence the set $E_c$ is the union of at most two arithmetic progressions in $\Z^2$ of spacing $(q,-a)$.  It thus suffices to show that $E_c$ is non-empty for at most $O( \frac{M^3}{X Q^2 q} )$ choices of $c$, and that 
\begin{equation}\label{sumc}
\sum_c \# E_c \ll \frac{HM^2}{XQq}.
\end{equation}

We begin with the first claim.  If $(m,n) \in E_c$ then $c = qn+am$ and $nm = X + O(H)$ and hence
\begin{equation}\label{qnm}
 c^2 - (qn-am)^2 = (qn+am)^2 - (qn-am)^2 = 4aqnm = 4aqX + O(aq H).
\end{equation}
On the other hand, we have
\begin{equation}\label{maq}
 qn-am = mq \left(\frac{n}{m}-\frac{a}{q}\right) \ll \frac{mq}{qQ} \ll \frac{M}{Q}.
\end{equation}
We thus have
$$ c^2 = 4aqX + O( aqH ) + O\left( \frac{M^2}{Q^2} \right).$$
From~\eqref{aq},~\eqref{qbound} we have
$$ aqH \ll \frac{X}{M^2}q^2 H \ll \frac{M^2}{Q^2} \frac{XHQ^4}{M^4} \ll \frac{M^2}{Q^2} $$
and thus
$$ c^2 = 4aqX + O\left( \frac{M^2}{Q^2} \right).$$
Also $\frac{M^2}{Q^2} \leq M^2 \ll X \leq aqX$.  Thus on taking square roots we have
$$ c = \sqrt{4aqX} + O\left( \frac{1}{\sqrt{aqX}} \frac{M^2}{Q^2} \right)$$
and hence by~\eqref{aq}
$$ c = \sqrt{4aqX} + O\left( \frac{M^3}{XQ^2q} \right)$$
giving the first claim.

It remains to prove~\eqref{sumc}.  We first consider the contribution of those $c$ for which
$$ c = \sqrt{4aqX} + O\left( \frac{1}{\sqrt{aqX}} aqH + 1\right),$$
so the total number of possible $c$ here is $O( \frac{1}{\sqrt{aqX}} aqH + 1 )$.  For a fixed such $c$, we then have from~\eqref{qnm} that 
$$ qn-am = O( \sqrt{aqH} ).$$
But once one fixes $c = qn+am$, the residue class of $qn-am$ modulo $q$ and modulo $a$ are both fixed, thus by the Chinese remainder theorem $qn-am$ is restricted to a single residue class modulo $aq$.  Thus the number of possible values of $qn-am$ is $O( \frac{\sqrt{aqH}}{aq} + 1 )$.  The net contribution of this case to~\eqref{sumc} is then
$$ \ll \left(\frac{1}{\sqrt{aqX}} aqH + 1\right) \left(\frac{\sqrt{aqH}}{aq} + 1\right)$$
which expands out to
$$ \ll \frac{H^{3/2}}{X^{1/2}} + \frac{a^{1/2} q^{1/2} H}{X^{1/2}} + \frac{H^{1/2}}{a^{1/2} q^{1/2}} + 1.$$
Using~\eqref{aq}, this becomes
$$ \ll \frac{H^{3/2}}{X^{1/2}} + \frac{q H}{M} + \frac{H^{1/2} M}{q X^{1/2}} + 1.$$
Thus we need to show that
$$ \frac{H^{3/2}}{X^{1/2}}, \frac{q H}{M}, \frac{H^{1/2} M}{q X^{1/2}}, 1 \ll \frac{HM^2}{XQq}$$
which on using $1 \leq q \leq Q$ rearranges to
$$ Q \ll \frac{M}{H^{1/4} X^{1/4}}, \frac{M}{X^{1/3}}, \frac{H^{1/2} M}{X^{1/2}}, \frac{H^{1/2} M}{X^{1/2}}$$
and the claim now follows from~\eqref{q-extra}.

Now we consider the contribution of the opposite case, in which $|c - \sqrt{4aqX}|$ exceeds a large multiple of $\frac{1}{\sqrt{aqX}} aqH + 1$.  Then $|c^2-4aqX|$ exceeds a large multiple of $aqH$, so from~\eqref{qnm} we have
$$ c^2 = 4aqX + O( (qn-am)^2 )$$
and thus if we restrict to a dyadic range $qn-am \in \pm [A,2A]$ for some $1 \leq A \ll \frac{M}{Q}$ that is a power of two (the upper bound coming from~\eqref{maq}) we have
$$ c = \sqrt{4aqX} + O\left( \frac{1}{\sqrt{aqX}} A^2 \right).$$
Thus for a fixed $A$, the total number of possible $c$ here is $O( \frac{1}{\sqrt{aqX}} A^2 )$ (note that we have already excluded those $c$ that lie within $O(1)$ of $\sqrt{4aqX}$).  On the other hand, once $c$ is fixed, we see from~\eqref{qnm} that $(qn-am)^2$ is constrained to an interval of length $O(aqH)$.  The quantity $qn-am$ is also constrained to lie in $\pm [A,2A]$ and to a single residue class modulo $aq$, so the squares $(qn-am)^2$ are separated by $\gg Aaq$ when $qn-am$ is positive, and similarly when $qn-am$ is negative.  Thus the total number of possible values of $qn-am$ available is $O( \frac{aqH}{Aaq} + 1 ) = O( \frac{H}{A} )$, since from~\eqref{qbound} one has $\frac{H}{A} \gg \frac{H}{M/Q} \geq 1$.  Thus the total contribution of this case to~\eqref{sumc} is
$$ \ll \sum_{\substack{1 \leq A \ll \frac{M}{Q} \\ A = 2^j}} \frac{A^2}{\sqrt{aqX}} \cdot \frac{H}{A} \ll \frac{1}{\sqrt{aqX}} H \frac{M}{Q}$$
which after applying~\eqref{aq} gives $O( \frac{HM^2}{XQq} )$ as required.
\end{proof}

Combining this with the pigeonhole principle we obtain

\begin{corollary}[Pigeonholing on a hyperbola neighborhood]\label{corx} Let $X, H, M, Q \geq 1$ be such that
$$ X^{1/3} \leq H \leq X, \quad M \ll X^{1/2}, \quad \text{and} \quad \frac{M}{H} \leq Q \leq \frac{M}{(HX)^{1/4}},$$
and let $J$ be a subinterval of $[M,2M]$.

Let $P_0$ be an arithmetic progression in $(X,X+H]$, and let $\beta_1, \beta_2 \colon \N \to \C$ be functions obeying the bounds
$$ \| \beta_1 \|_{\TV(\N;q_0)}, \| \beta_2 \|_{\TV(\N;q_0)} \leq 1/\delta$$
for some $1 \leq q_0 \leq 1/\delta$ and some\footnote{It is likely that with more effort the restriction on $\delta$ can be increased up to 1, but that we will not need to do so here.} $0 < \delta < 1/(\log X)$.
Let $f: \Z^2 \to \C$ be a $1$-bounded function such that
\begin{equation}\label{mj}
\left|\sum_{m \in J} \sum_{\substack{n \\ X < nm \leq X+H}} \beta_1(m) \beta_2(n) 1_{P_0}(nm) f(n,m)\right| \geq \delta H.
\end{equation}
Then for $\gg \delta^{O(1)} \frac{XQ^2}{M^2}$ pairs of coprime integers $q,a$ with $\delta^{O(1)} Q \ll q \leq Q$ and $a \asymp \frac{X}{M^2} q$, one can find an arithmetic progression $P$ in~\eqref{eaq0} of spacing $(q,-a)$ and length at most $\frac{HQ}{M}$ such that
$$ \left|\sum_{(m,n) \in P} f(n,m)\right|^* \gg \delta^{O(1)} \frac{HQ}{M}.$$
Here we extend the maximal sum notation~\eqref{maximal-sum} to sums over arithmetic progressions in $\Z^2$ in the obvious fashion.
\end{corollary}

\begin{proof} Let $q_0'$ be the spacing of $P_0$. We first claim that $q_0'\ll \delta^{-10}$. Indeed, by Shiu's bound (Lemma~\ref{shiu}) we have 
\begin{align*}
 \sum_{m\in J}\sum_{\substack{X<nm\leq X+H\\nm\equiv b(q_0')}}1
\leq \sum_{\substack{X<n\leq X+H\\n\equiv b(q_0')}}d_2(n)
\ll_{\varepsilon} d_2(q_0')\left((\log X)\frac{H}{q_0'} +X^{\varepsilon}\right),  
\end{align*}
and if $q_0'\gg \delta^{-10}$ then this together with the triangle inequality contradicts our assumption~\eqref{mj}. Now we may assume that $q_0'\ll \delta^{-10}$.

By Lemma~\ref{basic-prop}(iii), the left-hand side of~\eqref{mj} is bounded by
$$ \frac{1}{\delta} \left|\sum_{m \in J} \left(\sum_{\substack{n \\ X < nm \leq X+H}} \beta_2(n) 1_{P_0}(nm) f(n,m)\right)\right|^*$$
which by definition is equal to
$$ \frac{1}{\delta} \left|\sum_{m \in J} \sum_{\substack{n \\ X < nm \leq X+H}} 1_{P_1}(m) \beta_2(n) 1_{P_0}(nm) f(n,m)\right|$$
for some arithmetic progression $P_1 \subset J$.  Interchanging the $n$ and $m$ sums and using Lemma~\ref{basic-prop}(iii) again, we can bound this in turn by
$$ \frac{1}{\delta^2} \left|\sum_{m \in J} \sum_{\substack{n \\ X < nm \leq X+H}} 1_{P_1}(m) 1_{P_2}(n) 1_{P_0}(nm) f(n,m)\right|$$
for some arithmetic progression $P_2$.  
From Theorem~\ref{decomp} and the triangle inequality, we have
\begin{align*}
& \sum_{m \in J} \sum_{\substack{n \\ X < nm \leq X+H}} 1_{P_1}(m) 1_{P_2}(n) 1_{P_0}(nm) f(n,m) \\
&\quad \ll \sum_{q=1}^Q \sum_{\substack{a \asymp \frac{X}{M^2} q \\ (a,q)=1}} \frac{M^3}{XQ^2q} \sup_{P \in {\mathcal P}_{a,q}} \left|\sum_{(m,n) \in P} 1_{P_1}(m) 1_{P_2}(n) 1_{P_0}(nm) f(n, m)\right|
\end{align*}
and since the set $\{ (m,n) \in P: m \in P_1, n \in P_2, nm \in P_0 \}$ is the union of at most $O(\delta^{-O(1)})$ arithmetic progressions in $P$ (recalling that $q_0'\ll \delta^{-O(1)}$), we have
$$ \left|\sum_{(m,n) \in P} 1_{P_1}(m) 1_{P_2}(n) 1_{P_0}(nm) f(n, m)\right| \ll \delta^{-O(1)}\left|\sum_{(m,n) \in P} f(n, m)\right|^*.$$
We conclude that
\begin{equation}\label{qQ}
 \sum_{q=1}^Q \sum_{\substack{a \asymp \frac{X}{M^2} q \\ (a,q)=1}} \frac{M^3}{XQ^2q} \sup_{P \in {\mathcal P}_{a,q}} \left|\sum_{(m,n) \in P} f(n, m)\right|^* \gg \delta^{O(1)} H.
\end{equation}
As $f$ is $1$-bounded, we have here
\begin{equation}\label{mqq}
 \frac{M^3}{XQ^2q} \sup_{P \in {\mathcal P}_{a,q}} \left|\sum_{(m,n) \in P} f(n, m)\right|^* \leq \frac{M^3}{XQ^2 q} \frac{HQ}{M} = \frac{M^2 H}{XQq};
\end{equation}
since the number of $a$ associated to a fixed $q$ is $O(Xq/M^2)$, we conclude that, for any $q \leq Q$,
$$ \sum_{\substack{a \asymp \frac{X}{M^2} q \\ (a,q)=1}} \frac{M^3}{XQ^2q} \sup_{P \in {\mathcal P}_{a,q}} \left|\sum_{(m,n) \in P} f(n, m)\right|^* \ll \frac{H}{Q}.$$
Comparing this with~\eqref{qQ}, we conclude that
\begin{equation}\label{amq}
\sum_{\substack{a \asymp \frac{X}{M^2} q \\ (a,q)=1}} \frac{M^3}{XQ^2q} \sup_{P \in {\mathcal P}_{a,q}} \left|\sum_{(m,n) \in P} f(n, m)\right|^* \gg \delta^{O(1)} \frac{H}{Q}
\end{equation}
for $\gg \delta^{O(1)} Q$ choices of $1 \leq q \leq Q$.  By dropping small values of $q$, we may restrict attention to those $q$ with $\delta^{O(1)} Q \ll q \ll Q$.  For each such $q$, we combine~\eqref{mqq} with~\eqref{amq} to conclude that 
$$
\frac{M^3}{XQ^2q} \sup_{P \in {\mathcal P}_{a,q}} \left|\sum_{(m,n) \in P} f(n, m)\right|^* \gg \frac{M^2}{Xq} \delta^{O(1)} \frac{H}{Q}$$
for $\gg \delta^{O(1)} \frac{Xq}{M^2} \gg \delta^{O(1)} \frac{XQ}{M^2}$ choices of $a$, and the claim follows.
\end{proof}

We can now obtain a preliminary version of Theorem~\ref{inverse}(iv) (which basically corresponds to the case $A_{I_2}=1$, after some dyadic decomposition):

\begin{proposition}[Preliminary type $I_2$ inverse theorem]\label{d2-prelim} Let $X, H, M \geq 1$ be such that
$$ X^{1/3} \leq H \leq X \quad \text{and} \quad M \ll X^{1/2},$$
and let $J$ be a subinterval of $(M,2M]$. Let $0<\delta < 1/(\log X)$, let $P_0$ be an arithmetic progression in $(X,X+H]$, and let $\beta_1, \beta_2 \colon \N \to \C$ be functions obeying the bounds
$$ \| \beta_1 \|_{\TV(\N;q_0)}, \| \beta_2 \|_{\TV(\N;q_0)} \leq 1/\delta$$
for some $1 \leq q_0 \leq 1/\delta$. 

Let $G/\Gamma$ be a filtered nilmanifold of degree $d$, dimension $D$, and complexity at most $1/\delta$ for some $d,D \geq 1$, and let $F \colon G/\Gamma \to \C$ be a Lipschitz function of norm $1/\delta$ and mean zero, and $g \colon \Z \to G$ a polynomial map.  Suppose that
$$
\left|\sum_{m \in J} \sum_{\substack{n \\ X < nm \leq X+H}} \beta_1(m) \beta_2(n) 1_{P_0}(nm) F(g(nm)\Gamma)\right| \geq \delta H.$$
Then either
\begin{equation}\label{hqd}
H \ll_{d,D} \delta^{-O_{d,D}(1)} X^{1/3}
\end{equation}
or else there exists non-trivial horizontal character $\eta \colon G \to \R$ of Lipschitz norm $O_{d,D}(\delta^{-O_{d,D}(1)})$ such that
$$ \| \eta \circ g \|_{C^\infty(X,X+H]} \ll_{d,D} \delta^{-O_{d,D}(1)}.$$
\end{proposition}

\begin{proof}   We allow all implied constants to depend on $d,D$. We apply Corollary~\ref{corx} with 
$$ Q \coloneqq \left\lfloor \frac{M}{(HX)^{1/4}}\right\rfloor.$$
This gives that for $\gg \delta^{O(1)} XQ^2/M^2$ pairs $a,q$ with $q = O(Q)$ and $a = O( XQ/M^2)$, we have
$$ \left|\sum_{k=1}^K F( g( (n_0-ka)(m_0+kq) )\Gamma)\right|^* \gg \delta^{O(1)} \frac{HQ}{M}$$
for some integers $n_0,m_0$ and some $1 \leq K \leq \frac{HQ}{M}$.  

Applying the quantitative Leibman equidistribution theorem (Theorem~\ref{qlt}), we can find a non-trivial horizontal character $\eta: G \to \R$ of Lipschitz norm $O(\delta^{-O(1)})$ such that
\begin{equation}\label{ega}
 \| \eta\circ g( (n_0-\cdot a)(m_0+\cdot q) ) \|_{C^\infty([HQ/M])} \ll \delta^{-O(1)}.
\end{equation}
By pigeonholing we can make $\eta$ independent of $a,q$, so that~\eqref{ega} holds for  $\gg \delta^{O(1)} XQ^2/M^2$ pairs $a,q$ with $q = O(Q)$ and $a = O( XQ/M^2)$. Fix this choice of $\eta$. The map $P = \eta \circ g: \Z \to \R$ is a polynomial of degree at most $d$; say
$$ P(n) = \eta\circ g(n)=\sum_{0\leq j\leq d}\alpha_j(n-X)^j. $$  
Now suppose that~\eqref{hqd} fails. We will show that
\begin{equation}\label{qalphaj} 
\|q_0 \alpha_j\|_{\R/\Z} \ll \delta^{-O(1)}H^{-j} 
\end{equation}
for some $1 \leq q_0 \ll \delta^{-O(1)}$ and all $1 \leq j \leq d$.

We use downward induction on $j$. 
Extracting out the top degree coefficient $\alpha_d$ of $P$, we see that
$$ \| \alpha_d (qa)^d \|_{\R/\Z} \ll \delta^{-O(1)} (HQ/M)^{-2d}.$$
We apply the polynomial Vinogradov lemma (Lemma~\ref{vin}) twice. Since $HQ/M \ll \delta^{-O(1)}$ implies~\eqref{hqd}, we must have
$$ \| q_0 \alpha_d \|_{\R/\Z} \ll \delta^{-O(1)} (HQ/M)^{-2d} Q^{-d} (XQ/M^2)^{-d} = \delta^{-O(1)} H^{-2d} X^{-d} Q^{-4d} M^{4d} = \delta^{-O(1)} H^{-d}$$
for some $1 \leq q_0 \ll \delta^{-O(1)}$ by choice of $Q$.   This proves~\eqref{qalphaj} for $j=d$.

For the induction step, let $1 \leq j_0 < d$, and assume that~\eqref{qalphaj} has already been proved for $j \in \{j_0+1, \cdots,d\}$. Then the polynomials $n \mapsto q_0\alpha_j(n-X)^j$ has $C^{\infty}((X,X+H])$-norm $\ll \delta^{-O(1)}$ for $j \in \{j_0+1,\cdots,d\}$, and thus the polynomial $Q$ defined by
$$ Q(n) = q_0\Big(P(n) - \sum_{j=j_0+1}^d \alpha_j(n-X)^j\Big) = q_0\sum_{0\leq j\leq j_0}\alpha_j(n-X)^j $$
also satisfies the bound~\eqref{ega}. By repeating the analysis above with inspecting the top degree coefficient $q_0\alpha_{j_0}$ of $Q$ and applying twice the polynomial Vinogradov lemma, we deduce that
$$ \|q_1\cdot q_0\alpha_{j_0}\|_{\R/\Z} \ll \delta^{-O(1)} H^{-j_0} $$
for some $1 \leq q_1 \ll \delta^{-O(1)}$. This completes the induction step after replacing $q_0$ by $q_0q_1$.

Now that we have~\eqref{qalphaj}, it follows that $q_0P$ has $C^{\infty}((X,X+H])$-norm $\ll \delta^{-O(1)}$, and the claim follows after replacing $\eta$ by $q_0\eta$.
\end{proof}

Now we are ready to establish Theorem~\ref{inverse}(iv) in full generality, using an argument similar to that employed in Section~\ref{type-i-sec}.  Let $d, D, H, X, \delta, G/\Gamma, F, f, A_{I_2}$ be as in Theorem~\ref{inverse}(iv).  Henceforth we allow implied constants to depend on $d,D$. By Definition~\ref{struct-sum} we can write $f = \alpha * \beta_1 * \beta_2$ where $\alpha$ is supported on $[1,A_{I_2}]$ and obeys~\eqref{abound} for all $A$, and $\beta_1,\beta_2$ obey~\eqref{btv}.  From~\eqref{invo} we have
$$\left|\sum_{1 \leq a \leq A_{I_2}} \alpha(a) \sum_{m} \sum_{\substack{n \\ X/a < nm \leq X/a+H/a}} \beta_1(m) \beta_2(n) 1_{P_0}(anm) F(g(anm)\Gamma)\right| \geq \delta H$$
for some arithmetic progression $P_0 \subset (X,X+H]$.  Applying a dyadic decomposition in the $a,m,n$ variables, we may assume that $\alpha, \beta_1, \beta_2$ are supported in $(A,2A]$, $(M,2M]$, $(N,2N]$ for some $1 \leq A \leq A_{I_2}$ and $M,N \geq 1/2$, at the cost of worsening the above bound to
\begin{equation}\label{aaa}
\left|\sum_{a \in (A,2A]} \alpha(a) \sum_{m \in (M,2M]} \sum_{\substack{N < n \leq 2N \\ X/a < nm \leq X/a+H/a}} \beta_1(m) \beta_2(n) 1_{P_0}(anm) F(g(anm)\Gamma)\right| \geq \delta^{O(1)} H
\end{equation}
(here we use the hypothesis $\delta \leq \frac{1}{\log X}$).  By symmetry we may assume that $M \leq N$.  We may also assume that $AMN \asymp X$ since the sum is empty otherwise; this implies in particular that $M \ll (X/A)^{1/2}$.  We may also assume that
\begin{equation}\label{haxa}
H/A \geq \delta^{-C} (X/A)^{1/3}
\end{equation}
for some large constant $C$ (depending only on $d,D$), since otherwise we have~\eqref{hale} after some algebra. By~\eqref{aaa}, Cauchy--Schwarz, and the bound~\eqref{abound} we obtain
\begin{equation}
\label{eq:AsumI2}
\sum_{a \in (A,2A]} \left|\sum_{m \in (M,2M]} \sum_{\substack{N < n \leq 2N \\ X/a < nm \leq X/a+H/a}} \beta_1(m) \beta_2(n) 1_{P_0}(anm) F(g(anm)\Gamma)\right|^2 \geq \delta^{O(1)} H^2/A.
\end{equation}
For each $a \in (A,2A]$, we see from the triangle inequality and~\eqref{btv} that
\begin{align*}
&\sum_{m \in (M,2M]} \sum_{\substack{N < n \leq 2N \\ X/a < nm \leq X/a+H/a}} \beta_1(m) \beta_2(n) 1_{P_0}(anm) F(g(anm)\Gamma)\\
&\quad \ll \delta^{-O(1)} \sum_{m \in (M,2M]} \sum_{\substack{n \\ X/a < nm \leq X/a+H/a}} 1
\end{align*}
and hence by the bound~\eqref{hlog}
$$
\sum_{m \in (M,2M]} \sum_{\substack{n \in (N,2N] \\ X/a < nm \leq X/a+H/a}} \beta_1(m) \beta_2(n) 1_{P_0}(anm) F(g(anm)\Gamma) \ll \delta^{-O(1)} H/A.$$
Combining this with~\eqref{eq:AsumI2} implies that
$$ \left|\sum_{m \in (M,2M]} \sum_{\substack{n \in (N,2N] \\ X/a < nm \leq X/a+H/a}} \beta_1(m) \beta_2(n) 1_{P_0}(anm) F(g(anm)\Gamma)\right| \gg \delta^{O(1)} H/A$$
for $\gg \delta^{O(1)} A$ values of $a \in (A,2A]$.  Applying Proposition~\ref{d2-prelim} (and~\eqref{haxa}), we conclude that for each such $a$ there exists a non-trivial horizontal character $\eta \colon G \to \R$ of Lipschitz norm $O(\delta^{-O(1)})$ such that
$$ \| \eta \circ g(a \cdot) \|_{C^\infty(X/a,X/a+H/a]} \ll \delta^{-O(1)}.$$
This $\eta$ currently is permitted to vary in $a$, but there are only $O(\delta^{-O(1)})$ choices for $\eta$, so by the pigeonhole principle we may assume without loss of generality that $\eta$ is independent of $a$.  Applying Corollary~\ref{smooth-dilate} (and~\eqref{haxa}), we conclude that there exists $1 \leq q \ll \delta^{-O(1)}$ such that
$$ \| q \eta \circ g \|_{C^\infty(X,X+H]} \ll \delta^{-O(1)}$$
and the claim follows.

At this point we have proved all cases of Theorem~\ref{inverse} which are necessary for our main Theorem (Theorem~\ref{discorrelation-thm}).

\section{Controlling the Gowers uniformity norms}\label{gowers-sec}

In order to deduce our Gowers uniformity result in short intervals (Theorem~\ref{thm_gowers}) from Theorem~\ref{discorrelation-thm}, we wish to apply the inverse theorem for the Gowers norms to $\Lambda-\Lambda^{\sharp}$, $d_k-d_{k}^{\sharp}$, $\mu$. However, before we can apply the inverse theorem, we need to show that the functions $\Lambda-\Lambda^{\sharp}$, $d_k-d_{k}^{\sharp}$ possess pseudorandom majorants even when localized to \emph{short} intervals. In the case of long intervals, the existence of pseudorandom majorants for these functions follows from existing works~\cite{green-tao},~\cite{matthiesen-linear}, and the main purpose of this section is to show that these long interval majorants also work over short intervals $(X,X+X^{\theta}]$.

We begin by defining what we mean by pseudorandomness localized to a short range\footnote{Strictly speaking, $H$ does not need to be small in terms of $x$ in Definition~\ref{def_pseudoshort}, but that is the regime we are most interested in.}.

\begin{definition}[Pseudorandomness over short intervals]\label{def_pseudoshort}
Let $x,H\geq 1$. Let $D\in \mathbb{N}$ and $0<\eta<1$. We say that a function $\nu:\mathbb{Z}\to \mathbb{R}_{\geq 0}$ is \emph{$(D,\eta)$-pseudorandom at location $x$ and scale $H$} if the function $\nu_x(n):=\nu(x+n)$ satisfies the following. Let $\psi_1,\ldots, \psi_t$ be affine-linear forms, where each $\psi_i:\mathbb{Z}^d\to \mathbb{Z}$ has the form $\psi_i(\mathbf{x})=\dot{\psi_i}\cdot \mathbf{x}+\psi_i(0)$, with $\dot{\psi_i}\in \mathbb{Z}^d$ and $\psi_i(0)\in \mathbb{Z}$ satisfying $d,t\leq D$, $|\dot{\psi_i}|\leq D$ and $|\psi_i(0)|\leq DH$, and with $\dot{\psi_i}$ and $\dot{\psi_j}$ linearly independent whenever $i\neq j$. Then, for any convex body $K\subset [-H,H]^d$, 
\begin{align*}
\left|\sum_{\mathbf{n}\in K}\nu_x(\psi_1(\mathbf{n}))\cdots \nu_x(\psi_t(\mathbf{n}))-\textnormal{vol}(K)\right|\leq \eta H^d.    
\end{align*}
\end{definition}

\begin{remark}\label{rem_gowers}
We note that the $(D,\eta)$-pseudorandomness of $\nu$ at location $x$ and scale $H$ directly implies the short interval Gowers uniformity bound $\|\nu-1\|_{U^D(x,x+H]}\ll_D \eta^{1/2^D}$, just by the definition of the Gowers norm as a correlation along linear forms.
\end{remark}

Our notion of pseudorandomness in the ``long interval'' case $x=0$ differs from that of Green--Tao~\cite[Section 6]{green-tao} in two ways. Firstly, we do not need to impose the \emph{correlation condition}~\cite[Definition 6.3]{green-tao} (making use of the later work of Dodos and Kanellopoulos~\cite{dodos}). Secondly, we work with pseudorandom functions defined on the integers, as opposed to those defined on cyclic groups. The latter is only a minor technical convenience, as then we do not need to extend majorants defined on the integers into a cyclic group. The next lemma shows that the notion of pseudorandomness over the integers is very closely related to pseudorandomness over a cyclic group.

\begin{lemma}\label{le_pseudo2}
Let $x,H\geq 1$, $D\in \mathbb{N}$, and $0<\eta<1$. Suppose that $\nu:\mathbb{Z}\to \mathbb{R}_{\geq 0}$ is $(D,\eta)$-pseudorandom at location $x$ and scale $H$. Then  there exists a prime $H<H'\ll_D H$ and a function $\widetilde{\nu}:\mathbb{Z}/H'\mathbb{Z}\to\mathbb{R}_{\geq 0}$ such that $\nu(x+n)\leq 2\widetilde{\nu}(n)$ for all $n\in [0,H]$ (where $[0,H]$ is embedded into $\mathbb{Z}/H'\mathbb{Z}$ in the natural way) and such that $\widetilde{\nu}$ satisfies the following. Let $\psi_1,\ldots, \psi_t$ be affine-linear forms, where each $\psi_i:\mathbb{Z}^d\to \mathbb{Z}$  has the form $\psi_i(\mathbf{x})=\dot{\psi_i}\cdot \mathbf{x}+\psi_i(0)$, with $\dot{\psi_i}\in \mathbb{Z}^d$ and $\psi_i(0)\in \mathbb{Z}$ satisfying $t\leq D$, $|\dot{\psi_i}|\leq D$. Then
\begin{align}\label{eqqn7}
\sum_{\mathbf{n}\in (\mathbb{Z}/H'\mathbb{Z})^d}\widetilde{\nu}(\psi_1(\mathbf{n}))\cdots \widetilde{\nu}(\psi_t(\mathbf{n}))=(1+O_D(\eta))(H')^{d},    
\end{align}
where the affine-linear forms $\psi_j:(\mathbb{Z}/H'\mathbb{Z})^d\to \mathbb{Z}/H'\mathbb{Z}$ are induced from their global counterparts in the obvious way.
\end{lemma}

\begin{proof}
Let $H'\in [C_DH,2C_DH]$ be a prime for large enough $C_D\geq 1$. Take $\widetilde{\nu}(n)=(\frac{1}{2}+\frac{1}{2}\nu(x+n))1_{n\in [0,H]}+1_{(H,H')}(n)$, extended to an $H'$-periodic function. Then the claim~\eqref{eqqn7} follows  from the $(D,\eta)$-pseudorandomness of $\nu$ at location $x$ and scale $H$ by splitting $\widetilde{\nu}$ into its components.
\end{proof}

We then state the inverse theorem for unbounded functions that we are going to use.

\begin{proposition}[An inverse theorem for pseudorandomly bounded functions]\label{prop_inverse} Let $s\in \mathbb{N}$ and $0<\eta<1$. Let $I$ be an interval of length $\geq 2$. Let $f:I\to \mathbb{C}$ be a function, and suppose that the following hold.
\begin{itemize}
    \item There exists a function $\nu:I\to \mathbb{R}_{\geq 0}$ such that $\|\nu-1\|_{U^{2s}(I)}\leq \eta$ and $|f(n)|\leq \nu(n)$.   
    
    \item For any filtered $(s-1)$-step nilmanifold $G/\Gamma$ and any Lipschitz function $F:G/\Gamma\to \mathbb{C}$, we have
    \begin{align*}
    \sup_{g\in \Poly(\mathbb{Z}\to G)}\left|\frac{1}{|I|}\sum_{n\in I}f(n)\overline{F}(g(n)\Gamma)\right|\ll_{\|F\|_{\textnormal{Lip}},G/\Gamma} \eta.    
    \end{align*}
\end{itemize}
Then we have the Gowers uniformity estimate
\begin{align*}
 \|f\|_{U^s(I)}=o_{s;\eta\to 0}(1).   
\end{align*}
\end{proposition}

\begin{proof} Let $I=(X,X+H]$, where without loss of generality $X$ and $H$ are integers. The desired result follows from the work of Dodos and Kanellopoulos~\cite[Theorem 5.1]{dodos} (which gives the inverse theorem of~\cite[Proposition 10.1]{green-tao} under weaker hypotheses). Indeed, we can apply~\cite[Theorem 5.1]{dodos} to the function $n\mapsto f(X+n)$ on $[1,H]$, noting that the interval Gowers norm estimate $\|\nu-1\|_{U^{2s}(I)}=o_{\eta \to 0}(1)$ is equivalent to the cyclic group Gowers norm estimate $\|\widetilde{\nu}-1\|_{U^{2s}(\mathbb{Z}/N'\mathbb{Z})}=o_{\eta \to 0}(1)$ for all primes $N'\in [100sH,200sH]$, where $\widetilde{\nu}(n)$ is defined as $\nu(X+n)1_{n\in [1,H]}$ for $0\leq n<N'$ and extended periodically to $\mathbb{Z}/N'\mathbb{Z}$. 
\end{proof}

The following lemma tells us that if a function has a pseudorandom majorant over a long interval, and if the majorant is given by a type $I$ sum, then it in fact has a pseudorandom majorant over short intervals as well. This allows us to conveniently reduce the concept of pseudorandom majorants over short intervals to that over long intervals.

\begin{lemma}[Pseudorandomness over long intervals implies pseudorandomness over short intervals]\label{le_pseudo}
Let $\varepsilon \in (0,1)$, $D,k\in \mathbb{N}$ be fixed. Let $C\geq 1$ be large enough in terms of $k$ and $D$. Let $H\in [X^{\varepsilon},X/2]$ and $\eta\in ((\log X)^{-C},1/2)$, with $X\geq 3$ large enough. Let $\nu:\mathbb{Z}\to \mathbb{R}_{\geq 0}$ be $(D,\eta)$-pseudorandom at location $0$ and scale $H$. Also let $1\leq A,B\leq \log X$ be integers. 

Suppose that there is an exceptional set $\mathscr{S}\subset \mathbb{Z}$ and a sequence $\lambda_n$ such that 
\begin{align}\label{eqqn2}\begin{split}
\nu(n)&=\sum_{\substack{d\mid An+B\\d\leq X^{\varepsilon/(2D)}}}\lambda_d \quad \textnormal{for}\quad n\not \in \mathscr{S},\\
|\lambda_n|&\leq (\log X)^{k}d(n)^k\quad \textnormal{for all}\quad n,\\
|\nu(n)|&\leq (\log X)^kd(An+B)^k \quad \textnormal{for}\quad n \in \mathscr{S}.
\end{split}
\end{align} 
Also suppose that $\mathscr{S}$ is small in the sense that
\begin{align}\label{eqqn3}
|\mathscr{S}\cap [y-2DH,y+2DH]|\ll H/(\log X)^{4C}\textnormal{ for } y\in \{0,X\}
\end{align}
Then $\nu$ is $(D,2\eta)$-pseudorandom at location $X$ and scale $H$.
\end{lemma}

\begin{proof} By~\eqref{eqqn2}, we can write
\begin{align*}
\nu(n)&=1_{n\not \in \mathscr{S}} \sum_{\substack{d\mid An+B\\d\leq X^{\varepsilon/(2D)}}}\lambda_d+O((\log X)^kd(An+B)^{k}1_{n\in \mathscr{S}})\\
&=\sum_{\substack{d\mid An+B\\d\leq X^{\varepsilon/(2D)}}}\lambda_d+O((\log X)^kd(An+B)^{k+1}1_{n\in \mathscr{S}}).   \end{align*}
Hence, for any convex body $K\subset [-H,H]^d$ and for $x\in \{0,X\}$, we can split the sum
\begin{align*}
\sum_{\mathbf{n}\in K}\prod_{i=1}^t\nu_x(\psi_i(\mathbf{n}))
\end{align*}
(where $\nu_x(n):=\nu(x+n)$) as the main term 
\begin{align}\label{eqqn6}
\sum_{e_1,\ldots, e_{t}\leq X^{\varepsilon/(2D)}}\lambda_{e_1}\cdots \lambda_{e_{t}}\sum_{\mathbf{n}\in K} \prod_{i=1}^{t}1_{e_{i}\mid A(x+\psi_i(\mathbf{n}))+B}.
\end{align}
and $2^{t}-1$ error terms whose contribution is for some $j\leq t$ bounded using~\eqref{eqqn2} by
\begin{align}\label{eqqn5}
\ll (\log X)^{kt}\sum_{\mathbf{n}\in K}\prod_{i=1}^td(A(x+\psi_i(\mathbf{n}))+B)^{k+1}1_{x+\psi_j(\mathbf{n}) \in \mathscr{S}}.        
\end{align}
Now, using Cauchy--Schwarz, the inequality $\prod_{i=1}^tx_i\leq \sum_{i=1}^tx_i^t$,~\eqref{eqqn3}, and Shiu's bound (Lemma~\ref{shiu}),~\eqref{eqqn5} is
\begin{align*}
&\ll  (\log X)^{kt}\left(\sum_{\mathbf{n}\in K}1_{x+\psi_j(\mathbf{n}) \in \mathscr{S}}\right)^{1/2}\left(\sum_{\mathbf{n}\in K}\prod_{i=1}^td(A(x+\psi_i(\mathbf{n}))+B)^{2(k+1)}\right)^{1/2}\\
&\ll (\log X)^{kt}\left(\sum_{\mathbf{n}\in K}1_{x+\psi_j(\mathbf{n}) \in \mathscr{S}}\right)^{1/2} \left(\sum_{\mathbf{n}\in K}\sum_{i=1}^t d(A(x+\psi_i(\mathbf{n}))+B)^{2(k+1)t}\right)^{1/2}\\
&\ll H^d(\log X)^{kt-2C}(\log X)^{M_{D,k}}
\end{align*}
for some constant $M_{D,k}\geq 1$. If $C$ is large enough in terms of $D$ and $k$, this is $\ll H^d(\log X)^{-3C/2}$.

We lastly estimate the main term in~\eqref{eqqn6}. A lattice point counting argument as in~\cite[Appendix A]{green-tao} gives us
\begin{align*}
\sum_{\mathbf{n}\in K} \prod_{i=1}^{t}1_{e_{i}\mid A(x+\psi_i( \mathbf{n}))+B}&=\alpha_{A,B}(e_1,\ldots, e_{t})\textnormal{vol}(K)+O(H^{d-1})
\end{align*}
for some $\alpha_{A,B}(e_1,\ldots, e_{t})\in [0,1]$ independent of $x$ and $H$ (since the left-hand side is counting elements of $K$ in some shifted lattice $\mathbf{q}\mathbb{Z}+\mathbf{a}$). Combining this with the estimates $e_1\cdots e_{t}\leq X^{\varepsilon/2}\leq H^{1/2}$ and $|\lambda_d|\ll X^{o(1)}$, we see that
\begin{align}\begin{split}\label{eqqn1}
 &\sum_{\mathbf{n}\in K}\prod_{i=1}^t\nu_x(\psi_i(\mathbf{n}))= \sum_{e_1,\ldots, e_{t}\leq X^{\varepsilon/(2D)}}\lambda_{e_1}\cdots \lambda_{e_{t}} \alpha_{A,B}(e_1,\ldots, e_{t})\textnormal{vol}(K)+O(H^d(\log X)^{-3C/2}). 
 \end{split}
\end{align}
Since the main term on the right-hand side of~\eqref{eqqn1} is independent of $x\in \{0,X\}$, we see that 
 $$\sum_{\mathbf{n}\in K}\prod_{i=1}^t\nu_X(\psi_i(\mathbf{n}))=\sum_{\mathbf{n}\in K}\prod_{i=1}^t\nu_0(\psi_i(\mathbf{n}))+O(H^d(\log X)^{-3C/2}).$$ 
 Hence, using the assumption that $\nu$ is $(D,\eta)$-pseudorandom at location $0$ and scale $H$, $\nu$ must also be $(D,2\eta)$-pseudorandom at location $X$ and scale $H$.
\end{proof}

Lemma~\ref{le_pseudo}  leads to the existence pseudorandom majorants over short intervals for $W$-tricked versions of our functions of interest. Let us recall that, for any $w\geq 2$,
\begin{align*}
\Lambda_w(n):=\frac{W}{\varphi(W)}1_{(n,W)=1},    
\end{align*}
where $W=\prod_{p\leq w}p$. We note for later use that in this notation our model function $\Lambda^{\sharp}$ equals to $\Lambda_{R}$, where $R=\exp((\log X)^{1/10})$.

\begin{lemma}[Pseudorandom majorants over short intervals for $\Lambda-\Lambda_w$, $d_k-d_{k}^{\sharp}$]\label{le_pseudoinshort} Let $\varepsilon>0$ and $D,k\in \mathbb{N}$ be fixed. Let  $X\geq H\geq X^{\varepsilon}\geq 2$. Let $2\leq w\leq w(X)$, where $w(X)$ is a slowly growing function of $X$, and denote $W=\prod_{p\leq w}p$. Also let $w\leq \widetilde{w}\leq \exp((\log X)^{1/10})$.

\begin{enumerate}
    \item There exists a constant $C_0\geq 1$ such that each of the functions 
\begin{align}\label{eq_lambdaW}
&\frac{\varphi(W)}{W}\Lambda(Wn+b)/C_0,\quad \frac{\varphi(W)}{W}\Lambda_{\widetilde{w}}(Wn+b)
\end{align}
for $1\leq b\leq W$ with $(b,W)=1$, is majorized on $(X, X+H]$ by a $(D,\eta)$-pseudorandom function at location $X$ and scale $H$ for some $\eta=o_{w\to \infty}(1)$. In fact, the latter of the two functions is $(D,\eta)$-pseudorandom at location $X$ and scale $H$.
\item Let $W'$ be such that $W\mid W'\mid W^{\lfloor w\rfloor}$. Suppose that $H\geq X^{1/5+\varepsilon}$. There exists a constant $C_k\geq 1$ such that each of the functions
\begin{align}\label{eq_dkW}
\begin{split}
&(\log X)\frac{\varphi(W)}{W}\prod_{w\leq p\leq X}\left(1+\frac{k}{p}\right)^{-1}d_k(W'n+b)/C_k,\\
&(\log X)\frac{\varphi(W)}{W}\prod_{w\leq p\leq X}\left(1+\frac{k}{p}\right)^{-1}d_k^{\sharp}(W'n+b)/C_k
\end{split}
\end{align}
\end{enumerate}
for $1\leq b\leq W'$ with $(b,W')=1$, is majorized on $(X, X+H]$ by a $(D,\eta)$-pseudorandom function at location $X$ and scale $H$ for some $\eta=o_{w\to \infty}(1)$. 
\end{lemma}

\begin{remark}\label{rem_gowers2}
Note that if $\|\nu_1-1\|_{U^{D}(x,x+H]}\leq \eta$ and $\|\nu_2-1\|_{U^D(x,x+H]}\leq \eta$, then by the triangle inequality for the Gowers norms also $\|(\nu_1+\nu_2)/2-1\|_{U^{D}(x,x+H]}\leq \eta$. Hence, by Remark~\ref{rem_gowers}, Lemma~\ref{le_pseudo} in particular provides us a majorant $\nu$ for the difference of the two functions in~\eqref{eq_lambdaW} or~\eqref{eq_dkW} satisfying $\|\nu-1\|_{U^D(x,x+H]}=o_{w\to \infty}(1)$, allowing us to apply the inverse theorem (Proposition~\ref{prop_inverse}).
\end{remark}

\begin{proof}
(1) Let us first consider the function $\frac{\varphi(W)}{W}\Lambda(Wn+b)/C_0$. Let $R'=X^{\gamma}$ with $\gamma>0$ small enough in terms of $\varepsilon,D$. Let $\psi$ be a smooth function supported on $[-2,2]$ with $\psi(0)=-1$ and  $\int_{0}^{\infty}|\psi'(y)|^2\, dy=1$. Define
\begin{align*}
\Lambda_{R',\psi}(n):=-(\log R')\sum_{d\mid n}\mu(d)\psi\left(\frac{\log d}{\log R'}\right).    
\end{align*}
Put 
$$\nu_b(n):=\frac{\varphi(W)}{W}(\log R')^{-1}\Lambda_{R',\psi}(Wn+b)^2+2(\log X)1_{Wn+b\in S},$$
where $S$ is the set of perfect powers. Then $$\frac{\varphi(W)}{W}\Lambda(Wn+b)\leq 2\gamma^{-1}\nu_b(n)$$ for $X/2\leq n\leq X$, since $Wn+b$ being prime implies that $Wn+b$ has no divisors $1<d\leq X^{2\gamma}$.

From~\cite[Theorem D.3]{green-tao}  we see that $\nu_b$ is $(D,o_{w\to \infty}(1))$-pseudorandom at location $0$ and scale $H$ (since the term $2(\log X)1_{Wn+b\in S}$ has negligible contribution to the correlations that arise in the definition of pseudorandomness). Moreover, $\nu_b(n)$ can be expanded out as 
$$\sum_{\substack{d\mid Wn+b\\d\leq X^{4\gamma}}}\lambda_d+2(\log X)1_{Wn+b\in S}$$
for some 
$$|\lambda_n|\ll (\log X)\sum_{\substack{d_1,d_2\geq 1\\n=[d_1,d_2]}}1\ll (\log X)d(n)^2.$$ Hence, by Lemma~\ref{le_pseudo}, $\nu_b$ is $(D,o_{w\to \infty}(1))$-pseudorandom also at location $X$ and scale $H$ (since the set $\mathscr{S}:=\{n:Wn+b\in S\}$ certainly obeys~\eqref{eqqn3}).

For the case of $\frac{\varphi(W)}{W}\Lambda_{\widetilde{w}}(Wn+b)$, we can apply~\cite[Proposition 5.2]{tt-quant} to directly deduce that this function is $(D,o_{w\to \infty}(1))$-pseudorandom at location $0$ and scale $X$. To prove the $(D,o_{w\to \infty}(1))$-pseudorandomness of this function also at location $X$ and scale $H$, we show that it is well-approximated by a type $I$ sum. By M\"obius inversion,
\begin{align*}
\frac{\varphi(W)}{W}\Lambda_{\widetilde{w}}(Wn+b)=\frac{\varphi(W)}{W}\prod_{p\leq \widetilde{w}}\left(1-\frac{1}{p}\right)^{-1}\sum_{\substack{d\mid Wn+b\\d\mid P(\widetilde{w})}}\mu(d),   
\end{align*}
and by Lemma~\ref{le:FLS} we have
\begin{align*}
\sum_{X<n\leq X+H}\Big|\sum_{\substack{d\mid Wn+b\\d\mid P(\widetilde{w})\\d\geq X^{\varepsilon/(2D)}}}\mu(d)\Big|\ll H\frac{(\log X)^{2e}}{\exp(\frac{\varepsilon}{2D}\frac{\log X}{\log \widetilde{w}})}\ll H\exp(-(\log X)^{4/5}),    
\end{align*}
say. Hence $\frac{\varphi(W)}{W}\Lambda_{\widetilde{w}}(Wn+b)=\nu(n)+\eta(n)$, where $\nu$ is  of the form of Lemma~\ref{le_pseudo} and $\sum_{X<n\leq X+H}|\eta(n)|\ll H\exp(-(\log X)^{3/5})$, say. It suffices to show that $\nu$ is $(D,o_{w\to \infty}(1))$-pseudorandom at location $X$ and scale $H$, and this follows from Lemma~\ref{le_pseudo}.

(2)  Note that by~\eqref{eq:dkapp<<dk} we have $d_k^{\sharp}(n)\ll_k d_k(n)$ for all $n\geq 1$, so by Lemma~\ref{le_pseudo} it suffices to show that the function 
$$h(n):=(\log X)\frac{\varphi(W)}{W}\prod_{w\leq p\leq X}\left(1+\frac{k}{p}\right)^{-1}d_k(W'n+b)/C_k'$$
is for some $C_k'\geq 1$ majorized by a $(D,o_{w\to \infty}(1))$-pseudorandom function at location $0$ and scale $H$, which is of the form~\eqref{eqqn2} outside an exceptional set $\mathscr{S}$ satisfying~\eqref{eqqn3}.

By~\cite[Proposition 9.4]{matthiesen-linear}, for any $X\geq 2$ and $1\leq n\leq 2DX$, we have
\[
h(n) \ll \nu(n) + h(n) 1_{n \in \mathscr{S}},
\]
where $\nu$ is a certain $(D,o_{X\to \infty}(1))$-pseudorandom function at location $0$ and scale $X$, and $\mathscr{S}$ is defined in~\cite[Section 7]{matthiesen-linear} as
\begin{align*}
\mathscr{S}&=\mathscr{S}_1\cup\mathscr{S}_2,\\ 
\mathscr{S}_1:&=\left\{n\leq 2Dx:\,\, \exists \, p:\, v_p(n)\geq \max\left\{2,C_1\frac{\log \log X}{\log p}\right\}\right\}.\\
\mathscr{S}_2:&=\left\{n\leq 2DX:\,\,\prod_{p\leq X^{1/(\log \log X)^3}}p^{v_p(n)}\geq X^{\gamma/\log \log X}\right\}
\end{align*} 
Here $C_1$ can be taken arbitrarily large, so we may assume that $C_1>8C$ for any given constant $C$. To show that $\mathscr{S}$ satisfies~\eqref{eqqn3}, it suffices to show that for $j\in \{1,2\}$ we have
\begin{align}
|\mathscr{S}_j\cap [X-2DH,X+2DH]|&\ll H/(\log X)^{4C},\label{eqqn3b}\\
|\mathscr{S}_j\cap [-2DH,2DH]|&\ll H/(\log X)^{4C}.\label{eqqn3bb}
\end{align}
Let us prove~\eqref{eqqn3b}, the proof of~\eqref{eqqn3bb} is similar but easier.

We first prove~\eqref{eqqn3b} for $j=1$. By splitting into shorter intervals if necessary, we may assume that $H\leq X^{1/3}$, say. Note that the number of $n\in (X-2DH,X+2DH]$ satisfying $v_p(n)\geq \max\{2,C_1\frac{\log \log X}{\log p}\}$ for some $p$ is
\begin{align*}
&\ll \sum_{p< (\log X)^{4C}}H\exp(-C_1(\log \log X))+\sum_{(\log X)^{4C}\leq p\leq (4DH)^{1/2}}\frac{H}{p^2}\\
&+\sum_{(4DH)^{1/2}<p\leq (2X)^{1/2}}\left(\left\lfloor \frac{X+2DH}{p^2}\right\rfloor-\left\lfloor \frac{X-2DH}{p^2}\right\rfloor\right)\\
& \ll H(\log X)^{-4C}+\sum_{(4DH)^{1/2}<p\leq (2X)^{1/2}}\left(\left\lfloor \frac{X+2DH}{p^2}\right\rfloor-\left\lfloor \frac{X-2DH}{p^2}\right\rfloor\right),  
\end{align*}
since $C_1>8C$. 

We can trivially bound 
\begin{align*}
\sum_{(4DH)^{1/2}<p\leq H(\log X)^{-4C}}\left(\left\lfloor \frac{X+2DH}{p^2}\right\rfloor-\left\lfloor \frac{X-2DH}{p^2}\right\rfloor\right) &\ll  \sum_{(4DH)^{1/2}<p\leq H(\log X)^{-4C}}1\\
&\ll H(\log X)^{-4C}.
\end{align*}

Next, we bound 
\begin{align}\label{eqqn3c}
\sum_{H(\log X)^{4C}< p\leq (4DH)^{1/2}}\left(\left\lfloor \frac{X+2DH}{p^2}\right\rfloor-\left\lfloor \frac{X-2DH}{p^2}\right\rfloor\right).    
\end{align}
Note that for any $p\geq H(\log X)^{4C}$ there is at most one multiple of $p^2$ in $(X-2DH,X+2DH]$, so~\eqref{eqqn3c} is at most $|S(H(\log X)^{4C},(4DH)^{1/2})|$, where 
\begin{align*}
 S(t_1,t_2):=\{d\in (t_1,t_2]:\,\, md^2\in [X-2DH,X+2DH]\,\textnormal{ for some }\, m\in \mathbb{N}\}   
\end{align*}
In~\cite[p. 221]{filaseta}, it is proven for $H\geq X^{1/5+\varepsilon}$ that 
\begin{align*}
|S(H\log X,2\sqrt{X})|\ll X^{1/5}\log X,    
\end{align*}
so~\eqref{eqqn3c} is $\ll H(\log X)^{-4C}$.

Finally, we bound 
\begin{align}\label{eqqn3d}
&\sum_{H(\log X)^{-4C}\leq p\leq H(\log X)^{4C}}\left(\left\lfloor \frac{X+2DH}{p^2}\right\rfloor-\left\lfloor \frac{X-2DH}{p^2}\right\rfloor\right)\nonumber\\
&=\sum_{H(\log X)^{-4C}\leq p\leq H(\log X)^{4C}}\left(\frac{4DH}{p^2}-\left\{\frac{X+2DH}{p^2}\right\}+\left\{\frac{X-2DH}{p^2}\right\}\right).   
\end{align}
The first term in the sum gives a negligible contribution of $\ll (\log X)^{4C}$. Pick two $1$-periodic smooth functions $W^{-}, W^{+}$ such that $W^{-}(t)\leq \{t\}\leq W^{+}(t)$ for all $t\in \mathbb{R}$ and such that $W^{\pm}(t)$ differs from $\{t\}$ only in the region where $\|t\|_{\R/\Z}\leq (\log X)^{-8C}$, and $W^{\pm}$ satisfy the derivative bounds $\sup_{t}|(W^{\pm})^{(\ell)}(t)|\ll (\log X)^{8C\ell}$ for $1\leq \ell \leq 3$. Then~\eqref{eqqn3d} is
\begin{align*}
\leq O\left((\log X)^{4C}\right)+\sum_{H(\log X)^{-4C}\leq p\leq H(\log X)^{4C}}\left(-W^{-}\left(\frac{X+2DH}{p^2}\right)+W^{+}\left(\frac{X-2DH}{p^2}\right)\right).  
\end{align*}
By~\cite[Proposition 1.12(ii)]{singmaster} and the fact that for any $u,h\geq 0$ we have $\{u+h\}-\{u\}= h$ unless $\|u\|_{\R/\Z}\leq h$, the main term here is 
\begin{align*}
&\int_{H(\log X)^{-4C}}^{H(\log X)^{4C}}\left(W^{+}\left(\frac{X-2DH}{t^2}\right)-W^{-}\left(\frac{X+2DH}{t^2}\right)\right)\frac{dt}{\log t}+O(H(\log X)^{-4C})\\
&\ll \max_{\sigma\in \{-1,+1\}}\int_{H(\log X)^{-4C}}^{H(\log X)^{4C}}\left(\frac{4DH}{t^2}+1_{\|(X+2DH\sigma)/t^2\|_{\R/\Z}\leq (\log X)^{-8C}}\right)\frac{dt}{\log t}+H(\log X)^{-4C}\\
&\ll H(\log X)^{-4C},
\end{align*}
since the condition $\|(X+2DH\sigma)/t^2\|_{\R/\Z}\leq (\log X)^{-8C}$ for $t\in [H(\log X)^{-4C},H(\log X)^{4C}]$ holds in a union of intervals of total measure $\ll H(\log X)^{-4C}$.

Putting the above estimates together, we obtain~\eqref{eqqn3b} for $j=1$.

 Let us then prove~\eqref{eqqn3b} for $j=2$. We thus bound the number of integers $n\in I:=(X-2DH,X+2DH]$ that satisfy $\prod_{p\leq X^{1/(\log \log X)^3}}p^{v_p(n)}\geq X^{\gamma/\log \log X}$. Writing $v = X^{1/(\log \log X)^3},$ the number of such $n \in I$ is
\begin{equation}
\label{eq:excninI}
\ll \sum_{\substack{ab \in I \\ p \mid a \implies p > v \\ p \mid b \implies p \leq v \\ b \geq X^{\gamma/\log \log X}}} 1 \leq \sum_{\substack{ab \in I \\ p \mid a \implies p > v \\ p \mid b \implies p \leq v}} \left(\frac{b}{X^{\gamma/\log \log X}}\right)^{\frac{10C (\log \log X)^2}{\gamma \log X}} \ll \frac{1}{(\log X)^{10C}} \sum_{n \in I} g(n),
\end{equation}
where $g$ is the completely multiplicative function for which 
\[
g(p) = 
\begin{cases}
1 & \text{if $p > v$;} \\ 
p^{\frac{10C (\log \log X)^2}{\gamma \log X}} & \text{if $p \leq v$.} 
\end{cases}
\]
Then Shiu's bound (Lemma~\ref{shiu}) implies that~\eqref{eq:excninI} is $\ll H/(\log X)^{4C}$. This proves~\eqref{eqqn3b} for $j=2$.

Hence $|\mathscr{S}| \ll H/(\log X)^{4C}$, and in particular arguing as in the beginning of the proof of Lemma~\ref{le_pseudo} we see that the fact that $\nu$ is a $(D,o_{X\to \infty}(1))$-pseudorandom function at location $0$ and scale $X$ implies that so is $\nu(n) + h(n) 1_{n \in \mathscr{S}}$.

Hence it suffices to show that $\nu(n)$ is of the form~\eqref{eqqn2}. The majorant $\nu(n)$ is defined in~\cite[Section 7]{matthiesen-linear}, for some $\gamma>0$ small enough in terms of $D,k$, as 
\begin{align}\label{eq_nu}
\nu(n):=\sum_{u\mid n}d_k(u)\sum_{\kappa=4/\gamma}^{\lfloor (\log \log X)^3 \rfloor}\sum_{\lambda=\lceil \log(\kappa)/\log 2-2\rceil}^{\lfloor \log((\log \log X)^3)/\log 2\rfloor}2^{k\kappa}1_{u\in U(\lambda,\kappa)}h_{\gamma}\left(\frac{n}{\prod_{p\mid u}p^{v_p(n)}}\right),
\end{align}
where 
\begin{itemize}
    \item $U(\lambda,\kappa)$, defined in~\cite[Section 7]{matthiesen-linear}, is a set contained in $[1,X^{10\gamma^{1/2}}]$ and satisfying
    \begin{align*}
u\in U(\lambda,\kappa),u>1&\implies \omega(u)\geq \frac{\gamma \kappa (\lambda+3-(\log \kappa)/(\log 2))}{200}\\
1\in U(\lambda,\kappa)&\implies \kappa=4/\gamma;
\end{align*}
\item $h_{\gamma}(n)=\sum_{\ell\mid n}(d_k*\mu)(\ell)\chi\left(\frac{\log \ell}{\log X^{\gamma}}\right),$
where $\chi:\mathbb{R}\to [0,1]$ is some smooth function supported in $[-1,1]$.
\end{itemize} 

Therefore, in particular, in~\eqref{eq_nu} we have
\begin{align*}
\kappa\leq (200/\gamma)(\omega(u)+1),    
\end{align*}
so that
\begin{align*}
2^{k\kappa}\ll d(u)^{M}    
\end{align*}
for some constant $M=M_{k,\gamma}\geq 1$.
Inserting the definition of $h_{\gamma}$ into the definition of $\nu$, and setting $T=X^{10\gamma^{1/2}}$, we see that for some $|\lambda_u|\ll d(u)^{k+M}(\log \log X)^{O_{D,k}(1)}$ we have
\begin{align*}
 \nu(n)=\sum_{\substack{u\mid n\\u\leq T}}\lambda_u\sum_{\substack{\ell\mid n\\\ell\leq X^{\gamma}}}(d_k*\mu)(\ell)1_{(\ell,u)=1}\chi\left(\frac{\log \ell}{\log X^{\gamma}}\right).
\end{align*}
Writing $e=\ell u$, we see that for some $|\lambda_{e}'|\ll (\log \log X)^{O_{D,k}(1)}d(e)^{k+M+1}d_{k+1}(e)$ the function $\nu$ is of the form
\begin{align*}
\nu(n)=\sum_{\substack{e\mid n\\e\leq X^{10\gamma^{1/2}+\gamma}}}\lambda_{e}'.    
\end{align*}
Taking $\gamma$ small enough in terms of $D,k$, this is of the form required in Lemma~\ref{le_pseudo}, so appealing to that lemma we conclude that $\nu$ is $(D,o_{w\to \infty}(1))$-pseudorandom at location $X$ and scale $H$.
\end{proof}

We need two more lemmas before proving Theorem~\ref{thm_gowers}.

\begin{lemma}\label{le_wtrick} Let $D\in \mathbb{N}$ be fixed. Let $1\leq q\leq H^{1/4}$ be an integer. Let $X\geq H\geq 2$, and let $f:(X,X+H]\to \mathbb{C}$ be a function with $|f(n)|\ll H^{1/2^{D+2}}$. Then we have 
\begin{align*}
\|f\|_{U^D(X,X+H]}\leq \frac{1}{q}\sum_{1\leq a\leq q}\|f_{q,a}\|_{U^D(X/q,(X+H)/q]}+O(H^{-1/2}),    
\end{align*}
where $f_{q,a}(n):=f(qn+a)$.
\end{lemma}

\begin{proof}
Denote by $1_{a(q)}$ the indicator of the arithmetic progression $a\pmod{q}$. Then, by the triangle inequality for the Gowers norms, we have
\begin{align*}
\|f\|_{U^D(X,X+H]}\leq \sum_{1\leq a\leq q}\|f 1_{a(q)}\|_{U^D(X,X+H]}.    
\end{align*}
The claim now follows by making a linear change of variables $(n,\mathbf{h})=(qn'+a,q\mathbf{h}')$ in the definition of $\|f 1_{a(q)}\|_{U^D(X,X+H]}$.
\end{proof}

\begin{lemma}\label{le_upperbound}
Let $D,k\in \mathbb{N}$ and $\varepsilon>0$ be fixed, with $\varepsilon>0$ small enough. Let $X\geq H\geq X^{\varepsilon}$, and let $1\leq q\leq X^{\varepsilon^2}$ be an integer. Let $f(n)=(\log X)^{1-k}d_k(n)$. Then for $1\leq a\leq q$ with $(a,q)=1$ we have
\begin{align*}
\|f_{q,a}\|_{U^{D}(X,X+H]}\ll \left(\frac{\varphi(q)}{q}\right)^{k-1},   \end{align*}
where  $f_{q,a}(n):=f(qn+a)$.
\end{lemma}

\begin{proof} Let $g_{q,a}(n):=d_k(qn+a)$. By the definition of the interval Gowers norms and the fact that $\|1_{(X, X+H]}\|_{U^{D}(\mathbb{Z})}^{2^D} \asymp H^{D+1}$, we have
\begin{align}\label{eq_hsums}\begin{split}
\|g_{q,a}\|_{U^{D}(X,X+H]}^{2^D}&\asymp \frac{1}{H^{D+1}}\sum_{n}\sum_{h_1,\ldots, h_D}\prod_{\omega\in \{0,1\}^D}d_k(q(n+\omega\cdot \mathbf{h})+a)1_{(X,X+H]}(n+\omega\cdot \mathbf{h})\\
&\ll \frac{1}{H^{D+1}}\sum_{X < n\leq X+H}\sum_{\substack{|h_1|,\ldots, |h_D|\leq 2H\\h_i\textnormal{ distinct }}}\prod_{\omega\in \{0,1\}^D}d_k(q(n+\omega\cdot \mathbf{h})+a) + H^{-1/2}.
\end{split}
\end{align}
We can upper bound the correlation of these multiplicative functions using Henriot's bound~\cite[Theorem 3]{Henriot} (taking $x\to X,y\to H$, $\delta\to 2^{-D-2}$, $Q(n)\to \prod_{\omega\in \{0,1\}^D}(q(n+\omega\cdot \mathbf{h})+a)$ there), obtaining
\begin{align}\label{eq_henriot}
&\frac{1}{H}\sum_{X < n \leq X+H}\prod_{\omega\in \{0,1\}^D}d_k(q(n+\omega\cdot \mathbf{h})+a)\nonumber\\
&\ll \Delta_{\mathcal{D}}\prod_{p\leq X}\left(1-\frac{\rho_Q(p)}{p}\right)\prod_{\omega\in \{0,1\}^D}\sum_{\substack{n\leq X\\(n,\mathcal{D})=1}}\frac{d_k(n)\rho_{Q_{\omega}}(n)}{n}, 
\end{align}
where 
\begin{align*}
Q_{\omega}(u)&=q(u+\omega\cdot \mathbf{h})+a, \qquad Q=\prod_{\omega\in \{0,1\}^D}Q_{\omega},\\
\rho_{P}(n)&=|\{u\pmod n:\,\, P(u)\equiv 0\pmod n\}|,\\
\mathcal{D}&=\mathcal{D}(\mathbf{h})=(-1)^{2^D(2^D-1)/2} q^{2^{2D}-2^D}\prod_{\omega\neq \omega'}((\omega-\omega')\cdot \mathbf{h})=: (-1)^{2^{D-1}} q^{2^{2D}-2^D}\mathcal{D}',\\
\Delta_{\mathcal{D}}&=\prod_{p\mid \mathcal{D}}\left(1+\sum_{\substack{0\leq \nu_1,\ldots, \nu_{2^{D}}\leq 1\\(\nu_1,\ldots, \nu_{2^D})\neq (0,\ldots,0)}}d_k(p^{\nu_1})\cdots d_k(p^{\nu_{2^D}})\frac{|\{n\pmod{p^2}:\,\, p^{\nu_{j}}\mid \mid Q_{\omega_j}(n)\,\forall\, j\}|}{p^{2}}\right)\\
&\ll \prod_{p\mid \mathcal{D}'}\left(1+\frac{O_{D,k}(1)}{p}\right),
\end{align*}
where $\omega_1,\ldots, \omega_{2^D}$ is any ordering of $\{0,1\}^D$. 
In order to bound the various expressions above, note that
\begin{align*}
\prod_{p\leq X}\left(1-\frac{\rho_Q(p)}{p}\right)\ll \prod_{\substack{p\leq X\\p\nmid \mathcal{D}}}\left(1-\frac{2^D}{p}\right)\ll (\log X)^{-2^D}\prod_{p\mid \mathcal{D}'}\left(1+\frac{2^D}{p}\right)\cdot \left(\frac{q}{\varphi(q)}\right)^{2^D}
\end{align*}
and
\begin{align*}
\sum_{\substack{n\leq X\\(n,\mathcal{D})=1}}\frac{d_k(n)\rho_{Q_{\omega}}(n)}{n}\ll  \prod_{\substack{p\leq X\\p\nmid q}}\left(1+\frac{k}{p}\right)\ll (\log X)^{k}\left(\frac{\varphi(q)}{q}\right)^{k}.     
\end{align*}
We now conclude that~\eqref{eq_henriot} is
\begin{align*}
\ll (\log X)^{(k-1) \cdot 2^D}\left(\frac{\varphi(q)}{q}\right)^{(k-1)\cdot 2^D}\prod_{p\mid \mathcal{D}'}\left(1+\frac{O_{D,k}(1)}{p}\right)    
\end{align*}
By the inequality $\prod_{i=1}^kx_i\leq \sum_{i=1}^k x_i^k$ and an elementary upper bound for moments of $n/\varphi(n)$, we have
\begin{align*}
\sum_{\substack{|h_1|,\ldots, |h_D|\leq 2H\\h_i\textnormal{ distinct }}} \prod_{p\mid \mathcal{D}'(\mathbf{h})}\left(1+\frac{O_{D,k}(1)}{p}\right)\ll  \sum_{\substack{|h_1|,\ldots, |h_D|\leq 2H\\h_i\textnormal{ distinct }}} \sum_{\omega\in \{-1,0,1\}^D\setminus\{\mathbf{0}\}}\prod_{p\mid \omega\cdot \mathbf{h}}\left(1+\frac{1}{p}\right)^{O_{D,k}(1)}\ll H^D.    \end{align*}
The claim now follows by combining this with~\eqref{eq_hsums}.
\end{proof}

We are now ready to prove Theorem~\ref{thm_gowers}.

\begin{proof}[Proof of Theorem~\ref{thm_gowers}] (i) Let $H$ be as in Theorem~\ref{thm_gowers}(i).  By the triangle inequality for the Gowers norms, to prove~\eqref{erg14} it suffices to show that
\begin{align}\label{eq_lambda1}
\|\Lambda^{\sharp}-\Lambda_w\|_{U^s(X,X+H]}=o_{w\to \infty}(1).
\end{align}
and 
\begin{align}\label{eq_lambda3}
\|\Lambda-\Lambda^{\sharp}\|_{U^s(X,X+H]}=o_{X\to \infty}(1)
\end{align}
The first claim~\eqref{eq_lambda1} follows directly from Lemma~\ref{le_pseudoinshort} and Remark~\ref{rem_gowers}. 

We are then left with proving~\eqref{eq_lambda3} and~\eqref{erg14b}. 
Let $1\leq b\leq W'\leq \log X$ be integers. For $f=\Lambda-\Lambda^{\sharp}$, by Theorem~\ref{discorrelation-thm}  for any $x\in [X/(\log X)^A,X(\log X)^A]$,$H(\log X)^{-A}\leq H'\leq H$ and $G/\Gamma$, $F$ as in that theorem, we have
\begin{align}\label{eqqn11}\begin{split}
&\sup_{g\in \Poly(\mathbb{Z}\to G)}\left|\sum_{x < n \leq x+H'}f(W'n+b)\overline{F}(g(n)\Gamma) \right|\\
&= \sup_{g\in \Poly(\mathbb{Z}\to G)}\left|\sum_{\substack{W'x+b < n\leq W'(x+H')+b\\n\equiv b\pmod{W'}}}f(n)\overline{F}(g(\frac{n-b}{W'})\Gamma) \right|\\
&\ll_A H'/(\log X)^{A},
\end{split}
\end{align}
since there exists a polynomial sequence $\widetilde{g}:\mathbb{Z}\to G$ such that $\widetilde{g}(n)=g((n-b)/W')$ for all $n\equiv b\pmod{W'}$.

Now~\eqref{erg14b} follows  by combining the inverse theorem (Proposition~\ref{prop_inverse}) with the estimate~\eqref{eqqn11}, Lemma~\ref{le_pseudoinshort}, and Remark~\ref{rem_gowers2}. Lastly,~\eqref{erg14} follows from~\eqref{erg14b} and Lemma~\ref{le_wtrick}.

(ii) We then turn to the case $f=d_k-d_{k}^{\sharp}$.  Again, Theorem~\ref{discorrelation-thm} gives us the bound~\eqref{eqqn11}. Together with the inverse theorem (Proposition~\ref{prop_inverse}), Lemma~\ref{le_pseudoinshort} and Remark~\ref{rem_gowers2}, this implies~\eqref{dk-unifb}. 

Let
\begin{align*}
h(n):=(\log X)^{1-k}(d_k(n)-d_k^{\sharp}(n)).    
\end{align*}
 Then,  to prove~\eqref{dk-unif}, we must show that 
\begin{align*}
\|h\|_{U^D(X,X+H]}=o_{X\to \infty}(1).    
\end{align*}

Let $\widetilde{W}:=W^w$ with $w$ an integer tending to infinity slowly\footnote{Let us explain why we perform the $W$-trick for the divisor function with the modulus $\widetilde{W}:=W^w$ rather than with the modulus $W$. In order to apply the inverse theorem, we wish to find a modulus $W'$ such that $h(W'n+a)$ is pseudorandomly majorized for almost all $1\leq a\leq W'$. Since $|h(Wn+a)|\ll d_k((W',a))d_k(\frac{W'}{(W',a)}n+\frac{a}{((W',a))}),$ we want to show that this latter function is pseudorandomly majorized for almost all $1\leq a\leq W'$. By Lemma~\ref{le_pseudoinshort}, we thus want that $W\mid \frac{W'}{(W',a)}$ for almost all $1\leq a\leq W$. This property fails if $W'=W$ but holds if $W'=W^{w}$ with $w\to \infty$.}. By Lemma~\ref{le_wtrick}, we have
\begin{align}\label{eqqn10}
\left\|h\right\|_{U^s(X,X+H]}&\leq \frac{1}{\widetilde{W}}\sum_{1\leq a\leq \widetilde{W}}\left\|h_{\widetilde{W},a}\right\|_{U^s(X/\widetilde{W},(X+H)/\widetilde{W}]}+O(H^{-1/2})\nonumber\\
\begin{split}
&=\frac{1}{\widetilde{W}}\sum_{\substack{1\leq a\leq \widetilde{W}\\(a,\widetilde{W})\mid W^{w-1}}}\left\|h_{\widetilde{W},a}\right\|_{U^s(X/\widetilde{W},(X+H)/\widetilde{W}]}\\
&\quad \quad +\frac{1}{\widetilde{W}}\sum_{\substack{1\leq a\leq \widetilde{W}\\(a,\widetilde{W})\nmid W^{w-1}}}\left\|h_{\widetilde{W},a}\right\|_{U^s(X/\widetilde{W},(X+H)/\widetilde{W}]}
+O(H^{-1/2}).
\end{split}
\end{align}
The number of terms in the last sum is
\begin{align*}
\ll \sum_{p\leq w}\frac{\widetilde{W}}{p^w} \ll \frac{\widetilde{W}}{2^w},    
\end{align*}
so by Lemma~\ref{le_upperbound} the contribution of this sum is $\ll 2^{-w/2}$, say.  The first sum over $a$ in~\eqref{eqqn10} can further be written as
\begin{align}\label{eq:dW}
\sum_{\ell \mid W^{w-1}}d_k(\ell)\sum_{\substack{1\leq a\leq \widetilde{W}\\(a,\widetilde{W})=\ell}}\left\|\frac{h_{\widetilde{W},a}}{d_k(\ell)}\right\|_{U^s(X/\widetilde{W},(X+H)/\widetilde{W}]}    
\end{align}
Since $d_k^{\sharp}(m)\ll d_k(m)$, for $(a,\widetilde{W})=\ell$ we have
\begin{align*}
\left(\frac{W}{\varphi(W)}\right)^{k-1}\frac{h_{\widetilde{W},a}(n)}{d_k(\ell)}&\ll \left(\frac{W}{\varphi(W)}\right)^{k-1} (\log X)^{1-k} \frac{d_k(\widetilde{W}n+a)}{d_k(\ell)}\\
&=\left(\frac{W}{\varphi(W)}\right)^{k-1}(\log X)^{1-k} d_k\left(\frac{\widetilde{W}}{\ell}n+\frac{a}{\ell}\right),    
\end{align*}
and since $W\mid \frac{\widetilde{W}}{\ell}$, by  Lemma~\ref{le_pseudoinshort} and Mertens's theorem this function is pseudorandomly majorized by a $(D,o_{X\to \infty}(1))$-pseudorandom function at location $0$ and scale $H/\widetilde{W}$.  This combined with~\eqref{eqqn11} (with $\widetilde{W}/\ell$ in place of $W'$) and Proposition~\ref{prop_inverse} yields
\begin{align}\label{eq:hW}
\left\|\frac{h_{\widetilde{W},a}}{d_k(\ell)}\right\|_{U^D(X/\widetilde{W},(X+H)/\widetilde{W}]}=o_{w\to \infty}\left(\left(\frac{\varphi(W)}{W}\right)^{k-1}\right),    \end{align}
uniformly in $1\leq a\leq \widetilde{W}$ with $(\widetilde{W},a)=\ell$.  

Now the bound~\eqref{dk-unif} follows from~\eqref{eq:dW},~\eqref{eq:hW}, and the estimate
\begin{align*}
\sum_{\ell\mid W^{w-1}}d_k(\ell)\sum_{\substack{1\leq a\leq \widetilde{W}\\(a,\widetilde{W})=\ell}}\left(\frac{\varphi(W)}{W}\right)^{k-1}&\ll  \sum_{\ell\mid W^{w-1}}d_k(\ell)\frac{\widetilde{W}}{\ell}\left(\frac{\varphi(W)}{W}\right)^{k}\\
&\ll \widetilde{W}\prod_{p\mid w}\left(1+\frac{k}{p}+O\left(\frac{1}{p^2}\right)\right)\left(\frac{\varphi(W)}{W}\right)^{k}\ll \widetilde{W}.
\end{align*}

(iii) This case follows directly from the inverse theorem (Proposition~\ref{prop_inverse} with $\nu=1$) and Theorem~\ref{discorrelation-thm}(iv). 
\end{proof}

\section{Applications}\label{sec:apps}

In this section, we shall prove the applications stated in Section~\ref{sec:intro}.

\begin{proof}[Proof of Corollary~\ref{cor:polynomial}]
Parts (i) and (iii) follow immediately from Theorem~\ref{discorrelation-thm}, as polynomial phases are special cases of nilsequences. By Theorem~\ref{discorrelation-thm} and the triangle inequality, the proof of part (ii) reduces to proving that
\begin{align*}
\left|\sum_{X < n\leq X+H}\Lambda^{\sharp}(n)e(P(n))\right|\gg \frac{H}{(\log X)^{A}}    
\end{align*}
implies~\eqref{erg7}. Recalling from~\eqref{eq:lambdasharp2} that $\Lambda^{\sharp}(n)=\Lambda^{\sharp}_I(n)+E(n)$, where $\Lambda^{\sharp}_I$ is a $((\log X)^{O(1)},X^{\varepsilon})$ type $I$ sum and $\sum_{X<n\leq X+H}|E(n)|\ll_A H\log^{-A}X$, the claim follows from  the type $I$ estimate in~\cite[Proposition 2.1]{matomaki-shao}.
\end{proof}

\begin{proof}[Proof of Theorem~\ref{thm_ergodic}]
First note that, since $\log p=(1+o(1))\log N$ for $p\in (N,N+N^{\kappa}]$ and since the contribution of higher prime powers is negligible, we have
\begin{align}\label{erg6}
\mathbb{E}_{N < p\leq N+N^{\kappa}} f_1(T^{h_1p}x)\cdots f_k(T^{h_kp}x)   =\mathbb{E}_{N < n\leq N+N^{\kappa}} \Lambda(n)f_1(T^{h_1n}x)\cdots f_k(T^{h_kn}x)+o_{N\to \infty}(1).   
\end{align}    
Hence, it suffices to show that the right-hand side of~\eqref{erg6} converges in $L^2(\mu)$.

Let $w$ be a large parameter (which we will eventually send to infinity), and let $W=\prod_{p\leq w}p$. Let
\begin{align*}
\epsilon(n):=\Lambda(n)-\Lambda_{w}(n);   
\end{align*}
this is a function that has small Gowers norms over short intervals by Theorem~\ref{thm_gowers}.

We first claim that
\begin{align}\label{erg1}
\int_{X}\left|\mathbb{E}_{N < n\leq N+N^{\kappa}}\epsilon(n)f_1(T^{h_1n}x)\cdots f_k(T^{h_kn}x)\right|^2   \, d \mu(x) =o_{w\to \infty}(1).   
\end{align}
Since the average over $n$ in~\eqref{erg1} is bounded, it is enough to show for all bounded $f_0:X\to \mathbb{C}$ that
\begin{align}\label{erg3}
\int_{X}\mathbb{E}_{N < n\leq N+N^{\kappa}}\epsilon(n)f_0(x)f_1(T^{h_1n}x)\cdots f_k(T^{h_kn}x)   \, d \mu(x) =o_{w\to \infty}(\|f_0\|_{L^2(\mu)}).   
\end{align}

To prove this, we first make the changes of variables $n'=n+N$, $x=T^my$, with $m$ arbitrary, and use the $T$-invariance of $\mu$ to rewrite the left-hand side of~\eqref{erg3}  as
\begin{align}\label{erg2}
\int_{X}\mathbb{E}_{m\leq N^{\kappa}}\mathbb{E}_{n'\leq N^{\kappa}}\epsilon_N(n')f_0(T^m y)f(T^{m+h_1n'}T^{h_1N}y)\cdots f(T^{m+h_kn'}T^{h_kN}y)   \, d \mu(y),   
\end{align}
where $\epsilon_{N}(n'):=\epsilon(n'+N)$.
Since $f_i:X\to \mathbb{C}$ are bounded, we can appeal to the generalized von Neumann theorem in the form of~\cite[Lemma 2]{fhk} (after embedding $[N^{\kappa}]$ to $\mathbb{Z}/M\mathbb{Z}$ for some $M\ll N^{\kappa}$) to bound~\eqref{erg2} as
\begin{align*}
\ll \|\epsilon_N\|_{U^k([N^{\kappa}])} \|f_0\|_{L^2(\mu)}=o_{w\to \infty}(\|f_0\|_{L^2(\mu)}), 
\end{align*}
where for the second estimate we used Theorem~\ref{thm_gowers}. Now~\eqref{erg1} has been proved. Then let $w'>w$. By an argument identical to the proof of~\eqref{erg1}, but using in the end the  fact that $\|\Lambda_{w}-\Lambda_{w'}\|_{U^k[N,N+N^{\kappa}]}=o_{w\to \infty}(1)$ (which follows from Theorem~\ref{thm_gowers} and the triangle inequality, but could also be proved more directly), we see that also
\begin{align}\label{erg5}
 \int_{X}\left|\mathbb{E}_{N < n\leq N+N^{\kappa}}(\Lambda_{w}(n)-\Lambda_{w'}(n))f_1(T^{h_1n}x)\cdots f_k(T^{h_kn}x)\right|^2   \, d \mu(x) =o_{w\to \infty}(1).     
\end{align}

Consider now
\begin{align*}
\mathbb{E}_{N < n\leq N+N^{\kappa}}\Lambda_w(n)f_1(T^{h_1n}x)\cdots f_k(T^{h_kn}x).    
\end{align*}
This can be rewritten as
\begin{align*}
\frac{W}{\varphi(W)}\sum_{\substack{1\leq b\leq W\\(b,W)=1}}\mathbb{E}_{N/W< n\leq (N+N^{\kappa})/W}f_1(T^{h_1(Wn+b)}x)\cdots f_k(T^{h_k(Wn+b)}x)+o_{N\to \infty}(1).
\end{align*}
Since the sequence $((N/W,(N+N^{\kappa})/W])_N$ of intervals are translates of a F{\o}lner sequence, from~\cite[Theorem 1.1]{austin} it follows that there exists $\phi_{w, b} \colon X \to \mathbb{C}$ such that
\begin{align*}
\int_X \left|\mathbb{E}_{N/W < n\leq (N+N^{\kappa})/W}f_1(T^{h_1(Wn+b)}x)\cdots f_k(T^{h_k(Wn+b)}x) -\phi_{w,b}(x)\right|^2 d\mu(x) = o_{N \to \infty, w}(1).
\end{align*}
Hence there exists also $\phi_{w} \colon X \to \mathbb{C}$ such that
\begin{align}
\label{eq:L2convphiw}
\int_X \left|\mathbb{E}_{N < n\leq N+N^{\kappa}}\Lambda_w(n)f_1(T^{h_1n}x)\cdots f_k(T^{h_kn}x) - \phi_w(x)\right|^2 d\mu(x) = o_{N \to \infty, w}(1).    
\end{align}
By~\eqref{erg5}, for $w'>w$ we have
\begin{align*}
\|\phi_{w}-\phi_{w'}\|_{L^2(\mu)}=o_{w\to \infty}(1),    
\end{align*}
so the sequence $(\phi_w)_w$ is Cauchy in $L^2(\mu)$. Let $\phi\in L^2(\mu)$ be its limit. Then, denoting
\begin{align*}
F(x)= \mathbb{E}_{N < n\leq N+N^{\kappa}} \Lambda(n)f_1(T^{h_1n}x)\cdots f_k(T^{h_kn}x),   
\end{align*}
from the triangle inequality,~\eqref{erg1} and~\eqref{eq:L2convphiw}, we have
\begin{align*}
\|F -\phi\|_{L^{2}(\mu)}&=\|\phi_w -\phi\|_{L^{2}(\mu)}+o_{w\to \infty}(1)+o_{N\to \infty;w}(1)\\
&=o_{w\to \infty}(1)+o_{N\to \infty;w}(1).
\end{align*}
By sending $N, w\to \infty$ with $w$ tending to $\infty$ slowly enough, and recalling~\eqref{erg6}, this proves the claim of Theorem~\ref{thm_ergodic}, with the limit being $\phi$.
\end{proof}

For proving Theorem~\ref{thm_lineq}, we need the generalized von Neumann theorem, so we state here a version of it that is suitable for us.

\begin{lemma}[Generalized von Neumann theorem]\label{le_gvnt} Let
Let $s,d,t,L\geq 1$ be fixed, and let $D$ be large enough in terms of $s,d,t,L$. Let $\nu$ be $(D,o_{N\to \infty}(1))$-pseudorandom at location $0$ and scale $N$, and let $f_1,\ldots, f_t:\mathbb{Z}\to \mathbb{R}$ satisfy $|f_i(x)|\leq \nu(x)$ for all $i\in [t]$ and $x\in [N]$. Let $\Psi=(\psi_1,\ldots, \psi_t)$ be a system of affine-linear forms with integer coefficients in $s$-normal form such that all the linear coefficients of $\psi_i$ are bounded by $L$ in modulus and $|\psi_i(0)|\leq DN$. Let $K\subset [-N,N]^d$ be a convex body with $\Psi(K)\subset (0,N]^d$. Suppose that for some $\delta>0$ we have
\begin{align*}
\min_{1\leq i\leq t}\|f_i\|_{U^{s+1}[N]} \leq \delta.   
\end{align*}
Then we have
\begin{align*}
\sum_{\mathbf{n}\in K}\prod_{i=1}^tf_i(\psi_i(\mathbf{n}))=o_{\delta\to 0}(N^d).    
\end{align*}
\end{lemma}

\begin{proof}
Note that by  Lemma~\ref{le_pseudo2} there exists a prime $N'\ll N$ such that we have a majorant for $f_i$ on the cyclic group $\mathbb{Z}/N'\mathbb{Z}$ satisfying the $(D,D,D)$-linear forms condition of~\cite[Definition 6.2]{green-tao}. Then the claim follows from~\cite[Proposition 7.1]{green-tao},  observing that its proof only used the $(D,D,D)$-linear forms condition of~\cite[Definition 6.2]{green-tao} and not the correlation condition.
\end{proof}

\begin{proof}[Proof of Theorem~\ref{thm_lineq}]
 Let $w$ be a sufficiently slowly growing function of $X$, and let $W=\prod_{p\leq w}p$. Let $\mathbf{N}=(X,\ldots, X)\in \mathbb{R}^d$. We can write $K=\mathbf{N}+K'$, where $K'\subset (0,H]^d$ is a convex body. Now the sum~\eqref{erg10} becomes 
\begin{align}\label{erg8}
\sum_{\mathbf{n}\in K'\cap \mathbb{Z}^d}\prod_{i=1}^t \Lambda(\psi_i( \mathbf{n})+\dot{\psi_i}\cdot \mathbf{N}).
\end{align}
Writing $\Lambda=\Lambda_w+(\Lambda-\Lambda_w)$, this splits as the main term
\begin{align*}
\sum_{\mathbf{n}\in K'\cap \mathbb{Z}^d}\prod_{i=1}^t \Lambda_w(\psi_i( \mathbf{n})+\dot{\psi_i}\cdot \mathbf{N})    
\end{align*}
and $2^t-1$ error terms
\begin{align}\label{eq_psii}
\sum_{\mathbf{n}\in K'\cap \mathbb{Z}^d}\prod_{i=1}^t \Lambda_i(\psi_i( \mathbf{n})+\dot{\psi_i}\cdot \mathbf{N})
\end{align}
where $\Lambda_{i}\in \{\Lambda_w,\Lambda-\Lambda_w\}$ and at least one $\Lambda_i$ equals to $\Lambda-\Lambda_w$. Following~\cite[Section 5]{green-tao} verbatim, we see that the main term is
\begin{align*}
\textnormal{vol}(K\cap \Psi^{-1}(\mathbb{R}_{>0}^t))\prod_{p}\beta_p+o_{X\to \infty}(H^d).    
\end{align*}
Following~\cite[Section 4]{green-tao}, we may assume that the system of linear forms involved in~\eqref{eq_psii} is in $s$-normal form for some $s\ll_D 1$. 

We make the change of variables $\mathbf{n}=W\mathbf{m}+\mathbf{b}$ with $\mathbf{b}\in [0,W)^d$ in~\eqref{eq_psii} and abbreviate $M_{\mathbf{b},i}:=\dot{\psi_i}\cdot \mathbf{b}+\psi_i(0)$ to rewrite that sum as
\begin{align}\label{eq_psii2}\begin{split}
&\sum_{\mathbf{b}\in [0,W)^d} \sum_{\substack{\mathbf{m}\in \mathbb{Z}^d\\W\mathbf{m}+\mathbf{b}\in K'}}\prod_{i=1}^t \Lambda_i(\psi_i(W \mathbf{m}+\mathbf{b})+\dot{\psi_i}\cdot \mathbf{N})\\
&=\sum_{\mathbf{b}\in [0,W)^d} \sum_{\substack{\mathbf{m}\in \mathbb{Z}^d\\W\mathbf{m}+\mathbf{b}\in K'}}\prod_{i=1}^t \Lambda_i(W\dot{\psi_i}\cdot \mathbf{m}+\dot{\psi_i}\cdot \mathbf{b}+\psi_i(0)).\\
&=\left(\frac{W}{\varphi(W)}\right)^{t}\sum_{\substack{\mathbf{b}\in [0,W)^d\\(M_{\mathbf{b},i},W)=1\,\forall i\leq t}} \sum_{\substack{\mathbf{m}\in \mathbb{Z}^d\\\mathbf{m}\in (K'-\mathbf{b})/W}}\prod_{\substack{1\leq i\leq t\\\Lambda_i=\Lambda-\Lambda_w}} \left(\frac{\varphi(W)}{W}\Lambda(W\dot{\psi_i}\cdot \mathbf{m}+M_{\mathbf{b},i})-1\right)\\
&+o_{X\to \infty}(H^d),
\end{split}
\end{align}
where the error term comes from the contribution of integers in the support of $\Lambda$ that are not $w$-rough.

 By Theorem~\ref{thm_gowers}(i), uniformly for integers $1\leq M\leq X$ with $(M,W)=1$ we have
\begin{align*}
\max_{\substack{1\leq a\leq W\\(a,W)=1}}\left\|\frac{\varphi(W)}{W}\Lambda(W\cdot+M)-1\right\|_{U^{s+1}[0,H/W]}=o_{X\to \infty;s}(1). 
\end{align*}
Moreover, by Lemma~\ref{le_pseudoinshort} the function $\frac{\varphi(W)}{W}\Lambda(W\cdot+M)-1$ is majorized by a $(D,o_{X\to \infty}(1))$-pseudorandom measure $\nu_{M}$ at location $0$ and scale $H/W$ for any fixed $D\geq 1$. Hence, applying the generalized von Neumann theorem (Lemma~\ref{le_gvnt}, with $\nu=\frac{1}{t}\sum_{i\leq t}\nu_{M_{\mathbf{b},i}}$), we conclude that~\eqref{eq_psii2} is
\begin{align*}
\ll \left(\frac{W}{\varphi(W)}\right)^t\cdot W^d\left(\frac{\varphi(W)}{W}\right)^t \cdot o_{X\to \infty}\left(\left(\frac{H}{W}\right)^d\right)=o_{X\to \infty}(H^d),\end{align*}
completing the proof. 
\end{proof}

\begin{proof}[Proof of Corollary~\ref{cor_lineq}] This follows directly from Theorem~\ref{thm_lineq}, since the assumptions imply that $\beta_p>0$ for all $p$, and on the other hand $\beta_p=1+O_{t,d,L}(1/p^2)$ by~\cite[Lemmas 1.3 and 1.6]{green-tao}, so we have $\prod_{p}\beta_p>0$.
\end{proof}

\appendix

\section{Variants of the main result}\label{variants-app}

In this appendix discuss in more detail the variants of the main results described in Remark~\ref{variants-rem}.

\subsection{Results for the Liouville function}

It is an easy matter to replace the M\"obius function $\mu$ by the Liouville function $\lambda$ in our main results:

\begin{proposition}\label{liouville}  The results in Theorem~\ref{discorrelation-thm}(i), (iv) (and hence also Corollary~\ref{cor:polynomial}(i), (iv)) continue to hold if $\mu$ is replaced by $\lambda$.
\end{proposition}

\begin{proof}  We illustrate the argument for the estimate~\eqref{mobius-discor}, as the other estimates are proven similarly.  Under the hypotheses of Theorem~\ref{discorrelation-thm}(i), we wish to show that
$$ \sup_{g \in \Poly(\Z \to G)} \left| \sum_{X < n \leq X+H} \lambda(n) \overline{F}(g(n)\Gamma) \right|^* \ll_{A,\eps,d,D} \delta^{-O_{d,D}(1)} H \log^{-A} X.$$
Writing $\lambda(n) = \sum_{m \leq \sqrt{2X}: m^2|n} \mu(n/m^2)$ for $n \leq 2X$ and using the triangle inequality, we can bound the left-hand side by
$$ \sum_{m \leq \sqrt{2X}} \sup_{g \in \Poly(\Z \to G)} \left| \sum_{X/m^2 < n \leq X/m^2+H/m^2} \mu(n) \overline{F}(g(m^2 n)\Gamma) \right|^*.$$
If $m \leq X^{\varepsilon/10}$ (say), then by Theorem~\ref{discorrelation-thm}(i) (with $X, H, g$ replaced by $X/m^2$, $H/m^2$, $g(m^2 \cdot)$, and $\varepsilon$ reduced slightly) we have
$$ \sup_{g \in \Poly(\Z \to G)} \left| \sum_{X/m^2 < n \leq X/m^2+H/m^2} \mu(n) \overline{F}(g(m^2 n)\Gamma) \right|^*
\ll_{A,\eps,d,D} m^{-2} \delta^{-O_{d,D}(1)} H \log^{-A} X.$$
For $X^{\varepsilon/10} < m \ll \sqrt{X}$, we simply use the triangle inequality and the trivial bound $|\overline{F}(g(n)\Gamma)| \leq 1/\delta$ to conclude
$$ \sup_{g \in \Poly(\Z \to G)} \left| \sum_{X/m^2 < n \leq X/m^2+H/m^2} \mu(n) \overline{F}(g(m^2 n)\Gamma) \right|^* \ll \frac{1}{\delta} \left( \frac{H}{m^2} + 1 \right).$$
Summing in $m$, we obtain the claim after a brief calculation (since $H$ is significantly larger than $X^{1/2}$).
\end{proof}

\subsection{Results for the indicator function of the primes}

It is also easy to replace the von Mangoldt function $\Lambda$ with the indicator function $1_{\mathcal P}$ of the primes ${\mathcal P}$:

\begin{proposition} The results in Theorem~\ref{discorrelation-thm}(ii) (and hence also Corollary~\ref{cor:polynomial}(ii)) continue to hold if $\Lambda$ is replaced by $1_{\mathcal P}$, and $\Lambda^\sharp(n)$ is replaced by $\frac{1}{\log n} \Lambda^\sharp(n)$.
\end{proposition}

\begin{proof}  From~\eqref{mangoldt-discor} and Lemma~\ref{basic-prop}(iii) we have
$$ \sup_{g \in \Poly(\Z \to G)} \left| \sum_{X < n \leq X+H} \left(\frac{1}{\log n} \Lambda(n) - \frac{1}{\log n} \Lambda^\sharp(n)\right) \overline{F}(g(n)\Gamma) \right|^* \ll_{A,\eps,d,D} \delta^{-O_{d,D}(1)} H \log^{-A} X$$
and so by the triangle inequality it will suffice to show that
$$ \sum_{X < n \leq X+H} \left| 1_{\mathcal P}(n) - \frac{1}{\log n} \Lambda(n)\right| \ll_{A} H \log^{-A} X.$$
But the summand is supported on prime powers $p^j$ with $2 \leq j \ll \log X$ and $p \ll \sqrt{X}$, so there are at most $O( \sqrt{X} \log X )$ terms, each of which gives a contribution of $O(1)$.  Since $H$ is significantly larger than $X^{1/2}$, the claim follows.
\end{proof}

\subsection{Results for the counting function of sums of two squares}

It is a classical fact that the counting function
$$ r_2(n) \coloneqq \sum_{\substack{a,b \in \Z\\ a^2+b^2=n}} 1$$
can be factorized as $r_2(n) = 4 (1 * \chi_4)(n)$, where $\chi_4$ is the non-principal Dirichlet character of modulus $4$.  This is formally very similar to the divisor function $d_2(n) = (1*1)(n)$.  In this paper we use the Dirichlet hyperbola method to expand $d_2(n)$ for $X < n \leq X+H$ as
$$ d_2(n) = \sum_{\substack{R_2 \leq n_1 \leq n/R_2\\ n_1\mid n}} 1 + \sum_{\substack{n_1 < R_2\\ n_1\mid n}} 2 $$ 
with $R_2 \coloneqq X^{1/20}$, and approximate this function by the Type I sum
$$ d_2^\sharp(n) = \sum_{\substack{R_2 \leq n_1 < R_2^2\\ n_1\mid n}} \frac{\log n - \log R_2^2}{\log R_2} + \sum_{\substack{n_1 < R_2\\ n_1\mid n}} 2$$
(these are the $k=2$ cases of~\eqref{eq:genhypid},~\eqref{dks-def} respectively).  In a similar vein, we can expand
$$ r_2(n) = \sum_{\substack{R_2 \leq n_1 \leq n/R_2\\ n_1\mid n}} 4 \chi_4(n_1) + \sum_{\substack{n_1 < R_2\\ n_1\mid n}} 4(\chi_4(n_1) + \chi_4(n/n_1))$$
and then introduce the twisted Type I approximant
$$ r_2^\sharp(n) = \sum_{\substack{R_2 \leq n_1 < R_2^2\\ n_1\mid n}} 4 \chi_4(n_1) \frac{\log n - \log R_2^2}{\log R_2} + \sum_{\substack{n_1 < R_2\\ n_1\mid n}} 4(\chi_4(n_1) + \chi_4(n/n_1)).$$
We then have

\begin{proposition}  The $k=2$ results in Theorem~\ref{discorrelation-thm}(iii) continue to hold if $d_2, d_2^\sharp$ are replaced by $r_2$, $r_2^\sharp$ respectively.
\end{proposition}

This proposition is established by repeating the arguments used to establish Theorem~\ref{discorrelation-thm}(iii), but by inserting ``twists'' by the character $\chi_4$ at various junctures.  However, such twists are quite harmless (for instance, since $\|\chi_4\|_{\TV(P;4)} \ll 1$ for any arithmetic progression $P$, Proposition~\ref{basic-prop}(iii) allows one to insert this character into maximal sum estimates without difficulty), and there is no difficulty in modifying the arguments to accommodate this twist.

\subsection{Potential result for the indicator function of the sums of two squares}

Let $S = \{ n^2+m^2: n,m \in \Z\}$ be the set of numbers representable as sums of two squares.  The Dirichlet series for $S$ is equal to $\zeta(s)^{1/2} L(s,\chi_4)^{1/2}$ times a holomorphic function near $s=1$, and in particular extends into the classical zero-free region after making a branch cut to the left of $s=1$ on the real axis.  

By a standard Perron formula calculation, one can then obtain asymptotics of the form
$$ \sum_{n \leq x} 1_S(n) = x \sum_{j=0}^{A-1} B_{j} \log^{-j-1/2} x + O_A( x \log^{-A-1/2} x )$$
for any $A>0$ and some real constants $B_j$ which are in principle explicitly computable; see for instance~\cite[Theorem 1.1]{breteche-tenenbaum} for a recent treatment (in significantly greater generality) using the Selberg--Delange method.  Similar calculations give asymptotics of the form
$$ \sum_{\substack{n \leq x\\ n = a\ (q)}} 1_S(n) = x \sum_{j=0}^{A-1} B_{j,a,q} \log^{-j-1/2} x + O_A( x \log^{-A-1/2} x )$$
for any fixed residue class $a\ (q)$ and some further real constants $B_{j,a,q}$.  With further effort one can also localize such estimates to intervals $\{ X < n \leq X+H \}$ with $H$ not too small (e.g., $H = X^{5/8+\eps}$ or $H = X^{7/12+\eps}$). 

This suggests the existence of an approximant $1_S^{\sharp,A}$ for any given accuracy $A > 0$ that is well approximated by Type I sums, and is such that one has the major arc estimate
$$ \left|\sum_{X < n \leq X+H} 1_S(n) - 1_S^{\sharp,A}(n) \right|^* \ll_A H \log^{-A} x$$
(cf. Theorem~\ref{thm:major-arc}). For small $A$, it seems likely that one could construct $1_S^{\sharp,A}$ by a variant of the Cram\'er--Granville construction used to form $\Lambda^\sharp$; but for large $A$ it appears that the approximant is more difficult to construct (for instance one may have to use Fourier-analytic methods such as the delta method).  However, once such an approximant is constructed, we conjecture that the methods of this paper will produce  analogues of Theorem~\ref{discorrelation-thm}(ii) (and hence also of Corollary~\ref{cor:polynomial}(ii)) if $\Lambda, \Lambda^\sharp$ are replaced by $1_S, 1_S^{\sharp,A'}$ respectively, with $A'$ sufficiently large depending on $A$.  The main point is that a satisfactory analogue of the Heath--Brown decompositions in Lemma~\ref{hb-identity} for $1_{S}$ is known; see~\cite[Lemma 7.2]{shao-teravainen}.  

We do not foresee any significant technical issues with the remaining portions of the argument, though of course one would need to define the approximant $1_S^{\sharp, A}$ more precisely before one could say with certainty that the portions of the argument involving this approximant continue to be valid.

\subsection{Potential result for the indicator function of smooth numbers}

Let $0 < \eta < \frac{1}{2}$, let $X$ be large, and let $S_\eta$ denote the set of $X^\eta$-smooth integers, that is to say those numbers whose prime factors are all less than $X^\eta$.  Let $H \geq X^{\theta+\eps}$ with $\theta \coloneqq \frac{1}{2}+\eta$.  As is well known, the density of $S_\eta$ in $[X,X+H]$ is asymptotic to the Dickman function $\rho(1/\eta)$ evaluated at $1/\eta$.  We conjecture that the methods of this paper can be used to establish a bound of the form
$$ \sup_{g \in \Poly(\Z \to G)} \left| \sum_{X < n \leq X+H} (1_{S_\eta}(n) - \rho(1/\eta)) \overline{F}(g(n)\Gamma) \right|^* \ll_{\eps,d,D,\eta} \delta^{-O_{d,D}(1)} H \log^{-c} X$$
for some absolute constant $c>0$ under the hypotheses of Theorem~\ref{discorrelation-thm}. 

Indeed, a Heath--Brown type decomposition, involving only $(1, x^{1/2-\eta}, x^{1/2})$ type II sums and a (somewhat) small exceptional set, was constructed in~\cite[Lemma 11.5]{kmt}; the exceptional set was only shown to be small on long intervals such as $[1,X]$ in that paper, but it is likely that one can show the set to also be small on the shorter interval $\{ X < n \leq X+H\}$. 

There are however some further technical difficulties in implementing our methods here.  The first (and less serious) issue is that one would need to verify that the type II sums $f(n)$ produced by~\cite[Lemma 11.5]{kmt} obey the bound 
\eqref{eq-comb-lambda}; we believe that this is likely to be achievable after some computation.  The second and more significant difficulty is that one would need an approximant $1_{S_\eta}^\sharp$ obeying a major arc estimate of the shape
$$ \left|\sum_{X < n \leq X+H} 1_{S_\eta}(n) - 1^\sharp_{S_\eta}(n) \right|^* \ll_A H \log^{-A} X$$
for any $A>0$ (possibly after removing a small exceptional set from $S_\eta$), in the spirit of Theorem~\ref{thm:major-arc} and Corollary~\ref{cor:major-arc-Ramare}. 

The constant $\rho(1/\eta)$ is an obvious candidate for such an approximant, but unfortunately such an estimate is only valid for small values of $A$; see~\cite[Theorem 1.8]{ht}. Thus, as in the previous discussion for the indicator of the sums of two squares, a more complicated approximant is likely to be required; the function $\Lambda(x,y)$ appearing in~\cite[Theorem 1.8]{ht} will most likely become involved. See also~\cite{matthiesen-wang} for some recent estimates on the distribution of smooth numbers in short intervals or arithmetic progressions (in a slightly different regime in which the $X^\eta$ threshold for smoothness is replaced by a smaller quantity).

\bibliography{refs}
\bibliographystyle{plain}

\end{document}